%% file: thesis.tex
\documentclass[11pt,openright,twoside,a4paper,reqno]{amsbook}
\usepackage[a4paper,top=3cm,bottom=3cm,inner=4cm,outer=3cm]{geometry}
\usepackage{mathrsfs,amsmath,amsfonts,amsthm,amssymb,graphicx,pdfpages,tikz}
\usepackage{setspace,fancyhdr,calc,verbatim,bbm,enumerate}
\usepackage[all,2cell]{xy}\UseAllTwocells\SilentMatrices

\newcommand{\ovl}{\overline}
\newcommand{\ca}{\mathcal}
\newcommand{\set}{\mathbf{Set}}
\newcommand{\pt}[1]{(#1_0,#1_1,#1_2)}

\newcommand{\Sum}{\sum\limits_{i=0}^2 }

\newtheorem*{thrm*}{Theorem}
\theoremstyle{remark}
\newtheorem*{rmk*}{Remark}
\newtheorem*{lem*}{Lemma}
\theoremstyle{definition}
\newtheorem*{dfn*}{Definition}
\newtheorem*{cor*}{Corollary}
\theoremstyle{definition}
\newtheorem*{egs*}{Examples}
\newtheorem{prop*}{Proposition}

\theoremstyle{plain}
\newtheorem{thrm}{Theorem}[section]
\theoremstyle{plain}
\newtheorem{prop}[thrm]{Proposition}
\theoremstyle{remark}
\newtheorem{rmk}[thrm]{Remark}
\theoremstyle{plain}
\newtheorem{lem}[thrm]{Lemma}
\theoremstyle{plain}
\newtheorem{cor}[thrm]{Corollary}
\theoremstyle{definition}
\newtheorem{defn}[thrm]{Definition}
\theoremstyle{definition}
\newtheorem{egs}[thrm]{Examples}

\numberwithin{equation}{chapter}
\numberwithin{section}{chapter}

\makeatletter
\@addtoreset{thrm}{chapter}
\makeatother

\onehalfspace

\makeatletter
\renewcommand{\l@chapter}{\@tocline{0}{12pt}{0pt}{}{\bfseries}}
\renewcommand{\l@subsection}{\@tocline{2}{0pt}{2pc}{5pc}{}}
\makeatother

\begin{document}

\mathcode`\:="603A

\frontmatter

\thispagestyle{empty}

\include{coverpage}

\cleardoublepage

\pagestyle{fancy}
\fancyhf{} 
\renewcommand{\headrulewidth}{0pt} 
\pagenumbering{roman}
\fancyfoot[C]{{\footnotesize\thepage}}

\include{declaration}

\cleardoublepage
\include{abstract}
\cleardoublepage
\include{acknowledgements}
\cleardoublepage
\tableofcontents
\cleardoublepage

\mainmatter

\pagestyle{headings}
\pagenumbering{arabic}

\chapter{Introduction} \label{ch:intro}
\input{intro}
\cleardoublepage

\part{Constructive geometry}

\chapter{Projective planes}  \label{chaprojplanes}

In this chapter we approach projective planes from an analytic and a synthetic point of view. We construct the projective plane $\mathbb P(R)$ over a given local ring $R$ and demonstrate a few propositions satisfied by this structure. We present the coherent theory of preprojective planes whose axioms are satisfied by projective planes over local rings. This is followed by results on morphisms of preprojective planes and morphisms between projective planes over local rings. We present Desargues' and Pappus' axioms and show that they are satisfied by projective planes over local rings. The coherent theory of projective planes is then given as the theory of preprojective planes with the addition of Desargues' and Pappus' axiom.

\section{Points and lines}

\begin{defn}
Given a ring $R$ we define the set of \emph{points} of the projective plane over $R$ to be 
$$\mathbb P_{\text{pt}}(R)=\{(a_0,a_1,a_2)\in R^3| \text{inv}(a_0)\vee \text{inv}(a_1)\vee \text{inv}(a_2)\}/\sim$$
where $(a_0,a_1,a_2)\sim (b_0,b_1,b_2)$ iff there exists $r\in R$ such that $b_0=r a_0$, $b_1=ra_1$ and $b_2=ra_2$.
\end{defn}

Note that $r$ is necessarily invertible because one of the $b_i$'s is invertible. Since $r$ must be invertible, $\sim$ is an equivalence relation.

\begin{defn}
We say that two points $A$ and $B$ are \emph{apart} from each other and we write
$$A\# B$$
when for some representatives $(a_0,a_1,a_2)$ of $A$ and $(b_0,b_1,b_2)$ of $B$ the determinant of one of the three minors of the matrix
$$\begin{pmatrix}
a_0 & b_0 \\
a_1 & b_1 \\
a_2 & b_2 \end{pmatrix}$$
is invertible.
\end{defn}

Let $A\# B$ and let $(a_0, a_1,a_2)$ and $(b_0,b_1,b_2)$ be representatives of $A$ and $B$ respectively such that the matrix described above has an invertible $2\times 2$ minor. Suppose that $\pt{a'}$ and $\pt{b'}$ are also representatives of $A$ and $B$ respectively. Then, $\mathbf{a'}=r \mathbf a$ and $\mathbf{b'}=s \mathbf b$ for some invertible $r$ and $s$. Without loss of generality, suppose that $a_0 b_1-a_1 b_0=\lambda$ is invertible. Then, $a'_0 b'_1-a'_1 b'_0=\lambda r s$ is also invertible. Hence, $A\# B$ iff for \emph{any} representatives $(a_0,a_1,a_2)$ of $A$ and $(b_0,b_1,b_2)$ of $B$ the determinant of one of the three minors of the matrix
$$\begin{pmatrix}
a_0 & b_0 \\
a_1 & b_1 \\
a_2 & b_2 
\end{pmatrix}$$
is invertible.


\begin{defn}
We define the set of \emph{lines} of the projective plane over a ring $R$ in the same way and denote it by $\mathbb P_{\text{li}}(R)$. We also define a $\#$ relation on the set of lines in the same way we did for the set of points and we also denote it by $\#$.
\end{defn}

We usually denote points by capital Latin letters and their representatives by the corresponding lower case letters. For example, $(a_0,a_1,a_2)$ will usually be a representation of the point $A$. Lines are usually denoted by the lower case letters $k$, $l$, $m$, $n$ and their representatives by the corresponding Greek letters. For example, $(\kappa_0,\kappa_1,\kappa_2)$ is usually a representation of the line $k$.

\section{Incidence}

\begin{defn}
Given a point $A$ and a line $l$, we say that $A$ lies on $l$, and we write
$$A\in l$$
when for some representations $(a_0,a_1,a_2)$ of $A$ and $(\lambda_0, \lambda_1, \lambda_2)$ of $l$, $\sum\limits_{i=0}^2  \lambda_i a_i=0$. 
\end{defn}

Note that if $\mathbf a$ and $\mathbf{a'}$ are in the same equivalence class (i.e. $\mathbf{a'}=r \mathbf a$ for some invertible $r$) and $\boldsymbol{\lambda}$ and $\boldsymbol{\lambda '}$ are in the same equivalence class (i.e. $\boldsymbol{\lambda '}=t \boldsymbol{\lambda}$ for some invertible $t$) then $\sum\limits_{i=0}^2  \lambda_i a_i=rt\sum\limits_{i=0}^2  \lambda'_i a'_i$. therefore $\sum\limits_{i=0}^2  \lambda_i a_i=0$ iff $\sum\limits_{i=0}^2  \lambda'_i a'_i=0$. Hence $A\in l$ iff for \emph{any} representations $(a_0,a_1,a_2)$ of $A$ and $(\lambda_0, \lambda_1, \lambda_2)$ of $l$, $\sum\limits_{i=0}^2  \lambda_i a_i=0$.

\begin{defn}
Given a point $A$ and a line $l$, we say that $A$ lies outside of $l$, and we write
$$A \notin l$$
when for some representations $(a_0,a_1,a_2)$ of $A$ and $(\lambda_0, \lambda_1, \lambda_2)$ of $l$, $\sum\limits_{i=0}^2  \lambda_i a_i$ is invertible. 
\end{defn}

For similar reasons as above, $A\notin l$ iff for \emph{any} representations $(a_0,a_1,a_2)$ of $A$ and $(\lambda_0, \lambda_1, \lambda_2)$ of $l$, $\sum\limits_{i=0}^2  \lambda_i a_i$ is invertible.

It is common practice in projective geometry to abuse notation and write $(a,b,c)$ for the equivalence class represented by $(a,b,c)$ for both points and lines. From now on, we shall also adopt this notation.

\begin{defn}
Given a ring $R$, the \emph{projective plane over} $R$, denoted by $\mathbb P(R)$ is the structure consisting of the sets $\mathbb P_{\text{pt}} (R)$, $\mathbb P_{\text{li}} (R)$, the two $\#$ relations, and the $\in$ and $\notin$ relations.
\end{defn}

Note that for $R$ a geometric field in $\set$ (i.e. a ring which satisfies $\text{inv}(0)\vdash \bot$ and $\top\vdash_x (x=0)\vee \text{inv}(x)$), the above construction gives the classical projective plane over the field $R$: $\in$ becomes the incidence relation, $\#$ becomes the inequality relation, and $\notin$ becomes the complement of $\in$.

\section{Duality}

Given a ring $R$, the set of points of its projective plane is isomorphic to its set of lines via the isomorphism which sends a point represented by $\pt{a}$ to the line represented by $\pt{a}$. This isomorphism and its inverse preserve the relation $\#$. Also, a point represented by $\pt{a}$ lies on/outside a line represented by $\pt{\lambda}$ iff the point represented by $\pt{\lambda}$ lies on/outside the line represented by $\pt{a}$. Hence, given a theorem that holds on the projective plane over a ring $R$, we automatically know that its \emph{dual theorem} also holds. The dual theorem is the theorem we acquire when in the statement of the theorem, we replace points with lines, lines with points, and reverse the order of incidence and non-incidence relations. So, for example in the statement of a theorem we would replace the phrase ``the point $A$ lies on the line $l$'' with the phrase ``the line $A$ passes through the point $l$''.

\section{A few propositions and remarks} \label{projprop}

Note that in this section, we do not make any assumptions about the ring $R$. In particular, the propositions we prove here also hold when $R$ is the zero ring. However, we begin this section with the following proposition.

\begin{prop} 
Given the projective plane over a ring $R$, the following are equivalent:
\begin{enumerate}
\item $R$ is a non-trivial ring (it satisfies $\text{inv}(0)\vdash \bot$).
\item For $A$ a point of the projective plane, $A\#A\vdash_{A} \bot$.
\item For $l$ a line of the projective plane, $l\#l\vdash_{l} \bot$.
\item For $A$ a point and $l$ a line of the projective plane, $A\in l \wedge A\notin l \vdash_{A,l} \bot$.
\end{enumerate}
\end{prop}

\begin{proof}
1 implies 2, 3 and 4 because $0$ is not invertible in a non-trivial ring.

Suppose $0=1$ in $R$. Then, every point of the projective plane is apart from itself. Therefore, if we assume 2 then $R$ is a non-trivial ring. Dually, if we assume 3 again $R$ is a non-trivial ring.

Also, if we suppose $0=1$ in $R$, then the point $(0,0,1)$ lies on and apart from the line $(0,0,1)$. Hence, if we assume 4 then $R$ is a non-trivial ring.
\end{proof}

\begin{prop} \label{propexistsuniqueline}
 On the projective plane over a ring $R$, let $A$ and $B$ be points such that $A\# B$. Then, there exists a unique line passing through both $A$ and $B$. Dually, for $k$ and $l$ lines on the projective plane such that $k\# l$, there exists a unique point lying on both $k$ and $l$.
\end{prop}

\begin{proof}
 Let $A=(a_0,a_1,a_2)$, $B=(b_0,b_1,b_2)$  be points  such that $A\#B$. Without loss of generality, we assume that $a_0 b_1 - a_1 b_0$ is invertible. Let $l$ be the line $(\lambda_0,\lambda_1,\lambda_2)$, where $\lambda_0=a_1 b_2 -a_2 b_1$, $\lambda_1= a_2 b_0 - a_0 b_2$ and $\lambda_2= a_0 b_1 - a_1 b_0$ and notice that $\lambda_2$ is invertible. For  $x_0$, $x_1$, $x_2$ in $R$, 
$\sum\limits_{i=0}^2  \lambda_i x_i= \det \begin{pmatrix}
x_0 & a_0 & b_0 \\
x_1 & a_1 & b_1 \\
x_2 & a_2 & b_2  
\end{pmatrix}$.
Hence, $\sum\limits_{i=0}^2 \lambda_i a_i=0$ and $\sum\limits_{i=0}^2  \lambda_i b_i= 0$, therefore $A,B\in l$.

Suppose $m$ represented by $(\mu_0,\mu_1,\mu_2)$ is a line such that $A,B\in m$, so that $\sum\limits_{i=0}^2 \mu_i a_i=0$ and $\sum\limits_{i=0}^2  \mu_i b_i= 0$. Let $\lambda'_i= \mu_2 \lambda_2^{-1} \lambda_i$ for $i=0,1,2$. Note that $\lambda'_2=\mu_2$, $\sum\limits_{i=0}^2 \lambda'_i a_i=0$ and $\sum\limits_{i=0}^2  \lambda'_i b_i= 0$, therefore we have the equations
\[
(\lambda'_0-\mu_0)a_0 + (\lambda'_1-\mu_1)a_1=\sum\limits_{i=0}^2 \lambda'_i a_i-\sum\limits_{i=0}^2 \mu_i a_i=0,
\]
\[
(\lambda'_0-\mu_0)b_0 + (\lambda'_1-\mu_1)b_1=\sum\limits_{i=0}^2 \lambda'_i b_i-\sum\limits_{i=0}^2 \mu_i b_i=0.
\]

Taking $b_1$ times the first equation minus $a_1$ times the second we see that $(\lambda'_0-\mu_0)(a_0 b_1-a_1 b_0)=0$. By assumption, $a_0 b_1-a_1 b_0$ is invertible, therefore $\mu_0=\lambda'_0$. By a symmetric argument, $\mu_1=\lambda'_1$ and we've already seen that $\mu_2=\lambda'_2$. Therefore, for each $i$, $\mu_i= \mu_2 \lambda_2^{-1} \lambda_i$, hence $\pt{\lambda}$ and $\pt{\mu}$ represent the same line, and therefore $l=m$.

Therefore, given points $A$, $B$ that are apart from each other, there is a unique line passing through both $A$ and $B$. By the duality principle, given lines $k$, $l$ that are apart from each other, there is a unique point lying on both $k$ and $l$.
\end{proof}

Given points $A$, $B$ of a projective plane over a ring such that $A\#B$, we denote the unique line through $A$ and $B$ by $\ovl{AB}$. Dually, given lines $k$ and $l$ of a projective plane over a ring such that $k\#l$, we denote the unique point lying on both $k$ and $l$ by $k\cap l$ and we call it the \emph{intersection} of $k$ and $l$.

\begin{prop}
 For any ring $R$, the following hold for points and lines of its projective plane:

\begin{enumerate}
\item For any line $l$ there exist points $A$, $B$, $C$ lying on $l$ such that $A\#B\#C\#A$.

\item $\top \vdash \exists A,B,C,l. A\#B \wedge A,B\in l\wedge C\notin l$.

\item $\top \vdash_l \exists A. A\notin l$.
\end{enumerate}

Dually, the following hold:
\begin{enumerate}
\item For any point $A$ there exist lines $k$, $l$, $m$ passing through $A$ such that $k\#l\#m\#k$.

\item $\top \vdash \exists k,l,m,A. k\#l \wedge A\in k,l \wedge A\notin m$.

\item $\top \vdash_A \exists l. A\notin l$.
\end{enumerate}
\end{prop}

\begin{proof}
We only prove the first three listed results, since by the duality principle the second list of results hold.
\begin{enumerate}
\item Let $l$ be the line represented by $(\lambda_0,\lambda_1,\lambda_2)$, and without loss of generality assume that $\lambda_0$ is invertible. Then, the points $(\lambda_2,0,-\lambda_0)$, $(\lambda_1,-\lambda_0,0)$, $(\lambda_1+\lambda_2,-\lambda_0,-\lambda_0)$ lie on $l$ and they are all apart from each other.

\item The points $(0,0,1)$ and $(1,0,1)$ are apart from each other and they lie on the line $(0,1,0)$ and the point $(0,1,1)$ lies outside the line $(0,1,0)$.

\item Let $l$ be the line represented by $\pt{\lambda}$ and without loss of generality assume that $\lambda_0$ is invertible. Then, $(1,0,0)\notin l$.
 \end{enumerate}
\end{proof}

\begin{prop} \label{proplastaxiom}
Let $A$ and $B$ be points and let $k$ and $l$ be lines of the projective plane over a ring $R$. Then, the following sequent holds:
$$A\# B \wedge l\#m \wedge A\in l \wedge B\in l \wedge B\in m  \vdash_{A,B,l,m} A\notin m.$$
\end{prop}

\begin{proof}
Let $A=\pt{a}$, $B=\pt{b}$ be two points and let $l=\pt{\lambda}$ and $m=\pt{\mu}$ be lines, such that $A\#B$, $l\#m$ and $A,B\in l$ and $B\in m$.
$l\#m$, so without loss of generality we assume that $\lambda_0\mu_1-\lambda_1\mu_0$ is invertible. One of the $\lambda_i$'s is invertible, therefore we consider each of the three cases.

In the case where $\lambda_0$ is invertible, and by the construction of the unique line passing through points that are apart from each other, $\lambda_0 = r(a_1b_2-a_2b_1)$ for some invertible element $r$, hence $(a_1b_2-a_2b_1)$ is invertible. Since $A,B\in l$ and $B\in m$ we have the following:
\[
\lambda_0 b_2 \Sum\mu_i a_i = \lambda_0 b_2 \Sum\mu_i a_i+\mu_0a_2\Sum\lambda_ib_i -\mu_0b_2\Sum\lambda_ia_i -\lambda_0a_2\Sum\mu_ib_i
\]
\[
=(\lambda_0\mu_1-\lambda_1\mu_0)(a_1b_2-a_2b_1).
\]
The right hand side is invertible by assumption, hence $\Sum\mu_i a_i$ is invertible, and therefore $A\notin m$.

The case where $\lambda_1$ is invertible is symmetric to the above case where $\lambda_0$ is invertible.

In the case where $\lambda_2$ is invertible, $a_0b_1 - a_1b_0$ is also invertible. Since $A,B\in l$ and $B\in m$ we have the following:
$$
(\lambda_0b_0 +\lambda_1b_1 )\Sum\mu_ia_i=
$$

$$=(\lambda_0b_0 +\lambda_1b_1 )\Sum\mu_ia_i+\mu_2b_2\Sum\lambda_ia_i -\mu_2a_2 \Sum\lambda_ib_i-(\lambda_0a_0+\lambda_1a_1)\Sum\mu_ib_i
$$

$$
=-(\lambda_0 \mu_1 -\lambda_1 \mu_0)(a_0b_1 - a_1b_0).
$$
The right hand side is invertible by assumption, hence $\Sum\mu_ia_i$ is invertible, and therefore $A\notin m$. 
\end{proof}

\begin{prop} \label{propdetcol}
Let $A=(a_0,a_1,a_2)$, $B=(b_0,b_1,b_2)$ and $C=(c_0,c_1,c_2)$ be points of the projective plane over a ring $R$, such that there exists a line $l$  passing through all $A$, $B$ and $C$. Then, 
$$\det \begin{pmatrix}
a_0 & b_0 & c_0 \\
a_1 & b_1 & c_1 \\
a_2 & b_2 & c_2  
\end{pmatrix}
=0.$$

Dually, let $k=\pt{\kappa}$, $l=\pt{\lambda}$ and $m=\pt{\mu}$ be lines of $\mathbb P(R)$ such that there exists a point $A$  lying on all $k$, $l$ and $m$. Then, 
$$\det \begin{pmatrix}
\kappa_0 & \lambda_0 & \mu_0 \\
\kappa_1 & \lambda_1 & \mu_1 \\
\kappa_2 & \lambda_2 & \mu_2  
\end{pmatrix}
=0.$$
\end{prop}

\begin{proof}
Let $A=(a_0,a_1,a_2)$, $B=(b_0,b_1,b_2)$ and $C=(c_0,c_1,c_2)$ be points of $\mathbb P(R)$. Let $l=(\lambda_0,\lambda_1,\lambda_2)$ be a line of $\mathbb P(R)$ such that $A,B,C\in l$. Without loss of generality, assume that $\lambda_0$ is invertible. Then $a_0=\lambda_0^{-1}(-\lambda_1 a_1 -\lambda_2 a_2)$, $b_0=\lambda_0^{-1}(-\lambda_1 b_1 -\lambda_2 b_2)$ and $c_0=\lambda_0^{-1}(-\lambda_1 c_1 -\lambda_2 c_2)$, therefore 
\[
\det \begin{pmatrix}
a_0 & b_0 & c_0 \\
a_1 & b_1 & c_1 \\
a_2 & b_2 & c_2  
\end{pmatrix}
=-\lambda_0^{-1}
\det 
\begin{pmatrix}
\lambda_1 a_1 +\lambda_2 a_2 & \lambda_1 b_1 +\lambda_2 b_2 & \lambda_1 c_1 +\lambda_2 c_2 \\
a_1 & b_1 & c_1 \\
a_2 & b_2 & c_2  
\end{pmatrix}
=0.
\]

The second part of the proposition is true by the duality principle.
\end{proof}

\begin{prop} \label{propiffcollinear}
Let $A=\pt a$, $B=\pt b$ and $C=\pt c$ be points of the projective plane over a ring $R$ such that $A\#B$. Then $C\in \ovl{AB}$ iff
$$\det
 \begin{pmatrix}
a_0 & b_0 & c_0 \\
a_1 & b_1 & c_1 \\
a_2 & b_2 & c_2  
\end{pmatrix}=0.
$$

Dually, let $k=\pt{\kappa}$, $l=\pt{\lambda}$ and $m=\pt{\mu}$ be lines of $\mathbb P(R)$ such that $k\# l$. Then $k\cap l\in m$ iff 
$$\det \begin{pmatrix}
\kappa_0 & \lambda_0 & \mu_0 \\
\kappa_1 & \lambda_1 & \mu_1 \\
\kappa_2 & \lambda_2 & \mu_2  
\end{pmatrix}
=0.$$
\end{prop}

\begin{proof}
Let the line $\ovl{AB}$ have coordinates $\pt r=(a_2b_1-a_1b_2, a_0b_2 -a_2b_0,a_1b_0-a_0b_1)$. $C\in \ovl{AB}$ iff $r_0 c_0 +r_1 c_1 +r_2 c_2 =0$.

$r_0 c_0 +r_1 c_1 +r_2 c_2= - \det \begin{pmatrix}
a_0 & b_0 & c_0 \\
a_1 & b_1 & c_1 \\
a_2 & b_2 & c_2  
\end{pmatrix}$, hence $C\in \ovl{AB}$ iff the determinant is $0$.
\end{proof}

\begin{lem} \label{lemprojcol}
Let $R$ be a ring and let $A$, $B$ and $C$ be points of $\mathbb P(R)$ represented by $\mathbf{a}=(a_0,a_1,a_2)$, $\mathbf{b}=(b_0,b_1,b_2)$ and $\mathbf{c}=(c_0,c_1,c_2)$ respectively, such that $A\# B$. Then, $C\in \ovl{AB}$ iff there exist $x$ and $y$ in $R$ such that $\mathbf{c}=x \mathbf{a}+ y\mathbf{b}$. 
\end{lem}

\begin{proof}
Let $A$, $B$ and $C$ be points of $\mathbb P(R)$ represented by $\mathbf{a}=(a_0,a_1,a_2)$, $\mathbf{b}=(b_0,b_1,b_2)$ and $\mathbf{c}=(c_0,c_1,c_2)$ with $A\#B$. The line $\ovl{AB}$ can be represented by $\pt{\lambda}=(a_1 b_2 -a_2 b_1, a_2 b_0 - a_0 b_2, a_0 b_1 - a_1 b_0)$. Without loss of generality, let us assume that $\lambda_2= a_0 b_1 - a_1 b_0$ is invertible.

Suppose that $C\in \ovl{AB}$. Then, $c_2=\lambda_2^{-1}(\lambda_0 c_0+\lambda_1 c_1)$.
The matrix $\begin{pmatrix}
a_0 & a_1 \\ b_0 & b_1
       \end{pmatrix}$ is invertible because its determinant is invertible. We define $x$ and $y$ by 
$$\begin{pmatrix} x \\ y \end{pmatrix}
=
\begin{pmatrix} a_0 & a_1 \\ b_0 & b_1 \end{pmatrix}
^{-1}
\begin{pmatrix} c_0 \\c_1 \end{pmatrix}.$$
Therefore,
$\begin{pmatrix} c_0 \\c_1 \end{pmatrix}
=x
\begin{pmatrix} a_0 \\a_1 \end{pmatrix}
+y
\begin{pmatrix} b_0 \\b_1 \end{pmatrix}
$.
Let $\mathbf{c'}= x\mathbf{a}+y\mathbf{b}$. Note that $\Sum \lambda_i c'_i=0$ by Proposition \ref{propiffcollinear}, because $\det \begin{pmatrix}
c'_0 & a_0 & b_0 \\
c'_1 & a_1 & b_1 \\
c'_2 & a_2 & b_2  
\end{pmatrix}=0$. Therefore, $c'_2=\lambda_2^{-1}(\lambda_0 c'_0+\lambda_1 c'_1)=\lambda_2^{-1}(\lambda_0 c_0+\lambda_1 c_1)=c_2$. Hence, $\mathbf c= \mathbf{c'}= x\mathbf{a}+y\mathbf{b}$.

Conversely, suppose that $\mathbf{c}=x \mathbf{a}+ y\mathbf{b}$.  Then, $\det \begin{pmatrix}
a_0 & b_0 & c_0 \\
a_1 & b_1 & c_1 \\
a_2 & b_2 & c_2  
\end{pmatrix}=0$, and therefore $C\in\ovl{AB}$ by Proposition \ref{propiffcollinear}.  
\end{proof}

\begin{lem} \label{lemptslines}
Let $A=\pt{a}$, $B=\pt{b}$ be points and let $l=\pt{\lambda}$ be a line of a projective plane over a ring $R$. If $A$, $B$ lie on $l$ then 
$$(a_1 b_2 -a_2 b_1, a_2 b_0 - a_0 b_2, a_0 b_1 - a_1 b_0)$$
is a multiple of $\pt{\lambda}$.
\end{lem}

\begin{proof}
Let $A$, $B$ and $l$ be as in the statement of the lemma. The line $l$ passes through some points $P=\pt{p}$ and $Q=\pt{q}$ which are apart from each other. Then, $(p_1 q_2 -p_2 q_1, p_2 q_0 - p_0 q_2, p_0 q_1 - p_1 q_0)=r\pt{\lambda}$ for some invertible $r$ in $R$ by Proposition \ref{propexistsuniqueline}. $A$ and $B$ lie on $\ovl{PQ}$, therefore by Lemma \ref{lemprojcol}, $\mathbf{a}=x\mathbf{p}+y\mathbf{q}$, and $\mathbf{b}=x'\mathbf{p}+y'\mathbf{q}$ for some $x$, $y$, $x'$, $y'$ in $R$. Hence, 
\begin{displaymath}
\begin{split}
a_0 b_1 - a_1 b_0 & =\det\begin{pmatrix}
a_0 & b_0 \\
a_1 & b_1
\end{pmatrix}\\
& =
\det\begin{pmatrix}
xp_0+yq_0 & x'p_0+y'q_0 \\
xp_0+yq_0 & x'p_1+y'q_1
\end{pmatrix} \\
& =(xy'-x'y)(p_0 q_1 - p_1 q_0) \\
& =r(xy'-x'y)\lambda_2.
\end{split}
\end{displaymath}
Similarly, $a_1 b_2 -a_2 b_1=r(xy'-x'y)\lambda_0$ and $a_2 b_0 - a_0 b_2=r(xy'-x'y)\lambda_1$. Hence, $(a_1 b_2 -a_2 b_1, a_2 b_0 - a_0 b_2, a_0 b_1 - a_1 b_0)= r(xy'-x'y)\pt{\lambda}$ as required.
\end{proof}

\begin{rmk}
The converse is not true. For a counterexample, consider the projective plane over $\mathbb Z/(4)$ and let $A=(2,0,1)$, $B=(0,2,1)$ and $l=(1,1,0)$. When $A\#B$ then $(a_1 b_2 -a_2 b_1, a_2 b_0 - a_0 b_2, a_0 b_1 - a_1 b_0)$ is a multiple of $\pt{\lambda}$ by a unit by the construction of the unique line through two points that are apart from each other in Proposition \ref{propexistsuniqueline}.
\end{rmk}

\begin{prop} 
Given the projective plane over a ring $R$ the following are equivalent:
\begin{enumerate}
\item  $R$ satisfies $\text{inv}(x+y) \vdash_{x,y} \text{inv}(x) \vee \text{inv}(y)$.
\item For $A$, $B$ and $C$ points of the projective plane, $A\# B \vdash_{A,B,C} A\# C \vee B\# C$.
\item For $k$, $l$ and $m$ lines of the projective plane, $k\# l \vdash_{k,l,m} k\# m \vee l\# m$.
\item $A\notin l \vdash_{A,B,l} A\# B \vee B\notin l$.
\item $A\notin l \vdash_{A,l,m} l\#m \vee A\notin m$.
\item $A\# B \wedge l\#m \vdash_{A,B,l,m} A\notin l \vee B\notin m \vee A\notin m \vee B \notin l$.
\end{enumerate}
\end{prop}

\begin{proof}
\begin{description}
 \item [$1\Rightarrow 2$: ]

Let $A=(a_0,a_1,a_2)$, $B=(b_0,b_1,b_2)$ and $C=(c_0,c_1,c_2)$ be points such that $A\#B$.
Without loss of generality, assume that $a_0 b_1 - a_1 b_0$ is invertible.
By 1, $a_0 b_1$ is invertible or $a_1 b_0$ is.
Without loss of generality, assume that $a_0 b_1$ is invertible, i.e. $a_0$ and $b_1$ are invertible.
At least one of $c_0$, $c_1$ and $c_2$ is invertible.

In the case where $c_0$ is invertible, $b_0(a_0 c_1-a_1 c_0)-a_0(b_0 c_1 - b_1 c_0)= c_0 (a_0 b_1 -a_1 b_0)$ is also invertible.
By 1, at least one of $b_0(a_0 c_1-a_1 c_0)$ and $a_0(b_0 c_1 - b_1 c_0)$ is invertible.
In the first case $A\#C$ and in the second case $B\#C$.

The case where $c_1$ is invertible is symmetric.

Finally, in the case where $c_2$ is invertible, $a_0 c_2$ is invertible. Therefore, either $a_0 c_2 - a_2 c_0$ is invertible or $a_2 c_0$ is invertible. In the first case $A\# C$. In the second case $c_0$ is also invertible, and therefore it reduces to a case considered above.

\item [$2\Rightarrow 1$]

Let $x$, $y$ be in $R$ such that $x+y$ is invertible. Then, $(x,0,1)\#(-y,0,1)$, therefore $(0,0,1)$ is apart from $(x,0,1)$ or $(0,0,1)$ is apart from $(-y,0,1)$. In the first case $x$ is invertible and in the second case $y$ is invertible. Hence, 1 is satisfied.

\item [$1\Leftrightarrow 3$]

3 is dual to 2, and since 1 iff 2, then also 1 iff 3.

\item [$1\Rightarrow 4$]

Let $A=\pt{a}$, $B=\pt{b}$ be two points and let $l=\pt{\lambda}$ be a line, such that $A\notin l$, i.e. $\sum\limits_{i=0}^2 \lambda_i a_i$ is invertible. Without loss of generality, let us assume that $b_0$ is invertible. Then, the sum
$a_0(\lambda_0 b_0 + \lambda_1b_1 +\lambda_2 b_2) + \lambda_2(a_2 b_0 -a_0 b_2)+\lambda_1 (a_1b_0-a_0b_1)$ is invertible because it is equal to $b_0\sum\limits_{i=0}^2 \lambda_ia_i$. By 1, at least one of the three summands $a_0(\lambda_0 b_0 + \lambda_1b_1 +\lambda_2 b_2)$, $\lambda_2(a_2 b_0 -a_0 b_2)$ and $\lambda_1 (a_1b_0-a_0b_1)$ is invertible. If the first one is invertible, then $B\notin l$. If the second or the third one is invertible, then $A\#B$.

\item [$4\Rightarrow 1$]

Let $x$, $y$ be in $R$ such that $x+y$ is invertible. Then, $(x,0,1)\notin (1,0,y)$. By 4, the point $(0,0,1)$ is either apart from $(x,0,1)$ or lies outside $(1,0,y)$. In the first case $x$ is invertible and in the second case $y$ is invertible. Therefore, 1 is satisfied.

\item [$1 \Leftrightarrow 5$]

5 is dual to 4, and since 1 iff 4, then also 1 iff 5.

\item [$2,3 \Rightarrow 6$]

Let $A=\pt{a}$, $B=\pt{b}$ be points and let $l=\pt{\lambda}$, $m=\pt{\mu}$ be lines such that $A\# B$ and $l\# m$. Let $C$ be the unique intersection point of $l$ and $m$. Then by 2, at least one of $A$ and $B$ is apart from $C$. Without loss of generality, let us assume that $A\#C$ and let $k$ be the line $\ovl{AC}$. By 3, $k$ is apart from at least one of $l$ and $m$. Without loss of generality, let us assume that $k\#l$. Therefore, 
$A\# C \wedge k\#l \wedge A\in k \wedge C\in k \wedge C\in l$. Hence, by Proposition \ref{proplastaxiom} $A\notin l$.

\item [$6 \Rightarrow 1$]

Let $x$, $y$ be in $R$ such that $x+y$ is invertible. Then, the points $(x,0,1)$ and $(-y,0,1)$ are apart from each other and the lines $(1,0,0)$ and $(0,1,0)$ are apart from each other. By the four cases given in the conclusion of 6, at least one of $x$, $y$ and $0$ is invertible. In the case where $0$ is invertible, also $x$ is invertible. Hence, in all of the four cases at least one of $x$ and $y$ is invertible.
\end{description}
\end{proof}

\begin{rmk}
Not all rings satisfy the condition given in 1. For example in $\mathbb Z/(6)$, $3+2$ is invertible but neither $2$ nor $3$ are invertible.
\end{rmk}

We shall use the notion of local ring as used in the formulation of topos theory. It is the following.

\begin{defn}
A \emph{local ring} is a commutative ring that is a non-trivial ring ($\text{inv}(0)\vdash \bot$), and  satisfies the sequent $\text{inv}(x+y) \vdash_{x,y} \text{inv}(x) \vee \text{inv}(y)$.
\end{defn}

The following statements hold for projective planes over a field but not in general for projective planes over local rings:

\begin{enumerate}
 \item $\top \vdash_{A,B} \exists l. A,B\in l$

The above sequent is not satisfied by the projective plane over the local ring $\mathbb Z [X,Y]_{(X,Y)}$. Consider the points $(0,0,1)$ and $(X,Y,1)$. Suppose that they both lie on the line $(\lambda_0,\lambda_1, \lambda_2)$. Then $\lambda_2=0$ and $\lambda_0 X+\lambda_1 Y=0$. Without loss of generality suppose that $\lambda_0$ is invertible. Then $X=-\lambda_0^{-1}\lambda_1Y$ which implies that there exist  $a,b\in \mathbb Z[X,Y]$ with $a\notin (X,Y)$ and $aX=bY$, contradicting the fact that $(Y)$ is a prime ideal.

\item $A,B\in l,m \vdash_{A,B,l,m} (A=B)\vee (l=m)$

The above sequent is not satisfied by the projective plane over the local ring $\mathbb Z/(4)$. Consider the points $A=(2,2,1)$ and $B=(2,0,1)$, and the lines $l=(1,0,2)$ and $m=(1,2,2)$.

\item Let $A=(a_0,a_1,a_2)$, $B=(b_0,b_1,b_2)$ and $C=(c_0,c_1,c_2)$ be points, such that $\det \begin{pmatrix}
a_0 & b_0 & c_0 \\
a_1 & b_1 & c_1 \\
a_2 & b_2 & c_2  
\end{pmatrix}=0$.
Then, there exists a line $l$ containing the points $A$, $B$ and $C$.

The above statement is not satisfied by the projective plane over the local ring $\mathbb Z [X,Y]_{(X,Y)}$. Consider the points from (1). $\det \begin{pmatrix}
0 & X & X\\
0 & Y & Y\\
1 & 1 & 1
\end{pmatrix}=0$ but there exists no line containing $(0,0,1)$ and $(X,Y,1)$.

\item $A\in l \wedge B\#A\vdash_{A,B,l} B\notin l \vee B \in l$

The above sequent is not satisfied by the projective plane over the local ring $\mathbb Z/(4)$. Consider the points $A=(1,0,0)$ and $B=(0,2,1)$, and the line $l=(0,1,0)$.
\end{enumerate}
In fact, we have the following proposition:

\begin{prop}
 \begin{enumerate}
  \item A projective plane over a local ring $R$ satisfies $\top \vdash_{A,B} \exists l. A,B\in l$ iff $R$ satisfies $\top\vdash_{a,b} \exists x. ax=b \vee a=bx$.

\item A projective plane over a local ring $R$ satisfies 
$$A,B\in l,m \vdash_{A,B,l,m} (A=B)\vee (l=m)$$
iff $R$ is an integral domain (a ring satisfying $ab=0\vdash_{a,b} a=0 \vee b=0$).

\item A local ring $R$ is an integral domain and it satisfies 
$$\top\vdash_{a,b} \exists x. ax=b \vee a=bx$$
iff its projective plane satisfies that for any three points $A=(a_0,a_1,a_2)$, $B=(b_0,b_1,b_2)$ and $C=(c_0,c_1,c_2)$, such that $\det \begin{pmatrix}
a_0 & b_0 & c_0 \\
a_1 & b_1 & c_1 \\
a_2 & b_2 & c_2  
\end{pmatrix}=0$,
 there exists a line $l$ containing the points $A$, $B$ and $C$.

\item  A projective plane over a local ring $R$ satisfies
$$A\in l \wedge B\#A\vdash_{A,B,l} B\notin l \vee B \in l$$
iff $R$ is a (geometric) field.
 \end{enumerate}

\end{prop}

\begin{proof}
\begin{enumerate}
 \item Suppose the projective plane over $R$ satisfies $\top \vdash_{A,B} \exists l. A,B\in l$. Let $a,b \in R$ and consider the points $(0,0,1)$ and $(a,b,1)$ of the projective plane. There exists a line $\pt{\lambda}$ which contains both points. $(0,0,1)$ lies on the line therefore $\lambda_2=0$. $(a,b,1)$ lies on the line, therefore 
$$\lambda_0 a + \lambda_1 b=0.$$
One of $\lambda_0$ and $\lambda_1$ is invertible, hence $a=\lambda_0^{-1}\lambda_1 b$ or $b=\lambda_0\lambda_1^{-1} b$.

Suppose that $R$ satisfies $\top\vdash_{a,b} \exists x. ax=b \vee a=bx$. Let $A$ and $B$ be two points of the projective plane over $R$ represented by $\pt{a}$ and $\pt{b}$ respectively. Without loss of generality suppose $a_2$ is invertible, and consider a representative of $A$ of the form $(a_0,a_1,1)$. If $b_2$ is not invertible, then $A\#B$, therefore there exists a unique line containing $A$ and $B$. Therefore, we consider the case where $b_2$ is invertible and consider a representative of $B$ of the form $(b_0,b_1,1)$. There exists $x$ in $R$ such that $(a_0-b_0)x=(a_1-b_1)$ or $(a_0-b_0)=(a_1-b_1)x$. Without loss of generality suppose the first case. Then $A$ and $B$ lie on the line $(x,1,-xa_0-a_1)$.

\item Suppose the projective plane over $R$ satisfies $A,B\in l,m \vdash_{A,B,l,m} (A=B)\vee (l=m)$. Let $a,b \in R$ such that $ab=0$ and consider the points $(0,0,1)$ and $(a,0,1)$ of the projective plane which lie on both lines $(0,1,0)$ and $(b,1,0)$. Then, either the two points are the same or the two lines are the same. In the first case $a=0$ and in the second case $b=0$.

Suppose that $R$ is an integral domain. Let $A$ and $B$ be two points of the projective plane over $R$ lying on both lines $l$ and $m$. As in the proof of 1, if the two points are apart from each other then they can only belong to a unique line, therefore without loss of generality we only consider the case where $A$ and $B$ are represented by $(a_0,a_1,1)$ and $(b_0,b_1,1)$ respectively. $A\in l$, therefore the first or second coordinate of $l$ must be invertible. Without loss of generality suppose that $l$ is represented by $(1,\lambda_1,\lambda_2)$. If $l\# m$ then $A=B$, therefore we consider the case where $l$ and $m$ are not apart from each other. Hence, $m$ is represented by $(1,\mu_0,\mu_1)$. $A$ and $B$ lie on both $l$ and $m$, therefore $a_0+\lambda_1a_1+\lambda_2=b_0+\lambda_1 b_1+\lambda_2=0$, and $a_0+\mu_1 a_1+\mu_2=b_0+\mu_1 b_1+\mu_2=0$. Therefore, $(\lambda_1-\mu_1)(a_1-b_1)=0$ and since $R$ is an integral domain, one of the two factors must be 0. If $\lambda_1=\mu_1$, then $l=m$ and if $a_1=b_1$ then $A=B$.

\item Suppose the projective plane over a local ring $R$ satisfies the specified condition. Then, given two points $A$ and $B$ represented by $\pt{a}$ and $\pt{b}$ respectively then $\det
\begin{pmatrix}
  a_0 & a_0 & b_0\\
a_1 & a_1 & b_1 \\
a_2 & a_2 & b_2
 \end{pmatrix}=0$.Therefore there exists a line $l$ which contains both $A$ and $B$. Therefore, the conditions of 1 are satisfied, hence $R$ satisfies $\top\vdash_{a,b} \exists x. ax=b \vee a=bx$.
Also, given $x, y$ in $R$ such that $xy=0$, then $\det
\begin{pmatrix}
  x & 0 & 0\\
0 & y & 0 \\
1 & 1 & 1
 \end{pmatrix}=0$, therefore there exists a line $\pt{\lambda}$ which contains the points $(x,0,1), (0,y,1)$ and $(0,0,1)$. $\lambda_2$ must be zero. Therefore, $\lambda_0 x=0$ and $\lambda_1 y=0$. One of the $\lambda_i$'s must be invertible therefore $x=0$ or $y=0$. Hence, $R$ is an integral domain.

Suppose now that $R$ is a local ring which is an integral domain and satisfies $\top\vdash_{a,b} \exists x. ax=b \vee a=bx$. Suppose $A=(a_0,a_1,a_2)$, $B=(b_0,b_1,b_2)$ and $C=(c_0,c_1,c_2)$ are points of the projective plane over $R$, such that $\det \begin{pmatrix}
a_0 & b_0 & c_0 \\
a_1 & b_1 & c_1 \\
a_2 & b_2 & c_2  
\end{pmatrix}=0$. If any two of the points are apart from each other the third point lies on the unique line defined by them, therefore we consider the case where they are not and without loss of generality we assume that $A$, $B$ and $C$ are represented by $A=(a_0,a_1,1)$, $B=(b_0,b_1,1)$ and $C=(c_0,c_1,1)$ respectively. By the proof of 1 there exists a line $l=\pt{\lambda}$ containing $A$ and $B$ whose first or second coordinate is invertible. Without loss of generality, suppose $l=(1,\lambda_1,\lambda_2)$. Then,
$$\det \begin{pmatrix}
0 & 0 & c_0-\lambda_1c_1-\lambda_2 \\
a_1 & b_1 & c_1 \\
1 & 1 & 1  
\end{pmatrix}=
\begin{pmatrix}
a_0 & b_0 & c_0 \\
a_1 & b_1 & c_1 \\
1 & 1 & 1  
\end{pmatrix}=0.$$
Therefore, $(\lambda_1c_1-\lambda_2)(a_1-b_1)=0$ and since $R$ is an integral domain one of the two factors must be zero. If $(\lambda_1c_1-\lambda_2)=0$, then $C\in l$. If $(a_1-b_1)=0$, then also $a_0=b_0$, therefore $A=B$, and then by 1 there exists some line containing both $A$ and $C$.

\item Suppose the projective plane over a local ring $R$ satisfies $A\in l \wedge B\#A\vdash_{A,B,l} B\notin l \vee B \in l$. Then given $x$ in $R$, the points $(0,0,1)$ and $(x,1,1)$ are apart from each other and $(0,0,1)$ lies on $(1,0,0)$. Therefore either $(x,1,1)$ lies on $(1,0,0)$ or outside $(1,0,0)$. In the first case, $x=0$ and in the second case $x$ is invertible. Therefore, $R$ is a field.

Suppose that $R$ is a field. Given a point $A$ and a line $l$ represented by $\pt{a}$ and $\pt{\lambda}$ respectively, $\Sum\lambda_i a_i$ is either $0$ or invertible, therefore either $A\in l$ or $A\notin l$.
\end{enumerate}

\end{proof}

\section{The theory of preprojective planes}

\begin{defn} 
The theory of \emph{preprojective planes} is written in a language with two sorts: points and lines. It has a binary relation $\#$ on points, a binary relation $\#$ on lines, and two relations $\in$ and $\notin$ between points and lines. The axioms of the theory of preprojective planes are the following:
\begin{itemize}

\item $\#$ is an apartness relation on points, i.e. for $A$, $B$, $C$ points the following hold:
\begin{enumerate}
\item $A\# A \vdash_A \bot$, 
\item $A\# B \vdash_{A,B} B\# A$,
\item $A\# B \vdash_{A,B,C} A\# C \vee B\# C$.
\end{enumerate}

\item $\#$ is an apartness relation on lines, i.e. for $k$, $l$, $m$ lines the following hold:
\begin{enumerate}
\item $k\# k \vdash_k \bot$, 
\item $k\# l \vdash_{k,l} l\# k$,
\item $k\# l \vdash_{k,l,m} k\# m \vee l\# m$.
\end{enumerate}

\item $\notin$ is in some sense a constructive complement of $\in$, i.e. for $A$, $B$ points and $k$, $l$ lines the following hold:
\begin{enumerate}
\item $A\in l \wedge A\notin l \vdash_{A,l} \bot$,
\item $A\notin k \vdash_{A,B,k} A\# B \vee B\notin k$,
\item $A\notin k \vdash_{A,k,l} k\# l \vee A\notin l$.
\end{enumerate}

\item There exists a unique function from the set of pairs of points that are apart from each other to lines such that the image of the pair contains both points, i.e. the following hold:
\begin{enumerate}
\item $A\# B \vdash_{A,B} \exists k. A\in k \wedge B\in k $,
\item $A\# B \wedge A,B\in k \wedge A,B\in l \vdash_{A,B,k,l} k=l$.
\end{enumerate}

\item Dually, there exists a unique function from the set of pairs of lines that are apart from each other to points such that the image of the pair lies on both lines, i.e. the following hold:
\begin{enumerate}
\item $k\# l \vdash_{k,l} \exists A. A\in k \wedge A\in l $,
\item $k\# l \wedge A,B\in k \wedge A,B\in l \vdash_{A,B,k,l} A=B$.
\end{enumerate}

\item We have the following three axioms which say that we have enough points and lines:
\begin{enumerate}
 \item $\top \vdash_l \exists A,B,C.A\#B\#C\#A \wedge A,B,C\in l$,
\item $\top \vdash_l \exists A. A\notin l$
\item $\top \vdash \exists A,B,C,l. A\#B \wedge A,B\in l\wedge C\notin l$.
\end{enumerate}

\item We also have the dual of the above axioms:
\begin{enumerate}
 \item $\top \vdash_A \exists k,l,m.k\#l\#m\#k \wedge A\in k,l,m$,
\item $\top \vdash_A \exists l. A\notin l$,
\item $\top \vdash \exists A,k,l,m. k\#l \wedge A\in k,l \wedge A\notin m$.
\end{enumerate}

\item The following self-dual axiom which says that given two points that are apart from each other and two lines that are apart from each other, then at least one of the two points lies outside at least one of the two lines: $A\# B \wedge l\#m \vdash_{A,B,l,m} A\notin l \vee B\notin m \vee A\notin m \vee B \notin l$.
\end{itemize}

\end{defn}

\begin{rmk} The apartness relation on lines is not necessary to describe this theory because from the above axioms we can prove that $k\# l$ iff $\exists A. A\in k \wedge A\notin l$. Hence, we can replace occurrences of the $\#$ relation on lines by the above relation and get a Morita equivalent theory.

Also, by the duality of the theory we can also replace the $\#$ relation on points with a relation which includes the $\in$ and $\notin$ relations. Hence, we can describe the theory of projective planes using only the incidence and non-incidence relations. 
\end{rmk}

\begin{rmk} Notice that 
$$A\notin l \text{ iff } \exists B, m. A \in m \wedge B \in m \wedge B\in l \wedge A\# B \wedge l\# m.$$
Thus, we can also describe the theory of projective planes without the non-incidence relation, i.e. by only using the incidence relation and the two $\#$ relations.
\end{rmk}

\begin{rmk} \label{rmkthrylines}
We can also consider a theory Morita equivalent to the theory of preprojective planes which has only one sort: points. The language of this theory has an apartness relation on points, a binary relation on points (which is an equivalence relation on pairs of points that are apart from each other) and two ternary relations on points. The first ternary relation (in our formulation of the theory) states that two of the three points are apart from each other and they all lie on a common line. The second ternary relation states that the three points are non-collinear (Definition \ref{defnnoncoll}).

The lines of the preprojective plane become equivalence classes of pairs of points which are apart from each other. The incidence and non-incidence relations of the preprojective planes can be retrieved using the two ternary relation on points.

Dually, we have a Morita equivalent theory whose only sort is the sort of lines. 
\end{rmk}

For every axiom of the theory of preprojective planes, its dual (in the sense described for the projective plane over a ring) is also an axiom. Hence, the \emph{duality principle} holds for preprojective planes: given a theorem for preprojective planes, the dual theorem also holds.

Here is some notation and conventions for points and lines of a preprojective plane:

\begin{itemize}
\item We say that a point $A$ is apart from a point $B$ whenever $A\# B$ and we say that a line $k$ is apart from a line $l$ whenever $k\# l$. We say that a point $A$ lies on a line $l$ or that $l$ passes through $A$ whenever $A\in l$. We say that a point $A$ lies outside a line $l$ or that $l$ passes outside $A$ whenever $A\notin l$.

\item For $A$, $B$ points such that $A\#B$, by the axioms of preprojective planes there exists a unique line that passes through $A$ and $B$. We denote this line by $\ovl{AB}$.

\item For $l$, $m$ lines such that $l\#m$, by the axioms of preprojective planes there exists a unique point that lies on both lines. We call this point the intersection of $l$ and $m$ and denote it by $l\cap m$.
\end{itemize}

\begin{lem}
The projective plane over a local ring satisfies all the above conditions, therefore it is a preprojective plane. 
\end{lem}

\begin{proof}
This is true by the results of  Section \ref{projprop}.
\end{proof}

\section{Non-collinear points, non-concurrent lines}

\begin{defn} \label{defnnoncoll}
We say that three points $A$, $B$, $C$ of a preprojective plane are \emph{non-collinear} when $A\#B\#C\#A$ and $A\notin \ovl{BC}$, $B\notin \ovl{CA}$ and $C\notin \ovl{AB}$.

Let us also consider the dual of the above definition. Three lines $k$, $l$, $m$ of the preprojective plane are \emph{non-concurrent} when $k\#l\#m\#k$ and $(l\cap m) \notin k$, $(m\cap k) \notin l$ and $(k\cap l) \notin m$.
\end{defn}

It is clear from the definitions that both of the above relations are symmetric on the three points and the three lines respectively.

\begin{prop} \label{propnonc}
Let $A$, $B$ and $C$ be points of a preprojective plane such that $B\#C$ and $A\notin \ovl{BC}$. Then, $A\# B$, $\ovl{AB}\#\ovl{BC}$ and $C\notin \ovl{AB}$.

Dually, let $k$, $l$ and $m$ be lines of a preprojective plane such that $l\#m$ and $(l\cap m)\notin k$, then $k\# l$, $(k\cap l) \# (l \cap m)$ and $(k\cap l) \notin m$.
\end{prop}

\begin{proof}
 We only prove the first part of the proposition since the second one is dual to it.

Let $A$, $B$ and $C$ be as in the statement of the proposition.
$A\notin \ovl{BC}$, therefore $B$ is either apart from $A$ or it lies outside $\ovl{BC}$, and since $B\in\ovl{BC}$ we conclude that $A\#B$. By a symmetric argument, $C$ is apart from $A$.

$A\notin \ovl{BC}$, therefore $\ovl{AB}$ is either apart from $\ovl{BC}$ or it passes outside $A$, and since $A\in \ovl{AB}$ we conclude that $\ovl{AB} \# \ovl{BC}$. Now, $B\# C$ and $\ovl{AB}\# \ovl{BC}$, hence at least one of  $B$ and $C$ lies outside at least one of $\ovl{AB}$ and $\ovl{BC}$. Since $B$ lies on both $\ovl{AB}$ and $\ovl{BC}$, and $C\in \ovl{BC}$ we conclude that $C\notin \ovl{AB}$.
\end{proof}

\begin{lem} \label{lemnonc}
Let $A$, $B$ and $C$ be points of a preprojective plane. Then, the following are equivalent:
\begin{enumerate}
\item $A$, $B$, $C$ are non-collinear,
\item $B\#C$ and $A\notin \ovl{BC}$,
\item $A\# B\# C$ and $\ovl{AB}\#\ovl{BC}$.
\end{enumerate}

Dually, let $k$, $l$ and $m$ be lines of a preprojective plane. Then the following are equivalent
\begin{enumerate}
\item $k$, $l$, $m$ are non-concurrent,
\item $l\# m$ and $(l\cap m) \notin k$,
\item $k\#l\#m$ and $(k\cap l) \# (l \cap m)$
\end{enumerate}
\end{lem}

\begin{proof}
We only prove the first part of the lemma because the second one is dual to it. 

Clearly, (1) implies (2), and by Proposition \ref{propnonc} (2) implies (3).

(2) $\Rightarrow$ (1): $B\#C$ and $A\notin \ovl{BC}$, therefore by Proposition \ref{propnonc}, $A\#B$ and $C\notin \ovl{AB}$. By the symmetry of the proposition it is also true that $C\# A$ and $B\notin \ovl{AC}$. Hence, $A$, $B$ and $C$ are non-collinear.

(3) $\Rightarrow$ (2): $A\# B$ and $\ovl{AB}\#\ovl{BC}$, therefore at least one of $A$ and $B$  lies outside at least one of $\ovl{AB}$ and $\ovl{BC}$. Since both $A$ and $B$ lie on $\ovl{AB}$ and $B\in \ovl{BC}$, we conclude that $A\notin \ovl{BC}$.
\end{proof}

\begin{rmk}
In the statement of the above lemma, we only list three equivalent conditions for non-collinear points (and non-concurrent lines). The definitions of non-collinear (and non-concurrent) are symmetric, therefore the list of equivalent conditions may be extended to include all the permutations of $A$, $B$, $C$ (and $k$, $l$, $m$)  in conditions (2) and (3). 
\end{rmk}

\begin{lem}
Let $A$, $B$, $C$ be non-collinear points of a preprojective plane. Then, the lines $\ovl{AB}$, $\ovl{BC}$, $\ovl{CA}$ are non-concurrent.

Dually, for $k$, $l$, $m$ non-concurrent lines of a preprojective plane, the points $k\cap l$, $l\cap m$, $m\cap k$ are non-collinear.
\end{lem}

\begin{proof}
 We only prove the first part of the lemma since the second one is dual to it.

Let $A$, $B$, $C$ be non-collinear points of the projective plane. By the equivalent conditions (and their symmetric ones) in Lemma \ref{lemnonc}, we see that $\ovl{AB}\# \ovl{BC} \# \ovl{CA}$. Therefore, by condition (2) of Lemma \ref{lemnonc} for $\ovl{AB}$, $\ovl{BC}$, $\ovl{CA}$ to be non-concurrent it is sufficient to prove that $\ovl{AB}\cap \ovl{BC}$ is apart from $\ovl{BC} \cap \ovl{CA}$, i.e. that $B\# C$. This is true because $A$, $B$, $C$ are non-concurrent.
\end{proof}

\begin{lem} \label{lemnoncd}
Let $A$, $B$, $C$, $D$ be points of a preprojective plane and let $A$, $B$, $C$ be non-collinear. Then, at least one combination of $D$ with two of the points $A$, $B$, $C$ gives a non-collinear triple.

Dually, let $k$, $l$, $m$, $n$ be lines of a preprojective plane and let $k$, $l$, $m$ be non-concurrent. Then, at least one combination of $n$ with two of the lines $k$, $l$, $m$ gives a non-concurrent triple.
\end{lem}

\begin{proof}
We only prove the first part of the lemma, since the second part is dual to it.

Let $A$, $B$, $C$ be non-collinear points. Then, $A\#B$, hence $D$ is apart from at least one of $A$ and $B$. Without loss of generality, let us assume that $D\#B$. By Lemma \ref{lemnonc}, $\ovl{AB}\#\ovl{BC}$, therefore at least one of $D$ and $B$ lies outside at least one of $\ovl{AB}$ and $\ovl{BC}$. Hence $D$ lies outside at least one of $\ovl{AB}$ and $\ovl{BC}$. In the first case $D$, $A$, $B$ are non-collinear, and in the second case $D$, $B$, $C$ are non-collinear.
\end{proof}

\begin{lem} \label{lemomega39}
Let $A$, $B$ be points of a preprojective plane and let $A\#B$. Then, there exists a point $C$ on the line $\ovl{AB}$ such that $A\#C\#B$.
\end{lem}

\begin{proof}
Let $A$, $B$ be points of the preprojective plane such that $A\#B$. Then, there exist points $P$, $Q$, $R$ lying on $\ovl{AB}$ such that $P\#Q\#R\#P$.
$P\#Q$ implies that $A\#P \vee A\#Q$.
$Q\#R$ implies that $A\#Q \vee A\#R$.
$R\#P$ implies that $A\#R \vee A\#P$.
By combining the three, we see that $A$ is apart from at least two of the points $P$, $Q$ and $R$. Similarly, $B$ is apart from at least two of the points $P$, $Q$ and $R$. Therefore, both $A$ and $B$ are apart from at least one of the points $P$, $Q$ and $R$, and hence the result.
\end{proof}

\begin{lem} \label{lemomega4}
In a preprojective plane, there exist points $O$, $A$, $B$, $I$ such that every subset of three of them is non-collinear.

Dually, a preprojective plane contains lines $k$, $l$, $m$, $n$ such that every subset of three of them is non-concurrent.
\end{lem}

It is easier to visualize the following proof by considering the projective plane over a local ring and letting $O=(0,0,1)$, $X=(1,0,1)$, $Y=(0,1,1)$, $A=(1,0,0)$ and $B=(0,1,0)$, and then $I=(1,1,1)$. 

\begin{proof}
By the axioms of preprojective planes, a preprojective plane contains three non-collinear points $O$, $X$, $Y$. By Lemma \ref{lemomega39}, there exists a point $A$ on the line $\ovl{OX}$ such that $O\#A\#Y$, and there exists a point $B$ on the line $\ovl{OY}$ such that $O\#B\#X$. Notice that $\ovl{OA}=\ovl{OX}\#\ovl{OY}=\ovl{OB}$, therefore $O$, $A$, $B$ are non-collinear, and also $O$, $Y$, $A$ are non-collinear.

Therefore, $A\notin \ovl{OB}=\ovl{XB}$, hence $A\notin \ovl{XB}$ which implies that $\ovl{YA}\#\ovl{XB}$. Let $I$ be the intersection of $\ovl{YA}$ and $\ovl{XB}$.

$I\in \ovl{XB}$ and $A\notin \ovl{XB}$, therefore $I\#A$. By a symmetric argument, $I\#B$. Therefore, $A\notin \ovl{XB}=\ovl{IB}$. Hence, $A$, $B$, $I$ are non-collinear.

$O$, $Y$, $A$ are non-collinear, therefore $Y\notin \ovl{OA}=\ovl{XA}$. Hence $X\notin \ovl{YA}=\ovl{IA}$ which implies that $I\notin \ovl{XA}=\ovl{OA}$. Hence, $O$, $I$, $A$ are non-collinear.

By a symmetric argument, $O$, $I$, $B$ are non-collinear, therefore $O$, $A$, $B$, $I$ are points of the preprojective plane such that every subset of three of them is non-collinear.

By the duality principle, the preprojective plane also contains lines $k$, $l$, $m$, $n$ such that every subset of three of them is non-concurrent.
\end{proof}

\section{Morphisms of preprojective planes} \label{secprojmorph}


\begin{defn}
Given two preprojective planes, a morphism between them is a structure-preserving homomorphism: It consists of a function $f_P$ from the set of points of the first to the set of points of the second and a function $f_L$  from the set of lines of the first to the set of lines of the second, such that they preserve the two $\#$ relations, $\in$ and $\notin$.
\end{defn}

The identity morphism on points and lines of a preprojective plane is always an endomorphism of the preprojective plane.

\begin{lem} \label{lemprojptsen}
 A morphism of preprojective planes is uniquely determined by the morphism on points.

Dually, a morphism of preprojective planes is uniquely determined by the morphism on lines.
\end{lem}

\begin{proof}
Let $P$, $L$ be the set of points and the set of lines respectively of a preprojective plane, and let $P'$, $L'$ be the set of points and the set of lines respectively of a second preprojective plane. Suppose we are given a morphism from the first to the second preprojective plane, such that $f_P:P\to P'$ is the morphism on points and $f_L:L\to L'$ is the morphism on lines.

Given a line $l$ in $L$ there exist points $A$ and $B$ in $P$ that are apart from each other and lie on $l$. A morphism of preprojective planes preserves the $\#$ relation on points and the incidence relation, therefore $f_P(A)$ and $f_P(B)$ are apart from each other and lie on $f_L(l)$. $\ovl{f_P(A)f_P(B)}$ is the unique line through $f_P(A)$ and $f_P(B)$, therefore $f_L(l)=\ovl{f_P(A)f_P(B)}$. Hence, $f_L$ is uniquely determined by $f_P$.

Dually, $f_P$ is uniquely determined by $f_L$.
\end{proof}


\begin{prop} \label{propptslines}
Let $\ca P$ and $\ca P'$ be preprojective planes with sets of points $P$ and $P'$ respectively. Let $f_P:P\to P'$ be a function such that:
\begin{enumerate}
\item $f_P$ preserves the $\#$ relation on points,
\item for $A$, $B$, $C$ points of $\ca P$ such that $A\# B$ and $C\in \ovl{AB}$, then ($f_P(A)\# f_P(B)$ and) $f_P(C)\in \ovl{f_P(A)f_P(B)}$,
\item given three non-collinear points $A$, $B$, $C$ of $\ca P$, then the points $f_P(A)$, $f_P(B)$, $f_P(C)$ are also non-collinear.
\end{enumerate}
Then, there is a unique morphism of preprojective planes $\ca P\to \ca P'$ whose morphism on points is $f_P$.
\end{prop}

\begin{proof}
Let $\ca P$ and $\ca P'$ be preprojective planes with sets of points $P$ and $P'$ respectively and with sets of lines $L$ and $L'$ respectively. Let $f_P:P\to P'$ be a function satisfying the above conditions.

We define a function $f_L:L\to L'$ in the following way. Given a line $k$ in $L$, there exist points $A$ and $B$ on $k$ such that $A\#B$, and by 1 $f_P(A)\# f_P(B)$, hence we define $f_L(l)$ to be $\ovl{f_P(A)f_P(B)}$. By 2, the definition of $f_L$ does not depend on the choice of $A$ and $B$. Also given a point $C$ of $\ca P$, if $C\in k$ then $f_P(C)\in f_L(k)$. $f_P$ sends non-collinear points to non-collinear points, therefore given $C\notin k$ in $(P,L)$, then $f_P(C)\notin f_L(k)$.

Given lines $k$ and $l$ in $L$ such that $k\# l$, let $A$ be their intersection point. There exist $B$ and $C$ lying on $k$ and $l$ respectively such that they are both apart from $A$. Therefore, the points $A$, $B$, $C$ are non-collinear. Hence, $f_P(A)$, $f_P(B)$, $f_P(C)$ are non-collinear. $f_L(k)=\ovl{f_P(A)f_P(B)}\#\ovl{f_P(A)f_P(B)}=f_L(l)$. Therefore, $f_L$ preserves the $\#$ relation on lines, and $(f_P, f_L)$ is a homomorphism of preprojective planes.
\end{proof}

\begin{rmk}
Notice that the above lemma and proposition may be deduced directly by using the alternative formulation of the theory of projective planes which only has one sort and which is mentioned in Remark \ref{rmkthrylines}.
\end{rmk}

\section{Morphisms between projective planes over rings} \label{secmorprojring}

Let $\alpha:R \to S$ be a ring homomorphism. Note that $\alpha$ sends invertible elements to invertible elements. Given $\mathbf{a}=\pt{a}\in R^3$ with one invertible coordinate then $(\alpha(a_0),\alpha(a_1),\alpha(a_2))$ also has an invertible coordinate. We will write $\alpha(\mathbf{a})$ to mean $(\alpha(a_0),\alpha(a_1),\alpha(a_2))$. Given $\mathbf{a}$ and $\mathbf{b}$ in $R^3$ that represent the same point/line in $\mathbb P(R)$, then $\mathbf{b}=r\mathbf{a}$ for some $r$ in $R$. Therefore, $\alpha(\mathbf{b})=\alpha(r)\alpha(\mathbf{a})$ which implies that $\alpha(\mathbf{a})$ and $\alpha(\mathbf{b})$ represent the same point/line in $\mathbb P(S)$. Therefore, $\alpha$ determines a morphism from the points of $\mathbb P(R)$ to the points of $\mathbb P(S)$ and a morphism from the lines of $\mathbb P(R)$ to the lines of $\mathbb P(S)$.

\begin{prop} \label{propringproj}
A ring homomorphism $\alpha:R \to S$ determines a morphism from $\mathbb P(R)$ to $\mathbb P(S)$ (in the way described above).
\end{prop}

\begin{proof}
Given $\pt{a},\pt{b}\in R^3$ such that $\mathbf{a}\#\mathbf{b}$, assume without loss of generality that $a_0b_1 -a_1b_0$ is invertible. $\alpha(a_0b_1 -a_1b_0)$ is invertible, therefore $\alpha(a_0)\alpha(b_1) -\alpha(a_1)\alpha(b_0)$ is invertible, hence $\alpha(\mathbf{a})\#\alpha(\mathbf{b})$. Hence, these morphisms preserve $\#$ on points and lines.

Given a point $\pt{a}$ belonging to a line $\pt{\lambda}$, $\Sum \lambda_i a_i=0$, therefore $\Sum \alpha(\lambda_i)\alpha(a_i)=0$. Therefore, $\alpha(\mathbf{a})$ lies on $\alpha(\boldsymbol{\lambda})$. Given a point $\pt{a}$ lying outside a line $\pt{\lambda}$, $\Sum \lambda_i a_i$ is invertible, therefore $\Sum \alpha(\lambda_i)\alpha(a_i)$ is invertible. Therefore, $\alpha(\mathbf{a})$ lies on $\alpha(\boldsymbol{\lambda})$. Hence, these morphisms preserve the relations $\in$ and $\notin$.

Hence, these morphisms on points and lines give a homomorphism between the two projective planes.
\end{proof}

\begin{lem} \label{lemringpart}
 Let $\phi: \mathbb P(R) \to \mathbb P(S)$ be a morphism of projective planes that sends the points $(1,0,0)$, $(0,1,0)$, $(0,0,1)$, $(1,1,1)$ to the points $(1,0,0)$, $(0,1,0)$, $(0,0,1)$, $(1,1,1)$ respectively. Then, there exists a unique ring homomorphism $\sigma :R \to S$ such that $\phi=\mathbb P(\sigma)$.
\end{lem}

\begin{proof}
First, notice that any such ring homomorphism $\sigma$ must be unique because $\mathbb{P}(\sigma)(a,0,1)=(\sigma(a),0,1)$.

In the proof that follows, we will repeatedly use arguments of the form: given points $A\#B$, and $l$ the unique line through $A$ and $B$, then $\phi(A)\#\phi(B)$ and $\phi(l)$ is the unique line through $\phi(A)$ and $\phi(B)$ (and the dual version of this).

By using the above arguments to the unique lines through any pair of the points $(1,0,0)$, $(0,1,0)$, $(0,0,1)$ and $(1,1,1)$ we conclude that $\phi$ maps the lines $(1,0,0)$, $(0,1,0)$, $(0,0,1)$, $(0,1,-1)$, $(1,0,-1)$, $(1,-1,0)$ to the lines $(1,0,0)$, $(0,1,0)$, $(0,0,1)$, $(0,1,-1)$, $(1,0,-1)$, $(1,-1,0)$ respectively.

By considering further intersections of pairs of the above lines, we conclude that $\phi$ maps the points $(1,0,1)$, $(0,1,1)$ to $(1,0,1)$, $(0,1,1)$ respectively.

For $a\in R$, $(a,0,1)\in(0,1,0)$ and $(a,0,1)\notin (0,0,1)$ therefore $\phi(a,0,1)\in (0,1,0)$ and $\phi(a,0,1)\notin (0,0,1)$. Hence, $\phi(a,0,1)=(\sigma(a),0,1)$, for some $\sigma(a)\in S$. We shall think of $\sigma$ as a function $\sigma:R\to S$.
 
The points $(a,0,1)$ and $(0,1,0)$ are mapped to the points $(\sigma(a),0,1)$ and $(0,1,0)$ respectively, therefore the line $(1,0,-a)$ is mapped to $(1,0,-\sigma(a))$.

$(a,a,1)$ is the unique intersection point of the lines $(1,0,-a)$ and $(1,-1,0)$, therefore it is mapped to $(\sigma(a),\sigma(a),1)$.

$(0,1,-a)$ is the unique line through $(a,a,1)$ and $(1,0,0)$, therefore it is mapped to $(0,1,-\sigma(a))$.

$(0,a,1)$ is the unique intersection point of the lines $(0,1,-a)$ and $(1,0,0)$ therefore it is mapped to $(0,\sigma(a),1)$.

$(a,b,1)$ is the unique intersection point of the lines $(1,0,-a)$ and $(0,1,-b)$, therefore it is mapped to $(\sigma(a),\sigma(b),1)$ (the intersection of $(1,0,-\sigma(a))$ and $(0,1,-\sigma(b))$.

\vspace{0.1 in}
\textbf{Hence, we have constructed a function $\sigma:R\to S$, such that $\phi(a,b,1)=(\sigma(a),\sigma(b),1)$ and we shall now prove that $\sigma$ is a ring homomorphism.}
\vspace{0.1 in}

Notice that $\sigma(0)=0$ because $\phi(0,0,1)=(0,0,1)$ and $\sigma(1)=1$ because $\phi(1,1,1)=(1,1,1)$.

\vspace{0.1 in}
\textbf{Proving that $\sigma$ commutes with addition:}

Given the points $(a,0,1)$ and $(b,0,1)$, we construct the point $(a+b,0,1)$ geometrically in the following way.
The line $(1,0,-a)$ is the unique line through the points $(a,0,1)$ and $(0,1,0)$. 
The point $(a,1,1)$ is the unique intersection point of the lines $(1,0,-a)$ and $(0,1,-1)$. 

The line $(1,b,-b)$ is the unique line through $(b,0,1)$ and $(0,1,1)$.
The point $(-b,1,0)$ is the unique intersection of $(1,b,-b)$ and $(0,0,1)$.

The line $(1,b,-a-b)$ is the unique line through $(a,1,1)$ and $(-b,1,0)$.
And finally, $(a+b,0,1)$ is the unique intersection point of $(0,1,0)$ and $(1,b,-a-b)$.

\begin{center}
  \begin{tikzpicture}
    \draw 
    (0,0) node[left] {$(0,0,1)$} -- (0,3)
    (0,0) -- (6,0)
    (0,2.5) node[left] {$(0,1,1)$} -- (6,2.5)

    (2,0) -- (2,2.5) node[above] {$(a,1,1)$}
    (3,0) -- (0,2.5)
    (2,2.5)-- (5,0) node[below] {$(a+b,0,1)$}
    (1.75,0) node[below] {$(a,0,1)$}
    (3.25,0) node[below] {$(b,0,1)$};
  \end{tikzpicture}
\end{center}

If we replace $a$ with $\sigma(a)$, and $b$ with $\sigma(b)$ in the above construction we end up with the point $(\sigma(a)+\sigma(b),0,1)$ of $\mathbb P(S)$. And if we replace all the lines and points mentioned above with their images through $\phi$ we end up with the point $\phi((a+b,0,1))=(\sigma(a+b),0,1)$ because $\phi$  preserves intersection points of lines that are apart from each other and lines through points that apart from each other. But the effect of $\phi$ in the above construction is exactly replacing $a$ with $\sigma(a)$ and $b$ with $\sigma(b)$ (because of the points that $\phi$ preserves). Hence, $(\sigma(a+b),0,1)=(\sigma(a)+\sigma(b),0,1)$, i.e. $\sigma(a+b)=\sigma(a)+\sigma(b)$.

\vspace{0.1 in}
\textbf{Proving that $\sigma$ commutes with multiplication:}

Similarly, it suffices to give an appropriate geometric construction of $(ab,0,1)$ from the points $(a,0,1)$ and $(b,0,1)$ to prove that $\sigma(ab)=\sigma(a)\sigma(b)$. It is done in the following way.

$(1,0,-a)$ is the unique line through $(a,0,1)$ and $(0,1,0)$.
$(a,a,1)$ is the unique intersection point of $(1,-1,0)$ and $(1,0,-a)$.

$(1,b-1,-b)$ is the unique line through $(1,1,1)$ and $(b,0,1)$.
$(1-b,1,0)$ it the unique intersection point of $(1,b-1,-b)$ and $(0,0,1)$.

$(1,b-1,-ab)$ is the unique line through $(1-b,1,0)$ and $(a,a,1)$.
And finally, $(ab,0,1)$ is the unique intersection point of $(0,1,0)$ and $(1,b-1,-ab)$.

\begin{center}
  \begin{tikzpicture}
    \draw 
    (0,0) node[left] {$(0,0,1)$} -- (0,3)
    (0,0) -- (7,0)
    (0,0) -- (3,3)

    (2,0) -- (2,2) node[above] {$(a,a,1)$}
    (3,0) -- (1,1) node[above] {$(1,1,1)$}
    (2,2)-- (6,0) node[below] {$(ab,0,1)$}
    (1.75,0) node[below] {$(a,0,1)$}
    (3.25,0) node[below] {$(b,0,1)$};
  \end{tikzpicture}
\end{center}

Hence, $\sigma(ab)=\sigma(a)\sigma(b)$.

\vspace{0.1 in}
\textbf{Proving that $\phi=\mathbb P(\sigma)$:}

By arguments symmetric to the ones above, we can prove that $\phi(a,1,b))=(\tau(a),1,\tau(b))$ for some function $\tau:R\to S$. By symmetry and the results about $\sigma$, we conclude that $\tau(0)=0$ and $\tau(1)=1$. Also, $\phi(1,a,1)=(1,\sigma(a),1)=(1,\tau(a),1)$, hence $\sigma(a)=\tau(a)$ for all $a$ in $R$. Similarly, $\phi(1,a,b)=(1,\sigma(a),\sigma(b))$. Finally, using the fact that $\sigma$ preserves multiplication, we conclude that for every point $\pt{a}$ of the projective plane, $\phi\pt{a}=(\sigma(a_0),\sigma(a_1),\sigma(a_2))$.
\end{proof}

\begin{rmk}
Notice that using the above proof we can also prove the following: Given the projective plane over a ring $R$ and the points $(1,0,0)$, $(0,1,0)$, $(0,0,1)$, $(1,1,1)$ and forgetting the coordinates of any points we can reconstruct a ring isomorphic to $R$. First, we construct the line $(0,1,0)$ and the point $(1,0,1)$. The underlying set of the ring is going to be the set of points on the line $(0,1,0)$  which are apart from the point $(1,0,0)$. We can prove that this set is $\{ (x,0,1) | x\in R\}$ which is clearly isomorphic to $R$. We define the additive unit of the ring to be $(0,0,1)$ and we define the multiplicative unit to be $(1,0,1)$ which matches the two units of $R$ via the isomorphism.  Given two points of this set we construct their sum and their product in a synthetic/geometric way, as in the above proof. This addition and multiplication commute with the isomorphism to $R$. Hence, the structure we have defined is isomorphic to the ring $R$.
\end{rmk}

Notice that up to this point we have not made any assumptions on the ring $R$.

Let us now consider a local ring $R$ and let $M$ be an invertible $3\times 3$ matrix over $R$. Let $\mathbf{a}=(a_0,a_1,a_2)\in R^3$ have an invertible coordinate. $M^{-1}M \mathbf{a}=\mathbf{a}$ has an invertible coordinate. $R$ is a local ring, therefore when a sum is invertible one of its summands must be invertible, hence $M\mathbf{a}$ must have an invertible coordinate.

Given $\mathbf{a}$ and $\mathbf{a'}$ in $R^3$ representing the same point, then $\mathbf{a'}=r\mathbf{a}$ for some $r\in R$. Hence, $M\mathbf{a'}=r M\mathbf{a}$, and therefore $M\mathbf{a'}$ and $M\mathbf{a}$ represent the same point in $\mathbb P(R)$.

Hence $M$ determines a morphism from the set of points of the projective plane to itself.

$(M^{-1})^T$ is also an invertible $3\times 3$ matrix, therefore by the duality principle it determines a morphism from the set of lines of the projective plane to itself.

\begin{prop} \label{propmatrproj}
 Any invertible $3\times 3$ matrix $M$ over a \emph{local} ring $R$ determines an automorphism of the projective plane over $R$ (in the way described above). 
\end{prop}

\begin{proof}
Let $M$ be an invertible $3\times 3$ matrix over $R$. Given $\pt{a}$, $\pt{b}$ in $R^3$ such that $\mathbf{a}\#\mathbf{b}$, assume without loss of generality that $a_0b_1 -a_1b_0$ is invertible. Therefore, the determinant of $\begin{pmatrix}
a_0 & b_0 & 0 \\
a_1 & b_1 & 0 \\
a_2 & b_2 & 1  
\end{pmatrix}$ is invertible. 
$$\begin{pmatrix}
  M\begin{pmatrix} a_0 \\ a_1 \\a_2 \end{pmatrix} & M\begin{pmatrix}b_0 \\ b_1 \\b_2 \end{pmatrix} & M\begin{pmatrix} 0 \\ 0 \\1 \end{pmatrix}
   \end{pmatrix}=
M
\begin{pmatrix}
a_0 & b_0 & 0 \\
a_1 & b_1 & 0 \\
a_2 & b_2 & 1  
\end{pmatrix},$$ therefore its determinant is also invertible. Hence,
$$\begin{pmatrix}
  M\begin{pmatrix} a_0 \\ a_1 \\a_2 \end{pmatrix} & M\begin{pmatrix}b_0 \\ b_1 \\b_2 \end{pmatrix}
\end{pmatrix}$$
has a $2\times 2$ minor with invertible determinant because $R$ is a local ring. Therefore, $M\mathbf{a}\# M\mathbf{b}$.

The same proof shows that lines that are apart from each other are mapped to lines that are apart from each other.

Given a point $A$ represented by $\pt{a}$ and a line $l$ represented by $\pt{\lambda}$, $A$ lies on $l$ iff $\boldsymbol{\lambda}^T \mathbf{a}=0$. Suppose $A\in l$, then \begin{displaymath}
\begin{split}
((M^{-1})^T\boldsymbol{\lambda})^T M\mathbf{a}= \boldsymbol{\lambda}^T M^{-1} M\mathbf{a}= \boldsymbol{\lambda}^T \mathbf{a}=0,
\end{split}
\end{displaymath}
therefore $M\mathbf{a}$ lies on $M\boldsymbol{\lambda}$.
 
$A$ lies outside $l$ iff $\boldsymbol{\lambda}^T \mathbf{a}$ is invertible. Suppose $A\notin l$, then 
$$((M^{-1})^T\boldsymbol{\lambda})^T M\mathbf{a}= \boldsymbol{\lambda}^T M^{-1} M\mathbf{a}= \boldsymbol{\lambda}^T \mathbf{a}$$ 
is invertible, therefore $M\mathbf{a}$ lies outside $M\boldsymbol{\lambda}$.

Hence, the described morphisms preserve the structure of projective plane over $R$. Furthermore, this is an automorphism of $\mathbb P(R)$ because $M^{-1}$ induces the inverse of this morphism.
\end{proof}

\begin{rmk}
Note that the above proposition is not necessarily true for rings that are not local. For example, consider the invertible matrix $M=\begin{pmatrix}
3 & 0 & 1 \\
0 & 1 & 0 \\
2 & 0 & 1  
\end{pmatrix}$ in $\mathbb Z/(6)$. This does not define a function from the set of points of the plane to itself (in the way explained above) because for example $M(1,0,0)=(3,0,2)$ that does not have an invertible coordinate.
\end{rmk}

Notice that for $M$ a $3\times 3$ matrix over a local ring $R$ and $\lambda$ an invertible element of $R$, the matrices $M$ and $\lambda M$ induce the same automorphism of $\mathbb P(R)$.

\begin{defn}
The \emph{projective general linear group} over a local ring $R$ is the group of invertible $3\times 3$ matrices quotiented by scalar multiplication by an invertible element of $R$. In this thesis, we denote this group as $H(R)$.
\end{defn}

\begin{defn}
Four points of a preprojective plane or a projective plane over a ring are \emph{in general position} when any  subset of three of them is a non-collinear triple.
\end{defn}

\begin{egs}
For $R$ a ring, the points $(1,0,0)$, $(0,1,0)$, $(0,0,1)$ and $(1,1,1)$ are points in general position.
\end{egs}

\begin{lem} \label{lemmatrixauto}
Let $R$ be a local ring and let $A$, $B$, $C$, $D$ be points in general position of $\mathbb P(R)$. Then, there exists a unique element of $H(R)$ inducing an automorphism of $\mathbb P(R)$ sending $(1,0,0)$ to $A$, $(0,1,0)$ to $B$, $(0,0,1)$ to $C$ and $(1,1,1)$ to $D$.
\end{lem}

\begin{proof}
Let $(a_0,a_1,a_2)$, $(b_0,b_1,b_2)$, $(c_0,c_1,c_2)$ and $(d_0,d_1,d_2)$ be representatives of $A$, $B$, $C$ and $D$. Then, $M$ must be of the form
\[
\begin{pmatrix}
\lambda a_0 & \mu b_0 & \nu c_0 \\
\lambda a_1 & \mu b_1 & \nu c_1 \\
\lambda a_2 & \mu b_2 & \nu c_2  
\end{pmatrix}
\]
for some invertible $\lambda, \mu$ and $\nu$ in $R$. Note that the determinant of $M$ is invertible because $A$, $B$, $C$ are non-collinear (and $\lambda$, $\mu$ and $\nu$ are invertible).

It is now sufficient to prove that there exist invertible $\lambda, \mu$ and $\nu$ in $R$ such that $M(1,1,1)=\lambda A +\mu B +\nu C=D$, i.e. 
\[
\begin{pmatrix}
a_0 & b_0 & c_0 \\
a_1 & b_1 & c_1 \\
a_2 & b_2 & c_2
\end{pmatrix}
\begin{pmatrix}
\lambda \\
\mu \\
\nu
\end{pmatrix}
=D.
\]
Therefore,  
\[
 \begin{pmatrix}
\lambda \\
\mu \\
\nu
\end{pmatrix}
=
\begin{pmatrix}
a_0 & b_0 & c_0 \\
a_1 & b_1 & c_1 \\
a_2 & b_2 & c_2
\end{pmatrix}^{-1}
D.
\]

Notice that
$
\begin{pmatrix}
b_0 & c_0 & d_0 \\
b_1 & c_1 & d_1 \\
b_2 & c_2 & d_2
\end{pmatrix} 
\begin{pmatrix}
 \mu \\
\nu \\
-1
\end{pmatrix}
=\lambda
\begin{pmatrix}
a_0 \\ a_1 \\ a_2
\end{pmatrix}
$. The matrix is invertible, therefore it maps the point $(\mu,\nu, -1)$ to a point of the projective plane. Hence, $\lambda A$ has an invertible component, therefore $\lambda$ is invertible. By symmetric arguments, $\mu$ and $\nu$ are also invertible.

Suppose a matrix $M'$ represents the same automorphism. Then, again by considering the images of $(1,0,0)$, $(0,1,0)$ and $(0,0,1)$ , we see that $M'$ is of the form 
\[
\begin{pmatrix}
\lambda' a_0 & \mu' b_0 & \nu' c_0 \\
\lambda' a_1 & \mu' b_1 & \nu' c_1 \\
\lambda' a_2 & \mu' b_2 & \nu' c_2  
\end{pmatrix}
\]
$N'(1,1,1)=\xi D$, therefore as before
\[
 \begin{pmatrix}
\lambda' \\
\mu' \\
\nu'
\end{pmatrix}
=
\begin{pmatrix}
a_0 & b_0 & c_0 \\
a_1 & b_1 & c_1 \\
a_2 & b_2 & c_2
\end{pmatrix}^{-1}
\xi D
=
\xi
 \begin{pmatrix}
\lambda\\
\mu \\
\nu
\end{pmatrix}
\]
Hence, 
\[
N=
\begin{pmatrix}
\xi \lambda a_0 & \xi \mu b_0 & \xi \nu c_0 \\
\xi \lambda a_1 & \xi \mu b_1 & \xi \nu c_1 \\
\xi \lambda a_2 & \xi \mu b_2 & \xi \nu c_2  
\end{pmatrix}
=\xi M
\]
and therefore $M$ and $N$ represent the same element of $H(R)$ and this element does not depend on the choice of representatives of $A$, $B$, $C$ and $D$.
\end{proof}

\begin{rmk}
Some condition on the ring $R$ is necessary for the theorem to hold, since it does not hold on the projective plane over $\mathbb Z/(6)$. Consider the points $A=(1,0,0)$, $B=(3,-1,0)$, $C=(3,2,1)$ and $D=(1,1,1)$ of $\mathbb P(\mathbb Z/(6))$ and notice that they are in general position. By going through the proof above we see that the matrix $M$ is necessarily (a multiple by a unit of) 
$\begin{pmatrix}
1 & 3 &3 \\
0 & -1 & 2 \\
0 & 0 & 1
\end{pmatrix}$. But this matrix sends the point $(3,1,2)$ to $(0,0,2)$ which is not a point, therefore it does not even represent an endomorphism of the projective plane over $\mathbb Z/(6)$.
\end{rmk}

Given a local ring $R$, let $\omega_4(R)$ be the set of quadruples of points in general position of $\mathbb P(R)$. The left $H(R)$-action on points of $\mathbb P(R)$ sends non-collinear points to non-collinear points, therefore it extends to a left action on $\omega_4(R)$.

\begin{thrm}
Given a local ring $R$, $\omega_4(R)$ is a left $H(R)$-torsor via the action described above.
\end{thrm}

\begin{proof}
Given $(A,B,C,D)$ and $(A',B',C',D')$ in $\omega_4(R)$ by Lemma \ref{lemmatrixauto} there exist unique $g$ and $g'$ in $H(R)$ such that $g$  and $g'$ send $\begin{pmatrix}
\begin{pmatrix} 1 \\ 0 \\ 0 \end{pmatrix}, \begin{pmatrix} 0 \\ 1 \\ 0 \end{pmatrix},\begin{pmatrix} 0 \\ 0 \\1 \end{pmatrix}, \begin{pmatrix} 1 \\ 1 \\ 1 \end{pmatrix}
\end{pmatrix}$ to $(A,B,C,D)$ and $(A',B',C',D')$ respectively.

Then, $g'  g^{-1}$ sends $(A,B,C,D)$ to $(A',B',C',D')$, therefore this $H(R)$-action is transitive.

Suppose that $h$ in $H(R)$ sends $(A,B,C,D)$ to $(A',B',C',D')$. Then, $h  g$ sends $\begin{pmatrix}
\begin{pmatrix} 1 \\ 0 \\ 0 \end{pmatrix},  \begin{pmatrix} 0 \\ 1 \\ 0 \end{pmatrix},  \begin{pmatrix} 0 \\ 0 \\1 \end{pmatrix},  \begin{pmatrix} 1 \\ 1 \\ 1 \end{pmatrix}
\end{pmatrix}$ to $(A',B',C',D')$. Hence, $h  g=g'$  by Lemma \ref{lemmatrixauto} and therefore $h=g'  g^{-1}$. Thus, $\omega_4(R)$ is an $H(R)$-torsor under this action.
\end{proof}

\begin{rmk}
Notice that we have an isomorphism $H(R)\to \omega_4$ which maps $h$ to the quadruple of points $\begin{pmatrix}
h\begin{pmatrix} 1 \\ 0 \\ 0 \end{pmatrix},& h\begin{pmatrix} 0 \\ 1 \\ 0 \end{pmatrix},& h\begin{pmatrix} 0 \\ 0 \\1 \end{pmatrix},&  h\begin{pmatrix} 1 \\ 1 \\ 1 \end{pmatrix}
\end{pmatrix}$.
Moreover, this isomorphism commutes with left $H(R)$-action on $H(R)$ via group multiplication and the left $H(R)$-action on $\omega_4(R)$ described above.
\end{rmk}

The following theorem completely describes the morphisms between projective planes over local rings.

\begin{thrm} \label{thrmprojmorph}
 Let $\phi: \mathbb P(R) \to \mathbb P(S)$ be a morphism of projective planes over the rings $R$ and $S$, where $S$ is local. Then, there exists a unique $g$ in $H(S)$ and a unique ring homomorphism $\alpha: R\to S$ such that $\phi =g \circ \mathbb P(\alpha)$.
\end{thrm}

\begin{proof}
$\phi$ sends the points $(1,0,0)$, $(0,1,0)$, $(0,0,1)$ and $(1,1,1)$ to the points $A$, $B$, $C$ and $D$ respectively. $\phi$ is a morphism of projective planes, therefore any three of the points $A$, $B$, $C$ and $D$  are non-collinear. By Lemma \ref{lemmatrixauto}, there exists a unique $g$ in $H(S)$ which induces an automorphism of the projective plane over $S$ which sending $(1,0,0)$, $(0,1,0)$, $(0,0,1)$ and $(1,1,1)$ to the points $A$, $B$, $C$ and $D$ respectively. $g^{-1}$ induces the inverse automorphism.

$g^{-1}\phi$ satisfies the conditions of Lemma \ref{lemringpart}, so it is of he form $\mathbb{P}(\alpha)$ for a unique ring homomorphism $\alpha:R\to S$. Hence, $\phi=g \circ \mathbb{P}(\alpha)$.

Suppose $\phi$ is also equal to $g'\circ \mathbb{P}(\alpha')$. Then, $g'$ sends $(1,0,0)$, $(0,1,0)$, $(0,0,1)$ and $(1,1,1)$ to the points $A$, $B$, $C$ and $D$, therefore $g=g'$. $g$ is an isomorphism, therefore $\mathbb P(\alpha)=\mathbb P(\alpha')$ and hence $\alpha=\alpha'$.
\end{proof}

\begin{rmk}
Let $\phi:\mathbb P(R)\to \mathbb P(S)$ be a morphism of projective planes such that $\phi=g\circ \mathbb P(\alpha)$ where $g$ is in $\mathbb H(S)$ and $\alpha:R\to S$ a ring homomorphisms. Let $\psi: \mathbb P(S) \to \mathbb P(T)$ be a second morphism of projective planes such that and $\psi=k\circ \mathbb P(\beta)$ where $g$ is in $H(T)$ and $\beta:S\to T$ is a ring homomorphism.Then, 
\begin{displaymath}
\begin{split}
\psi\circ \phi & = k\circ \mathbb P(\beta)\circ g \circ \mathbb P(\alpha) \\
& = k \circ (\beta(g)) \circ \mathbb P(\beta) \circ \mathbb P(\alpha) \\
& = (k\circ \beta(g)) \circ \mathbb P(\beta\circ \alpha),
\end{split}
\end{displaymath}
where $\beta(g)$ is the image of $g$ of $H(S)$ under $\beta$.
\end{rmk}

\section{Desargues' theorem on the projective plane} \label{secdes}

In this section we present our version Desargues' theorem which holds on projective planes over local rings. This version of Desargues' theorem can be written as a geometric sequent (in the language of preprojective planes) and is added as an axiom of the theory of projective planes in Section \ref{secthryproj}. In Theorem \ref{thrmsmallDes} and Theorem \ref{thrmbigDes} we use Desargues' theorem on the projective plane to prove the small and the big Desargues' theorems on the affine plane. The proofs of these theorems might be useful to a reader who wishes to understand the connection between this new version of Desargues' theorem and the ones that appears in classical treatments of the subject as in \cite{Harts}.

\begin{defn}
Given two lines $k$, $l$ and two points $A$, $B$ of a preprojective plane we say that they satisfy $\delta(k,l,A,B)$ when there exists a line $r$ and a point $X$ such that $X$ lies on each of the lines $k$, $l$ and $r$, and each of the points $A$, $B$ and $X$ lie on $r$.
\end{defn}

$\delta(k,l,A,B)$ can be written as a geometric formula in the following way 
$$\exists r \exists X. (A,B,X\in r) \wedge (X \in k,l,r).$$

\begin{center}
\begin{tikzpicture}
 \draw 
(0,0) node[left] {$A$} 
-- (2,2) node[above] {$B$} 

(0.5,0.5) node[left] {$r$}

(0.8,-0.8)
-- (1,1) node[above] {$X$}
-- (2.4,0.8)
(2,1) node {$l$}
(1,0) node {$k$} ;
\end{tikzpicture}
\end{center}

\begin{prop}
For $k$, $l$ lines and $A$, $B$ points of a preprojective plane the following hold:
\begin{enumerate}
\item $\delta(k,l,A,B) \vdash_{k,l,A,B} \delta(l,k,A,B)$.
\item $\delta(k,l,A,B) \vdash_{k,l,A,B} \delta(k,l,B,A)$.
 \item $\top \vdash_{k,A} \delta(k,k,A,A)$.
\item $A\#B \wedge (A\notin k \vee B\notin k) \vdash_{k,A,B} \delta(k,k,A,B)$.
\item $k\#l \wedge (A\notin k \vee A\notin l) \vdash_{k,l,A} \delta(k,l,A,A)$.
\end{enumerate}
\end{prop}

\begin{proof}
 1 and 2 are clearly true by the symmetry in the definition of $\delta$.

3. Let $k$ be a line and let $A$ be a point of a preprojective plane. There exist points $X$, $Y$ that are apart from each other and lie on $k$. $A$ is apart from at least one of $X$ and $Y$. Without loss of generality let us assume that $A\# X$. Then, $A$ and $X$ lie on $\ovl{AX}$ and $X$ lies on both $k$ and $\ovl{AX}$, therefore $\delta(k,k,A,A)$ is satisfied.

4. Let $A$, $B$ be points of the preprojective plane that are apart from each other and let $k$ be a line such that at least one of $A$ and $B$ lies outside $k$. Then, $\ovl{AB}$ is apart from $k$. Let $X$ be the intersection of $k$ and $\ovl{AB}$. $A$, $B$ and $X$ lie on $\ovl{AB}$, and $X$ lies on both $k$ and $\ovl{AB}$, hence $\delta(k,k,A,B)$ is satisfied. 

5 is true by the duality principle because it is dual to 4.
\end{proof}

\begin{rmk}
 Notice that given a line $k$ and two points $A$ and $B$ of a preprojective plane, $\delta(k,k,A,B)$ is not necessarily true because there might not be a line passing through both $A$ and $B$, and even in the case where $A\#B$ there might be no point lying on both $k$ and $\ovl{AB}$.

Dually, given two lines $k$ and $l$, and a point $A$ of a preprojective plane, $\delta(k,l,A,A)$ is not necessarily true because there might not be a point lying on both $k$ and $l$, and even in the case where $k\#l$ there might be no line through both $k\cap l$ and $A$.
\end{rmk}

\begin{lem} \label{lemS}
Given a local ring $R$, let $k=\pt{\kappa}$, $l=\pt{\lambda}$ be lines and let $A=\pt{a}$, $B=\pt{b}$ be points of $\mathbb P(R)$. If $\delta(k,l,A,B)$ is satisfied, then the determinant of the product

$$\begin{pmatrix}
 \kappa_0 & \kappa_1 & \kappa_2
\\ \lambda_0 & \lambda_1 & \lambda_2
\end{pmatrix}
\begin{pmatrix}
 a_0 & b_0 \\
a_1 & b_1 \\
a_2 & b_2
\end{pmatrix}$$
is $0$.
\end{lem}

\begin{proof}
In the following proof we write $\kappa_a$ for the inner product $\pt{\kappa}\cdot \pt{a}$ and similarly $\lambda_a=\pt{\lambda}\cdot \pt{a}$, $\kappa_b= \pt{\kappa} \cdot \pt{b}$, and $\lambda_b=\pt{\lambda} \cdot \pt{b}$.

Let $k$, $l$, $A$ and $B$ be as above and let $X$ be a point lying on both $k$ and $l$, and let $r$ be a line passing through all $A$, $B$ and $X$.  Let $X=\pt{x}$ and $r=\pt{r}$.

$k$, $l$ and $r$ pass through a common point $X$ therefore by Proposition \ref{propdetcol} the determinant of the matrix
$\begin{pmatrix}
\kappa_0 & \kappa_1 & \kappa_2 \\
\lambda_0 & \lambda_1 & \lambda_2 \\
r_0 & r_1 & r_2
\end{pmatrix}$ is $0$.

One of the coordinates of $\pt{r}$ is invertible. Without loss of generality, let us assume that $r_0$ is invertible since the other two cases are symmetric. Consider the product
$$\begin{pmatrix}
\kappa_0 & \kappa_1 & \kappa_2 \\
\lambda_0 & \lambda_1 & \lambda_2 \\
r_0 & r_1 & r_2
\end{pmatrix}
\begin{pmatrix}
 a_0 & b_0 & 1 \\
a_1 & b_1  & 0\\
a_2 & b_2  & 0 
\end{pmatrix}
= \begin{pmatrix}
\kappa_a & \kappa_b & \kappa_0 \\
\lambda_a & \lambda_b & \lambda_0 \\
0 & 0 & r_0
\end{pmatrix}$$
and observe that the determinant of the left hand side is $0$, therefore the determinant of the right hand side is also $0$. Hence, $r_0 (\kappa_a \lambda_b - \kappa_b \lambda_a)=0$, and since $r_0$ is invertible we conclude that $\kappa_a \lambda_b - \kappa_b \lambda_a = 0$. $\kappa_a \lambda_b - \kappa_b \lambda_a$ is the determinant of the product 
$\begin{pmatrix}
 \kappa_0 & \kappa_1 & \kappa_2
\\ \lambda_0 & \lambda_1 & \lambda_2
\end{pmatrix}
\begin{pmatrix}
 a_0 & b_0 \\
a_1 & b_1 \\
a_2 & b_2
\end{pmatrix}$,
hence the result.
\end{proof}

\begin{lem} \label{lemSiff}
Given a local ring $R$, let $k=\pt{\kappa}$, $l=\pt{\lambda}$ be lines and let $A=\pt{a}$, $B=\pt{b}$ be points of $\mathbb P(R)$, such that either $k\#l$ or $A\#B$, and at least one of the points $A$ and $B$ lies outside at least one of the lines $k$ and $l$. Then $\delta(k,l,A,B)$ is satisfied iff the determinant of the product

$$\begin{pmatrix}
 \kappa_0 & \kappa_1 & \kappa_2
\\ \lambda_0 & \lambda_1 & \lambda_2
\end{pmatrix}
\begin{pmatrix}
 a_0 & b_0 \\
a_1 & b_1 \\
a_2 & b_2
\end{pmatrix}$$
is $0$.
\end{lem}

\begin{proof}
The direct implication of the statement was already proved in the previous lemma. 

For the converse implication, let us suppose that the determinant of the product of matrices described above is $0$.  Let us define $\kappa_a$, $\lambda_a$, $\kappa_b$ and $\lambda_b$ as in the proof of the previous lemma. We only consider the case where $A\#B$ since the case where $k\# l$ is dual. Given that $A$ is apart from $B$, there exists a unique line $r=\ovl{AB}$ passing through both $A$ and $B$, with coordinates $\pt{r}=(a_2b_1-a_1b_2, a_0b_2 -a_2b_0,a_1b_0-a_0b_1)$. Without loss of generality, let us assume that $r_0=a_2b_1-a_1b_2$ is invertible. Let us consider the product 
$$\begin{pmatrix}                                                                                                                                                                                                                                                                                                                                                                                                                                                                                                                                                                                                                                                                                                                                                                                                                                                                                            
\kappa_0 & \kappa_1 & \kappa_2 \\
\lambda_0 & \lambda_1 & \lambda_2 \\
r_0 & r_1 & r_2
\end{pmatrix}
\begin{pmatrix}
a_0 & b_0 & 1\\
a_1 & b_1 & 0\\
a_2 & b_2 & 0
\end{pmatrix}
=
\begin{pmatrix}
 \kappa_a & \kappa_b & \kappa_0 \\
\lambda_a & \lambda_b & \lambda_0 \\
0 & 0 & r_0
\end{pmatrix}.
$$
Determinants commute with matrix multiplication therefore 
$$r_0 \det
\begin{pmatrix}                                                                                                                                                                                                                                                                                                                                                                                                                                                                                                                                                                                                                                                                                                                                                                                                                                                                                            
\kappa_0 & \kappa_1 & \kappa_2 \\
\lambda_0 & \lambda_1 & \lambda_2 \\
r_0 & r_1 & r_2
\end{pmatrix}
= r_0
\det
\begin{pmatrix}
 \kappa_a & \kappa_b \\
\lambda_a & \lambda_b 
\end{pmatrix}$$
and since $r_0$ is invertible we conclude that
$$\det
\begin{pmatrix}                                                                                                                                                                                                                                                                                                                                                                                                                                                                                                                                                                                                                                                                                                                                                                                                                                                                                            
\kappa_0 & \kappa_1 & \kappa_2 \\
\lambda_0 & \lambda_1 & \lambda_2 \\
r_0 & r_1 & r_2
\end{pmatrix}
= \det
\begin{pmatrix}
 \kappa_a & \kappa_b \\
\lambda_a & \lambda_b 
\end{pmatrix}.$$
The right hand side above is equal to the determinant of the product $$\begin{pmatrix}
 \kappa_0 & \kappa_1 & \kappa_2
\\ \lambda_0 & \lambda_1 & \lambda_2
\end{pmatrix}
\begin{pmatrix}
 a_0 & b_0 \\
a_1 & b_1 \\
a_2 & b_2
\end{pmatrix}$$
and therefore is equal to $0$. Hence,
$$\det
\begin{pmatrix}                                                                                                                                                                                                                                                                                                                                                                                                                                                                                                                                                                                                                                                                                                                                                                                                                                                                                            
\kappa_0 & \kappa_1 & \kappa_2 \\
\lambda_0 & \lambda_1 & \lambda_2 \\
r_0 & r_1 & r_2
\end{pmatrix}=0.$$

At least one of $A$ and $B$ lies outside at least one of $k$ and $l$. Without loss of generality, let us assume that $A\notin k$. Then, either $A\notin r$ or $k\# r$. $A$ lies on $r$, therefore $k$ is apart from $r$. Hence, by Proposition \ref{propiffcollinear}, there exists a point lying on all three lines, and therefore $\delta(k,l,A,B)$.
\end{proof}

\begin{rmk}
The extra conditions we have added above, make sure of the uniqueness of $X$ and $r$ such that all $A$, $B$ and $X$ lie on  all $k$, $l$ and $r$.
\end{rmk}

Before presenting the full version of Desargues' theorem we prove the following lemma which includes the cases we need for the full version of Desargues' theorem.

\begin{lem} \label{lempredes}
Let $R$ be a local ring and let $A$, $B$, $C$, $D$ be points, and $k$, $l$, $m$, $n$ lines of $\mathbb P(R)$ such that $\delta(k,l,A,B)$, $\delta(l,m,B,C)$, $\delta(m,n,C,D)$, $\delta(n,k,D,A)$, $\delta(k,m,B,D)$, and such that either $l\#n$ or $A\#C$, and at least one of the points $A$, $C$ lies outside at least one of the lines $l$, $n$. Then $\delta(l,n,A,C)$ is satisfied as long as any one of the following conditions hold:
\begin{enumerate}
 \item $B\notin k$ and $D\notin m$, \label{itemdescase1}
\item $D\notin k$, $D\notin m$ and $B\notin l$, \label{itemdescase2}
\item $B\notin k$, $D\notin k$ and $C\notin m$, \label{itemdescase3}
\item $D\notin k$, $C\notin m$ and $B\notin l$, \label{itemdescase4}
\item $A\notin k$, $B\notin l$, $D\notin n$ and $C\notin m$. \label{itemdescase5}
\end{enumerate}

\end{lem}

\begin{center}
\begin{tikzpicture}
 \draw 
(0,0) node[left] {$A$} 
-- (2,2) node[above] {$B$} 
-- (4,0) node[right] {$C$}
-- (2,-2) node[below] {$D$}
-- (0,0)

(1.5,1.5)
-- (3,1)
-- (3.5,-0.5)
-- (1,-1)
-- (1.5,1.5)
(2,1) node {$l$}
(3,0) node {$m$}
(2,-1) node {$n$}
(1,0) node {$k$} ;
\end{tikzpicture}
\end{center}

\begin{proof}
In the following proof we write $\kappa_a$ for the inner product $\pt{\kappa}\cdot \pt{a}$ and similarly 
$\lambda_a=\pt{\lambda}\cdot \pt{a}$,
$\nu_a=\pt{\nu}\cdot \pt{a}$,
$\kappa_b= \pt{\kappa} \cdot \pt{b}$,
$\lambda_b=\pt{\lambda} \cdot \pt{b}$,
$\mu_b=\pt{\mu}\cdot \pt{b}$,
$\lambda_c=\pt{\lambda} \cdot \pt{c}$,
$\mu_c=\pt{\mu}\cdot \pt{c}$,
$\nu_c=\pt{\nu}\cdot \pt{c}$,
$\mu_d=\pt{\mu}\cdot \pt{d}$,
$\nu_d=\pt{\nu}\cdot \pt{d}$ and
$\kappa_d= \pt{\kappa} \cdot \pt{d}$.

 The quadruples satisfying the relation $S$ give us the following five equations by Lemma \ref{lemS}:
\begin{displaymath}
\begin{split}
\kappa_a \lambda_b & = \kappa_b \lambda_a \\
\lambda_b \mu_c & = \lambda_c \mu_b \\
\mu_c \nu_d & = \mu_d \nu_c \\
\nu_d \kappa_a & = \nu_a \kappa_d \\
\kappa_b \mu_d & = \kappa_d \mu_b
\end{split}
\end{displaymath}

By Lemma \ref{lemSiff}, to prove that $\delta(l,n,A,C)$ holds it is sufficient to prove that $\lambda_a \nu_c= \nu_a \lambda_c$.

\begin{enumerate}
 \item  In case \ref{itemdescase1} where $B\notin k$ and $D\notin m$, $\kappa_b$ and $\mu_d$ are invertible. By the above equations 
$$\lambda_a \nu_c \kappa_b \mu_d =
\lambda_b \nu_d \kappa_a \mu_c =
\lambda_c \nu_a \kappa_d \mu_b =
\lambda_c \nu_a \kappa_b \mu_d.$$
$\lambda_a \nu_c= \nu_a \lambda_c$ because $\kappa_b$ and $\mu_d$ are invertible, therefore $\delta(l,n,A,C)$.


\item In case \ref{itemdescase2} where $D\notin k$, $D\notin m$ and $B\notin l$, $\kappa_d$, $\mu_d$ and $\lambda_b$  are invertible. By the above equations we conclude that
\begin{displaymath}
\begin{split}
\lambda_a \nu_c \mu_d \kappa_d \lambda_b = \lambda_a \nu_d \mu_c \kappa_d \lambda_b  = \lambda_a \nu_d \mu_b \kappa_d \lambda_c  = \lambda_a \nu_d \mu_d \kappa_b \lambda_c = \\
= \lambda_b \nu_d \mu_d \kappa_a \lambda_c = \lambda_b \nu_a \mu_d \kappa_d \lambda_c  =\lambda_c \nu_a \mu_d \kappa_d \lambda_b .
\end{split}
\end{displaymath}
$\lambda_a \nu_c= \nu_a \lambda_c$ because $\mu_d$ $\kappa_d$ and $\lambda_b$ are invertible, therefore $\delta(l,n,A,C)$



\item Case \ref{itemdescase3} is dual to case \ref{itemdescase2}.

\item In case  \ref{itemdescase4} : $D\notin k$, $C\notin m$ and $B\notin l$, therefore $\kappa_d$, $\mu_c$, $\lambda_b$ are invertible.

\begin{displaymath}
\begin{split}
\lambda_a \nu_c \kappa_d \mu_c \lambda_b =
\lambda_a \nu_c \kappa_d \mu_b \lambda_c =
\lambda_a \nu_c \kappa_b \mu_d \lambda_c =
\lambda_b \nu_c \kappa_a \mu_d \lambda_c = \\
=\lambda_b \nu_d \kappa_a \mu_c \lambda_c =
\lambda_b \nu_a \kappa_d \mu_c \lambda_c =
\lambda_c \nu_a \kappa_d \mu_c \lambda_b.
\end{split}
\end{displaymath}
$\kappa_d$, $\mu_c$ and $\lambda_b$ are invertible, therefore $\lambda_a \nu_c = \lambda_c \nu_a$. Hence, $\delta(l,n,A,C)$.

\item In case \ref{itemdescase5}: $A\notin k$, $B\notin l$, $D\notin n$ and $C\notin m$, therefore $\kappa_a$, $\lambda_b$, $\nu_d$ and $\mu_c$ are invertible. Hence,  $\kappa_b$, $\lambda_a$, $\nu_c$ and $\mu_d$ are also invertible.

\begin{displaymath}
\begin{split}
\lambda_a \nu_c \mu_d \kappa_a \lambda_b =
\lambda_a \nu_d \mu_c \kappa_a \lambda_b =
\lambda_a \nu_a \mu_c \kappa_d \lambda_b =
\lambda_a \nu_a \mu_b \kappa_d \lambda_c = \\
=\lambda_a \nu_a \mu_d \kappa_b \lambda_c =
\lambda_b \nu_a \mu_d \kappa_a \lambda_c =
\lambda_c \nu_a \mu_d \kappa_a \lambda_b.
\end{split}
\end{displaymath}
$\mu_d$, $\kappa_a$ and $\lambda_b$ are invertible, therefore $\lambda_a \nu_c = \lambda_c \nu_a$. Hence, $\delta(l,n,A,C)$.

\end{enumerate}
\end{proof}

\begin{thrm} \label{thrmnewdes}
(Desargues' theorem)
Let $R$ be a local ring and let $A$, $B$, $C$, $D$ be points, and $k$, $l$, $m$, $n$ lines of $\mathbb P(R)$ such that $\delta(k,l,A,B)$, $\delta(l,m,B,C)$, $\delta(m,n,C,D)$, $\delta(n,k,D,A)$, $\delta(k,m,B,D)$, and such that $l\#n$ or $A\#C$, and at least one of the points $A$, $C$ lies outside at least one of the lines $l$, $n$. Then $\delta(l,n,A,C)$ is satisfied when all of the following conditions hold:
\begin{enumerate}
 \item $B$ lies outside at least one of $k$, $l$ and $m$,
 \item $D$ lies outside at least one of $m$, $n$ and $k$,
 \item at least one of $D$, $A$ and $B$ lies outside $k$,
 \item at least one of $B$, $C$ and $D$ lies outside $m$.
\end{enumerate}
\end{thrm}
\begin{center}
\begin{tikzpicture}
 \draw 
(0,0) node[left] {$A$} 
-- (2,2) node[above] {$B$} 
-- (4,0) node[right] {$C$}
-- (2,-2) node[below] {$D$}
-- (0,0)

(1.5,1.5)
-- (3,1)
-- (3.5,-0.5)
-- (1,-1)
-- (1.5,1.5)
(2,1) node {$l$}
(3,0) node {$m$}
(2,-1) node {$n$}
(1,0) node {$k$} ;
\end{tikzpicture}
\end{center}

\begin{proof}
To prove this theorem we list all the cases and show that we have already considered them (or a symmetric version of them) in Lemma \ref{lempredes}.

When the four conditions above hold, then the following is also true:

$$B\notin k \vee B\notin m \vee D\notin k \vee D\notin m \vee (B\notin l \wedge D\notin n \wedge A\notin k \wedge C\notin m).$$

Therefore, we consider the five above cases.

The case where $B\notin l \wedge D\notin n \wedge A\notin k \wedge C\notin m$ is case \ref{itemdescase5} of Lemma \ref{lempredes}.

The other four cases ($B\notin k$, $B\notin m$, $D\notin k$ and $D\notin m$) are symmetric so we shall consider the case where $D\notin k$.

By condition 4, $B$ or $C$ or $D$ lie outside $m$:
\begin{itemize}
 \item $B\notin m$. This is symmetric to case \ref{itemdescase1} of Lemma \ref{lempredes}.
\item $C\notin m$

By condition 1, $B$ lies outside $k$, $l$ or $m$
\begin{itemize}
\item $B\notin k$. This is case \ref{itemdescase3} of Lemma \ref{lempredes}.
\item $B\notin l$. This is case \ref{itemdescase4} of Lemma \ref{lempredes}.
\item $B\notin m$. This case was already covered above ($D\notin k$ and $B\notin m$).
\end{itemize}

\item $D\notin m$

By condition 1, $B$ lies outside $k$, $l$ or $m$.
\begin{itemize}
\item $B\notin k$. This is case \ref{itemdescase1} of Lemma \ref{lempredes}.
\item $B\notin l$. This is case \ref{itemdescase2} of Lemma \ref{lempredes}.
\item $B\notin m$. This case was already covered above ($D\notin k$ and $B\notin m$).
\end{itemize}

\end{itemize}
\end{proof}

\begin{rmk}
Notice that Desargues' theorem given in the above form is self-dual.

The four conditions in the statement of Desargues' theorem are necessary. Consider the case of the projective plane over the rational numbers and the following configuration:
$A=(1,0,1)$,
$B=(0,0,1)$,
$C=(0,1,1)$,
$D=(1,1,1)$,
$k=(1,-2,0)$,
$l=(2,1,0)$,
$m=(-2,1,0)$
$n=(2,2,-3)$.
All the conditions are satisfied except from $B$ lying outside one of the lines $k$, $l$ or $m$ and in this case $\delta(l,n,A,C)$ does not hold.

\begin{center}
\begin{tikzpicture}
 \draw
(0,0) node[left] {$B$}
--(2,0) node[below] {$A$}
--(2,2) node[right] {$D$}
--(0,2) node[left] {$C$}
--(0,0)
--(2,1)
--(1,2)
--(0,0)
(0.5,-1)
--(-1,2)
(-0.5,1) node[left] {$l$}
(1,0.5) node[below] {$k$}
(0.7125,1.25) node[left] {$m$}
(1.5,1.5) node[right] {$n$}
;
\end{tikzpicture}
\end{center}

Note that the second condition in Desargues' theorem is symmetric to the first one and the last two are dual to the first two.
\end{rmk}

\section{Pappus' theorem on the projective plane} \label{secpap}

\begin{lem} \label{lemS3}
Let $A=\pt{a}$, $B=\pt{b}$, $X=\pt{x}$, $Y=\pt{x}$, $Z=\pt{z}$, $W=\pt{w}$ be points and let $k=\pt{\kappa}$ and $l=\pt{\lambda}$ be lines of a projective plane over a ring.
If $(X, Y\in k)$, $(Z, W\in l)$, and $\delta(k,l,A,B)$, then
$$\det
\begin{pmatrix}
x_0 & x_1 & x_2 \\
y_0 & y_1 & y_2 \\
a_0 & a_1 & a_2
\end{pmatrix}
\det
\begin{pmatrix}
z_0 & z_1 & z_2 \\
w_0 & w_1 & w_2 \\
b_0 & b_1 & b_2
\end{pmatrix}
=\det
\begin{pmatrix}
x_0 & x_1 & x_2 \\
y_0 & y_1 & y_2 \\
b_0 & b_1 & b_2
\end{pmatrix}
\det
\begin{pmatrix}
z_0 & z_1 & z_2 \\
w_0 & w_1 & w_2 \\
a_0 & a_1 & a_2
\end{pmatrix}.$$
\end{lem}

\begin{proof}
$\delta(k,l,A,B)$, hence by Lemma \ref{lemS} the determinant of
$$\begin{pmatrix}
 \kappa_0 & \kappa_1 & \kappa_2
\\ \lambda_0 & \lambda_1 & \lambda_2
\end{pmatrix}
\begin{pmatrix}
 a_0 & b_0 \\
a_1 & b_1 \\
a_2 & b_2
\end{pmatrix}$$
is $0$, or equivalently 
$$(\kappa_0 a_0+\kappa_1 a_1 +\kappa_2 a_2)(\lambda_0 b_0 + \lambda_1 b_1 +\lambda_2 b_2)=(\kappa_0 b_0+\kappa_1 b_1 +\kappa_2 b_2)(\lambda_0 a_0 + \lambda_1 a_1 +\lambda_2 a_2).$$
By the above lemma, $(x_1 y_2 -x_2 y_1, x_2 y_0 - x_0 y_2, x_0 y_1 - x_1 y_0)= r\pt{\kappa}$, and $(z_1 w_2 - z_2 w_1, z_2 w_0 - z_0 w_2, z_0 w_1 - z_1 w_0)= s\pt{\lambda}$.

Hence,
\begin{displaymath}
\begin{split}
\det
\begin{pmatrix}
x_0 & x_1 & x_2 \\
y_0 & y_1 & y_2 \\
a_0 & a_1 & a_2
\end{pmatrix}
& \det
\begin{pmatrix}
z_0 & z_1 & z_2 \\
w_0 & w_1 & w_2 \\
b_0 & b_1 & b_2
\end{pmatrix}= \\
& =r(\kappa_0 a_0+\kappa_1 a_1 +\kappa_2 a_2)s(\lambda_0 b_0 + \lambda_1 b_1 +\lambda_2 b_2) \\
& =r(\kappa_0 b_0+\kappa_1 b_1 +\kappa_2 b_2)s(\lambda_0 a_0 + \lambda_1 a_1 +\lambda_2 a_2) \\
& =\det
\begin{pmatrix}
x_0 & x_1 & x_2 \\
y_0 & y_1 & y_2 \\
b_0 & b_1 & b_2
\end{pmatrix}
\det
\begin{pmatrix}
z_0 & z_1 & z_2 \\
w_0 & w_1 & w_2 \\
a_0 & a_1 & a_2
\end{pmatrix}.
\end{split}
\end{displaymath}
\end{proof}

\begin{lem} \label{lemS4}
Let $A=\pt{a}$, $B=\pt{b}$, $X=\pt{x}$, $Y=\pt{x}$, $Z=\pt{z}$, $W=\pt{w}$ be points and let $k=\pt{\kappa}$ and $l=\pt{\lambda}$ be lines of the projective plane over a local ring, such that $X$, $Y$ lie on $k$ and $Z$, $W$ lie on $l$ . If $X\#Y$ and $Z\#W$, at least one of the points $A$ and $B$ lies outside at least one of the lines $k$ and $l$ and
$$\det
\begin{pmatrix}
x_0 & x_1 & x_2 \\
y_0 & y_1 & y_2 \\
a_0 & a_1 & a_2
\end{pmatrix}
\det
\begin{pmatrix}
z_0 & z_1 & z_2 \\
w_0 & w_1 & w_2 \\
b_0 & b_1 & b_2
\end{pmatrix}
=\det
\begin{pmatrix}
x_0 & x_1 & x_2 \\
y_0 & y_1 & y_2 \\
b_0 & b_1 & b_2
\end{pmatrix}
\det
\begin{pmatrix}
z_0 & z_1 & z_2 \\
w_0 & w_1 & w_2 \\
a_0 & a_1 & a_2
\end{pmatrix},$$
then  $\delta(k,l,A,B)$ holds.
\end{lem}

\begin{proof}
By the construction of a unique line through two points that are apart from each other in Proposition \ref{propexistsuniqueline},
$(x_1 y_2 -x_2 y_1, x_2 y_0 - x_0 y_2, x_0 y_1 - x_1 y_0)= r\pt{\kappa}$, and $(z_1 w_2 - z_2 w_1, z_2 w_0 - z_0 w_2, z_0 w_1 - z_1 w_0)= s\pt{\lambda}$ for invertible $r$ and $s$.

Hence,
\begin{displaymath}
\begin{split}
\det
\begin{pmatrix}
x_0 & x_1 & x_2 \\
y_0 & y_1 & y_2 \\
a_0 & a_1 & a_2
\end{pmatrix}
& \det
\begin{pmatrix}
z_0 & z_1 & z_2 \\
w_0 & w_1 & w_2 \\
b_0 & b_1 & b_2
\end{pmatrix}= \\
& =r(\kappa_0 a_0+\kappa_1 a_1 +\kappa_2 a_2)s(\lambda_0 b_0 + \lambda_1 b_1 +\lambda_2 b_2)
\end{split}
\end{displaymath}
 and
\begin{displaymath}
\begin{split}
\det
\begin{pmatrix}
x_0 & x_1 & x_2 \\
y_0 & y_1 & y_2 \\
b_0 & b_1 & b_2
\end{pmatrix}
& \det
\begin{pmatrix}
z_0 & z_1 & z_2 \\
w_0 & w_1 & w_2 \\
a_0 & a_1 & a_2
\end{pmatrix}= \\
& =r(\kappa_0 b_0+\kappa_1 b_1 +\kappa_2 b_2)s(\lambda_0 a_0 + \lambda_1 a_1 +\lambda_2 a_2).
\end{split}
\end{displaymath}
Therefore, since $r$ and $s$ are invertible
$$(\kappa_0 a_0+\kappa_1 a_1 +\kappa_2 a_2)(\lambda_0 b_0 + \lambda_1 b_1 +\lambda_2 b_2)=(\kappa_0 b_0+\kappa_1 b_1 +\kappa_2 b_2)(\lambda_0 a_0 + \lambda_1 a_1 +\lambda_2 a_2),$$
or equivalently the determinant of the product
$$\begin{pmatrix}
 \kappa_0 & \kappa_1 & \kappa_2
\\ \lambda_0 & \lambda_1 & \lambda_2
\end{pmatrix}
\begin{pmatrix}
 a_0 & b_0 \\
a_1 & b_1 \\
a_2 & b_2
\end{pmatrix}$$
is $0$. Hence by Lemma \ref{lemSiff}, $\delta(k,l,A,B)$ holds.
\end{proof}

\begin{thrm}
Given six points $A$, $B$, $C$, $D$, $E$, $F$ and six lines  $k_A$, $k_B$, $k_C$, $k_D$, $k_E$, $k_F$ of the projective plane over a local ring such that:
\begin{itemize}
\item $A$, $B$ lie on $k_A$,
\item $B$, $C$ lie on $k_B$,
\item $C$, $D$ lie on $k_C$,
\item $D$, $E$ lie on $k_D$,
\item $E$, $F$ lie on $k_E$,
\item $F$, $A$ lie on $k_F$.
\end{itemize}
Then, if $\delta(k_C,k_F, B,E)$ and $\delta(k_B, k_E, A,D)$ hold and:
\begin{itemize}
\item $A\#B \wedge D\#E$ or $k_B \# k_C \wedge k_F \# k_A$,
\item at least one of the points $C$ and $F$ lies outside at least one of the lines $k_A$ and $k_D$,
\end{itemize}
then $\delta(k_A,k_D, F, C)$ also holds.
\begin{center}
\begin{tikzpicture}
\draw
(2,-0.8)--(6,-2.4)
(2,1.2)--(8,4.8)
(2,1.2) node[above] {$E$}--(4,-1.6)  node[below] {}
(4,2.4)  node[above] {} --(6,-2.4) node[below] {$D$}
(4,2.4)--(2,-0.8) node[below] {$A$}
(8,4.8)  node[above] {$B$}--(4,-1.6)
(2,1.2) -- (6,-2.4)
(2,-0.8) -- (8,4.8)
(2.3,0.2) node {$F$}
(5.5,0.1) node {$C$};

\draw[red, thick, densely dotted]
(2.667, 0.267)--(5,0);
\end{tikzpicture}
\end{center}
\end{thrm}

\begin{proof}
Notice that the case where $k_B \# k_C \wedge k_F \# k_A$ is dual to the case where $A\#B \wedge D\#E$. Hence, it suffices to consider the second case.

Let $A=\pt{a}$, $B=\pt{b}$, $C=\pt{c}$, $D=\pt{d}$, \newline
$E=\pt{e}$ and $F=\pt{f}$.

$\delta(k_C,k_F, B,E)$ and $\delta(k_B, k_E, A,D)$ hold, hence by Lemma \ref{lemS3}:
$$\det
\begin{pmatrix}
c_0 & c_1 & c_2 \\
d_0 & d_1 & d_2 \\
b_0 & b_1 & b_2
\end{pmatrix}
\det
\begin{pmatrix}
f_0 & f_1 & f_2 \\
a_0 & a_1 & a_2 \\
e_0 & e_1 & e_2
\end{pmatrix}
=\det
\begin{pmatrix}
c_0 & c_1 & c_2 \\
d_0 & d_1 & d_2 \\
e_0 & e_1 & e_2
\end{pmatrix}
\det
\begin{pmatrix}
f_0 & f_1 & f_2 \\
a_0 & a_1 & a_2 \\
b_0 & b_1 & b_2
\end{pmatrix}$$
and
$$\det
\begin{pmatrix}
b_0 & b_1 & b_2 \\
c_0 & c_1 & c_2 \\
a_0 & a_1 & a_2
\end{pmatrix}
\det
\begin{pmatrix}
e_0 & e_1 & e_2 \\
f_0 & f_1 & f_2 \\
d_0 & d_1 & d_2
\end{pmatrix}
=\det
\begin{pmatrix}
b_0 & b_1 & b_2 \\
c_0 & c_1 & c_2 \\
d_0 & d_1 & d_2
\end{pmatrix}
\det
\begin{pmatrix}
e_0 & e_1 & e_2 \\
f_0 & f_1 & f_2 \\
a_0 & a_1 & a_2
\end{pmatrix}. $$

By combining the two above equalities and interchanging the same number of rows on both sides we conclude that
$$\det
\begin{pmatrix}
a_0 & a_1 & a_2 \\
b_0 & b_1 & b_2 \\
f_0 & f_1 & f_2 
\end{pmatrix}
\det
\begin{pmatrix}
e_0 & e_1 & e_2 \\
d_0 & d_1 & d_2 \\
c_0 & c_1 & c_2
\end{pmatrix}
=
\det
\begin{pmatrix}
a_0 & a_1 & a_2 \\
b_0 & b_1 & b_2 \\
c_0 & c_1 & c_2 
\end{pmatrix}
\det
\begin{pmatrix}
e_0 & e_1 & e_2 \\
d_0 & d_1 & d_2 \\
f_0 & f_1 & f_2
\end{pmatrix}.$$
Hence, by Lemma \ref{lemS4}  $\delta(k_A,k_D, F, C)$ also holds.
\end{proof}

\section{The theory of projective planes} \label{secthryproj}

\subsection*{Desargues' axiom}

A preprojective plane satisfies \emph{Desargues' axiom} when given $A$, $B$, $C$, $D$ points and $k$, $l$, $m$, $n$ lines such that:
\begin{enumerate}
\item $\delta(k,l,A,B)$, $\delta(l,m,B,C)$, $\delta(m,n,C,D)$, $\delta(n,k,D,A)$, $\delta(k,m,B,D)$ hold,
 \item $l\#n$ or $A\#C$,
 \item at least one of the points $A$, $C$ lies outside at least one of the lines $l$, $n$,
 \item $B$ lies outside at least one of $k$, $l$ and $m$,
 \item $D$ lies outside at least one of $m$, $n$ and $k$,
 \item at least one of $D$, $A$ and $B$ lies outside $k$,
 \item at least one of $B$, $C$ and $D$ lies outside $m$,
\end{enumerate}
then $\delta(l,n,A,C)$ holds.
\begin{center}
\begin{tikzpicture}
 \draw 
(0,0) node[left] {$A$} 
-- (2,2) node[above] {$B$} 
-- (4,0) node[right] {$C$}
-- (2,-2) node[below] {$D$}
-- (0,0)

(1.5,1.5)
-- (3,1)
-- (3.5,-0.5)
-- (1,-1)
-- (1.5,1.5)
(2,1) node {$l$}
(3,0) node {$m$}
(2,-1) node {$n$}
(1,0) node {$k$} ;
\end{tikzpicture}
\end{center}

\subsection*{Pappus' axiom}
A preprojective plane satisfies \emph{Pappus' axiom} when given six points $A$, $B$, $C$, $D$, $E$, $F$ and six lines  $k_A$, $k_B$, $k_C$, $k_D$, $k_E$, $k_F$ such that:
\begin{itemize}
\item $A$, $B$ lie on $k_A$,
\item $B$, $C$ lie on $k_B$,
\item $C$, $D$ lie on $k_C$,
\item $D$, $E$ lie on $k_D$,
\item $E$, $F$ lie on $k_E$,
\item $F$, $A$ lie on $k_F$.
\item $\delta(k_C,k_F, B,E)$ and $\delta(k_B, k_E, A,D)$ hold,
\item $A\#B \wedge D\#E$ or $k_B \# k_C \wedge k_F \# k_A$
\item at least one of the points $C$ and $F$ lies outside at least one of the lines $k_A$ and $k_D$,
\end{itemize}
then $\delta(k_A,k_D, F, C)$ also holds.
\begin{center}
\begin{tikzpicture}
\draw
(2,-0.8)--(6,-2.4)
(2,1.2)--(8,4.8)
(2,1.2) node[above] {$E$}--(4,-1.6)  node[below] {}
(4,2.4)  node[above] {} --(6,-2.4) node[below] {$D$}
(4,2.4)--(2,-0.8) node[below] {$A$}
(8,4.8)  node[above] {$B$}--(4,-1.6)
(2,1.2) -- (6,-2.4)
(2,-0.8) -- (8,4.8)
(2.3,0.2) node {$F$}
(5.5,0.1) node {$C$};

\draw[red, thick, densely dotted]
(2.667, 0.267)--(5,0);
\end{tikzpicture}
\end{center}

\begin{defn}
A \emph{projective plane} is a preprojective plane that satisfies Desargues' axiom and Pappus' axiom.
\end{defn}

For $R$ a local ring, $\mathbb P(R)$ is a projective plane by the results of Section \ref{secdes} and Section \ref{secpap}.

Both Desargues' axiom and Pappus' axiom can be expressed as geometric sequents in the language of preprojective planes, therefore the theory of projective planes is a geometric theory. Both Desargues' axiom and Pappus' axiom are self-dual, therefore the theory of projective planes satisfies the duality principle.

\chapter{Affine planes} \label{chaaffplanes}

In this chapter, we approach affine planes from an analytic and a syntactic point of view and we show some of their links to projective planes. We first construct the affine plane $\mathbb A(R)$ over a given local ring $R$. We also construct an affine plane structure $\mathfrak A(\ca P,l)$ from a given projective plane $\ca P$ with a chosen line $l$ and we demonstrate that affine planes over local rings are always of this form. We prove a few results satisfied by these structures. We present the coherent theory of preaffine planes whose axioms are satisfied by both the structures mentioned here. We continue with results on morphisms of preaffine planes, morphism between preaffine planes of the form $\mathfrak A(\ca P,l)$ and morphisms between projective planes over local rings. We present Desargues' big and small axioms, and Pappus' axiom on the affine plane and show that they are satisfied by affine planes over local rings. We also prove some further versions of Desargues' theorem which are used in proofs of Chapter \ref{chalocal}.  The coherent theory of affine planes is given as the theory of preaffine planes with the addition of Desargues' big and small axioms, and Pappus' axiom.

\section{Points}

\begin{defn}
Given a ring $R$ we define the set of \emph{points} of its affine plane to be the set of points of the projective plane over $R$ which lie outside the line $(0,0,1)$ and we denote it by $\mathbb A_{\text{pt}}(R)$.
\end{defn}

Since each point lying outside the line $(0,0,1)$ has an invertible third coordinate, it can be represented by $(a_0,a_1,1)$ for unique $a_0$ and $a_1$. So, we will write $(a_0,a_1)$ for a point of the affine plane meaning the point represented by $(a_0,a_1,1)$. Henceforth, we think of $\mathbb A_{\text{pt}}(R)$ as $R^2$.

\begin{defn}
We say that two points $A=(a_0,a_1)$ and $B=(b_0,b_1)$ are \emph{apart} from each other and we write
$$A\# B$$
when at least one of $(a_0-a_1)$ and $(b_0-b_1)$ is invertible.
\end{defn}

\begin{lem} \label{lemapartaffine}
For $R$ a \emph{local ring}, the $\#$ relation on $\mathbb A_{\text{pt}}(R)$ as described above is the restriction of the $\#$ relation on $\mathbb P_{\text{pt}}(R)$ from the previous chapter: two points $(a_0,a_1)$ and $(b_0,b_1)$ of the affine plane are apart from each other iff the points $(a_0,a_1,1)$ and $(b_0,b_1,1)$ of the projective plane are apart from each other.
\end{lem}

\begin{proof}
 Suppose that $(a_0,a_1)$ is apart from $(b_0,b_1)$. One of the $2\times 2$ minors of the matrix
$\begin{pmatrix}
  a_0 & b_0\\
a_1 & b_1\\
1 &1
 \end{pmatrix}$ has invertible determinant. Therefore, $a_0-b_0$, $a_1-b_1$ or $a_0b_1-a_1b_0$ is invertible. The first two cases prove our claim, therefore we consider the third case. Notice that $a_0b_1-a_1b_0=a_0(b_1-a_1)+a_1(a_0-b_0)$, therefore at least one of the two summands is invertible. Hence, $a_0-b_0$ or $a_1-b_1$ is invertible.

The converse is clear. 
\end{proof}

\section{Lines}

\begin{defn} We define the set of \emph{lines} of the affine plane over a ring $R$ to be the set of lines of its projective plane that are apart from the line $(0,0,1)$ and we denote it by $\mathbb A_{\text{li}} (R)$.
\end{defn}

Hence, a line represented by $\pt{\lambda}$ of the projective plane belongs to the affine plane iff at least one of $\lambda_0$ and $\lambda_1$ is invertible. We write $\pt{\lambda}$ to  mean the line represented by $\pt{\lambda}$, as we did for lines of projective planes.

\begin{defn}
The $\#$ relation on lines of the projective plane over a ring $R$ restricts to a relation on lines of the affine plane which we also denote by $\#$.
\end{defn}

Therefore, two lines represented by $\pt{\lambda}$ and $\pt{\kappa}$ of the affine plane over a ring are apart from each other when the determinant one of the three minors of the matrix
$$\begin{pmatrix}
\kappa_0 & \lambda_0 \\
\kappa_1 & \lambda_1 \\
\kappa_2 & \lambda_2
\end{pmatrix}$$
is invertible.

\begin{defn}
We say that two lines $k$ and $l$ of the affine plane are \emph{parallel} and we write
$$k\parallel l$$
 when for some representatives $\pt{\kappa}$ and $\pt{\lambda}$ of $k$ and $l$ respectively, there exists $r\in R$ such that $(\mu_0,\mu_1)=(r \lambda_0,r \lambda_1)$.
\end{defn}

\begin{lem} \label{lemparallel}
Given two lines of the affine plane over a ring $R$, represented by $(\lambda_0,\lambda_1,\lambda_2)$ and $(\mu_0,\mu_1,\mu_2)$, the following are equivalent:
\begin{enumerate}
\item  the intersections of the two lines with the line $(0,0,1)$ on the projective plane coincide,
\item there exists $r\in R$ such that $(\mu_0,\mu_1)=(r\lambda_0,r\lambda_1)$,
\item
$\det \begin{pmatrix}
\lambda_0 & \mu_0 \\
\lambda_1 & \mu_1\\
\end{pmatrix}
=0$.
\end{enumerate}
\end{lem}

\begin{proof}
$1 \Leftrightarrow 2:$
The intersection of the lines $\pt{\lambda}$ and $(0,0,1)$ is the point $(\lambda_1,-\lambda_0,0)$ and the intersection of the lines $\pt{\mu}$ and $(0,0,1)$ is the point $(\mu_1,-\mu_0,0)$.
These two points are the same iff there exists $r\in R$ such that $\mu_0=r\lambda_0$ and $\mu_1=r\lambda_1$. 

$2\Rightarrow 3:$
Suppose that $r$ in $R$ is such that $(\mu_0,\mu_1)=(r\lambda_0,r\lambda_1)$. Then,
$$\det \begin{pmatrix}
\lambda_0 & \mu_0 \\
\lambda_1 & \mu_1\\
\end{pmatrix}=
\det \begin{pmatrix}
r\mu_0 & \mu_0 \\
r\mu_1 & \mu_1\\
\end{pmatrix}
=0.$$

$3\Rightarrow 2:$
Suppose that $\det \begin{pmatrix}
\lambda_0 & \mu_0 \\
\lambda_1 & \mu_1\\
\end{pmatrix}=0$. Then, $\lambda_0\mu_1=\lambda_1\mu_0$. $\pt{\lambda}$ is a line of the affine plane therefore $\lambda_0$ or $\lambda_1$ is invertible. In the first case, $(\mu_0,\mu_1)=(r\lambda_0,r\lambda_1)$ for $r=\lambda_0^{-1}\mu_0$. In the second case $(\mu_0,\mu_1)=(r\lambda_0,r\lambda_1)$ for $r=\lambda_1^{-1}\mu_1$.
\end{proof}

Notice that by the above, two lines $k$ and $l$ of the affine plane are parallel iff for \emph{any} representatives $\pt{\kappa}$ and $\pt{\lambda}$ of $k$ and $l$ respectively, there exists $r\in R$ such that $(\mu_0,\mu_1)=(r\lambda_0,r\lambda_1)$.

\section{Incidence}

\begin{defn}
The $\in$ and $\notin$ relations on the projective plane over a ring $R$ restrict to relations on the corresponding affine plane. We use the same names and notation for the restrictions of these relations to the affine plane over $R$.
\end{defn}


Therefore, a point $(x,y)$ of the affine plane lies on the line represented by $\pt{\lambda}$ iff $\lambda_0 x +\lambda_1 y + \lambda_2=0$. Also, a point $(x,y)$ of the affine plane lies outside the line represented by $\pt{\lambda}$ iff $\lambda_0 x +\lambda_1 y + \lambda_2$ is invertible.

We say that two lines $k$ and $l$ \emph{intersect} at a point $A$, when $A$ lies on both lines $k$ and $l$.

\begin{defn}
Given a ring $R$, the \emph{affine plane} over $R$, denoted by $\mathbb A(R)$ is the structure consisting of the sets $\mathbb A_{\text{pt}} (R)$, $\mathbb A_{\text{li}} (R)$, the two $\#$ relations, the $\parallel$ relation, and the $\in$ and $\notin$ relations.
\end{defn}

Note that for $R$ a geometric field in $\set$, the above construction gives the classical affine plane over the field $R$: $\in$ becomes the incidence relation, $\#$ becomes the inequality relation, $\notin$ becomes the complement of $\in$ and $\parallel$ is the usual parallel relation for affine planes.

On affine planes over a field, two lines are either parallel or they intersect but this is no longer true for affine planes over local rings. For example, on the affine plane over $\mathbb Z/(4)$, the lines $(1,0,2)$ and $(1,2,1)$ are apart from each other and they are not parallel but they have no intersection point. Their unique intersection point on the projective plane is $(0,1,2)$.

\section{Preaffine planes from preprojective planes with a line} \label{secprojtoaff}

Instead of proving propositions for affine planes over \emph{local rings} directly as we did for projective planes, we continue by describing an alternative construction of the affine plane over a local ring. We will later use this construction and the results about preprojective planes to prove statements about the affine plane over a local ring.

\begin{defn}
Given a preprojective plane $\ca P$ and a line $l_{\infty}$ of $\ca P$ we define the preaffine plane induced by them to be the following structure (consisting of two sets and five relations):
\begin{itemize}
 \item Its set of points is the set of points of $\ca P$  that lie outside the line $l_{\infty}$.
\item Its set of lines are the lines that are apart from $l_\infty$.
\item The $\#$ relations on lines and points, the $\in$ and the $\notin$ relations of the affine plane are the restrictions of the ones on $\ca P$.
\item Two lines of the preaffine plane are parallel ($\parallel$) when their (unique) intersections with $l_\infty$ coincide.
\end{itemize}
We denote this structure by $\mathfrak A(\ca P,l_\infty)$.
\end{defn}

Notice that the projective plane over a local ring is a preprojective plane. The affine plane over a local ring $R$ is isomorphic to the preaffine plane induced by the projective plane over $R$ and the line $(0,0,1)$. This is true by the definitions and results of the previous section and in particular, by Lemma \ref{lemapartaffine} and Lemma \ref{lemparallel}.

\begin{lem}
Let $\ca P$ be a preprojective plane with a line $l_\infty$. Let $A$ be a point of $\mathfrak A(\ca P,l_\infty)$, and let $l$ be a line of $\ca P$ such that $A\in l$ in $\ca P$. Then, $l$ is also a line of $\mathfrak A(\ca P,l_\infty)$.
\end{lem} 

\begin{proof}
$A$ is a point of $\mathfrak A(\ca P,l_\infty)$, therefore it is a point of $\ca P$ such that $A\notin l_\infty$. $A\in l$, therefore $l\# l_\infty$. Hence $l$ is a line of $\mathfrak A(\ca P,l_\infty)$.
\end{proof}


\begin{lem} \label{lemuniv}
Given a preprojective plane $\ca P$ with a line $l_\infty$, the structure $\mathfrak A(\ca P,l_\infty)$ without the $\parallel$ relation is a substructure of $\ca P$. Let $\phi$ and $\psi$ be geometric formulae in the language of projective planes with no quantifiers. Suppose that $\phi\vdash_{\mathbf{a}}\psi$ holds for $\ca P$. Then, $\phi\vdash_{\mathbf{a}}\psi$ also holds for $\mathfrak A(\ca P,l_\infty)$.
\end{lem}

\begin{proof}
Let $\phi$ and $\psi$ be as above and suppose that $\phi\vdash_{\mathbf{a}}\psi$ holds for $\ca P$. 

Given $\mathbf{a}$ in $\mathfrak A(\ca P,l_\infty)$, then $\mathbf a$ satisfies $\phi$ in $\mathfrak A(\ca P,l_\infty)$ iff $\mathbf a$ satisfies $\phi$ in $\ca P$. This can be seen by \cite[1.1.8]{Marker} restricted to geometric formulae (and in that case the proof is in constructive logic). Similarly, $\mathbf a$ satisfies $\psi$ in $\mathfrak A(\ca P,l_\infty)$ iff $\mathbf a$ satisfies $\psi$ in $\ca P$.

Hence, given $\mathbf{a}$ in $\mathfrak A(\ca P,l_\infty)$ satisfying $\phi$ in $\mathfrak A(\ca P,l_\infty)$, then $\mathbf{a}$ satisfies $\phi$ in $\ca P$. Therefore, $\mathbf a$ satisfies $\psi$ in $\ca P$, and hence $\mathbf a$ also satisfies $\psi$ in $\mathfrak A(\ca P,l_\infty)$. Therefore, $\phi\vdash_{\mathbf{a}}\psi$ is satisfied in $\mathfrak A(\ca P,l_\infty)$.
\end{proof}

\begin{prop}
Let $\ca P$ be a preprojective plane with a line $l_\infty$. Then, $\#$ on points and $\#$ on lines are both apartness relations on $\mathfrak A(\ca P,l_\infty)$.
\end{prop}

\begin{proof}
This is true by Lemma \ref{lemuniv} because the sequents expressing that $\#$ is an apartness relation are geometric, quantifier free and they are satisfied by $\ca P$.
\end{proof}

\begin{prop}
Let $\ca P$ be a preprojective plane with a line $l_\infty$. For $A$, $B$ points and $k$, $l$ lines of $\mathfrak A(\ca P,l_\infty)$ the following hold:
\begin{enumerate}
\item $A\in l \wedge A\notin l \vdash_{A,l} \bot$,
\item $A\notin k \vdash_{A,B,k} A\# B \vee B\notin k$,
\item $A\notin k \vdash_{A,k,l} k\# l \vee A\notin l$.
\end{enumerate}
\end{prop}

\begin{proof}
This is true because $\ca P$ satisfies the above sequents and they are geometric and quantifier free. Hence, $\mathfrak A(\ca P,l_\infty)$ also satisfies them by Lemma \ref{lemuniv}.
\end{proof}

\begin{prop}
Let $\ca P$ be a preprojective plane with a line $l_\infty$. Then, in $\mathfrak A(\ca P,l_\infty)$ there exists a unique function from the set of pairs of points that are apart from each other to lines such that the image of the pair contains both points. Equivalently $\mathfrak A(\ca P,l_\infty)$ satisfies:
\begin{enumerate}
\item $A\# B \vdash_{A,B} \exists k. A\in k \wedge B\in k $,
\item $A\# B \wedge A,B\in k \wedge A,B\in l \vdash_{A,B,k,l} k=l$.
\end{enumerate}
\end{prop}

\begin{proof}
For the first sequent, let $A$ and $B$ be points of $\mathfrak A(\ca P,l_\infty)$ such that $A\# B$. In $\ca P$ there exists a (unique) line $l$ which contains both of them. In $\ca P$, $A\in l$ and $A\notin l_\infty$, therefore $l\# l_\infty$, hence $l$ is a line of the substructure $\mathfrak A(\ca P,l_\infty)$.

The second sequent of the proposition holds because it is a quantifier free sequent which holds for $\ca P$ therefore it also holds for $\mathfrak A(\ca P,l_\infty)$ by Lemma \ref{lemuniv}.
\end{proof}

\begin{prop}
Let $\ca P$ be a preprojective plane with a line $l_\infty$. Then, in $\mathfrak A(\ca P,l_\infty)$ two lines that are apart from each other have at most one intersection point:
$$k\# l \wedge A,B\in k \wedge A,B\in l \vdash_{A,B,k,l} A=B.$$
\end{prop}

\begin{proof}
This is true by Lemma \ref{lemuniv}.
\end{proof}

\begin{prop}
Let $\ca P$ be a preprojective plane with a line $l_\infty$. Then, $\mathfrak A(\ca P,l_\infty)$ satisfies the following sequents:
\begin{enumerate}
 \item $\top \vdash_l \exists A,B.A\#B \wedge A,B\in l$,
\item $\top \vdash_A \exists k,l,m .k\#l\#m\#k \wedge A\in k, l, m$,
\item $\top \vdash \exists A,B,C,l. A\#B \wedge A,B\in l\wedge C\notin l$,
\item $\top \vdash_l \exists A. A\notin l$.
\end{enumerate}
\end{prop}

\begin{proof}
\begin{enumerate}
\item  A line $l$ of $\mathfrak A(\ca P,l_\infty)$ is a line of $\ca P$ which is apart from $l_\infty$. Let $D$ be the unique intersection point of $l$ and $l_\infty$ in $\ca P$. In $\ca P$, there exist points $A$, $B$, $C$ on $l$ such that $A\# B\# C \# A$. $\#$ is an apartness relation, therefore $D$ is apart from at least two of $A$, $B$ and $C$. Without loss of generality, let us assume that $D$ is apart from both $A$ and $B$. $A\#D$ and $l\#l_\infty$, therefore at least one of $A$ and $D$ lies outside from at least one of $l$ and $l_\infty$. Hence, $A\notin l_\infty$, and similarly $B\notin l_\infty$, and therefore both $A$ and $B$ are points of $\mathfrak A(\ca P,l_\infty)$.

\item A point $A$ of $\mathfrak A(\ca P,l_\infty)$ is a point of $\ca P$ which lies outside $l_\infty$. In $\ca P$, there exist lines $k$, $l$, $m$ that pass through $A$ and that are all apart from each other. $A$ lies on each one of them and lies outside $l_\infty$, hence $k$, $l$ and $m$ are all apart from $l_\infty$, and therefore they are lines of $\mathfrak A(\ca P,l_\infty)$.

\item A preprojective plane $\ca P$ contains at triple of non-concurrent lines $k$, $l$, $m$. By Lemma \ref{lemnoncd}, at least one combination of $l_\infty$ with two of the lines $k$, $l$, $m$ gives a triple of non-concurrent lines. Without loss of generality let us assume that $k$, $l$ and $l_\infty$ are non-concurrent. Then $k$ and $l$ are in $\mathfrak A(\ca P,l_\infty)$. Moreover, $A=k\cap l$ lies outside $l_\infty$, and therefore is in $\mathfrak A(\ca P,l_\infty)$.  By 1 of this proposition (and using that $\#$ is an apartness relation on points), there exists a point $B$ on $l$ and $C$ on $k$ such that $A\#B$ and $A\#C$ (and such that $B$ and $C$ are points of $\mathfrak A(\ca P,l_\infty)$). By the results on non-concurrent lines of preprojective planes, we see that $C\notin \ovl{AB}=l$. Hence, $A$, $B$, $C$ and $l$ are as required.

\item Given a line $l$, let $A$, $B$, $C$ be non-collinear points of $\mathfrak A(\ca P,l_\infty)$. Then $\ovl{AB}$, $\ovl{BC}$, $\ovl{CA}$ are non-concurrent lines of $\ca P$. Therefore, by Lemma  \ref{lemnoncd}, at least one combination of $l$ with two of $\ovl{AB}$, $\ovl{BC}$, $\ovl{CA}$ gives a triple of non-concurrent lines. Without loss of generality, let us assume that the lines $l$, $\ovl{AB}$, $\ovl{CA}$ are non-concurrent. Then, $A\notin l$ as required.
\end{enumerate}

\end{proof}

\begin{prop}
Let $\ca P$ be a preprojective plane with a line $l_\infty$. Then, in $\mathfrak A(\ca P,l_\infty)$ $\parallel$ is an equivalence relation:
\begin{enumerate}
\item $\top \vdash_k k\parallel k$,
\item $k\parallel l\vdash_{k,l} l\parallel k$,
\item $k\parallel l \wedge l \parallel m\vdash_{k,l,m} k\parallel m$.
\end{enumerate}
\end{prop}

\begin{proof}
 This is clear from the definition of $\parallel$.
\end{proof}

\begin{prop}
Let $\ca P$ be a preprojective plane with a line $l_\infty$. Let $A$ be a point and $k$ a line in $\mathfrak A(\ca P,l_\infty)$. Then, there exists a unique line through $A$ parallel to $k$:
\begin{enumerate}
\item $\top\vdash_{A,k} \exists l. A\in l \wedge k\parallel l$,
\item $A\in k \wedge A\in l\wedge k\parallel l\vdash_{A,k,l} k=l$.
\end{enumerate}
\end{prop}

\begin{proof}
Given a point $A$ and a line $k$ of $\mathfrak A(\ca P,l_\infty)$, let $D$ be $k\cap l_\infty$ in $\ca P$ and notice that $A\# D$. A line in $\mathfrak A(\ca P,l_\infty)$ which passes through $A$ and is parallel to $k$ is exactly a line of $\ca P$ through $A$ and $D$. $A\# D$ hence there exists a unique such line.
\end{proof}

\begin{prop}
Let $\ca P$ be a preprojective plane with a line $l_\infty$. Then, $\mathfrak A(\ca P,l_\infty)$ satisfies:
\begin{enumerate}
\item  $A\# B \wedge l\#m \vdash_{A,B,l,m} A\notin l \vee B\notin m \vee A\notin m \vee B \notin l$.
\item  $k\#l \wedge k\parallel l \vdash_{A,k,l} A\notin k \vee A\notin l$.
\end{enumerate}
\end{prop}

\begin{proof}
\begin{enumerate}
\item This sequent is identical to an axiom of preprojective planes. It is quantifier free therefore it also holds on $\mathfrak A(\ca P,l_\infty)$ by Lemma \ref{lemuniv}.
\item Let $k$ and $l$ be lines that are parallel and apart from each other. Let $B$ be their intersection point in $\ca P$. $k$ and $l$ are parallel, hence $B$ lies on $l_\infty$.  Let $A$ be a point of $\mathfrak A(\ca P,l_\infty)$. Then, $A$ lies outside $l_\infty$, and therefore $A\# B$. By the axioms of preprojective planes, at least one of $A$ and $B$ lies outside at least one of $k$ and $l$. $B$ lies on both $k$ and $l$, therefore $A$ lies outside at least one of the lies $k$ and $l$. 
\end{enumerate} 
\end{proof}

\begin{prop}
Let $\ca P$ be a preprojective plane with a line $l_\infty$. Then, $\mathfrak A(\ca P,l_\infty)$ satisfies $$(A\in l,m) \wedge (k\parallel l) \wedge (l\#m) \vdash_{A,k,l,m} (k\#m) \wedge (\exists B. B\in k \wedge B\in m).$$
\end{prop}

\begin{proof}
Let $l$ and $m$ be lines of $\mathfrak A(\ca P,l_\infty)$ that are apart from each other and intersect at a point $A$ of $\mathfrak A(\ca P,l_\infty)$. $A\notin l_\infty$, therefore $l$, $m$, $l_\infty$ is a triple of non-concurrent lines in $\ca P$. Let $D=l\cap l_\infty$, and notice that $D\notin m$. $k$ is parallel to $l$, therefore $D=k\cap l_\infty$, hence the triple of lines $(k,m,l_\infty)$ is also non-concurrent. Hence, the lines $k$ and $m$ are apart from each other and their intersection lies outside $l_\infty$, or equivalently $k\#m$ and they intersect on $\mathfrak A(\ca P,l_\infty)$.
\end{proof}

\section{The theory of preaffine planes}

\begin{defn} 
The theory of \emph{preaffine planes} is written in a language with two sorts: points and lines. It has a binary relation $\#$ on points, the binary relations $\#$ and $\parallel$ on lines, and two relations $\in$ and $\notin$ between points and lines. The axioms of the theory of preaffine planes are the following:

\begin{itemize}

 \item $\#$ is an apartness relation on points, i.e. for $A$, $B$, $C$ points the following hold:
\begin{enumerate}
\item $A\# A \vdash_A \bot$, 
\item $A\# B \vdash_{A,B} B\# A$,
\item $A\# B \vdash_{A,B,C} A\# C \vee B\# C$.
\end{enumerate}

\item $\#$ is an apartness relation on lines, i.e. for $k$, $l$, $m$ lines the following hold:
\begin{enumerate}
\item $k\# k \vdash_k \bot$, 
\item $k\# l \vdash_{k,l} l\# k$,
\item $k\# l \vdash_{k,l,m} k\# m \vee l\# m$.
\end{enumerate}

\item $\notin$ is in some sense a constructive complement of $\in$, i.e. for $A$, $B$ points and $k$, $l$ lines the following hold:
\begin{enumerate}
\item $A\in l \wedge A\notin l \vdash_{A,l} \bot$,
\item $A\notin k \vdash_{A,B,k} A\# B \vee B\notin k$,
\item $A\notin k \vdash_{A,k,l} k\# l \vee A\notin l$.
\end{enumerate}

\item There exists a unique function from the set of pairs of points that are apart from each other to lines such that the image of the pair contains both points, i.e. the following hold:
\begin{enumerate}
\item $A\# B \vdash_{A,B} \exists k. A\in k \wedge B\in k $,
\item $A\# B \wedge A,B\in k \wedge A,B\in l \vdash_{A,B,k,l} k=l$.
\end{enumerate}

\item The dual of the above does not hold for affine planes, but it is still true that two lines that are apart from each other have at most one intersection point:
$k\# l \wedge A,B\in k \wedge A,B\in l \vdash_{A,B,k,l} A=B$.

\item We have the following three axioms which say that we have enough points and lines:
\begin{enumerate}
 \item $\top \vdash_l \exists A,B.A\#B \wedge A,B\in l$,
\item $\top \vdash \exists A,B,C,l. A\#B \wedge A,B\in l\wedge C\notin l$,
\item $\top \vdash_l \exists A. A\notin l$.
\end{enumerate}

\item The following three axioms say that $\parallel$ is an equivalence relation:
\begin{enumerate}
\item $\top \vdash_k k\parallel k$,
\item $k\parallel l\vdash_{k,l} l\parallel k$,
\item $k\parallel l \wedge l \parallel m\vdash_{k,l,m} k\parallel m$.
\end{enumerate}

\item The following two axioms say that given a point $A$ and a line $k$ there exists a unique line through $A$ parallel to $k$:
\begin{enumerate}
\item $\top\vdash_{A,k} \exists l. A\in l \wedge k\parallel l$,
\item $A\in k \wedge A\in l\wedge k\parallel l\vdash_{A,k,l} k=l$.
\end{enumerate}

\item These two axioms give us a way of introducing the $\notin$ relation:
\begin{enumerate}
\item  $A\# B \wedge l\#m \vdash_{A,B,l,m} A\notin l \vee B\notin m \vee A\notin m \vee B \notin l$,
\item  $k\#l \wedge k\parallel l \vdash_{A,k,l} A\notin k \vee A\notin l$.
\end{enumerate}

\item Finally this axiom gives us a condition for two lines to intersect on the affine plane: $A\in l,m \wedge k\parallel l \wedge l\#m \vdash_{A,k,l,m} k\#m \wedge (\exists B. B\in k \wedge B\in m)$.
\end{itemize}
\end{defn}

Given two points $P$ and $Q$ that are apart from each other we write $\ovl{PQ}$ for the unique line through $P$ and $Q$.

\begin{defn}
We say that three points $A$, $B$, $C$ of a preaffine plane are \emph{non-collinear} when $A\#B\#C\#A$ and $A\notin \ovl{BC}$, $B\notin \ovl{CA}$ and $C\notin \ovl{AB}$. 
\end{defn}

As in the case of preprojective planes, the definition of non-collinear points is symmetric on the three points. We do not give a definition of non-concurrent lines on preaffine planes because there is not a notion of non-concurrent lines which has similar properties as the one of non-concurrent lines of a preprojective plane. The reason for that is that lines of a preaffine plane that are apart from each other do not necessarily intersect and that breaks the symmetry satisfied by non-concurrent lines on a preprojective plane.

Lemma \ref{lemnonc} and  Lemma \ref{lemnoncd} also hold for non-collinear points of a preaffine plane and we restate them here.

\begin{lem} \label{lemnoncaff}
Let $A$, $B$ and $C$ be points of a preaffine plane. Then, the following are equivalent:
\begin{enumerate}
\item $A$, $B$, $C$ are non-collinear,
\item $B\#C$ and $A\notin \ovl{BC}$,
\item $A\# B\# C$ and $\ovl{AB}\#\ovl{BC}$.
\end{enumerate}
\end{lem}

\begin{proof}
In the proof of the above statement for preprojective planes in \ref{lemnonc}, we only used axioms of preprojective planes which also hold for preaffine planes. Hence, the proof is still valid for points of preaffine planes.
\end{proof}

\begin{lem} \label{lemnoncdaff}
Let $A$, $B$, $C$, $D$ be points of a preaffine plane and let $A$, $B$, $C$ be non-collinear. Then, at least one combination of $D$ with two of the points $A$, $B$, $C$ gives a non-collinear triple.
\end{lem}

\begin{proof}
The proof of the above for preprojective planes in Lemma \ref{lemnoncd} only uses axioms of preprojective planes which also hold for preaffine planes. Therefore, the proof is still valid for preaffine planes.
\end{proof}


\begin{thrm}
Given a preprojective plane $\ca P$ with a line $l_\infty$, the structure $\mathfrak A(\ca P, l_\infty)$ is a preaffine plane. In particular, for $R$ a local ring, $\mathbb A(R)$ is a preaffine plane.
\end{thrm}

\begin{proof}
Given a preprojective plane $\ca P$ with a line $l_\infty$, $\mathfrak A(\ca P, l_\infty)$ satisfies the axioms of preaffine planes by the results of the previous section.

Given a local ring $R$, $\mathbb P(R)$ is a preprojective plane, $\mathbb A(R)$ is isomorphic to  $\mathfrak A(\mathbb P(R), (0,0,1))$ and therefore it also satisfies the axioms of preaffine planes.
\end{proof}

\section{Morphisms of preaffine planes}

\begin{defn}
Given two preaffine planes, a morphism between them is a structure-preserving homomorphism: It consists of a function $f_P$ from the set of points of the first to the set of points of the second and a function $f_L$  from the set lines of the first to the set of lines of the second, such that they preserve the two $\#$ relations, $\parallel$, $\in$ and $\notin$.
\end{defn}

\begin{lem} \label{lemaffptsen}
 A morphism of preaffine planes is uniquely determined by the morphism on points.
\end{lem}

The proof is identical to the proof of the corresponding lemma for preprojective planes (Lemma \ref{lemprojptsen}).

\begin{proof}
Let $P$, $L$ be the set of point and the set of lines respectively of a preaffine plane, and let $P'$, $L'$ be the set of points and the set of lines respectively of a second preaffine plane. Suppose we are given a morphism from the first to the second preprojective plane, such that $f_P:P\to P'$ is the morphism on points and $f_L:L\to L'$ is the morphism on lines.

Given a line $l$ in $L$ there exist points $A$ and $B$ in $P$ that are apart from each other and lie on $l$. A morphism of preaffine planes preserves the $\#$ relation on points and the incidence relation, therefore $f_P(A)$ and $f_P(B)$ are apart from each other and lie on $f_L(l)$. $\ovl{f_P(A)f_P(B)}$ is the unique line through $f_P(A)$ and $f_P(B)$, therefore $f_L(l)=\ovl{f_P(A)f_P(B)}$. Hence, $f_L$ is uniquely determined by $f_P$.
\end{proof}


\begin{prop}  \label{proppreaffmor}
Let $\ca A$ and $\ca A'$ be preaffine planes with sets of points $P$ and $P'$ respectively. Let $f_P:P\to P'$ be a function such that:
\begin{enumerate}
\item $f_P$ preserves the $\#$ relation on points,
\item for $A$, $B$, $C$ points of $\ca P$ such that $A\# B$ and $C\in \ovl{AB}$, then ($f_P(A)\# f_P(B)$ and) $f_P(C)\in \ovl{f_P(A)f_P(B)}$,
\item given three non-collinear points $A$, $B$, $C$ of $\ca P$, then the points $f_P(A)$, $f_P(B)$, $f_P(C)$ are also non-collinear,
\item whenever $A\#B$ and $C\# D$ in $P$, such that $\ovl{AB}\parallel \ovl{CD}$, then $\ovl{f_P(A) f_P(B)}\parallel \ovl{f_P(C) f_P(D)}$.
\end{enumerate}
Then, there is a unique homomorphism between the two preaffine planes whose morphism on points is $f_P$.
\end{prop}

The proof is almost identical to the proof of the corresponding proposition for preprojective planes (Proposition \ref{propptslines}).

\begin{proof}
Let $\ca A$ and $\ca A'$ be preaffine planes with sets of points $P$ and $P'$ respectively and with sets of lines $L$ and $L'$ respectively. Let $f_P:P\to P'$ be a function satisfying the above conditions.

We define a function $f_L:L\to L'$ in the following way. Given a line $k$ in $L$, there exist points $A$ and $B$ on $k$ such that $A\#B$, and by 1 $f_P(A)\# f_P(B)$, hence we define $f_L(l)$ to be $\ovl{f_P(A)f_P(B)}$. By 2, the definition of $f_L$ does not depend on the choice of $A$ and $B$. Also, by 2 given a point $C$ of $\ca A$, if $C\in k$ then $f_P(C)\in f_L(k)$. $f_P$ sends non-collinear points to non-collinear points, therefore given $C\notin k$ in $(P,L)$, then $f_P(C)\notin f_L(k)$. If $A\#B$ and $C\# D$ in $P$, such that $\ovl{AB}\parallel \ovl{CD}$, then $\ovl{f_P(A) f_P(B)}\parallel \ovl{f_P(C) f_P(D)}$. Hence $f_L$ preserves the $\parallel$ relation.

Let $k$ and $l$ be lines of $\ca A$, such that $k\# l$. There exist points $A$ and $B$ lying on $k$ such that $A\#B$. $A\#B$ and $k\#l$, therefore at least one of $A$ and $B$ lies outside at least one of the lines $k$ and $l$. Both $A$ and $B$ lie on $k$, therefore at least one of them lies outside $l$. Without loss of generality, suppose $A\notin l$. Then $f_P(A)$ lies on $f_P(k)$ and outside $f_P(l)$. Hence, $f_P(k)\#f_P(l)$, and therefore $f_L$ preserves the apartness relation on lines. Hence $(f_P, f_L)$ is a morphism of preaffine planes.
\end{proof}

\begin{lem} \label{lemprojaffmorph}
Given two preprojective planes $\ca P$ and $\ca P'$ which contain the lines $l_\infty$ and $l'_\infty$ respectively, a morphism from $\ca P$ to $\ca P'$ that sends the line $l_\infty$ to the line $l'_\infty$ restricts to a morphism from the preaffine plane $\mathfrak A(\ca P,l_\infty)$ to $\mathfrak A(\ca P',l'_\infty)$.
\end{lem}

\begin{proof}
Let $\phi$ be a morphism from $\ca P$ to $\ca P'$. $\phi$ preserves the two $\#$ relations, and the $\in$ and $\notin$ relations.

The points of $\mathfrak A(\ca P,l_\infty)$ are the points of $\ca P$ which lie outside $l_\infty$. Given a point $A$ of $\ca P$ which lies outside $l_\infty$, $\phi(A)$ lies outside $\phi(l_\infty)=l'_\infty$, hence $\phi(A)$ is a point of $\mathfrak A(\ca P',l'_\infty)$.

The lines of $\mathfrak A(\ca P,l_\infty)$ are lines of $\ca P$ which are apart from $l_\infty$. Given a line $k$ of $\ca P$ which is apart from $l_\infty$, $\phi(k)$ is apart from $\phi(l_\infty)=l'_\infty$, hence $\phi(k)$ is a line of $\mathfrak A(\ca P',l'_\infty)$.

The restriction of $\phi$ to $\mathfrak A(\ca P,l_\infty)$ still preserves the two $\#$ relations, and the $\in$ and $\notin$ relations, therefore it is sufficient to prove that it preserves the $\parallel$ relation on lines to prove that it is a morphism of preaffine planes.

Suppose $k$ and $m$ are parallel lines of $\mathfrak A(\ca P,l_\infty)$. Let $D$ be $k\cap l_\infty$ on $\ca P$. Then $D$ is also the intersection of $m$ with $l_\infty$. $\phi$ is a morphism of preprojective planes therefore $\phi(D)$ is the intersection of $\phi(k)$ with $\phi(l_\infty)$ and also the intersection of $\phi(m)$ with $\phi(l_\infty)$. $\phi(l_\infty)=l'_\infty$, therefore $\phi(k)\parallel \phi(m)$.
\end{proof}

\section{Morphisms of projective planes from morphisms of preaffine planes} \label{secmorproaff}

In this section, $\ca P$ and $\ca Q$ are \emph{projective} planes containing the lines $k_\infty$ and $l_\infty$ respectively. $\phi$ is a morphism of affine planes from $\mathfrak A(\ca P,k_\infty)$ to $\mathfrak A(\ca Q,l_\infty)$. The goal of this section is to prove Theorem \ref{thrmaffprojmorph} which states that $\phi$ can be extended uniquely to a morphism of projective planes $\psi: \ca P \to \ca Q$. The following lemma will enable us to define $\psi$ on points.

\begin{lem} \label{lempsiwd}
Let $A$ be a point and let $k$, $l$, $m$ be lines of $\ca P$ such that the lines $k$, $l$ and $m$ pass through $A$, are apart from $k_\infty$, and $k\#l$. Then, ($\phi(k)\#\phi(l)$ and) the intersection of the lines $\phi(k)$ and $\phi(l)$  lies on $\phi(m)$.
\end{lem}

\begin{proof}
Let us first consider the case where $A\notin k_\infty$. Given $k$ and $l$ as above, $\phi(A)$ is the unique intersection point of the lines $\phi(k)$ and $\phi(l)$ because $\phi$ preserves the $\in$ relation. $\phi(A)\in \phi(m)$, hence the intersection point of $\phi(k)$ and $\phi(l)$ lies on $\phi(m)$.

Whenever a point $B$ of $\ca P$ is such that $B\notin k_\infty$, then either $A\notin k_\infty$ or $A\#B$. In the first case, the result is proved by the above. Hence, in the rest of the proof, whenever we are given such a point we will assume that it is apart from $A$.

$k\# l$, therefore at least one of $k$ and $l$ is apart from $m$. Without loss of generality, we assume that $l\#m$. There exist points $R$, $R'$ lying on $m$ and outside $k_\infty$ such that $R\#R'$. There also exist a point $P$ lying on $k$ and outside $k_\infty$. $P$ is apart from at least one of $R$ and $R'$. Without loss of generality, we assume that $P\#R$. There exist points $Q$ and $B$ lying on $l$ and outside $k_\infty$ such that $Q\# B$. $P$, $Q$, $B$, $R$ and $R'$ all lie outside $k_\infty$ and therefore by an earlier argument we may assume that they are all apart from $A$. $k\# l$, therefore we can also conclude that $P\notin l$, and therefore $l\#\ovl{PR}$. Hence, at least one of $Q$ and $B$ lies outside (at least on of the lines $l$ and) $\ovl{PR}$. Without loss of generality, we assume that $Q\notin \ovl{PR}$. Then, $(A,Q,R)$, $(A,P,Q)$ and $(P,Q,R)$ are all triples of non-collinear points. Therefore, the line $\ovl{QR}$ is apart from both $l$ and $m$, and the line $\ovl{PQ}$ is apart from both $k$ and $l$.

\begin{center}
\begin{tikzpicture}
\draw 
(0,0) node[below] {$Q'$} -- (2,0)
(0,0) -- (0,2)
(2,0) node[below] {$R'$}
(4.25,2.25) -- (1,-1) node[left] {$m$}
(3.75,2.75)--(-1,1.8) node[left] {$k$}
(0,2) node[above] {$P'$}--
(2,0)
(3.75,2.25) -- (-1,-0.6) node[left] {$l$}
(2.5,2.5) node[above] {$P$} -- 
(2.5,1.5) node[below] {$Q$} --
(3.5,1.5) node[below] {$R$} --
(2.5,2.5)
(0,1) node[left] {$b$}
(1,0) node[below] {$a$} ;
\end{tikzpicture}
\end{center}

In $\mathfrak A(\ca P,k_\infty)$: Let $a$ be the line through $R'$ parallel to $\ovl{QR}$ and let $Q'$ be the intersection of $a$ and $l$ (the intersection exists because $a$ is parallel to $\ovl{QR}$ which is apart from and intersects with $l$). Let $b$ be the line through $Q'$ parallel to $\ovl{PQ}$, and let $P'$ be the intersection of $b$ and $k$ (the intersection exists because $b$ is parallel to $\ovl{PQ}$ which is apart from and intersects with $k$). $Q'$ and $P'$ (as points of $\ca P$) lie outside $k_\infty$ and therefore by an earlier argument we may assume that they are both apart from $A$. Notice that since $\ovl{PQ}$ and $\ovl{QR}$ are apart from each other (and they intersect), then their parallel line $\ovl{P'Q'}$ and $\ovl{Q'R'}$ are apart from each other and they intersect. Hence, the points $P'$, $Q'$ and $R'$ are non-collinear.

Consider the following configuration
\begin{center}
\begin{tikzpicture}
 \draw 
(0,0) node[left] {$P$} 
-- (2,2) node[above] {$Q$} 
-- (4,0) node[right] {$R$}
-- (2,-2) node[below] {$A$}
-- (0,0)

(1.5,1.5)
-- (3,1)
-- (3.5,-0.5) node[right] {$R'$}
-- (1,-1)
-- (1.5,1.5)
(0.9,-1.1) node[left] {$P'$}
(2,1) node {$k_\infty$}
(3,0) node {$a$}
(2,-1.1) node {$\ovl{P'R'}$}
(0.9,0) node {$b$} ;
\end{tikzpicture}
\end{center}
and notice that the conditions of Desargues' axiom are satisfied by the above discussion. By Desargues' axiom, $\delta(k_\infty,\ovl{P'R'},P,R)$ is satisfied in $\ca P$. Therefore in $\mathfrak A(\ca P,k_\infty)$, $\ovl{PR}$ is parallel to $\ovl{P'R'}$. 

Let $C$ be the unique intersection of $\phi(k)$ and $\phi(l)$ in $\ca P$. Consider the following configuration
\begin{center}
\begin{tikzpicture}
 \draw 
(0,0) node[left] {$C$} 
-- (2,2) node[above] {$\phi(P)$} 
-- (4,0) node[right] {$\phi(R)$}
-- (2,-2) node[below] {$\phi(Q)$}
-- (0,0)

(1.5,1.5)
-- (3,1)
-- (3.5,-0.5) 
-- (1,-1) node[left] {$\phi(Q')$}
-- (1.5,1.5) node[left] {$\phi(P')$}
(2.1,0.9) node {$\phi(\ovl{P'R'})$}
(3,0) node {$l_\infty$}
(2,-1.1) node {$\phi(a)$}
(0.8,0) node {$\phi(b)$} ;
\end{tikzpicture}
\end{center}
and notice that the conditions of Desargues' axiom are satisfied because $\phi$ is an affine plane homomorphism and because $\phi(P'R')\parallel \ovl{\phi(P)\phi(R)}$. By Desargues' axiom, $\delta(\phi(\ovl{P'R'}),\phi(a),C,\phi(R))$ is satisfied in $\ca P$. Hence, $\phi(R')$ (which is $\phi(\ovl{P'R'})\cap\phi(a)$), $\phi(R)$ and $C$ lie on a common line. $\phi(m)$ is the unique line through $\phi(R)$ and $\phi(R')$, therefore $C\in \phi(m)$ as required.
\end{proof}

\begin{lem} \label{lempsi}
Let $\phi:\mathfrak A(\ca P,k_\infty) \to \mathfrak A(\ca Q,l_\infty)$ be a morphism of preaffine planes as above. Then, there exists a unique function $\psi:\ca P_{\text{pt}}\to \ca Q_{\text{pt}}$ such that for any point $A$ and line $l$ of $\ca P$, such that $l\# k_\infty$, if $A\in l$ then $\psi(A)\in \phi(l)$. Moreover, $\psi$ extends $\phi$, i.e. whenever $A\notin k_\infty$, $\psi(A)=\phi(A)$.
\end{lem}

\begin{proof}
Given a point $A$ of $\ca P$, there exist lines $k$, $l$, $m$ passing through $A$ such that $k\#l\#m\#k$. At least two of $k$, $l$ and $m$ are apart from $k_\infty$. Without loss of generality, we assume that $k$ and $l$ that are apart from $k_\infty$. Note that by the requirements of $\psi$, $\psi(A)$ needs to lie on both lines $\phi(k)$ and $\phi(l)$. Hence, we define $\psi(A)$ to be the intersection of $\phi(k)$ and $\phi(l)$. Let $k'$ and $l'$ be a second pair of lines, such that $k'\#l'\#k_\infty \# k'$ and $A=k'\cap l'$. By Lemma \ref{lempsiwd} the intersection of $\phi(k)$ and $\phi(l)$ lies on both $\phi(k')$ and $\phi(l')$. Hence, $\phi(k)\cap \phi(l)=\phi(k')\cap \phi(l')$, and therefore $\psi$ does not depend on the choice of $k$ and $l$.

By Lemma \ref{lempsiwd}, given a line $m$ of $\ca P$ which is apart from $k_\infty$, $A\in m$ implies that $\psi(A)\in \phi(m)$. Moreover, if $A\notin k_\infty$ and $k$ and $l$ are as above, then $\phi(A)$ lies on both $\phi(k)$ and $\phi(l)$, hence $\psi(A)=\phi(k)\cap \phi(l)=\phi(A)$. 
\end{proof}

Let $\psi:\ca P_{\text{pt}}\to \ca Q_{\text{pt}}$ be the unique morphism satisfying the conditions of the above lemma. By Proposition \ref{propptslines}, the morphism $\psi$ extends to a unique morphism of projective planes iff:
\begin{enumerate}
\item $\psi$ preserves and reflects $\#$,
\item if $A$, $B$, $C$ are points such that $A\#B$ and $C\in \ovl{AB}$, then ($\psi(A)\#\psi(B)$ and) $\psi(C)\in \ovl{\psi(A)\psi(B)}$,
\item given three non-collinear points, then their images under $\psi$ are non-collinear.
\end{enumerate}

We shall prove the three above points to show that the morphism $\psi$ extends to a unique morphism of projective planes.

\begin{lem}
$\psi$ preserves the apartness relation on points. 
\end{lem}

\begin{proof}
Let $A$ and $B$ be points of $\mathfrak A(\ca P,k_\infty)$, such that $A\#B$.

Let us first consider the case where $B\notin k_\infty$. Then $\psi(B)\notin l_\infty$, therefore either $\psi(A)\# \psi(B)$ or $\psi(A) \notin l_\infty$. In the first case, our claim is proved. In the second case, we conclude that $A\notin k_\infty$, therefore $\psi(A)=\phi(A)\# \phi(B)=\psi(B)$.

For the general case, given two points $A$ and $B$ of $\ca P$ such that $A\# B$, there exists $C$ on the affine plane such that $C\notin \ovl{AB}$. $\ovl{AC}$ and $\ovl{BC}$ are both lines of the affine plane, and $\psi(A)\in \phi(\ovl{AC})$ and $\psi(B)\in \phi(\ovl{BC})$. Furthermore, $\phi(\ovl{AC}) \#\phi(\ovl{BC})$, and they intersect at $\psi(C)$. By the above $\psi(A)\#\psi(C)$, hence $\psi(A)\notin \phi(\ovl{BC})$,  and therefore $\psi(A)\# \psi(B)$.
\end{proof}

\begin{lem}
If $A$, $B$ and $C$ are points of $\ca P$ such that $A\#B$ and $C\in \ovl{AB}$, then ($\psi(A) \# \psi(B)$ and), $\psi(C)\in \ovl{\psi(A)\psi(B)}$.
\end{lem}

\begin{proof}
Let us first consider the case where $\ovl{AB}\# k_\infty$. Then $\psi(A)$, $\psi(B)$ and $\psi(C)$ lie on $\phi(\ovl{AB})$, therefore $\psi(C)\in \phi(\ovl{AB})=\ovl{\psi(A)\psi(B)}$.

Whenever we have a point $D$ of $\ca P$ such that $D\notin k_\infty$, then either $\ovl{AB}\# k_\infty$ or $D\notin \ovl{AB}$. In the first case $\psi(C)\in \ovl{\psi(A) \psi(B)}$. Hence, in the rest of the proof, whenever we have such a $D$ we shall assume that it lies outside $\ovl{AB}$ (hence $D$ is apart from each of $A$, $B$ and $C$). $C$ is apart from at least one of $A$ and $B$. Without loss of generality, we assume that $A\#C$.

There exists a line $m$ in $\ca P$ such that $C\notin m$. Then, either $C\notin k_\infty$ or $m \# k_\infty$. In the first case, $\ovl{AB}\# k_\infty$ which is a case we have covered above, therefore we may assume that $m\# k_\infty$. There exist points $P$ and $P'$ lying on $m$ and outside $k_\infty$ such that $P\#P'$. Notice that $P$ and $P'$ are both apart from $C$. Let $k=\ovl{PC}$ and $l=\ovl{P'C}$. There exists a point $Q$ on $\ovl{PC}$ that lies outside $k_\infty$ such that $P\# Q$. Let $m'$ be the line through $Q$ parallel to $m$, and let $Q'$ be the intersection of $m'$ and $l$.

By the running assumptions, $P$, $Q$, $P'$ and $Q'$ are all apart from each of $A$, $B$ and $C$ and they all lie outside from $\ovl{AB}$. Also, the lines $k$, $l$, $m$ and $m'$ are all apart from $\ovl{AB}$.

\begin{center}
\begin{tikzpicture}
\draw
(4,4)-- (5,4)
(-1,4)--(0,4) node[above] {$C$}
--(0,0) node[left] {$P$}
--(2,0) node[right] {$P'$}
--(0,4) -- (4,4) node[above] {$B$}
--(0,1) node[left] {$Q$}
--(1.5,1) node[right] {$Q'$}
--(4,4)
(0,0)-- (2,4) node[above] {$A$}
--(2,0)
(0,0) -- (0,-1) node[left] {$k$}
(2,0) -- (2.5,-1) node[left] {$l$}
(0.6,1.7) node {$R$}
(2, 1.5) node[right] {$R'$}
(0.75,1) node[below] {$m'$}
(1,0) node[below] {$m$}
;
\end{tikzpicture}
\end{center}


$P\notin \ovl{AB}$, therefore $B\notin \ovl{PA}$, hence $\ovl{PA}\# \ovl{QB}$. Let $R$ be the intersection of $\ovl{PA}$ and $\ovl{QB}$. $R$ lies outside $\ovl{AB}$, hence either $R\notin k_\infty$ or $\ovl{AB}\# k_\infty$. We have already proved the result for the case where $\ovl{AB}\# k_\infty$, therefore we may assume that $R\notin k_\infty$. Similarly, $\ovl{P'A}$ is apart from $\ovl{Q'B}$, and we define their intersection to be $R'$. As before, $R'\notin \ovl{AB}$ and we may assume that $R'\notin k_\infty$.

The following configuration
\begin{center}
\begin{tikzpicture}
 \draw 
(0,0) node[left] {$R'$} 
-- (2,2) node[above] {$A$} 
-- (4,0) node[right] {$R$}
-- (2,-2) node[below] {$B$}
-- (0,0)

(1.5,1.5)
-- (3,1) node[right] {$P$}
-- (3.5,-0.5) node[right] {$Q$}
-- (1,-1) node[left] {$Q'$}
-- (1.5,1.5) node[left] {$P'$}
(2,1) node {$m$}
(3,0) node {$k$}
(2,-1) node {$m'$}
(1,0) node {$l$} ;
\end{tikzpicture}
\end{center}
satisfies the conditions of Desargues' axiom, therefore $\delta(m,m',R',R)$ is satisfied in $\ca P$. Hence, $\ovl{RR'}$ is parallel to both $\ovl{PP'}$ and $\ovl{QQ'}$ in $\mathfrak A(\ca P,k_\infty)$.

Let us consider the images of the points mentioned above via $\psi$. The following configuration
\begin{center}
\begin{tikzpicture}
  \draw 
(0,0) node[left] {$\psi(A)$} 
-- (2,2) node[above] {$\phi(R)$} 
-- (4,0) node[right] {$\psi(B)$}
-- (2,-2) node[below] {$\phi(R')$}
-- (0,0)

(1.5,1.5)
-- (3,1) node[right] {$\phi(Q)$}
-- (3.5,-0.5) node[right] {$\phi(Q')$}
-- (1,-1) node[left] {$\phi(P')$}
-- (1.5,1.5) node[left] {$\phi(P)$}
(2,1) node {$\phi(k)$}
(2.8,0) node {$\phi(m')$}
(2,-1.1) node {$\phi(l)$}
(0.75,0) node {$\phi(m)$} ;
\end{tikzpicture}
\end{center}
satisfies the conditions of Desargues' axiom because $\phi$ preserves the $\parallel$ relation. Hence, by Desargues' axiom $\delta(\phi(k),\phi(l),\psi(A),\psi(B))$ is satisfied in $\ca Q$. Therefore, $\psi(C)$ (which is the intersection of $\phi(k)$ and $\phi(l)$) lies on $\ovl{\psi(A)\psi(B)}$.
\end{proof}

\begin{lem}
If $A$, $B$, $C$ are non-collinear points of $\ca P$, then $\psi(A)$, $\psi(B)$, $\psi(C)$ are non-collinear points of $\ca Q$.
\end{lem}

\begin{proof}
The lines $\ovl{AB}$, $\ovl{BC}$ and $\ovl{AC}$ are apart from each other, hence at least two of them are apart from $k_\infty$. Without loss of generality, let us assume that $\ovl{BC}$ and $\ovl{AC}$ are apart from $k_\infty$. Then, $\psi(C)$ lies on the intersection of $\phi(\ovl{BC})$ and $\phi(\ovl{AC})$. $\psi(B)$ lies on $\ovl{BC}$ and is apart from $\psi(C)$ by an earlier lemma. Hence, $\psi(B)$ lies outside the line $\phi(\ovl{AC})$ which is the unique line through $\psi(A)$ and $\psi(C)$. Therefore, the points $\psi(A)$, $\psi(B)$ and $\psi(C)$ are non-collinear. 
\end{proof}

\begin{thrm} \label{thrmaffprojmorph}
Let $\ca P$ and $\ca Q$ be two projective planes containing the lines $k_\infty$ and $l_\infty$ respectively. A morphism of preaffine planes from $\mathfrak A(\ca P,k_\infty)$ to $\mathfrak A(\ca Q,l_\infty)$ can be extended uniquely to a morphism of projective planes from $\ca P$ to $\ca Q$.
\end{thrm}

\begin{proof}
By Lemma \ref{lempsi},  $\phi$ can be extended to $\psi$ on points and $\psi$ is the unique such extension which preserves collinear points. The next lemmas prove that $\psi$ preserves the apartness relation on points, collinear points and non-collinear points. Hence, by Proposition \ref{propptslines} it determines a unique morphism of projective planes from $\ca P$ to $\ca Q$.
\end{proof}

\section{Morphisms between affine planes over rings}

\begin{prop}
Given a ring homomorphism $\alpha:R \to S$, the induced morphism from the projective plane over $R$ to the projective plane over $S$ described in \ref{secprojmorph} restricts to one from the affine plane over $R$ to the affine plane over $S$.
\end{prop}

\begin{proof}
By Proposition \ref{propringproj} and Lemma \ref{lemprojaffmorph}.
\end{proof}

\begin{defn}
The \emph{affine group} over a local ring $R$, denoted by $G(R)$ is the group of invertible  $3\times 3$ matrices over $R$ of the form
$$\begin{pmatrix}
a_0 & b_0 & c_0 \\
a_1 & b_1 & c_1 \\
0 & 0 & 1
\end{pmatrix}.$$
\end{defn}

Given a local ring $R$, $G(R)$ is a subgroup of the projective general linear group $H(R)$. The inclusion $G(R)\hookrightarrow H(R)$ sends a matrix $M$ of $G(R)$ to the element of $H(R)$ represented by $M$. Therefore, an element of $G(R)$ induces an automorphism of the projective plane $\mathbb P(R)$ as described in \ref{secprojmorph}.

\begin{prop} \label{prop382}
Let $R$ be a local ring and let $M$ be an invertible matrix over $R$ of the form 
$\begin{pmatrix}
a_0 & b_0 & c_0 \\
a_1 & b_1 & c_1 \\
0 & 0 & 1
\end{pmatrix}$.
Then the induced automorphism of $\mathbb P(R)$ described in \ref{secprojmorph} restricts to an automorphism of $\mathbb A(R)$.
\end{prop}

\begin{proof}
By Proposition \ref{propmatrproj}, $M$ induces an automorphism of the projective plane $\mathbb P(R)$ which acts on points via left matrix multiplication by $M$ and acts on lines via left matrix multiplication by the matrix  
$$(M^{-1})^T=
( a_0b_1-a_1b_0)^{-1}
\begin{pmatrix}
b_1 & -a_1 & 0 \\
-b_0 & a_0 & 0 \\
 b_0 c_1 -b_1c_0& c_0a_1-c_1a_0 & a_0b_1-a_1b_0
\end{pmatrix},$$
and therefore it sends the line $(0,0,1)$ to the line $(0,0,1)$. By Lemma \ref{lemprojaffmorph}, this restricts to an endomorphism of preaffine planes $\mathfrak A(\mathbb P(R),(0,0,1))$ which is an automorphism. $\mathbb A(R)$ is isomorphic to $\mathfrak A(\mathbb P(R),(0,0,1))$, therefore $M$ also induces an automorphism of $\mathbb A(R)$. 
\end{proof}

As we did in the beginning of the chapter, we can express points of $\mathbb A(R)$ in the form $(x,y,1)$ where $x$, $y$ are in $R$. Then, the morphism of affine planes induced by a matrix $M$ of $G(R)$ acts on points via left matrix multiplication, and on lines via left matrix multiplication by the matrix $(M^{-1})^T$.


\begin{lem} \label{lemmatrixautoaf}
Let $A=(a_0,a_1)$, $B=(b_0,b_1)$ and $C=(c_0,c_1)$ be points of the affine plane over a local ring $R$ that are non-collinear. Then, the matrix
$$M=\begin{pmatrix}
a_0-c_0 & b_0-c_0 & c_0 \\
a_1-c_1 & b_1-c_1 & c_1 \\
0 & 0 & 1
\end{pmatrix}$$
is the unique element of $G(R)$ which induces (in the way described above) an automorphism sending $(1,0)$ to $A$, $(0,1)$ to $B$ and $(0,0)$ to $C$.
\end{lem}

\begin{proof}
The points $A=(a_0,a_1)$, $B=(b_0,b_1)$ and $C=(c_0,c_1)$ are non-collinear, therefore the matrix
$\begin{pmatrix}
a_0 & b_0 & c_0 \\
a_1 & b_1 & c_1 \\
1 & 1 & 1
  \end{pmatrix}$
has an invertible determinant. $M=\begin{pmatrix}
a_0-c_0 & b_0-c_0 & c_0 \\
a_1-c_1 & b_1-c_1 & c_1 \\
0 & 0 & 1
\end{pmatrix}$ has the same determinant, therefore it is also invertible, and therefore is a member $G(R)$. By Proposition \ref{prop382}, $M$ induces an automorphism of the affine plane which sends $(1,0)$ to $A$, $(0,1)$ to $B$ and $(0,0)$ to $C$.

Suppose that a matrix $M'=\begin{pmatrix}
x_0 & y_0 & z_0 \\
x_1 & y_1 & z_1 \\
0 & 0 & 1
\end{pmatrix}$
of $G(R)$ also induces an automorphism of $\mathbb A(R)$ which sends $(1,0)$ to $A$, $(0,1)$ to $B$ and $(0,0)$ to $C$. Then,
$(a_0,a_1)=M'(1,0)=(x_0+z_0,x_1+z_1)$,
$(b_0,b_1)=M'(1,0)=(y_0+z_0,y_1+z_1)$ and
$(c_0,c_1)=M'(0,0)=(z_0,z_1)$. By these equations, we can see that $M=M'$, and therefore conclude that $M$ is the unique member of $G(R)$ sending $(1,0)$ to $A$, $(0,1)$ to $B$ and $(0,0)$ to $C$.
\end{proof}

Given a local ring $R$, let $\omega(R)$ be the set of triples of points that are non-collinear. The left $G(R)$-action on points of $\mathbb A(R)$ sends non-collinear points to non-collinear points, therefore it extends to a left action on $\omega(R)$.

\begin{thrm} \label{thrmleftGtors}
Given a local ring $R$, $\omega(R)$ is a left $G(R)$-torsor via the action described above.
\end{thrm}

\begin{proof}
Given $(A,B,C)$ and $(A',B',C')$ in $\omega(R)$ by Lemma \ref{lemmatrixautoaf} there exist unique $g$ and $g'$ in $G(R)$ such that $g$  and $g'$ send $\begin{pmatrix}
\begin{pmatrix} 1 \\ 0  \end{pmatrix}, & \begin{pmatrix} 0 \\ 1 \end{pmatrix}, & \begin{pmatrix} 0 \\ 0 \end{pmatrix}
\end{pmatrix}$ to $(A,B,C)$ and $(A',B',C')$ respectively.

Then, $g'\circ g^{-1}$ sends $(A,B,C)$ to $(A',B',C')$, therefore this $G(R)$-action is transitive.

Suppose that $h$ in $G(R)$ sends $(A,B,C)$ to $(A',B',C')$. Then, $h\circ g$ sends $\begin{pmatrix}
\begin{pmatrix} 1 \\ 0  \end{pmatrix}, & \begin{pmatrix} 0 \\ 1 \end{pmatrix}, & \begin{pmatrix} 0 \\ 0 \end{pmatrix}
\end{pmatrix}$ to $(A',B',C')$. Hence, $h\circ g=g'$  by Lemma \ref{lemmatrixautoaf} and therefore $h=g'\circ g^{-1}$. Thus, $\omega(R)$ is a $G(R)$-torsor under this action.
\end{proof}

\begin{rmk}
Notice that we have an isomorphism $G(R)\to \omega(R)$ which maps $g$ to the triple of points $(g(1,0), g(0,1), g(0,0))$. Moreover this isomorphism commutes with left $G(R)$-action on $G(R)$ via group multiplication and the left $G(R)$-action on $\omega(R)$ described above.
\end{rmk}

\begin{lem} \label{lemringpartaf}
Let $R$ and $S$ be local rings and let $\phi: \mathbb A(R) \to \mathbb A(S)$ be a morphism of affine planes that sends $(1,0)$, $(0,1)$ and $(0,0)$ to $(1,0)$, $(0,1)$ and $(0,0)$ respectively then there exists a unique ring homomorphism $\sigma :R \to S$ such that $\phi=\mathbb A(\sigma)$.
\end{lem}

\begin{proof}
Let $\phi: \mathbb A(R) \to \mathbb A(S)$ be a morphism of affine planes that sends $(1,0)$, $(0,1)$ and $(0,0)$ to $(1,0)$, $(0,1)$ and $(0,0)$. By Theorem \ref{thrmaffprojmorph}, it extends uniquely to a morphism of projective planes $\psi:\mathbb P(R)\to \mathbb P(S)$ which sends the line $(0,0,1)$ to $(0,0,1)$. $\psi$ sends the unique line through $(0,0,1)$ and $(1,0,1)$ to the unique line through $(0,0,1)$ and $(1,0,1)$, therefore $\psi$ sends the line $(0,1,0)$ to the line $(0,1,0)$. $\psi$ sends the unique intersection of $(0,1,0)$ and $(0,0,1)$ to the unique intersection of $(0,1,0)$ and $(0,0,1)$, therefore it sends the point $(1,0,0)$ to the point $(1,0,0)$. By similar arguments, $\psi$ sends the points to $(0,1,0)$ and $(1,1,1)$ to $(0,1,0)$ and $(1,1,1)$ respectively.

Hence, by Lemma \ref{lemringpart}, there exists a ring homomorphism $\sigma :R\to S$ such that $\psi=\mathbb P(\sigma)$. Therefore, $\phi=\mathbb P(\sigma)$. The uniqueness of $\sigma$ follows from the fact that given $x$ in $R$, $\phi(x,0)=(\sigma(x),0)$.
\end{proof}

\begin{rmk}
Alternatively, we could have proven this lemma in a similar way to the way we proved Lemma \ref{lemringpart}.
\end{rmk}

\begin{thrm} \label{thrmaffmorph}
 Let $\phi: \mathbb A(R) \to \mathbb A(S)$ be a morphism of affine planes over the local rings $R$ and $S$. Then there exists a unique invertible matrix $M$ over $S$ of the form
$\begin{pmatrix}
a_0 & b_0 & c_0 \\
a_1 & b_1 & c_1 \\
0 & 0 & 1
\end{pmatrix}$
and a unique ring homomorphism $\alpha: R\to S$ such that $\phi =M \mathbb A(\alpha)$.
\end{thrm}

\begin{proof}
$\phi$ sends the points $(1,0)$, $(0,1)$ and $(0,0)$ to the points $A$, $B$ and $C$ respectively. $\phi$ is a morphism of affine planes, therefore  $A$, $B$ and $C$ are non-collinear. By Lemma \ref{lemmatrixautoaf}, there exists a  unique matrix $M$ in $S$ of the appropriate form which induces an automorphism of the projective plane over $S$ which sends $(1,0)$, $(0,1)$ and $(0,0)$ to the points $A$, $B$ and $C$ respectively. $M^{-1}$ induces the inverse automorphism.

$M^{-1}\phi$ satisfies the conditions of Lemma \ref{lemringpartaf}, so it is of the form $\mathbb{A}(\alpha)$ for a unique ring homomorphism $\alpha:R\to S$.

Hence, $\phi=M\mathbb{A}(\alpha)$ for unique $M$ and $\alpha$.
\end{proof}

\begin{rmk}
Let $\phi:\mathbb A(R)\to \mathbb A(S)$ be a morphism of affine planes such that $\phi=g\circ \mathbb A(\alpha)$ where $g$ is in $\mathbb G(S)$ and $\alpha:R\to S$ a ring homomorphisms. Let $\psi: \mathbb A(S) \to \mathbb A(T)$ be a second morphism of affine planes such that and $\psi=k\circ \mathbb A(\beta)$ where $g$ is in $G(T)$ and $\beta:S\to T$ is a ring homomorphism.Then, 
\begin{displaymath}
\begin{split}
\psi\circ \phi & = k\circ \mathbb A(\beta)\circ g \circ \mathbb A(\alpha) \\
& = k \circ (\beta(g)) \circ \mathbb A(\beta) \circ \mathbb A(\alpha) \\
& = (k\circ \beta(g)) \circ \mathbb A(\beta\circ \alpha),
\end{split}
\end{displaymath}
where $\beta(g)$ is the image of $g$ of $G(S)$ under $\beta$.
\end{rmk}

\section{Desargues' axioms on the affine plane}


The classical theory of affine planes includes two versions of Desargues' axiom. Our theory of affine planes also includes two versions of Desargues' axiom. We state these two axioms and prove that they hold on preaffine planes constructed from projective planes with a line and therefore also on affine planes over local rings.

\subsection*{Desargues' small axiom}
A preaffine plane satisfies \emph{Desargues' small axiom} when given $k$, $l$, $m$, $n_A$, $n_A'$, $n_C$, $n_C'$ lines and $A$, $A'$, $B$, $B'$, $C$, $C'$ points of the preaffine plane such that:
\begin{enumerate}
\item $k\parallel l \parallel m$,
\item $n_A\parallel n_A'$,
\item $n_C\parallel n_C'$,
\item $A, B\in n_A$,
\item $B, C\in n_C$,
\item $A', B'\in n_A'$,
\item $B', C'\in n_C'$,
\item $A, A'\in k$,
\item $B, B'\in l$,
\item $C, C'\in m$,
\item $A\# C$,
\item $A'\#C'$,
\item $n_A\# l$,
\item $n_C\# l$,
\end{enumerate}
then $\ovl{AC}$ is parallel to $\ovl{A'C'}$.
\begin{center}
 \begin{tikzpicture}
\draw 
(0,-0.5) -- (0,6) node[right] {$k$} 
(3,-0.5) -- (3,6) node[right] {$l$} 
(5,-0.5) -- (5,6)  node[right] {$m$}

(0,0) node[ left] {$A$} 
--(3,1.5) node[ right] {$B$}
--(5,0) node[right] {$C$}
(1.5,0.75) node[above] {$n_A$}
(4,0.75) node[above] {$n_C$}

(0,4)  node[left] {$A'$} 
--(3,5.5) node[right] {$B'$}
--(5,4) node[right] {$C'$}
(1.5,4.75) node[above] {$n'_A$}
(4,4.75) node[above] {$n'_C$};

\draw[red, thick, densely dotted] 
(0,0) -- (5,0)
(0,4) -- (5,4);
\end{tikzpicture}
\end{center}

\begin{thrm} \label{thrmsmallDes}
Let $\ca P$ be a preprojective plane satisfying Desargues' axiom and let $l_\infty$ be a line of $\ca P$. Then, the preaffine plane $\mathfrak A(\ca P, l_\infty)$ satisfies Desargues' small axiom.
\end{thrm}

\begin{proof}
Let $\ca P$ be a preprojective plane satisfying Desargues' axiom and let $l_\infty$ be a line of $\ca P$. Let $k$, $l$, $m$ be lines and $A$, $A'$, $B$, $B'$, $C$, $C'$ points of the preaffine plane $\mathfrak A(\ca P, l_\infty)$ satisfying the conditions listed in Desargues' small axiom above.

Let $P$ be the intersection of $k$ (or $l$ or $m$) with $l_\infty$. We apply (the projective) Desargues' axiom on the following configuration of $\ca P$.

\begin{center}
\begin{tikzpicture}
 \draw 
(0,0) node[left] {$A$} 
-- (2,2) node[above] {$P$} 
-- (4,0) node[right] {$C$}
-- (2,-2) node[below] {$B$}
-- (0,0)

(1,1)
-- (3,1)
-- (3,-1)
-- (1,-1)
-- (1,1)
(2,1) node[above] {$\ovl{A'C'}$}
(3,0) node[right] {$n_C'$}
(2,-1) node[above] {$l_\infty$}
(1,0) node[right] {$n_A'$} ;
\end{tikzpicture}
\end{center}
The picture is a bit misleading because $P$ lies on $l_\infty$. The conditions of Desargues' axiom are all satisfied, therefore $\delta(\ovl{A'C'},l_\infty, A,C)$ is satisfied, i.e. $\ovl{AC}\parallel \ovl{A'C'}$.
\end{proof}

\begin{rmk}
Given a local ring $R$, the preaffine plane $\mathbb A(R)$ is isomorphic to $\mathfrak A(\mathbb P(R),(0,0,1))$. $\mathbb P(R)$ satisfies Desargues' axiom, therefore $\mathbb A(R)$ satisfies Desargues' small axiom.
\end{rmk}

\subsection*{Desargues' big axiom}
A preaffine plane satisfies \emph{Desargues' big axiom} when given $k$, $l$, $m$, $n_{AB}$, $n_{BC}$, $n_{AC}$ lines and $P$, $A$, $A'$, $B$, $B'$, $C$, $C'$ points of the preaffine plane such that:
\begin{enumerate}
\item $P$ lies on all three lines $k,l,m$,
\item $A, B \in n_{AB}$, 
\item $B, C \in n_{BC}$,
\item $A, C \in n_{AC}$,
\item $A, A'\in k$,
\item $B, B'\in l$,
\item $C, C'\in m$,
\item $P$ lies outside both $n_{AB}$ and $n_{BC}$,
\item $B'$ lies on the line through $A'$ parallel to $n_{AB}$,
\item $C'$ lies on the line through $B'$ parallel to $n_{BC}$,
\end{enumerate}
then $C'$ lies on the line through $A'$ parallel to $n_{AC}$.

\begin{center}
\begin{tikzpicture}
\draw 
(0,0) node[left] {$P$} -- (6,3) node[right] {$k$} 
(0,0) -- (7,0) node[right] {$l$} 
(0,0) -- (6,-1.5)  node[right] {$m$};

\draw 
(3, -0.75) node[ below] {$C'$} 
--(5,0) node[ below] {$B'$}
(1.5,-0.375)  node[below] {$C$} 
--(2.5,0) node[ below] {$B$};

\draw 
(5,0)--(4.5,2.25) node[above] {$A'$}
 (2.5,0)--(2.25,1.125) node[above] {$A$};

\draw[red, thick, densely dotted] 
(4.5,2.25)--(3,-0.75) 
(2.25,1.125)--(1.5,-0.375);
\end{tikzpicture}
\end{center}

\begin{rmk}
Notice that in the case where $A\#C$ and $A'\#C'$, the conclusion of the axiom can be changed to $\ovl{AC}\parallel \ovl{A'C'}$.
\end{rmk}

\begin{thrm} \label{thrmbigDes}
Let $\ca P$ be a preprojective plane satisfying Desargues' axiom and let $l_\infty$ be a line of $\ca P$. Then the preaffine plane $\mathfrak A(\ca P,l_\infty)$ satisfies Desargues' big axiom.
\end{thrm}

\begin{proof}
Let $\ca P$ be a preprojective plane satisfying Desargues' axiom and let $l_\infty$ be a line of $\ca P$. Let $k$, $l$, $m$, $n_{AB}$, $n_{BC}$, $n_{AC}$ be lines of $\mathfrak A(\ca P, l_\infty)$ and let $P$, $A$, $A'$, $B$, $B'$, $C$, $C'$ be points of $\mathfrak A(\ca P, l_\infty)$ satisfying the conditions of the above axiom.

We apply (the projective) Desargues' axiom on the following configuration of $\ca P$.

\begin{center}
\begin{tikzpicture}
 \draw 
(0,0) node[left] {$A'$} 
-- (2,2) node[above] {$P$} 
-- (4,0) node[right] {$C'$}
-- (2,-2) node[below] {$B'$}
-- (0,0)

(1,1)
-- (3,1)
-- (3,-1)
-- (1,-1)
-- (1,1)
(2,1) node[above] {$n_{AC}$}
(3,0) node[right] {$n_{BC}$}
(2,-1) node[above] {$l_\infty$}
(1,0) node[right] {$n_{AB}$} ;
\end{tikzpicture}
\end{center}
The conditions of Desargues' axiom are all satisfied, therefore $\delta(n_{AC},l_\infty, A',C')$ is satisfied, i.e. $C'$ lies on the line through $A'$ parallel to $n_{AC}$.
\end{proof}

\begin{rmk}
Given a local ring $R$, the preaffine plane $\mathbb A(R)$ is isomorphic to $\mathfrak A(\mathbb P(R),(0,0,1))$. $\mathbb P(R)$ satisfies Desargues' axiom, and therefore $\mathbb A(R)$ satisfies Desargues' big axiom.
\end{rmk}

\section{Further versions of Desargues' theorem} \label{secfurthdes}

In this section, we prove some further versions of Desargues' theorem on the affine plane which are consequences of Desargues' big and small axioms. These versions are going to be used in the proofs of results about dilatations and translations in Chapter \ref{chalocal}.

\begin{thrm} \label{thrmdesmall}
Let $\ca A$ be a preaffine plane satisfying Desargues' small axiom.
Let $k$, $l$, $m$ be lines of $\ca A$ and let $A$, $A'$, $B$, $B'$, $C$, $C'$ be points of $\ca A$. If the following are true:
\begin{enumerate}
\item $k\parallel l \parallel m$,
\item $k \# l \# m$,
\item $A\# C$,
\item $A'\#C'$,
\item $A, A'\in k$,
\item $B, B'\in l$,
\item $C, C'\in m$,
\end{enumerate}
we can conclude that $A\#B\#C$ and $A'\#B'\#C'$. Then if $\ovl{AB}$ is parallel to $\ovl{A'B'}$ and $\ovl{BC}$ parallel to $\ovl{B'C'}$, then $\ovl{AC}$ is parallel to $\ovl{A'C'}$.
\end{thrm}
\begin{center}
\begin{tikzpicture}
\draw 
(0,0) -- (0,3.7) node[right] {$k$} 
(3,0) -- (3,3.7) node[right] {$l$} 
(4,0) -- (4,3.7)  node[right] {$m$}; 

\draw 
(0,0.4) node[ left] {$A'$} 
--(3,1.75) node[ right] {$B'$}
(0,1.5)  node[left] {$A$} 
--(3,2.85) node[right] {$B$};

\draw 
(3,1.75)--(4,1.1) node[right] {$C'$}
 (3,2.85)--(4,2.2) node[right] {$C$};

\draw[red, thick, densely dotted] 
(4,1.1)--(0,0.4) 
(4,2.2)--(0,1.5);
\end{tikzpicture}
\end{center}

\begin{proof}
 This theorem is an immediate consequence of Desargues' small axiom. The conditions listed here imply the ones of Desargues' small axiom with $n_A=\ovl{AB}$, $n_A'=\ovl{A'B'}$, $n_C=\ovl{BC}$ and $n_C'=\ovl{B'C'}$. Hence, $\ovl{AC}$ is parallel to $\ovl{A'C'}$.
\end{proof}

We now proceed to prove a few lemmas that will enable us to prove Theorem \ref{thrmdesmall4}. In the classical treatment of the affine plane, Theorem \ref{thrmdesmall4} is a consequence of the above theorem. The classical proof considers two cases. In the first case, there exists a line $k'$ on the affine plane which is parallel to $k$ and apart from the lines $k$ and $n$ and then the above theorem proves the result. In the second case no such line $k'$ exists, which in the classical case implies that the affine plane is isomorphic to the affine plane over the finite field of order 2, where the theorem still holds. This approach is used implicitly in \cite[Theorem 2.17]{Artin}. When working with affine planes over local rings it is no longer true that we can consider these two separate cases: for example the affine plane over $\mathbb Z/(n)$ where $n$ is a power of $2$ does not contain three parallel lines which are all apart from each other. Hence, the above argument would fail in our case. We can still prove Theorem \ref{thrmdesmall4} using Desargues' small and big axioms.




\begin{lem} \label{lempar1}
Let $\ca A$ be a preaffine plane satisfying Desargues' big and small axioms.
 Let $k$ and $l$ be lines of $\ca A$ and let $A_0$, $A_1$, $A_2$, $B_0$, $B_1$ and $B_2$ be points of $\ca A$. If the following are true:
 \begin{enumerate}
  \item $k\# l$,
\item $k\parallel l$,
\item $A_0\# A_1 \# A_2$,
\item $B_0 \# B_1 \# B_2$,
\item $A_0, A_1, A_2 \in k$,
\item $B_0, B_1, B_2 \in l$
\end{enumerate}
then all the points on $k$ are apart from all the points on $l$. If moreover $\ovl{A_0B_0}\parallel \ovl{A_1B_1} \parallel \ovl{A_2B_2}$ and $\ovl{A_1B_0}\parallel \ovl{A_2B_1}$, then $\ovl{A_0B_1}\parallel \ovl{A_1B_2}$.
\end{lem}

\begin{center}
\begin{tikzpicture}
\draw 
(0,-0.5) -- (0,3.5) node[left] {$k$}
(2,-0.5) -- (2,3.5) node[left] {$l$}

(0,0) node[left] {$A_0$}--(2,0) node[right] {$B_0$}
(0,1.5) node[left] {$A_1$}--(2,1.5) node[right] {$B_1$}
(0,3) node[left] {$A_2$}--(2,3) node[right] {$B_2$}
(0,1.5)--(2,0)
(0,3) -- (2,1.5);

\draw[red, thick, densely dotted]
 (0,0)--(2,1.5)
(0,1.5) -- (2,3);
\end{tikzpicture}
\end{center}

\begin{proof}
 $B_1\notin \ovl{A_1A_2}$, therefore $\ovl{A_2B_1}$ is apart from $\ovl{A_1B_1}$. Hence, $\ovl{A_2B_1}$ is also apart and intersects $\ovl{A_0B_0}$, so let their intersection point be $X_0$.

Let $m$ be the line through $X_0$ parallel to $l$. $k$ and $l$ are apart from each other and parallel, hence $B_1\notin k$. Therefore, $B_1\notin \ovl{A_1A_2}$ which implies that $A_1\notin \ovl{A_2B_1}$. Hence, the lines $\ovl{A_1B_0}$ and $\ovl{A_2B_1}$ are apart from each other (and parallel), therefore $X_0\notin \ovl{A_1B_0}$. Hence, $X_0\#B_0$. $A_0\notin \ovl{B_0B_1}$, and therefore $B_1\notin \ovl{A_0B_0}=\ovl{B_0X_0}$. Hence, $B_1\notin \ovl{B_0X_0}$, and therefore $X_0\notin \ovl{B_0B_1}=l$. $X_0\notin l$ and $X_0\in m$, hence $l\# m$.

Let $X_1$ be the intersection of $\ovl{A_1B_1}$ and $m$, and let $X_2$ be the intersection of $\ovl{A_0B_1}$ and $m$.
\begin{center}
\begin{tikzpicture}
\draw 
(0,-0.5) -- (0,3.5) node[left] {$k$}
(2,-0.5) -- (2,3.5) node[left] {$l$}
(4,-0.5) -- (4,3.5) node[left] {$m$}

(0,0) node[left] {$A_0$}--(2,0) node[right] {$B_0$}--(4,0) node[right] {$X_0$}
(0,1.5) node[left] {$A_1$}--(2,1.5) node[right] {$B_1$} -- (4,1.5) node[right] {$X_1$}
(0,3) node[left] {$A_2$}--(2,3) node[right] {$B_2$}
(4,3) node[right] {$X_2$}
(2,1.5)--(4,0)
(0,1.5)--(2,0)
(0,3) -- (2,1.5)

 (0,0)--(2,1.5)--(4,3);

\end{tikzpicture}
\end{center}

We apply big Desargues' axiom on the concurrent lines $\ovl{A_0B_0}$, $\ovl{A_0B_1}$ and $\ovl{A_0A_1}$ and the triangles $B_0B_1A_1$ and $X_0X_2A_2$. $\ovl{A_1B_0} \parallel \ovl{A_2B_1}$ and $\ovl{B_0B_1}\parallel m$, therefore we conclude that $X_2$ lies on the line through $A_2$ parallel to $\ovl{A_1B_1}$. Hence, $X_2\in \ovl{A_2B_2}$.

By big Desargues' axiom on the concurrent lines $\ovl{B_1A_1}$, $\ovl{B_1A_0}$, $\ovl{B_1B_0}$ and the triangles $A_1A_0B_0$ and $X_1X_2B_1$. $\ovl{A_1A_0}\parallel \ovl{X_1X2}$ and $\ovl{A_0B_0}\parallel \ovl{X_2B_1}$, hence $\ovl{A_1B_0}\parallel \ovl{B_2X_1}$. Therefore $\ovl{B_2X_1}$ is also parallel to $\ovl{B_1X_0}$.

We are now applying small Desargues' axiom on the parallel lines $k$, $l$ and $m$, and the triangles $A_0X_0B_1$ and $A_1X_1B_2$. The appropriate lines are parallel and the conditions of the axiom are satisfied, hence we conclude that $\ovl{A_0B_1}\parallel \ovl{A_1B_2}$.
\end{proof}

\begin{lem} \label{lempar2}
Let $\ca A$ be a preaffine plane satisfying Desargues' big and small axioms.
Let $k$, $l$ be lines of $\ca A$ and let $A$, $B$, $C$, $D$, $A'$, $B'$, $C'$, $D'$ be points of $\ca A$. If the following hold:
\begin{enumerate}
\item $k\# l$,
\item  $k\parallel l$,
\item $A\# C$, $B\#D$, $A'\# C'$ and $B'\# D'$,
\item $A, A', C, C'\in k$,
\item $B, B', D, D'\in l$,
\end{enumerate}
then all the points on $k$ are apart from all the points on $l$. If $\ovl{AB}\parallel \ovl{CD} \parallel  \ovl{A'B'}\parallel \ovl{C'D'}$ and $\ovl{BC}\parallel \ovl{B'C'}$, then $\ovl{AD}$ is parallel to $\ovl{A'D'}$.
\end{lem}

\begin{center}
\begin{tikzpicture}
\draw 
(0,-0.5) -- (0,5) node[left] {$k$}
(2,-0.5) -- (2,5) node[left] {$l$}

(0,0) node[left] {$A$}--(2,0) node[right] {$B$}--
(0,1.5) node[left] {$C$}--(2,1.5) node[right] {$D$}

(0,3) node[left] {$A'$}--(2,3) node[right] {$B'$}--
(0,4.5) node[left] {$C'$} -- (2,4.5) node[right] {$D'$} ;

\draw[red, thick, densely dotted]
(0,0)--(2,1.5)
(0,3) -- (2, 4.5);
\end{tikzpicture}
\end{center}

\begin{proof}
Let $W$ be the intersection of $k$ and the line through $D$ parallel to $\ovl{BC}$. Let $V$ be the intersection of $l$ and the line through $W$ parallel to $\ovl{AB}$ as in the following picture.
\begin{center}
\begin{tikzpicture}
\draw 
(0,-0.5) -- (0,3.5) node[left] {$k$}
(2,-0.5) -- (2,3.5) node[left] {$l$}
(0,0) node[left] {$A$}--(2,0) node[right] {$B$}
(0,1.5) node[left] {$C$}--(2,1.5) node[right] {$D$}
(0,3) node[left] {$W$}--(2,3) node[right] {$V$}
(0,1.5)--(2,0)
(0,3) -- (2,1.5)
(0,0)--(2,1.5);
\end{tikzpicture}
\end{center}
Then, by Lemma \ref{lempar1}, $\ovl{AD}\parallel \ovl{CV}$.

$A\#C$, therefore $C'$ is apart from at least one of $A$ and $C$. Because of the symmetry of the two cases we shall only consider the case where $A\#C'$. (In particular to prove the lemma in the case where $C'\# C$ we just replace in the remaining proof $A$, $B$, $C$ and $D$ by $C$, $D$, $W$ and $V$ respectively and we use that $\ovl{AD}\parallel \ovl{CV}$).

 $B' \notin \ovl{A'C'}$, therefore $\ovl{B'C'}$ is apart from $\ovl{A'B'}$. Hence, $\ovl{B'C'}$ is also apart and intersects $\ovl{AB}$, so let their intersection point be $X$.

Let $m$ be the line through $X$ parallel to $k$. Let $Y$ be the intersection of $m$ and $\ovl{AD}$.

\begin{center}
\begin{tikzpicture}
\draw 
(0,-0.5) -- (0,5) node[left] {$k$}
(2,-0.5) -- (2,5) node[left] {$l$}
(6,-0.5) -- (6,5) node[left] {$m$}

(0,0) node[left] {$A$}--(2,0) node[right] {$B$}--
(0,1.5) node[left] {$C$}--(2,1.5) node[right] {$D$}

(0,3) node[left] {$A'$}--(2,3) node[right] {$B'$}--
(0,4.5) node[left] {$C'$} -- (2,4.5) node[right] {$D'$}

(2,0) -- (6,0) node[right] {$X$}
--(2,3)
(0,0)--(6,4.5) node[right]{$Y$};
\end{tikzpicture}
\end{center}
$k\#l$, therefore $B'\notin \ovl{AC'}$ which implies that $A\notin \ovl{B'C'}$ and therefore that $A\# X$. $B\notin \ovl{AC}$, therefore $C\notin \ovl{AB}=\ovl{AX}$. Hence, $X\notin \ovl{AC}$, and therefore $k\#m$.

By Desargues' big axiom on the concurrent lines $\ovl{AD}$, $\ovl{AB}$, $\ovl{AC}$ and the triangles $DBC$ and $YXC'$ we conclude that $Y$ lies on the line through $C'$ parallel to $\ovl{CD}$, i.e. $Y\in \ovl{C'D'}$.

We can now apply Desargues' big axiom on the concurrent lines $\ovl{C'D'}$, $\ovl{C'B'}$, $\ovl{C'A'}$, and the triangles $A'B'D'$ and $AXY$. Hence, we conclude that $\ovl{A'D'}$ is parallel to $\ovl{AY}$. Therefore, $\ovl{AD}\parallel \ovl{A'D'}$.
\end{proof}

\begin{rmk}
Notice that if we did not know that $k\#m$ (or equivalently that $A\#C'$) in the last application of Desargues' big axiom above the conclusion would have been that $Y$ lies on the line through $A$ parallel to $\ovl{A'D'}$. In the case where $A=Y$, this would not have given us any extra information and in particular we would not have been able to conclude that $\ovl{AD} \parallel \ovl{A'D'}$.
\end{rmk}

\begin{lem} \label{lempar3}
Let $\ca A$ be a preaffine plane satisfying Desargues' big and small axioms.
Let $A, B, C, D, A', B', C', D'$ be points of $\ca A$ and let $k$ and $l$ be lines of $\ca A$ such that:
\begin{enumerate}
 \item $k\parallel l$,
\item $k\# l$,
\item $B\#D$,
\item $A, A', C, C'\in k$,
\item $B, B', D, D'\in l$,
\end{enumerate}
Then, $A\# B\# C\# D \# A$ and $A'\# B'\# C'\# D' \# A'$. If $\ovl{AB}\parallel\ovl{A'B'}$, $\ovl{BC}\parallel\ovl{B'C'}$, and $\ovl{CD}\parallel \ovl{C'D'}$, then $\ovl{AD}$ is parallel to $\ovl{A'D'}$.
\end{lem}

\begin{center}
\begin{tikzpicture}
\draw 
(0,-0.5) -- (0,6) node[left] {$k$}
(2,-0.5) -- (2,6) node[left] {$l$}

(0,0) node[left] {$A$}--(2,0) node[right] {$B$}--
(0,1) node[left] {$C$}--(2,1.5) node[right] {$D$}

(0,3.5) node[left] {$A'$}--(2,3.5) node[right] {$B'$}--
(0,4.5) node[left] {$C'$} -- (2,5) node[right] {$D'$} ;

\draw[red, thick, densely dotted]
(0,0)--(2,1.5)
(0,3.5) -- (2, 5);
\end{tikzpicture}
\end{center}

\begin{proof}
Let $n$ be the line through $D$ parallel to $\ovl{BC}$. The conditions of the lemma make sure that the intersections we are going to mention exist and are unique. Let $X$ be the intersection of $n$ and $\ovl{AB}$ and let $Y$ be the intersection of $n$ and $k$. Let $m$ be the line through $X$ parallel to $k$.

Let $n'$ be the line through $D'$ parallel to $\ovl{BC}$. Let $X'$ be the intersection of $n'$ and $m$ and let $Y'$ be the intersection of $n'$ and $k$.

\begin{center}
\begin{tikzpicture}
\draw 
(0,-0.5) -- (0,6.5) node[left] {$k$}
(2,-0.5) -- (2,6.5) node[left] {$l$}
(5,-0.5) -- (5,6.5) node[left] {$m$}

(0,0) node[left] {$A$}--(2,0) node[right] {$B$}--
(0,1) node[left] {$C$}--(2,1.5) node[right] {$D$} --
(0,2.5) node[left] {$Y$}
(2,0) -- (5,0) node[right] {$X$}
-- (2,1.5) 

(3.5,0.75) node[above] {$n$}

(0,3.5) node[left] {$A'$}--(2,3.5) node[right] {$B'$}--
(0,4.5) node[left] {$C'$} -- (2,5) node[right] {$D'$} --
(0,6) node[left] {$Y'$} 
(2,5) -- (5, 3.5) node[right] {$X'$}

(3.5,4.25) node[above] {$n'$};
\end{tikzpicture}
\end{center}
$k$ and $l$ are parallel and apart from each other, hence $Y\notin l=\ovl{BD}$. Hence, $B\notin \ovl{YD}$, and therefore $\ovl{BC}\# \ovl{YD}$. Hence, $B\notin \ovl{YD}$ and $B\# X$. $A\notin \ovl{BD}$, therefore $D\notin \ovl{AB}=\ovl{BX}$. Hence, $X\notin \ovl{BD}=l$. $X$ lies on $m$ and outside $l$, therefore $l\#m$.

Since $l\#m$, $X$ is apart from both $B$ and $D$ and $X'$ is apart from both $B'$ and $D'$.

We are now going to apply Lemma \ref{lempar2} on the parallel lines $l$ and $k$, and on the quadrilaterals $BCDY$ and $B'C'D'Y'$. The appropriate conditions are satisfied therefore $\ovl{BY}\parallel \ovl{B'Y'}$.

Next, we apply Desargues' small axiom on the lines $m$, $k$, $l$, and the triangles $XYB$ and $X'Y'B'$. All the conditions of the axiom are satisfied, therefore $\ovl{XB}\parallel \ovl{X'B'}$, hence $X'\in \ovl{A'B'}$.

Finally, we apply Desargues' small axiom on the lines $k$, $m$, $l$, and the triangles $AXD$ and $A'X'D'$, and we conclude that $\ovl{AD}\parallel \ovl{A'D'}$.
\end{proof}

The following lemma is the same as Lemma \ref{lempar3} without the condition that $B\#D$.

\begin{lem} \label{lempar4}
Let $\ca A$ be a preaffine plane satisfying Desargues' big and small axioms.
Let $A, B, C, D, A', B', C', D'$ be points of $\ca A$ and $k$ and $l$ be lines of $\ca A$ such that:
\begin{enumerate}
 \item $k\parallel l$,
\item $k\# l$,
\item $A, A', C, C'\in k$,
\item $B, B', D, D'\in l$,
\end{enumerate}
Then, $A\# B\# C\# D \# A$ and $A'\# B'\# C'\# D' \# A'$. If $\ovl{AB}\parallel\ovl{A'B'}$, $\ovl{BC}\parallel\ovl{B'C'}$, and $\ovl{CD}\parallel \ovl{C'D'}$, then $\ovl{AD}$ is parallel to $\ovl{A'D'}$. 

\end{lem}

\begin{center}
\begin{tikzpicture}
\draw 
(0,-0.5) -- (0,6) node[left] {$k$}
(2,-0.5) -- (2,6) node[left] {$l$}

(0,0) node[left] {$A$}--(2,0) node[right] {$B$}--
(0,1) node[left] {$C$}--(2,1.5) node[right] {$D$}

(0,3.5) node[left] {$A'$}--(2,3.5) node[right] {$B'$}--
(0,4.5) node[left] {$C'$} -- (2,5) node[right] {$D'$} ;

\draw[red, thick, densely dotted]
(0,0)--(2,1.5)
(0,3.5) -- (2, 5);
\end{tikzpicture}
\end{center}

\begin{proof}
 There exists a point $Y$ on $l$ such that $B\#Y$. Given such a point, either $B\#D$ or $Y\#D$. In the first case the lemma is proved by Lemma \ref{lempar3}. In the second case let $Y'$ be the intersection of $l$ and the line through $C'$ which is parallel to $\ovl{CY}$.

\begin{center}
\begin{tikzpicture}
\draw 
(0,-0.5) -- (0,6) node[left] {$k$}
(2,-0.5) -- (2,6) node[left] {$l$}

(0,0) node[left] {$A$}--(2,0) node[right] {$B$}--
(0,1) node[left] {$C$}--(2,1.5) node[right] {$D$}
(0,1) -- (2,0.75) node[right] {$Y$}

(0,3.5) node[left] {$A'$}--(2,3.5) node[right] {$B'$}--
(0,4.5) node[left] {$C'$} -- (2,5) node[right] {$D'$} 
(0,4.5) -- (2, 4.25) node[right] {$Y'$};
\end{tikzpicture}
\end{center}
By Lemma \ref{lempar3} on the parallel lines $k$, $l$ and the quadrilaterals $ABCY$ and $A'B'C'Y'$, we conclude that $\ovl{AY}\parallel \ovl{A'Y'}$. Finally, by Lemma \ref{lempar3} on the parallel lines $k$, $l$ and the quadrilaterals $AYCD$ and $A'Y'C'D'$, we conclude that $\ovl{AD}\parallel \ovl{A'D'}$.
\end{proof}

We are now ready to prove Desargues' theorem on four parallel lines.

\begin{thrm} \label{thrmdesmall4}
Let $\ca A$ be a preaffine plane satisfying Desargues' big and small axioms.
Let $k$, $l$, $m$, $n$ be lines of $\ca A$ and let $A$, $A'$, $B$, $B'$, $C$, $C'$, $D$, $D'$ be points of $\ca A$. If the following are true:
\begin{enumerate}
\item $k,l,m,n$ are parallel to each other,
\item $k\#l\#m\#n$,
\item $A\#D$,
\item $A'\#D'$,
\item $A, A'\in k$,
\item $B, B'\in l$,
\item $C, C'\in m$,
\item $D, D'\in n$,
\end{enumerate}
then $A\# B\# C\# D$ and $A'\# B'\# C'\# D'$.  If $\ovl{AB}\parallel\ovl{A'B'}$, $\ovl{BC}\parallel\ovl{B'C'}$, and $\ovl{CD}\parallel \ovl{C'D'}$, then $\ovl{AD}$ is parallel to $\ovl{A'D'}$.
\end{thrm}

\begin{center}
 \begin{tikzpicture}
\draw 
(0,-0.5) -- (0,6) node[right] {$k$} 
(3,-0.5) -- (3,6) node[right] {$l$} 
(5,-0.5) -- (5,6)  node[right] {$m$} 
(6.5,-0.5) -- (6.5,6) node[right] {$n$}

(0,0) node[ left] {$A$} 
--(3,1.5) node[ right] {$B$}
--(5,0) node[right] {$C$}
--(6.5,1) node[right] {$D$}

(0,4)  node[left] {$A'$} 
--(3,5.5) node[right] {$B'$}
--(5,4) node[right] {$C'$}
-- (6.5,5) node[right] {$D'$};

\draw[red, thick, densely dotted] 
(0,0) -- (6.5,1)
(0,4) -- (6.5,5);
\end{tikzpicture}
\end{center}

\begin{proof}
Let $X$ be the intersection of $\ovl{CD}$ and $l$, and let $Y$ be the intersection of $\ovl{AB}$ and $m$. Also, let $X'$ be the intersection of $\ovl{C'D'}$ and $l$, and let $Y'$ be the intersection of $\ovl{A'B'}$ and $m$. All these are shown in the following picture.

\begin{center}
 \begin{tikzpicture}
\draw 
(0,-1.5) -- (0,8) node[right] {$k$} 
(3,-1.5) -- (3,8) node[right] {$l$} 
(5,-1.5) -- (5,8)  node[right] {$m$} 
(6.5,-1.5) -- (6.5,8) node[right] {$n$}

(0,0) node[left] {$A$} 
--(3,1.5) node[right] {$B$}
--(5,0) node[right] {$C$}
--(6.5,1) node[right] {$D$}
(5,0) -- (3,-1.3333) node[left] {$X$}
(3,1.5) -- (5,2.5) node[right] {$Y$}

(0,4)  node[left] {$A'$} 
--(3,5.5) node[right] {$B'$}
--(5,4) node[right] {$C'$}
-- (6.5,5) node[right] {$D'$}
(5,4) -- (3,2.6667) node[left] {$X'$}
(3,5.5) -- (5,6.5) node[right] {$Y'$};
\end{tikzpicture}
\end{center}

By applying Lemma \ref{lempar4} on the parallel lines $l$, $m$ and the quadrilaterals $XCBY$ and $X'C'B'Y'$, we conclude that $\ovl{XY}\parallel \ovl{X'Y'}$.

We now apply Desargues' small axiom on the parallel lines $m$, $l$, $n$ and the triangles $YXD$ and $Y'X'D'$. Hence, $\ovl{YD}\parallel \ovl{Y'D'}$.

Finally, we apply Desargues' small axiom on the parallel lines $k$, $l$, $n$ and the triangles $AYD$ and $A'Y'D'$. We conclude that $\ovl{AD}\parallel \ovl{A'D'}$. 
\end{proof}

The following theorem is an immediate consequence of Desargues' big axiom and it is in a form that it is convenient for later applications.

\begin{thrm} \label{thrmdesbig}
Let $\ca A$ be a preaffine plane satisfying Desargues' big axiom.
Let $k$, $l$, $m$ be lines of $\ca A$ and let $P$, $A$, $A'$, $B$, $B'$, $C$, $C'$ be points of  $\ca A$ such that:
\begin{enumerate}
\item $P$ lies on all three lines $k,l,m$,
\item $k\#l\#m$,
\item $P$ is apart from the points $A$, $A'$, $B$, $B'$, $C$ and $C'$,
\item $A, A'\in k$,
\item $B, B'\in l$,
\item $C, C'\in m$,
\item $A\# C$ and $A'\# C'$.
\end{enumerate}
Then, $A\# B\# C$ and $A'\# B'\# C'$. If $\ovl{AB}$ parallel to $\ovl{A'B'}$ and $\ovl{BC}$ parallel to $\ovl{B'C'}$, then $\ovl{AC}$ is parallel to $\ovl{A'C'}$.
\end{thrm}

\begin{tikzpicture}
\draw 
(0,0) node[left] {$P$} -- (6,3) node[right] {$k$} 
(0,0) -- (7,0) node[right] {$l$} 
(0,0) -- (6,-1.5)  node[right] {$m$};

\draw 
(3, -0.75) node[ below] {$C'$} 
--(5,0) node[ below] {$B'$}
(1.5,-0.375)  node[below] {$C$} 
--(2.5,0) node[ below] {$B$};

\draw 
(5,0)--(4.5,2.25) node[above] {$A'$}
 (2.5,0)--(2.25,1.125) node[above] {$A$};

\draw[red, thick, densely dotted] 
(4.5,2.25)--(3,-0.75) 
(2.25,1.125)--(1.5,-0.375);
\end{tikzpicture}

\begin{thrm} \label{thrmdesbig4}
Let $\ca A$ be a preaffine plane satisfying Desargues' big axiom.
Let $k$, $l$, $m$, $n$ be lines of $\ca A$ and let $P$, $A$, $A'$, $B$, $B'$, $C$, $C'$, $D$, $D'$ be points of $\ca A$. If the following are true:
\begin{enumerate}
\item $P$ lies on all three lines $k,l,m,n$,
\item $k\#l\#m\#n$,
\item $A\#D$,
\item $A'\#D'$,
\item $P$ is apart from the points $A$, $A'$, $B$, $B'$, $C$, $C'$, $D$ and $D'$,
\item $A, A'\in k$,
\item $B, B'\in l$,
\item $C, C'\in m$,
\item $D, D'\in n$.
\end{enumerate}
Then, $A\# B\# C\# D$ and $A'\# B'\# C'\# D'$. If $\ovl{AB}$ parallel to $\ovl{A'B'}$, $\ovl{BC}$ parallel to $\ovl{B'C'}$ and $\ovl{CD}$ parallel to $\ovl{C'D'}$, then $\ovl{AD}$ is parallel to $\ovl{A'D'}$.
\end{thrm}

\begin{center}
\begin{tikzpicture}
\draw 
(0,0) node[left] {$P$} -- (7,3.5) node[right] {$k$} 
(0,0) -- (7,0) node[right] {$l$} 
(0,0) -- (6,-1.5)  node[right] {$m$}
(0,0) -- (5,6.25) node[right] {$n$};

\draw 
(3, -0.75) node[ below] {$C'$} 
--(5,0) node[ below] {$B'$}
(1.5,-0.375)  node[below] {$C$} 
--(2.5,0) node[ below] {$B$};

\draw 
(5,0)--(6,3) node[above] {$A'$}
 (2.5,0)--(3,1.5) node[above] {$A$};

\draw
(3, -0.75)--(4,5) node[above] {$D'$}
(1.5,-0.375)--(2,2.5) node[above] {$D$};

\draw[red, thick, densely dotted] 
(6,3)--(4,5)
(3,1.5)--(2,2.5);
\end{tikzpicture}
\end{center}

\begin{proof}
On the affine plane $\ca A$, there exist three lines through $P$ that are apart from each other. Each of $k$ and $n$ is apart from at least two of these three lines, and therefore both $k$ and $n$ are apart from at least one of these three lines. Hence there exists a line $l'$ passing through $P$ such that $k\# l' \# n$. $k\# l'$, therefore $m\# k$ or $m\# l'$. 

Let us first consider the case where $m\# k$. Then, $A\# C$ and $A' \# C'$ and by Theorem \ref{thrmdesbig} on the concurrent lines $k$, $l$ and $m$, and since $\ovl{AB}\parallel \ovl{A'B'}$ and $\ovl{BC}\parallel \ovl{B'C'}$, we conclude that $\ovl{AC}\parallel \ovl{A'C'}$. By Theorem \ref{thrmdesbig} on the concurrent lines $k$, $m$ and $n$, and since $\ovl{AC}\parallel \ovl{A'C'}$ and $\ovl{CD}\parallel \ovl{C'D'}$, we conclude that $\ovl{AD}\parallel \ovl{A'D'}$.

Let us now consider the case where $m\# l'$. Since $n\# l'$, either $l\#n$ or $l\# l'$. In the first case, $\ovl{AD}\parallel \ovl{A'D'}$ by an argument symmetric to the one for $m\# k$ covered above. Hence, we shall consider the case where $k\#l'$, i.e. $l'$ is apart from all four lines $k$, $l$, $m$ and $n$. There exists a point $X$ on $l'$ such that $P\# X$ and notice that $\ovl{AX}\# l'$. Let $l_X$ be the line through $A'$ parallel to $ovl{AX}$. Then, $l_X$ is apart from $l'$ and they intersect, and let their intersection point be $X'$. By Theorem \ref{thrmdesbig}, on the concurrent lines $l$, $k$, $l'$, we see that $\ovl{BX}\parallel\ovl{B'X'}$. By Theorem \ref{thrmdesbig}, on the concurrent lines $m$, $l$, $l'$, we see that $\ovl{CX}\parallel\ovl{C'X'}$. By Theorem \ref{thrmdesbig}, on the concurrent lines $n$, $m$, $l'$, we see that $\ovl{DX}\parallel\ovl{D'X'}$. By Theorem \ref{thrmdesbig}, on the concurrent lines $k$, $l'$, $n$, and since $\ovl{AX}\parallel \ovl{A'X'}$ and $\ovl{XD}\parallel \ovl{X'D'}$ we conclude that $\ovl{AD}\parallel\ovl{A'D'}$.
\end{proof}

Finally, we state the following version of Desargues' theorem. We shall prove this version in Section \ref{secfurdes} once we have established a few results about translations and in particular Theorem \ref{thrmexistrans}. Notice that Section \ref{secfurdes} could have been added just after Theorem \ref{thrmexistrans} and in particular before Theorem \ref{thrmexistdil} whose proof uses Theorem \ref{thrmdes5}.

\begin{thrm} \label{thrmdes5}
Let $\ca A$ be a preaffine plane satisfying Desargues' big and small axioms.
Let $k$, $l$, $m$ be lines of $\ca A$ and let $P$, $A$, $A'$, $B$, $B'$, $C$, $C'$,  $D$, $D'$ be points of $\ca A$ such that:
\begin{enumerate}
\item $P$ lies on all three lines $k,l,m$,
\item $k\#l\#m$,
\item $P$ is apart from the points $A$, $A'$, $B$, $B'$, $C$ and $C'$,
\item $A, A'\in k$,
\item $B, B'\in l$,
\item $C, C'\in m$,
\item $D$ lies outside the lines $\ovl{AB}$ and $\ovl{BC}$,
\item $D'$ lies outside the lines $\ovl{A'B'}$ and $\ovl{B'C'}$.
\end{enumerate}
Then, $D\# A\# B\# C\# D \# B$ and $D'\# A'\# B'\# C'\# D' \# B'$. If $\ovl{AB}\parallel\ovl{A'B'}$, $\ovl{BC}\parallel \ovl{B'C'}$, $\ovl{BD}\parallel \ovl{B'D'}$ and $\ovl{CD}\parallel \ovl{C'D'}$, then $\ovl{AD}$ is parallel to $\ovl{A'D'}$.
\end{thrm}

\begin{center}
 \begin{tikzpicture}
\draw 
(0,0) node[left] {$P$} -- (7,3.5) node[right] {$k$} 
(0,0) -- (7,0) node[right] {$l$} 
(0,0) -- (6,-1.5)  node[right] {$m$};

\draw 
(3, -0.75) node[ below] {$C'$} 
--(5,0) node[ below] {$B'$}
(1.5,-0.375)  node[below] {$C$} 
--(2.5,0) node[ below] {$B$};

\draw 
(5,0)--(6,3) node[above] {$A'$}
 (2.5,0)--(3,1.5) node[above] {$A$};

\draw
(3, -0.75)--(4,5) node(D')[above] {$D'$}
(1.5,-0.375)--(2,2.5) node[above] {$D$};

\draw 
(5,0)--(4,5)
(2.5,0)--(2,2.5);

\draw[red, thick, densely dotted] 
(6,3)--(4,5)
(3,1.5)--(2,2.5);
\end{tikzpicture}
\end{center}

\section{Pappus' axiom on the affine plane}

\subsection*{Pappus' axiom}

A preaffine plane satisfies \emph{Pappus' axiom} when given  $k$, $l$ lines and $P$, $A$, $A'$, $B$, $B'$, $C$, $C'$ points of the preaffine plane such that:
\begin{enumerate}
\item $P, A, B, C \in k$,
\item $P, A', B', C' \in l$,
\item $k\# l$,
\item \label{aaaa} the points $A$, $A'$, $B$, $B'$, $C$, $C'$ are apart from $P$,
\item $\ovl{AB'}\parallel \ovl{BC'}$,
\item $\ovl{A'B}\parallel \ovl{B'C}$,
\end{enumerate}
then $\ovl{AA'}$ is parallel to $\ovl{CC'}$.

\begin{tikzpicture}
\draw
(0,0) node[left] {$P$} -- (5,3) node[right]{k}
(0,0)--(5,-2) node[right] {l};

\draw
(1,0.6) node[above] {$A$}--(2,-0.8)  node[below] {$B'$}
(2,1.2)  node[above] {$B$} --(4,-1.6) node[below] {$C'$};

\draw
(2,1.2)--(1,-0.4) node[below] {$A'$}
(4,2.4)  node[above] {$C$}--(2,-0.8);

\draw[red, thick, densely dotted]
(1,0.6)--(1,-0.4)
(4,2.4)--(4,-1.6);
\end{tikzpicture}

\begin{rmk}
Point \ref{aaaa} in the statement of the axiom is equivalent to the points $A$, $B$, $C$ lying outside $l$, and the points $A', B', C'$ lying outside $k$. 
\end{rmk}

\begin{thrm} \label{thrmpappu}
Let $\ca P$ be a preprojective plane satisfying Pappus' axiom, and let $l_\infty$ be a line of $\ca P$. Then, the preaffine plane $\mathfrak A(\ca P,l_\infty)$ satisfies Pappus' axiom.
\end{thrm}

\begin{proof}
Let $\ca P$ be a preprojective plane satisfying Pappus' axiom and let $l_\infty$ be a line of $\ca P$. Let $k$, $l$ be lines and $P$, $A$, $A'$, $B$, $B'$, $C$, $C'$ points of the preaffine plane $\mathfrak A(\ca P, l_\infty)$ satisfying the conditions listed in Pappus' axiom above.  Let $X$ be the intersection of $\ovl{AB'}$ and $l_\infty$. $\ovl{BC'}$ is parallel to $\ovl{AB'}$, therefore $X\in \ovl{BC'}$. Let $Y$ be the intersection of $\ovl{A'B}$ and $l_\infty$. $\ovl{B'C}$ is parallel to $\ovl{A'B}$, therefore $Y\in \ovl{B'C}$. We apply (the projective) Pappus' axiom on the six points $Y$, $X$, $C'$, $A'$, $A$, $C$ and the six lines $l_\infty$, $\ovl{C'B}$, $l$, $\ovl{A'A}$, $k$, $\ovl{CB'}$, i.e. we apply Pappus' axiom  on the following configuration
\begin{center}
\begin{tikzpicture}
\draw
(8,4.8)--
(2,1.2) node[above] {$A$}--(4,-1.6)  node[below] {}
(4,2.4)  node[above] {$B'$} --(6,-2.4) node[below] {$A'$}
(4,2.4)--(2,-0.8) node[below] {$Y$}
-- (6,-2.4)
(8,4.8)  node[above] {$X$}--(4,-1.6) node[below] {$B$}
(2,1.2) -- (6,-2.4)
(2,-0.8) -- (8,4.8)
(2.3,0.2) node {$C$}
(5.5,0.1) node {$C'$};
\end{tikzpicture}
\end{center}
to conclude that $\delta(l_\infty,\ovl{A'A}, C, C')$ is satisfied in $\ca P$. Hence, in $\mathfrak A(\ca P,l_\infty)$ the line $\ovl{AA'}$ is parallel to $\ovl{CC'}$.
\end{proof}

\begin{rmk}
Given a local ring $R$, the preaffine plane $\mathbb A(R)$ is isomorphic to $\mathfrak A(\mathbb P(R),(0,0,1))$. $\mathbb P(R)$ satisfies Pappus' axiom, therefore $\mathbb A(R)$ satisfies Pappus' axiom.
\end{rmk}

\begin{defn}
An \emph{affine plane} is a preaffine plane that satisfies Desargues' big and small axioms, and Pappus' axiom.
\end{defn}

For $R$ a local ring, we have already proved that the preaffine plane $\mathbb A(R)$ satisfies Desargues' small and big axioms, and Pappus' axiom. Hence, $\mathbb A(R)$ is an affine plane.

Both versions of Desargues' axiom, and Pappus' axiom can be expressed as geometric sequents in the language of preaffine planes, therefore the theory of affine planes is a geometric theory.

\chapter{Constructing the local ring from an affine plane} \label{chalocal}

The goal of this chapter is to construct a local ring from a given affine plane. The construction uses the notions of dilatations and translations which we define here. The underlying set of the local ring $\text{Tp}$ we construct is the set of trace preserving homomorphisms of the group of translations. We show that given three non-collinear points we have an isomorphism from the affine plane over the local ring of trace preserving homomorphisms to the original affine plane. There is also an action on the set of three non-collinear points which makes it a right $G(\text{Tp})$-torsor. We then return to the construction of the local ring to show that it can be thought of as a construction involving finite limits and finite colimits, therefore it is preserved by inverse images of geometric morphisms. We also show that any alternative construction of the local ring gives an isomorphic ring as long as it satisfies certain properties. The final section of this chapter concerns a version of Desargues' theorem whose proof uses some results on translations and which is used earlier in the chapter.

The constructions and arguments in this and the next chapter sometimes lie outside geometric logic but they are still constructive and in particular they can be carried out in any topos with a natural number object.

\section{Dilatations} \label{secdil}

\begin{defn}
Let $\ca A$ be an affine plane and let $\ca A_{\text{pt}}$ be its set of points. A \emph{dilatation} of $\ca A$ is a morphism $\sigma:\ca A_{\text{pt}}\to \ca A_{\text{pt}}$ such that for any two points $P$ and $Q$ of $\ca A$, if $P\#Q$ then $\sigma(P)\#\sigma(Q)$ and $\ovl{PQ}\parallel \ovl{\sigma(P)\sigma(Q)}$. 
\end{defn}

\begin{prop} \label{propdilhom}
The following hold for a dilatation $\sigma$ of an affine plane $\ca A$:
\begin{enumerate}
 \item for any three points such that $P\#Q$ and $R\in \ovl{PQ}$, then $\sigma(R)\in \ovl{\sigma(P)\sigma(Q)}$,
\item for any three non-collinear points $P$, $Q$ and $R$, then $\sigma(P)$, $\sigma(Q)$, $\sigma(R)$ are non-collinear.
\end{enumerate}
\end{prop}

\begin{proof}
\begin{enumerate}
\item Let $P\#Q$ be points of a preaffine plane and let $\sigma$ be a dilatation of the plane. Given a point $R$ on $\ovl{PQ}$, it is either apart from $P$ or apart from $Q$. Without loss of generality, let us assume it is apart from $P$. Then, $\sigma(P) \# \sigma(R)$ and $\ovl{\sigma(P)\sigma(R)}$ is parallel to $\ovl{PR}$. We know that $\ovl{PR}=\ovl{PQ}$ is also parallel to $\ovl{\sigma(P)\sigma(Q)}$. Hence, both $\ovl{\sigma(P)\sigma(Q)}$ and $\ovl{\sigma(P)\sigma(R)}$ are lines through $\sigma(P)$ parallel to $\ovl{PQ}$. There is a unique such line, hence the two lines are equal, and therefore $\sigma(R)\in \ovl{\sigma(P)\sigma(Q)}$.

\item  Let $P$, $Q$ and $R$ be three non-collinear points of an affine plane and let $\sigma$ be a dilatation of the plane. Then, $P\#Q\#R$ and $\ovl{PQ}\# \ovl{QR}$. Hence, $\sigma(P)\#\sigma(Q)\#\sigma(R)$. $\ovl{PQ}$, $\ovl{QR}$ are apart from each other and they intersect and $\ovl{PQ}\parallel \ovl{\sigma(P)\sigma(Q)}$ and $\ovl{QR}\parallel \ovl{\sigma(Q)\sigma(R)}$. Therefore, $\ovl{\sigma(P)\sigma(Q)}$ and $\ovl{\sigma(Q)\sigma(R)}$ are apart from each other. By results on non-collinear points, we conclude that $\sigma(P)$, $\sigma(Q)$ and $\sigma(R)$ are non-collinear because $\sigma(P)\#\sigma(Q)\#\sigma(R)$ and $\ovl{\sigma(P)\sigma(Q)}\# \ovl{\sigma(Q)\sigma(R)}$.
\end{enumerate}
\end{proof}

\begin{rmk}
We have now proved that all the conditions of Proposition \ref{proppreaffmor} are satisfied therefore a dilatation can be uniquely extended to an endomorphism of the affine plane such that every line is mapped to a parallel line. Conversely, given an endomorphism of an affine plane which maps every line to a parallel line, then its morphism on points is a dilatation.
\end{rmk}

\begin{egs}
 \begin{enumerate}
  \item  The identity on points is a dilatation.
\item Given a local ring $R$, consider the automorphism of $\mathbb A(R)$ induced by a matrix $\begin{pmatrix} r & 0 & a \\
0 & r & b \\
0 & 0 & 1
\end{pmatrix}$ where $r$, $a$ and $b$ in $R$ and $r$ is invertible. The morphism on lines is $\begin{pmatrix} 1 & 0 & 0 \\
0 & 1 & 0 \\
-a & -b & r
\end{pmatrix}$ hence it sends a line to a parallel line and therefore the morphism on points is a dilatation. We shall prove in Proposition \ref{propdilring} that these are all the dilatations of $\mathbb A(R)$.
 \end{enumerate}
\end{egs}

\begin{lem} \label{lemcompdil}
 The composite of two dilatations is a dilatation.
\end{lem}

\begin{proof}
 Let $\sigma$ and $\tau$ be dilatations. Let $P$ and $Q$ be points of the affine plane such that $P\#Q$. $\tau$ is a dilatation, therefore $\tau(P)\#\tau(Q)$ and the line $\ovl{PQ}$ is parallel to $\ovl{\tau(P)\tau(Q)}$. Since $\sigma$ is also a dilatation, $\sigma (\tau (P))\#\sigma (\tau (Q))$. Also, the line $\ovl{\sigma (\tau (P))\sigma (\tau (Q))}$ is parallel to $\ovl{\tau(P)\tau(Q)}$, hence  it is also parallel to $\ovl{PQ}$, and therefore $\sigma \circ\tau$ is a dilatation.
\end{proof}

\begin{thrm} \label{thrmundil}
A dilatation is uniquely determined by the images of two points $P$ and $Q$ such that $P\#Q$.
\end{thrm}

\begin{proof}
Let $\sigma$ be a dilatation, and let $P$ and $Q$ be two points that are apart from each other. We pick a point $R$ not on the line $\ovl{PQ}$. Let the points $P'$, $Q'$ and $R'$ be the images of the points $P$, $Q$ and $R$. Then $R'$ lies on the line $l$ through $P'$ and parallel to $\ovl{PR}$. It also lies on the line $m$ through $Q'$ parallel to $\ovl{QR}$. The lines $\ovl{PR}$ and $\ovl{QR}$ are apart from each other and they intersect, therefore $l\#m$. Hence, $R'$ is the unique intersection of $l$ and $m$ and is uniquely determined by the images of $P$ and $Q$. Notice also that $R'\notin \ovl{P'Q'}$.

Any other point lies outside at least one of the lines $\ovl{PQ}$, $\ovl{QR}$ and $\ovl{RP}$, therefore we apply again the above argument to conclude that its image through $\sigma$ is uniquely determined by $P'$, $Q'$ and $R'$. Hence $\sigma$ is uniquely determined by the images of $P$ and $Q$.
 \end{proof}

\begin{cor}
 If a dilatation $\sigma$ has two fixed points $P$ and $Q$ such that $P\#Q$, then $\sigma$ is the identity.
\end{cor}

\begin{proof}
 This is an immediate consequence of the above theorem because the identity is a dilatation.
\end{proof}

\begin{prop} \label{propdilring}
Let $R$ be a local ring. A dilatation of $\mathbb A(R)$ is exactly left multiplication by a matrix of the form $\begin{pmatrix} r & 0 & a \\
0 & r & b \\
0 & 0 & 1
\end{pmatrix}$.
\end{prop}

\begin{proof}
Let $\sigma$ be a dilatation of $\mathbb A(R)$. Then, by Theorem \ref{thrmundil} $\sigma$ is uniquely determined by $\sigma(0,0)$ and $\sigma(1,0)$. Let $\sigma(0,0)$ be $(a,b)$ where $a$, $b$ in $R$. Then, $\sigma(1,0)$ is apart from $(a,b)$ and on the line through $(a,b)$ parallel to $(0,1,0)$. Hence, $\sigma(1,0)=(a+r,b)$ for some invertible $r$ in $R$. The dilatation represented by the matrix $\begin{pmatrix} r & 0 & a \\
0 & r & b \\
0 & 0 & 1
\end{pmatrix}$ is a dilatation which sends $(0,0)$ to $(a,b)$ and $(1,0)$ to $(a+r,b)$, hence it is the dilatation $\sigma$.
\end{proof}

\begin{rmk}
Notice that for dilatations given in the above form, composition coincides with matrix multiplication.
\end{rmk}

\section{Translations} \label{sectr}

\begin{defn}
Let $\tau$ be a dilatation of an affine plane $\ca A$. Let $P$, $Q$ and $D$ be three non-collinear points of $\ca A$. Let $l$ be the line through $D$ parallel to $\ovl{PQ}$, let $m$ be the line through $Q$ and parallel to $\ovl{PD}$, and let $D'$ be the (unique) intersection point of the lines $l$ and $m$.

$\tau$ is a \emph{translation} when for any $P$, $Q$, $D$, $l$, $m$ and $D'$ as above the following two conditions are satisfied:
\begin{itemize}
\item If $\ovl{\tau(P)\tau(Q)}$ is apart from $\ovl{PQ}$, then ($P\#\tau(P)$, $Q\# \tau(Q)$ and) $\ovl{P\tau(P)}$ is parallel to $\ovl{Q\tau(Q)}$.
\item If $\ovl{\tau(P)\tau(Q)}$ is apart from $\ovl{DD'}$, then ($D\#\tau(P)$, $D'\#\tau(Q)$ and) $\ovl{D\tau(P)}$ is parallel to $\ovl{D'\tau(Q)}$.
\end{itemize}
\end{defn}

\begin{center}
\begin{tikzpicture}
\draw
(0,5) node[above] {$P$}
--(3,5) node[above] {$D$}
(0,3) node[left] {$\tau(P)$}
(1,2) node[below] {$Q$}
--(4,2) node[below] {$D'$}
(1,0) node[left] {$\tau(Q)$}
(0,5)--(1,2)
(3,5)--(4,2)
(0,3)--(1,0);
\end{tikzpicture}
\end{center}

Notice that in the above definition $P\#D\#D'\#Q$, the lines $\ovl{PQ}$, $\ovl{DD'}$ are apart from each other and parallel. Hence $\ovl{\tau(P)\tau(Q)}$ is apart from at least one of the lines $\ovl{PQ}$ and $\ovl{DD'}$ and parallel to both of them because $\tau$ is a dilatation.

\begin{egs}
\begin{enumerate}
 \item The identity map is a translation.
\item Given a local ring $R$, translations of the affine plane over $R$ are exactly maps of points of the form $\tau(x,y)=(x+a,y+b)$ for fixed $a$ and $b$ in $R$. This will be proved in Proposition \ref{proptring}.
\end{enumerate}
\end{egs} 

\begin{lem}
The composite of two translations is a translation.
\end{lem}

\begin{proof}
 Let $\sigma$ and $\tau$ be two translations. $\sigma\circ \tau$ is a dilatation by Lemma \ref{lemcompdil}. Let the points $P$, $Q$ and $D$ be such that $P\#Q$ and $D\notin \ovl{PQ}$. Define the lines $l$ and $m$ and the point $D'$ as in the definition of translations.

\vspace{0.1 in}

{\bf Consider the case where $\ovl{\sigma(\tau(P))\sigma(\tau(Q))}\# \ovl{PQ}$.}

Then, either $\ovl{\tau(P)\tau(Q)}\#\ovl{PQ}$ or $\ovl{\tau(P)\tau(Q)}$ is apart from both the lines $\ovl{DD'}$ and $\ovl{\sigma(\tau(P)\sigma(\tau(Q))}$.

In the case where $\ovl{\tau(P)\tau(Q)}\# \ovl{PQ}$, we know that $\ovl{P\tau(P)}$ is parallel to $\ovl{Q\tau(Q)}$ because $\tau$ is a translation. $\ovl{\tau(P)\tau(Q)}\# \ovl{PQ}$  also implies that $P\notin\ovl{\tau(P)\tau(Q)}$. So, since $\sigma$ is a translation $\ovl{P\sigma(\tau(P))}\parallel\ovl{Q\sigma(\tau(Q))}$.

In the case where $\ovl{\tau(P)\tau(Q)}$ is apart from both $\ovl{DD'}$ and $\ovl{\sigma(\tau(P)\sigma(\tau(Q))}$, we know that $\ovl{D\tau(P)}\parallel \ovl{D'\tau(Q)}$ because $\tau$ is a translation, and that $\ovl{\tau(P)\sigma(\tau(P))}\parallel \ovl{\tau(Q)\sigma(\tau(Q))}$ because $\sigma$ is a translation. We can now use Theorem \ref{thrmdesmall4}, on the four parallel lines $\ovl{PQ}$, $\ovl{DD'}$, $\ovl{\tau(P)\tau(Q)}$, $\ovl{\sigma(\tau(P)\sigma(\tau(Q))}$. $\ovl{PD}\parallel \ovl{QD'}$,  $\ovl{D\tau(P)}\parallel \ovl{D'\tau(Q)}$ and $\ovl{\tau(P)\sigma(\tau(P))}\parallel \ovl{\tau(Q)\sigma(\tau(Q))}$, hence $\ovl{P\sigma(\tau(P))}\parallel \ovl{Q\sigma(\tau(Q))}$.

\vspace{0.1 in}

{\bf Consider the case where $\ovl{\sigma(\tau(P))\sigma(\tau(Q))}\#\ovl{DD'}$.}

Then, either $\ovl{\tau(P)\tau(Q)}\# \ovl{DD'}$ or $\ovl{\tau(P)\tau(Q)}$ is apart from both the lines  $\ovl{PQ}$ and $\ovl{\sigma(\tau(P))\sigma(\tau(Q))}$.

In the case where $\ovl{\tau(P)\tau(Q)}\# \ovl{DD'}$, we know that $\ovl{D\tau(P)}\parallel\ovl{D'\tau(Q)}$ because $\tau$ is a translation. Hence, $\ovl{D\sigma(\tau(P))}\parallel\ovl{D\sigma(\tau(Q))}$ because $\sigma$ is a translation.

In the case where $\ovl{\tau(P)\tau(Q)}$ is apart from both $\ovl{PQ}$ and $\ovl{\sigma(\tau(P))\sigma(\tau(Q))}$, we know that $\ovl{P\tau(P)}$ is parallel to $\ovl{Q\tau(Q)}$ because $\tau$ is a translation. Also, $\ovl{\tau(P)\sigma(\tau(P))}$ is parallel to $\ovl{\tau(Q)\sigma(\tau(Q))}$ because $\sigma$ is a translation. Hence we can apply Theorem \ref{thrmdesmall4} on the four parallel lines $\ovl{DD'}$, $\ovl{PQ}$, $\ovl{\tau(P)\tau(Q)}$, $\ovl{\sigma(\tau(P)\sigma(\tau(Q))}$. $\ovl{DP}\parallel \ovl{D'Q}$,  $\ovl{P\tau(P)}\parallel \ovl{Q\tau(Q)}$ and $\ovl{\tau(P)\sigma(\tau(P))}\parallel \ovl{\tau(Q)\sigma(\tau(Q))}$, hence $\ovl{D\sigma(\tau(P))}$ is parallel to $\ovl{D'\sigma(\tau(Q))}$.

\vspace{0.1 in}

Therefore, $\sigma\circ \tau$ is a translation.
\end{proof}

\begin{thrm}
 A translation is uniquely determined by the image of one point.
\end{thrm}

\begin{proof}
 Let $\tau$ be a translation, and let $P'$ be the image of $P$. Let $Q$ be a point such that $P\#Q$, and $D$ a point not on the line $\ovl{PQ}$. Let the lines $l$, $m$ and the point $D'$ be as in the definition of translations.
 
 If $\ovl{P'\tau(Q)}$ is apart from $\ovl{DD'}$, then $\tau(Q)$ is the unique intersection point of the line through $D'$ parallel to $\ovl{DP}$ and the line through $P'$ parallel to $PQ$. $\tau(Q)$ lies on both these lines because $\tau$ is a translation (and a dilatation), and the two lines are apart from each other and have an intersection because they are parallel to lines that are apart from each other and have an intersection point.

 If $\ovl{P'\tau(Q)}$ is apart from the line $\ovl{PQ}$, then $\tau(Q)$ is the intersection of the line through $Q$ parallel to $\ovl{P\tau(P)}$, and the line through $P$ and parallel to $\ovl{PQ}$.

 Hence, we know the images of two points that are apart from each other, therefore we have uniquely determined the dilatation $\tau$.
\end{proof}

The following lemma describes a specific class of translations. Its proof is much longer than the corresponding proof for classical affine planes (over fields). The complication is due to having no way to deal separately with the case where the affine plane does not have three lines that are apart from each other. In the classical case, it is easy to prove that the only such plane is isomorphic to the affine plane over the field $\mathbb Z/(2)$.

\begin{lem}\label{lemlinetranslation}
A dilatation $\tau$ such that $P\#\tau(P)$ for some point $P$ is a translation iff $Q\#\tau(Q)$ and $\ovl{P\tau(P)}\parallel\ovl{Q \tau(Q)}$ for all points $Q$ of the affine plane.
\end{lem}

In the following proof we often use the fact that whenever $A$, $B$, $C$, $D$ are points such that $A\#B$, $C\#D$ and the lines $\ovl{AB}$, $\ovl{CD}$ are parallel and apart from each other then $A\#C$ and $B\#D$. If moreover $\ovl{AC}\parallel \ovl{BD}$, then $\ovl{AC}\#\ovl{BD}$.

\begin{proof}
Let us start with the direct implication. Let $\tau$ be a translation and let $P$ be a point of the affine plane such that $P\# \tau(P)$. There exists a point lying outside $\ovl{P\tau(P)}$, so let $R$ be such a point. Then, $\tau(P)\#\tau(R)$ and $\tau(P)\notin \ovl{PR}$, hence $\ovl{PR}\#\ovl{\tau(P)\tau(R)}$. Therefore by the definition of translations (since there exists $D$ lying outside $\ovl{PR}$), $\ovl{P\tau(P)}\parallel \ovl{R\tau(R)}$. Given any point $Q$, it lies outside at least one of the two (parallel and apart from each other) lines $\ovl{P\tau(P)}$ and $\ovl{R\tau(R)}$. By repeating the above argument we prove that $Q\#\tau(Q)$ and $\ovl{Q\tau(Q)}$ is parallel to both $\ovl{P\tau(P)}$ and $\ovl{R\tau(R)}$.




Let us now consider the converse implication. Let $\tau$ be a dilation such that for any two points $P$ and $Q$, $P\#\tau(P)$, $Q\#\tau(Q)$ and $\ovl{P\tau(P)}\parallel\ovl{Q \tau(Q)}$. Let $P$, $Q$ be points of the affine plane that are apart from each other. Let $D$ be a point lying outside $\ovl{PQ}$ and let $D'$ as in the definition of translations, i.e. $D\#D'\#Q$, $\ovl{PD}\parallel \ovl{QD'}$ and $\ovl{PQ}\parallel \ovl{DD'}$. To prove that $\tau$ is a translation we need to prove that if $\ovl{DD'}\#\ovl{\tau(P)\tau(Q)}$, then $\ovl{D\tau(P)}\parallel \ovl{D'\tau(Q)}$, therefore we assume that $\ovl{DD'}\#\ovl{\tau(P)\tau(Q)}$. The line through $D'$ parallel to $\ovl{QD}$ is apart from $\ovl{PD}$ and they intersect, and we name their intersection point $D''$ as in the following picture. In the picture, we also show the images of the eight points via $\tau$.

\begin{center}
\begin{tikzpicture}
\draw
(0,5) node[left] {$P$}
--(6,5) node[right] {$D''$}
(0,3) node[left] {$\tau(P)$}
--(6,3) node[right] {$\tau(D'')$}
(1,2) node[above] {$Q$}
--(4,2) node[above] {$D'$}
(1,0) node[left] {$\tau(Q)$}
--(4,0) node[right] {$\tau(D')$}

(0,3)--(0,5)
(3,3) node[below] {$\tau(D)$}
--(3,5) node[above] {$D$}
(6,3) --(6,5)
(1,0)--(1,2)
(4,0)--(4,2)

(0,5)--(1,2)
--(3,5)
-- (4,2)
-- (6,5)

(0,3)
--(1,0)
--(3,3)
--(4,0)
--(6,3);
\end{tikzpicture}
\end{center}

{\bf Let us first consider the case where $\ovl{PQ}\#\ovl{\tau(P)\tau(Q)}$.}

Then, by Desargues' small axiom on the parallel lines $\ovl{\tau(P)\tau(Q)}$, $\ovl{PQ}$, $\ovl{DD'}$ (which are all apart from each other), and since $\ovl{\tau(P)P}\parallel \ovl{\tau(Q)Q}$ and $\ovl{PD}\parallel \ovl{QD'}$, we conclude that $\ovl{\tau(P)D}\parallel \ovl{\tau(Q)D'}$ as required.

\vspace{0.1 in}

$D\notin \ovl{PQ}$, hence either $D\notin \ovl{P\tau(P)}$ or $\ovl{P\tau(P)}\#\ovl{PQ}$. In the second case, $Q\notin\ovl{P\tau(P)}$, which implies that ($Q\notin \ovl{P\tau(P)}$, $\ovl{Q\tau(Q)}\#\ovl{P\tau(P)}$ and therefore also that) $\ovl{PQ}\#\ovl{\tau(P)\tau(Q)}$. Hence, it reduces to a case we have considered above.

\vspace{0.1 in}

{\bf Therefore, we shall now consider the case where $D\notin \ovl{P\tau(P)}$.}

Then, $\ovl{D\tau(D)}\#\ovl{P\tau(P)}$, therefore $\ovl{PD}\#\ovl{\tau(P)\tau(D)}$. Hence, we can now use Theorem \ref{thrmdesmall4} on the four parallel lines $\ovl{\tau(P)\tau(D)}$, $\ovl{\tau(Q)\tau(D')}$, $\ovl{\tau(D)\tau(D'')}$, $\ovl{DD''}$, which in the written order each one of them is apart from the next. $\ovl{\tau(P)\tau(Q)}\parallel\ovl{\tau(D)\tau(D')}$, $\ovl{\tau(Q)\tau(D)}\parallel \ovl{\tau(D')\tau(D'')}$ and $\ovl{\tau(D)D}\parallel \ovl{\tau(D'')D''}$, hence $\ovl{\tau(P)D}\parallel \ovl{\tau(D)D''}$.

\vspace{0.1 in}
{\bf Hence to prove that $\ovl{\tau(Q)D'}$ is parallel to $\ovl{\tau(P)D}$ it suffices to prove that it is parallel to $\ovl{\tau(D)D''}$.}

(We wish to use Desargues' small axiom on the three parallel lines $\ovl{\tau(Q)\tau(D)}$, $\ovl{\tau(D')\tau(D'')}$, $\ovl{D'D''}$, but to do that we need to prove that both $\ovl{\tau(Q)\tau(D')}$ and $\ovl{\tau(D') D'}$ are apart from $\ovl{\tau(D')\tau(D'')}$. )

$D\notin \ovl{PQ}$, hence the lines $\ovl{PD}$ and $\ovl{QD}$ are apart from each other and they intersect.  $\ovl{\tau(Q)\tau(D')}$ is parallel to $\ovl{PD}$ and $\ovl{\tau(D')\tau(D'')}$ is parallel to $\ovl{QD}$, hence $\ovl{\tau(Q)\tau(D')}$ is apart from $\ovl{\tau(D')\tau(D'')}$. 

The line $\ovl{PQ}$ is apart from $\ovl{QD}$, hence $\ovl{Q\tau(Q)}$ is apart from at least one of the two lines. In the case where $\ovl{Q\tau(Q)}\#\ovl{PQ}$, we see that $\tau(Q)\notin \ovl{PQ}$ which implies that $\ovl{\tau(P)\tau(Q)}\#\ovl{PQ}$, and therefore reduces to a case we have considered above. Thus, it remains to consider the case where $\ovl{Q\tau(Q)}$ is apart from $\ovl{QD}$ (and they intersect at $Q$). $\ovl{D'\tau(D')}\parallel \ovl{Q\tau(Q)}$ and $\ovl{\tau(D')\tau(D'')}\parallel \ovl{QD}$, hence $\ovl{D'\tau(D')}\# \ovl{\tau(D')\tau(D'')}$.

Let us now apply Desargues' small axiom on the three parallel lines $\ovl{\tau(Q)\tau(D)}$, $\ovl{\tau(D')\tau(D'')}$, $\ovl{D'D''}$. Both $\ovl{\tau(Q)\tau(D')}$ and $\ovl{\tau(D') D'}$ are apart from $\ovl{\tau(D')\tau(D'')}$.  $\ovl{\tau(Q)\tau(D')}\parallel \ovl{\tau(D)\tau(D'')}$ and $\ovl{\tau(D') D'}\parallel \ovl{\tau(D'')D''}$, therefore $\ovl{\tau(Q)D'}\parallel \ovl{\tau(D)D''}$. $\ovl{\tau(D)D''}\parallel \ovl{\tau(P)D}$, hence $\ovl{\tau(Q)D'}\parallel \ovl{\tau(P)D}$, and that proves that $\tau$ is indeed a translation.
\end{proof}

\begin{prop} \label{proptring}
Let $R$ be a local ring. A dilatation $\begin{pmatrix} r & 0 & a \\
0 & r & b \\
0 & 0 & 1
\end{pmatrix}$ of $\mathbb A(R)$ is a translation iff $r=1$.
\end{prop}

\begin{proof}
Let $\tau$ be a dilatation of the form $\begin{pmatrix} 1 & 0 & a \\
0 & 1 & b \\
0 & 0 & 1
\end{pmatrix}$ and note that $\tau(0,0)=(a,b)$. $(0,0)\#(1,0)$ hence $(a,b)$ is apart from at least one of $(0,0)$ and $(1,0)$. Let us first consider the case where $(a,b)\#(0,0)$, i.e. where at least one of $a$ and $b$ is invertible. Then for each point $(x,y)$ of $\mathbb A(R)$, $(x,y)\#\tau(x,y)=(x+a,y+b)$. Also the line through $(0,0)$ and $\tau(0,0)$ is represented $(b,-a,1)$ and is parallel to the line through $(x,y)$ and $\tau(x,y)$ which is the line represented by $(b,-a,-xb+ay)$. Hence, by Lemma \ref{lemlinetranslation}, $\tau$ is a translation. For the case where $(a,b)\#(1,0)$, we use the above case to show that $\begin{pmatrix} 1 & 0 & 1 \\
0 & 1 & 0 \\
0 & 0 & 1
\end{pmatrix}$ and $\begin{pmatrix} 1 & 0 & a-1 \\
0 & 1 & b \\
0 & 0 & 1
\end{pmatrix}$ are translations. Then, the composite $\begin{pmatrix} 1 & 0 & 1 \\
0 & 1 & 0 \\
0 & 0 & 1
\end{pmatrix}
\begin{pmatrix} 1 & 0 & a-1 \\
0 & 1 & b \\
0 & 0 & 1
\end{pmatrix}= 
\begin{pmatrix} 1 & 0 & a \\
0 & 1 & b \\
0 & 0 & 1
\end{pmatrix}$
is also a translation.

Conversely, suppose that a translation $\tau$ is of the form $\begin{pmatrix} r & 0 & a \\
0 & r & b \\
0 & 0 & 1
\end{pmatrix}$ where $r$ is invertible. A translation is uniquely determined by the image of a point. Both $\begin{pmatrix} r & 0 & a \\
0 & r & b \\
0 & 0 & 1
\end{pmatrix}$ and $\begin{pmatrix} 1 & 0 & a \\
0 & 1 & b \\
0 & 0 & 1
\end{pmatrix}$ are translations sending $(0,0)$ to $(a,b)$, therefore they are equal. Hence, $r=1$.
\end{proof}

\begin{lem}\label{lemexistrans}
In an affine plane, for any two points $P$ and $P'$ that are apart from each other there exists a (unique) translation sending $P$ to $P'$.
\end{lem}

\begin{proof}
Let $P$ and $P'$ be points of an affine plane that are apart from each other. We pick a point $Q$ lying outside the line $\ovl{PP'}$. Let $Q'$ to be the intersection point of the line through $Q$ parallel to $\ovl{PP'}$ and the line through $P'$ parallel to $\ovl{PQ}$. It is clear that $Q\#Q'$. The lines $\ovl{PP'}$ and $\ovl{QQ'}$ are parallel and apart from each other, therefore any point lies outside at least one of them.

We define the map $\tau$ by sending a point $R$ that lies outside the line $\ovl{PP'}$ to the intersection point of the line through $R$ parallel to $\ovl{PP'}$ and the line through $P'$ parallel to $\ovl{PR}$. A point $R$ that lies outside the line $\ovl{QQ'}$ is sent to the intersection point of the line through $R$ parallel to $\ovl{QQ'}$ and the line through $Q'$ parallel to $\ovl{QR}$. Note that if a point lies outside both the lines $\ovl{PP'}$ and $\ovl{QQ'}$ the two definitions agree by Theorem \ref{thrmdesmall}. Also, notice that $\tau(P)=P'$ and $\tau(Q)=Q'$.

$R\#\tau(R)$ and $\ovl{R\tau(R)}$ is parallel to both $\ovl{PP'}$ and $\ovl{QQ'}$. Therefore by Lemma \ref{lemlinetranslation}, to prove that $\tau$ is a translation it suffices to show that it is a dilatation, i.e. we need to prove that for any two points $R$, $S$ that are apart from each other, $\tau(R)\#\tau(S)$, and $\ovl{RS}\parallel\ovl{\tau(R)\tau(S)}$.

Given a point $R$, it lies outside $\ovl{PP'}$ or outside $\ovl{QQ'}$ because the two lines are parallel and apart from each other. Without loss of generality, suppose that $R\notin \ovl{PP'}$. Let $R'$ be $\tau(R)$.

Then $\ovl{PR}\#\ovl{RR'}$ because $\ovl{PP'}\parallel\ovl{RR'}$, and the lines $\ovl{PR}$ and $\ovl{PP'}$ are apart from each other and intersect at $P$. For any point $S$ such that $R\#S$ we know that $S$ lies outside $\ovl{RR'}$ or outside $\ovl{PR}$. Let $S'$ be $\tau(S)$.

In the case where $S\notin\ovl{RR'}$, we know that $\ovl{SS'}$ is parallel and apart from $\ovl{RR'}$, so $S'$ lies outside $\ovl{RR'}$, hence $R'\#S'$. We also know that $S$ lies outside $\ovl{PP'}$ or $\ovl{QQ'}$.
In the case where $S\notin\ovl{PP'}$, we know that $\ovl{PR}\parallel\ovl{P'R'}$ and $\ovl{PS}\parallel\ovl{P'S'}$, therefore by Theorem \ref{thrmdesmall}, $\ovl{RS}\parallel\ovl{R'S'}$. 
In the case where $S\notin\ovl{QQ'}$, we know that $\ovl{PR}\parallel\ovl{P'R'}$, $\ovl{PQ}\parallel\ovl{P'Q'}$ and $\ovl{QS}\parallel\ovl{Q'S'}$, therefore by Theorem \ref{thrmdesmall4}, $\ovl{RS}\parallel\ovl{R'S'}$.

In the case where $S\notin\ovl{PR}$, we also know that $S$ must lie outside one of the two lines $\ovl{PP'}$ and $\ovl{RR'}$ because they're parallel lines that are apart from each other. We've already covered the case where $S$ is not on $\ovl{RR'}$, therefore we assume that $S$ lies outside both $\ovl{PP'}$ and $\ovl{PR}$. In this case $S'$ lies on the line $l$ through $P'$ parallel to $\ovl{PS}$, and $R'$ is on the line $m$ through $P'$ parallel to $\ovl{PR}$. $l\#m$ because these two lines are parallel to lines that are apart from each other and have an intersection point. $P'$ is the intersection point of $l$ and $m$ and $S'$ is apart from $P'$ therefore $S'\notin m$ and hence $S'\#R'$. Also, $\ovl{RS}$ is parallel to $\ovl{R'S'}$ by Theorem \ref{thrmdesmall}.
\end{proof}

For $P$ and $P'$ points of an affine plane which are apart from each other, we denote by $\tau_{PP'}$ the unique translation sending $P$ to $P'$. Note that by Lemma \ref{lemlinetranslation}, a translation $\tau$ is of the form $\tau_{PP'}$ for $P\#P'$ iff for some point $Q$, $Q\#\tau(Q)$. 

\begin{thrm} \label{thrmexistrans}
 In an affine plane, for any two points $P$ and $P'$ there exists a (unique) translation sending $P$ to $P'$.
\end{thrm}

\begin{proof}
 Given $P$, there exists a point $Q$ that is apart from $P$, hence $P\#P'$ or $P'\#Q$. In the first case, by the above theorem $\tau_{PP'}$ is the appropriate translation. In the second case, the translation is the composite $\tau_{QP'}\circ\tau_{PQ}$.
\end{proof}

For $P$ and $P'$ points of an affine plane, we denote by $\tau_{PP'}$ the unique translation sending $P$ to $P'$.

\begin{cor} \label{cortwotrans}
Let $\tau$ be a translation. Then, either there exist points $P\#Q$  such that $\tau$ is $\tau_{PQ}$ or there exist points $P\#Q\#R$ such that $\tau$ is $\tau_{QR}\circ\tau_{PQ}$.
\end{cor}

\begin{proof}
Given a translation $\tau$, pick a point $P$. Let $R=\tau(P)$, and let $Q$ be a point such that $P\#Q$. If $P\# R$, then $\tau=\tau_{PR}$. If $Q\#R$, then $\tau=\tau_{QR}\circ\tau_{PQ}$.
\end{proof}

\begin{lem}
 The inverse of a translation is a translation.
\end{lem}

\begin{proof}
Given a translation $\tau$, pick a point $P$ and let $P'=\tau(P)$. Then, by Theorem \ref{thrmexistrans} $\tau=\tau_{PP'}$. The inverse of $\tau$ is $\tau_{P'P}$ because $P'$ is a fixed point for $\tau\tau_{P'P}$ and $P$ is a fixed point for $\tau_{P'P}\tau$.
\end{proof}

Translations are closed under composition and they are invertible, hence they form a group. Moreover by Theorem \ref{thrmexistrans}, the points of the affine plane are a torsor over translations via the natural action of translations on points.

\begin{lem}
The group of translations is abelian.
\end{lem}

\begin{proof}
Let us first consider translations of the form $\tau_{PQ}$ where $P\#Q$ (i.e. translations such that $\tau(P)\#P$ for some (any) point $P$). So we consider two translations such that $\tau_{PQ}$ and $\tau_{QR}$ such that $P\#Q\#R$.

There exists a point $S$ such that $S\notin \ovl{PQ}$. $\tau_{QS}(\tau_{PQ}(P))=S$ by the definition of the translations. $\tau_{QS}(P)$ is the unique intersection of the line through $P$ parallel to $\ovl{QS}$ and the line through $S$ parallel to $\ovl{PQ}$ by the construction of the translation in Lemma \ref{lemexistrans}. By the same construction (and because $\tau_{QS}(P)\notin\ovl{PQ}$),  $\tau_{PQ}(\tau_{QS}(P))$ is the unique intersection of the line through $\tau_{QS}(P)$ parallel to $\ovl{PQ}$ and the line through $Q$ parallel to $\ovl{P\tau_{QS}(P)}$. $S$ belongs to both these lines, therefore $\tau_{PQ}(\tau_{QS}(P))=S=\tau_{QS}(\tau_{PQ}(P))$. A translation is uniquely determined by the image of a point, therefore $\tau_{PQ}\tau_{QS}=\tau_{QS}\tau_{PQ}$.

Now, we return to the two translations $\tau_{PQ}$ and $\tau_{QR}$ such that $P\#Q\#R$. The lines $\ovl{PQ}$ and $\ovl{QS}$ are apart from each other and $R$ is apart from $Q$, therefore it either lies outside $\ovl{PQ}$ or outside $\ovl{QS}$.

In the case where $R\notin\ovl{PQ}$ we proceed as above and conclude that the two translations commute.

In the case where $R$ lies outside $\ovl{QS}$, it is sufficient to prove that  $\tau_{QR}\tau_{PQ}\tau_{SP}=\tau_{PQ}\tau_{QR}\tau_{SP}$ because $\tau_{SP}$ is invertible.
$\tau_{QR}\tau_{PQ}\tau_{SP}=\tau_{QR}\tau_{SQ}$ by definition and $\tau_{QR}\tau_{SQ}=\tau_{SQ}\tau_{QR}$ by the previous argument because $R\notin \ovl{QS}$. Notice that $\tau_{SQ}\tau_{QR}=\tau_{PQ}\tau_{SP}\tau_{QR}$. Let $R'=\tau_{PS}(R)$. Then, $R'\#R$ and $\ovl{RR'}\parallel\ovl{PS}$. $\ovl{PS}\#\ovl{PQ}$ and they intersect and $\ovl{RR'}\parallel\ovl{PS}$, hence $\ovl{PQ}\#\ovl{RR'}$. Hence, $\ovl{QR}$ is apart from $\ovl{PQ}$ or $\ovl{RR'}$. In the first case, we conclude that $R\notin \ovl{PQ}$ so we return to the previous argument. In the case where $\ovl{QR}\#\ovl{RR'}$, $R'$ lies outside $\ovl{QR}$, and since $\tau_{SP}=\tau_{RR'}$, we conclude that $\tau_{PQ}\tau_{SP}\tau_{QR}=\tau_{PQ}\tau_{RR'}\tau_{QR}=\tau_{PQ}\tau_{QR}\tau_{RR'}=\tau_{PQ}\tau_{QR}\tau_{SP}$, therefore $\tau_{QR}\tau_{PQ}=\tau_{PQ}\tau_{QR}$.

Hence, we have proved that $\tau\tau'=\tau'\tau$ for any two translations of the form $\tau_{PQ}$ with $P\#Q$. The result for any two translations follows from this result and Corollary \ref{cortwotrans}.
\end{proof}

\begin{prop} \label{proptn}
Given two points $A$ and $B$ of an affine plane $\ca A$ and translations $\tau_1$, $\tau_2$, $\tau_3$ and $\tau_4$ then the following hold:
\begin{enumerate}
\item $\tau_1(A)=\tau_2(A)$ implies that $\tau_1(B)=\tau_2(B)$.
\item $\tau_1(A)\#\tau_2(A)$ implies that $\tau_1(B)\#\tau_2(B)$.
\item $\tau_1(A)\#\tau_2(A)$ and $\tau_3(A)\in \ovl{\tau_1(A)\tau_2(A)}$ implies that $\tau_1(B)\#\tau_2(B)$ and that $\tau_3(B)\in \ovl{\tau_1(B)\tau_2(B)}$.
\item $\tau_1(A)$, $\tau_2(A)$, $\tau_3(A)$ non-collinear implies that $\tau_1(B)$, $\tau_2(B)$, $\tau_3(B)$ are non-collinear.
\item $\tau_1(A)\#\tau_2(A)$, $\tau_3(A)\#\tau_4(A)$ and $\ovl{\tau_1(A)\tau_2(A)}\parallel \ovl{\tau_3(A)\tau_4(A)}$ together imply that $\tau_1(B)\#\tau_2(B)$, $\tau_3(B)\#\tau_4(B)$ and $\ovl{\tau_1(B)\tau_2(B)}\parallel \ovl{\tau_3(B)\tau_4(B)}$.
\end{enumerate}
\end{prop}

\begin{proof}
For each one of them we use Proposition \ref{propdilhom} on the translation $\tau_{AB}$ and use that the group of translations is abelian.
\end{proof}

\begin{rmk}
We will often use the above proposition in the case where $\tau_1$ is the identity.
\end{rmk}

Given an affine plane $\ca A$, we can view the set of translations of $\ca A$ as the set of points for a new affine plane which we denote by $\ca A_{\text{Tn}}$. We can define a $\#_{\text{Tn}}$ relation on these points by $\tau_1\#_{\text{Tn}} \tau_2$ when for some (any) point $A$ of $\ca A$, $\tau_2(A)\# \tau_2(A)$.

We define the set of lines to be the quotient of $\#_{\text{Tn}}$ by the relation which relates $(\tau_1,\tau_2)$ to $(\tau_3,\tau_4)$ when for some (any) point $A$ of $\ca A$, $\tau_3(A)\in \ovl{\tau_1(A)\tau_2(A)}$ and $\tau_4(A)\in \ovl{\tau_1(A)\tau_2(A)}$ (or equivalently when  $\ovl{\tau_1(A)\tau_2(A)}=\ovl{\tau_3(A)\tau_4(A)}$).

A translation $\tau_3$ lies on the line represented by $(\tau_1,\tau_2)$ iff $\tau_3(A)\in \ovl{\tau_1(A)\tau_2(A)}$ for some (any) $A$. A translation $\tau_3$ lies outside the line represented by $(\tau_1,\tau_2)$ iff $\tau_3(A)\notin \ovl{\tau_1(A)\tau_2(A)}$ for some (any) $A$. The lines represented by $(\tau_1,\tau_2)$ and $(\tau_3,\tau_4)$ are apart from each other iff for some (any) point $A$ of $\ca A$, the lines $\ovl{\tau_1(A)\tau_2(A)}\#\ovl{\tau_3(A)\tau_4(A)}$. The lines represented by $(\tau_1,\tau_2)$ and $(\tau_3,\tau_4)$ are parallel iff for some (any) point $A$ of $\ca A$, the lines $\ovl{\tau_1(A)\tau_2(A)}\parallel\ovl{\tau_3(A)\tau_4(A)}$. Notice that none of the definitions depends on the choice of representatives of the lines.

\begin{lem} \label{lemplanetrans}
Given an affine plane $\ca A$ and a point $A$, then there is a structure-preserving isomorphism from $\ca A_{\text{Tn}}$ to $\ca A$. 
\end{lem}

\begin{proof}
Let us consider the morphism $f_A: \ca A_{\text{Tn}}\to \ca A$ defined in the following way. The morphism on points sends a translation $\tau$ to the point $\tau(A)$. The morphism on lines sends a line represented by $(\tau_1,\tau_2)$ to $\ovl{\tau_1(A)\tau_2(A)}$ and this does not depend on the choice of representative for the line. It is clear from the above definitions that this is a structure-preserving homomorphism.

The inverse of the above morphism $g_A$, maps a point $B$ of $\ca A$ to the translation $\tau_{AB}$. Given a line $l$ of $\ca A$, there exist points $B\#C$ on $l$. $g_A$ maps the line $l$ to the line represented by $(\tau_{AB},\tau_{AC})$. Notice that this morphism does not depend on the choice of $B$ and $C$. This is also a structure-preserving homomorphism and it is easy to confirm that it is the inverse of the above morphism.
\end{proof}

An affine plane $\ca A$ contains a point $A$, therefore using that point we create a structure preserving isomorphism from the structure $\ca A$ to the structure $\ca A_{\text{Tn}}$. Hence, $\ca A_{\text{Tn}}$ satisfies the axioms of affine planes because $\ca A$ does.

By creating the group of translations we have constructed a new affine plane which comes with a chosen point. That means that it is not necessarily isomorphic to the original affine plane. In Chapter \ref{chaZG}, we see that the object of points of the generic affine plane is not pointed, hence it is not isomorphic to the one constructed using the group of translations.

\begin{thrm} \label{thrmexistdil}
 Let $P$, $Q$ and $Q'$ be points of an affine plane such that $Q\#P\#Q'$ and $Q'\in \ovl{PQ}$. There exists a unique dilatation $\sigma$ with $P$ a fixed point and $\sigma(Q)=Q'$.
\end{thrm}

\begin{proof}
 A dilatation is uniquely determined by the images of two points that are apart from each other, so the uniqueness of $\sigma$ follows from that.

Pick a point $R$ lying outside $\ovl{PQ}$. Let $k$ be the line through $Q'$ parallel to $\ovl{QR}$. Define $\sigma(R)=R'$ to be the intersection point of $k$ with $\ovl{PR}$. $k\parallel\ovl{QR}\#\ovl{PQ}$, therefore $k\#\ovl{PQ}$ and they intersect at $Q'$. $P\#Q'$, therefore $P$ lies outside $k$. $R'\in k$ therefore $R'\#P$. Notice also that $R'\#Q'$.

Let $A$ be any point, then $A\#P$ or $A\#Q$.

If $A\#P$, then $A$ lies outside at least one of the two lines $\ovl{PQ}$ and $\ovl{PR}$. If it lies outside $\ovl{PQ}$ we define $\sigma(A)$ to be the intersection of $\ovl{PA}$ and the line through $Q'$ parallel to $\ovl{QA}$ and then $\sigma(A)\#P$ and $\sigma(A)\#Q'$ by a similar argument with the one used above. If it lies outside $\ovl{PR}$ we define $\sigma(A)$ to be the intersection of $\ovl{QR}$ and the line through $R'$ parallel to $\ovl{RA}$ and similarly $\sigma(A)\#P$ and $\sigma(A)\#R'$. Note that if $A$ lies outside both $\ovl{PQ}$ and $\ovl{PR}$, then the two definitions are equivalent by Theorem \ref{thrmdesbig}.

If $A\#Q$ then $A$ lies outside one of the two lines $\ovl{PQ}$ and $\ovl{QR}$. In the first case, we define $\sigma(A)$ as before. If $A\notin \ovl{QR}$, we define $\sigma(A)$ to be the intersection of the line through $Q'$ parallel to $\ovl{QA}$ and the line through $R'$ parallel to $\ovl{RA}$. Again if $A$ lies outside both $\ovl{PQ}$ and $\ovl{QR}$ or outside both $\ovl{PR}$ and $\ovl{QR}$, then the two definitions of $\sigma(A)$ agree by Theorem \ref{thrmdesbig}.

We now need to prove that $\sigma$ is a dilatation, i.e. for $A$ and $B$ two points such that $A\#B$, then $\sigma(A)\#\sigma(B)$ and $\ovl{AB}\parallel \ovl{\sigma(A)\sigma(B)}$

Let $A$, $B$ be any two points such that $A\#B$. Then, at least one of the two points is apart from $P$. Without loss of generality, suppose that $A\#P$ and that $A\notin \ovl{PQ}$. $\sigma(A)$ is the intersection point of $\ovl{PA}$ and the line through $Q'$ parallel to $\ovl{QA}$, and by a previous discussion $\sigma(A)$ is apart from both $P$ and $Q'$. $B\#A$, therefore $B$ lies outside at least one of the two lines $\ovl{PA}$ and $\ovl{AQ}$.

In the case where $B\notin\ovl{PA}$, then $B\#P$ therefore $\sigma(B)$ is apart from $P$ and it lies outside $\ovl{PA}$ (in both the case where $B\notin\ovl{PQ}$ and the case where $B\notin\ovl{PR}$). Since $B\notin\ovl{PA}$, we also know that $\ovl{PA}\#\ovl{PB}$, and they intersect at $P$ therefore $\sigma(B)\notin\ovl{PA}$, hence $\sigma(A)\#\sigma(B)$. In the case where $B\notin\ovl{PQ}$, by Theorem \ref{thrmdesbig} we know that $\ovl{AB}\parallel\ovl{\sigma(A)\sigma(B)}$. In the case where $B\notin\ovl{PR}$, then $\ovl{AB}\parallel\ovl{\sigma(A)\sigma(B)}$ by Theorem \ref{thrmdesbig4}.

In the case where $B\notin\ovl{AQ}$, then $B\#Q$, therefore $B$ lies outside at least one of the lines $\ovl{PQ}$ or $\ovl{QR}$ and in both cases $\sigma(B)$ is apart from $Q'$ and $\ovl{QB}\parallel\ovl{Q'\sigma(B)}$. Since $\ovl{Q'\sigma(A)}\parallel \ovl{QA}$ and $\ovl{QB}\#\ovl{QA}$, then $\ovl{Q'\sigma(B)}\#\ovl{Q'\sigma(A)}$, hence $\sigma(A)\#\sigma(B)$. In the case where $B\notin \ovl{PQ}$, $\ovl{AB}\parallel\ovl{\sigma(A)\sigma(B)}$ by Theorem \ref{thrmdesbig}. In the case where $B\notin \ovl{QR}$, then by Theorem \ref{thrmdes5}, $\ovl{AB}$ is parallel $\ovl{\sigma(A)\sigma(B)}$.
\end{proof}

\begin{rmk}
In all the cases where we are taking intersections above it is easy to show that the lines we are intersecting are apart from each other and have an intersection point (usually because they are parallel to lines with this property).
\end{rmk}

\begin{thrm}
 Given $P$, $Q$, $P'$, $Q'$ points of the affine plane such that $P\#Q$, $P'\#Q'$ and $\ovl{PQ}\parallel \ovl{P'Q'}$, there exists a unique dilatation $\sigma$ such that $\sigma(P)=P'$ and $\sigma(Q)=Q'$.
\end{thrm}
 
\begin{proof}
 Let $R=\tau_{PP'}(Q)$. $R\# P'$ and $Q'\in \ovl{P'R}$. By Theorem \ref{thrmexistdil}, there exists a dilatation $\pi$ such that fixes $P'$ and sends $R$ to $Q'$. Hence $\pi\circ \tau_{PP'}$ is a dilatation which sends $P$ to $P'$ and $Q$ to $Q'$, and since a dilatation is uniquely determined by the images of two points that are apart from each other, $\pi\circ \tau_{PP'}$ is the unique such dilatation.
\end{proof}

A simple consequence of this theorem is the following lemma.

\begin{lem}
 Dilatations are invertible and their inverses are also dilatations.
\end{lem}

\begin{proof}
Given a dilatation $\sigma$, let $P$ and $Q$ be two points that are apart from each other, and let $P'$ and $Q'$ be their images. Then, $P'\#Q'$ and $\ovl{PQ}\parallel \ovl{P'Q'}$, therefore by the above theorem there exists a unique dilatation $\sigma'$ which sends $P'$ and $Q'$ to $P$ and $Q$ respectively. $P$ and $Q$ are fixed points of $\sigma'\circ \sigma$, therefore it is the identity. $P'$ and $Q'$ are fixed points of $\sigma \circ \sigma'$, therefore it is the identity. Hence, $\sigma'$ is the inverse of $\sigma$.
\end{proof}

\begin{lem} \label{lemnormtrans}
Translations form a normal subgroup of the group of dilatations, i.e. given a dilatation $\sigma$ and a translation $\tau$, the dilatation $\sigma\circ\tau\circ\sigma^{-1}$ is a translation. Moreover, given a dilatation $\sigma$ and points $P\#Q$ of an affine plane, then $\sigma\tau_{PQ}\sigma^{-1}(P)$ lies on $\ovl{PQ}$. 
\end{lem}

\begin{proof}
Let $\sigma$ be a dilatation and let $\tau$ be a translation. $\sigma\circ\tau\circ\sigma^{-1}$ is a dilatation because dilatations are closed under composition and taking inverses.

Let us consider the case where the translation $\tau$ is of the form $\tau_{PQ}$ where $P\#Q$. Then, by Lemma \ref{lemlinetranslation}, $\tau(R)\#R$ for every point $R$. Therefore, $\sigma^{-1}(R)\#\tau(\sigma^{-1}(R))$ for every point $R$, and applying the dilatation $\sigma$ to both points, we know that $R\#\sigma(\tau(\sigma^{-1}(R)))$. Therefore, by Lemma \ref{lemlinetranslation} it is now sufficient to prove that for any point $R$, the line $\ovl{R\sigma(\tau(\sigma^{-1}(R)))}$ is parallel to $\ovl{P\sigma(\tau(\sigma^{-1}(P)))}$. The lines $\ovl{\sigma^{-1}(P)\tau(\sigma^{-1}(P))}$ and $\ovl{\sigma^{-1}(R)\tau(\sigma^{-1}(R))}$ are parallel because $\tau$ is a translation. Also, $\ovl{\sigma^{-1}(P)\tau(\sigma^{-1}(P))}\parallel\ovl{P\sigma(\tau(\sigma^{-1}(P)))}$ and $\ovl{\sigma^{-1}(R)\tau(\sigma^{-1}(R))}\parallel\ovl{R\sigma(\tau(\sigma^{-1}(R)}$ because $\sigma$ is a dilatation. Hence, $\ovl{R\sigma(\tau(\sigma^{-1}(R)))}\parallel \ovl{P\sigma(\tau(\sigma^{-1}(P)))}$
by the transitivity of $\parallel$, therefore $\sigma\tau\sigma^{-1}$ is a translation.

By the above and using the transitivity of $\parallel$, we can also see that $\ovl{P\sigma(\tau(\sigma^{-1}(P)))}\parallel \ovl{\sigma^{-1}(R)\tau(\sigma^{-1}(R))}$. For $R=\sigma(P)$ and since $\tau(P)=Q$, the above relation becomes $\ovl{P\sigma(\tau(\sigma^{-1}(P)))}\parallel \ovl{PQ}$. Hence, $\sigma\tau_{PQ}\sigma^{-1}(P)$ lies on $\ovl{PQ}$.

Let us now consider the case where $\tau$ is any translation. Then, either there exist points $P\#Q$ such that $\tau=\tau_{PQ}$ or there exist points $P\# Q\# R$  such that $\tau=\tau_{QR}\circ\tau_{PQ}$. The first case is covered above. In the second case, $\sigma \circ \tau \circ \sigma^{-1}= (\sigma \tau_{QR} \sigma^{-1}) \circ (\sigma \tau_{PQ} \sigma^{-1})$ and therefore it is a translation because it is the composite of two translations.
\end{proof}

\section{The local ring of trace preserving homomorphisms} \label{secaffring}

\begin{defn}
 Let $\sigma$ be a dilatation, and $P$ a point. Any line containing $P$ and $\sigma(P)$ is called a \emph{trace} of $\sigma$ at $P$. 
\end{defn}

\begin{rmk}
If $P\#\sigma (P)$ then the trace of $\sigma$ at $P$ is unique.

Given $A$ and $B$ points that are apart from each other, $\ovl{AB}$ is the unique trace of $\tau_{AB}$.
\end{rmk}

\begin{thrm} \label{thrmtrace}
 Let $\sigma$ be a dilatation, $P$ a point and $l$ a trace of $\sigma$ at $P$. If $Q\in l$, then $\sigma (Q)$ also lies on $l$.
\end{thrm}

\begin{proof}
There exists a point $R\in l$ with $P\#R$ (because any line contains two points which are apart from each other), therefore $l=\ovl{PR}$. $\ovl{\sigma(P)\sigma(R)}$ is parallel to $\ovl{PR}$, therefore the two lines must be equal because $l$ already contains $\sigma(P)$. Hence $\sigma(R)\in l$.

Since $P\#R$, $Q$ must be apart from at least one of the points $P$ or $R$. Hence, by repeating the above argument we prove that $\sigma(Q)\in l$.
\end{proof}

\begin{cor}
 If two traces of a dilatation have a unique intersection point then that is a fixed point for the dilatation.
\end{cor}

Given an endomorphism $\alpha$ of the group of translations, and a translation $\tau$, we write $\tau^\alpha$ for the image of $\tau$ via $\alpha$.

\begin{defn}
An endomorphism $\alpha$ of the group of translations is \emph{trace preserving} when it is a group homomorphism and for every translation $\tau$ the traces of $\tau$ are also traces for $\tau^{\alpha}$, i.e. given a point $P$ and a line $l$ such that $P, \tau(P)\in l$, then $\tau^{\alpha}(P)\in l$.
\end{defn}

\begin{egs}
\begin{enumerate}
\item The identity morphism on translations is a trace preserving homomorphism and it will be denoted by $1$.
\item The constant morphism mapping all translations to the identity translation is a trace preserving homomorphism and it will be denoted by $0$.
\end{enumerate}
\end{egs}

\begin{lem}
An endomorphism $\alpha$ of the group of translations is trace preserving iff it is a group homomorphism and for every two points $A$ and $B$ that are apart from each other $\tau_{AB}^{\alpha}(A)\in \ovl{AB}$.
\end{lem}

\begin{proof}
 The direct implication is true by the definition of trace preserving homomorphisms.

Let us suppose that $\alpha$ is an endomorphism of the group of translation which is a group homomorphism and such that for any two points $A$ and $B$ that are apart from each other $\tau_{AB}^{\alpha}(A)\in \ovl{AB}$.
Let $\tau$ be a translation and let $l$ be a trace for $\tau$ at a point $P$. Let $R$ be $\tau(P)$ and notice that $R\in l$. We pick a point $Q$ on $l$ which is apart from $P$. Then, either $P\#R$ or $Q\#R$.
In the case where $P\#R$, $\tau=\tau_{PR}$ and $l=\ovl{PR}$, hence $\tau^{\alpha}(P)\in l$ and $l$ is a trace for $\tau^\alpha$ at $P$.
In the case where $Q\#R$, then $l=\ovl{PQ}=\ovl{QR}$. $\tau=\tau_{PR}=\tau_{QR}\tau_{PQ}$, and since $\alpha$ is a group homomorphism $\tau^\alpha=\tau^\alpha_{QR}\tau^\alpha_{PQ}$. $\tau_{PQ}^\alpha(P)\in l$ and $l$ is a trace for $\tau^\alpha_{QR}$ at $Q$, hence by Theorem \ref{thrmtrace} $\tau^\alpha_{QR}(\tau^\alpha_{PQ}(P))\in l$. Hence, $\tau^\alpha(P)\in l$ and $l$ is also a trace for $\tau^\alpha$ at $P$. Therefore, $\alpha$ is a trace preserving homomorphism.
\end{proof}

\begin{cor}
 $\tau\mapsto \sigma\tau\sigma^{-1}$ is a trace preserving homomorphism.
\end{cor}

\begin{proof}
The above morphism is an endomorphism of the group of translations by Lemma \ref{lemnormtrans}. It is a ring homomorphism and given points $A\#B$ of the affine plane, then $\sigma\tau_{AB}\sigma^{-1}(A)\in \ovl{AB}$ by Lemma \ref{lemnormtrans}. Hence, by the above Lemma, this endomorphism is a trace preserving homomorphism.
\end{proof}

\begin{lem}
 Let $\alpha$ be a trace preserving homomorphism and suppose that for some translation $\tau$ such that $P\#\tau(P)$, $\tau^\alpha(P)=P$. then $\alpha=0$.
\end{lem}

\begin{proof}
Let $Q=\tau(P)$ and pick a point $R$ lying outside the line $\ovl{P\tau(P)}$. Then $\tau_{PR}=\tau_{QR}\tau_{PQ}$, therefore $\tau_{PR}^\alpha=\tau_{QR}^\alpha\tau_{PQ}^\alpha$, hence $\tau_{PR}^\alpha(P)=\tau_{QR}^\alpha(P)$. $\tau_{PR}^\alpha(P)\in \ovl{PR}$ and $\tau_{QR}^\alpha(P)\in k$ where $k$ is the line through $P$ parallel to $\ovl{QR}$. The two lines are apart from each other and their unique intersection point is $P$, hence $\tau_{PR}^\alpha(P)=P$.

Let $\tau'$ be any translation. If $\tau'(P)\#P$, then $\tau'(P)$ lies outside at least one of the lines $\ovl{PQ}$ or $\ovl{PR}$, therefore by the above argument we again conclude that $(\tau')^\alpha(P)=P$.

Finally, any translation is a composite of two translations that send $P$ to a point apart from $P$, hence any translation is sent to the identity translation by $\alpha$. Therefore, $\alpha=0$.
\end{proof}

\begin{thrm} \label{thrmtraceunique}
 A trace preserving homomorphism is uniquely determined by the image of a translation $\tau$ where $P\#\tau(P)$ for some (any) point $P$. 
\end{thrm}

\begin{proof}
 Let $\tau$ be a translation and $P$ a point such that $P\#\tau(P)$ and let $\alpha$ and $\beta$ be trace preserving homomorphisms such that $\tau^\alpha=\tau^\beta$. Then $\tau^{\alpha-\beta}$ is the identity translation, hence $\tau^{\alpha-\beta}(P)=P$. By the above lemma, $\alpha-\beta=0$, hence $\alpha=\beta$.
\end{proof}

Let $R$ be a local ring. The group of translations of $\mathbb A(R)$ is isomorphic to the group of $3\times 3$ matrices over $R$ of the form $\begin{pmatrix} 1 & 0 & a \\
0 & 1 & b \\
0 & 0 & 1 \end{pmatrix}$
where $a$ and $b$ are in $R$. Then, for each $r$ in $R$ the homomorphism $\alpha_r$ which sends $\begin{pmatrix} 1 & 0 & a \\
0 & 1 & b \\
0 & 0 & 1 \end{pmatrix}$ to $\begin{pmatrix} 1 & 0 & ra \\
0 & 1 & rb \\
0 & 0 & 1 \end{pmatrix}$ is a trace preserving homomorphism.
Moreover, any trace preserving homomorphism of $\mathbb A(R)$ is of the above form. Given a trace preserving homomorphism $\alpha$, consider the translation $\tau$ which sends $(0,0)$ to $(1,0)$. Then, $\tau^{\alpha}(0,0)$ lies on the line $(1,0,0)$ therefore it $\tau^{\alpha}(0,0)=(r,0)$ for some $r$ in $R$. Then, both $\alpha$ and $\alpha_r$ send $\tau$ to $\tau^{\alpha}$, hence by the above theorem $\alpha=\alpha_r$. Therefore, we have proved the following:

\begin{prop}
The set of trace preserving homomorphisms of the affine plane over a local ring $R$ is isomorphic to the elements of the local ring $R$.
\end{prop}

We define $\text{Tp}$ to be the set of trace preserving endomorphisms of the group of translations, and we are going to give it a ring structure.

\begin{defn}
Let $T$ be the set of translations and let $\alpha$, $\beta$ be trace preserving preserving homomorphisms.

We define $\alpha+\beta:T\to T$ to be the endomorphism which sends a translation $\tau$ to the translation $\tau^{\alpha}\circ\tau^{\beta}$.

We define $\alpha\cdot \beta:T\to T$ to be the endomorphism which sends a translation $\tau$ to the translation $(\tau^{\beta})^{\alpha}$.

We define $0:T\to T$ to be the map sending all translations to the identity translation.

We define $1:T\to T$ to be the identity homomorphism.

We define $-1:T\to T$ to be the morphism sending a translation to its inverse (i.e. sending $\tau$ to $\tau^{-1}$ and the notation is not ambiguous).
\end{defn}

\begin{thrm}
 If $\alpha$ and $\beta$ belong to $\text{Tp}$, then $\alpha+\beta$ and $\alpha\cdot \beta$ belong to $\text{Tp}$. $\text{Tp}$ with $+$ and $\cdot$ becomes an associative ring with the identity as the multiplicative unit and the map that sends all the translations to the identity as the additive unit.
\end{thrm}

\begin{proof}

\begin{align*} (\tau\sigma)^{\alpha+\beta} &=(\tau\sigma)^{\alpha}(\tau\sigma)^{\beta} &\text{(by the definition of $\alpha+\beta$)}\\
&=\tau^{\alpha}\sigma^{\alpha}\tau^{\beta}\sigma^{\beta} &\text{(because $\alpha$ and $\beta$ are group homomorphisms)}\\
&=\tau^{\alpha}\tau^{\beta}\sigma^{\alpha}\sigma^{\beta} &\text{(because the group of translations is abelian)}\\
&=\tau^{\alpha+\beta}\sigma^{\alpha+\beta}
\end{align*}
Hence, $\alpha+\beta$ is a group homomorphism. Let $\tau$ be a translation, $P$ a point and let $l$ be a trace for $\tau$ at $P$. Then, $\tau^{\beta}(P)\in l$ because $\beta$ is trace preserving. Then $\tau(\tau^{\beta}(P))$ lies on $l$ by Theorem \ref{thrmtrace}, and since $\alpha$ is trace preserving $\tau^\alpha(\tau^{\beta}(P))$ lies on $l$, hence $\alpha+\beta$ is trace preserving an-d therefore belongs in $\text{Tp}$.

$(\tau\sigma)^{\alpha\cdot\beta}=((\tau\sigma)^{\beta})^{\alpha}=(\tau^{\beta}\sigma^{\beta})^{\alpha}=(\tau^{\beta})^\alpha(\sigma^{\beta})^{\alpha}=\tau^{\alpha\cdot\beta}\sigma^{\alpha\cdot\beta}$ by the definition of $\alpha\cdot \beta$ and the fact that $\alpha$ and $\beta$ are group homomorphism. Hence, $\alpha\cdot\beta$ is a group homomorphism. Also, suppose $l$ is a trace for a translation $\tau$ at $P$. Then, it is also a trace for $\tau^\beta$ because $\beta$ is trace preserving hence it is also a trace for $(\tau^\beta)^\alpha$ because $\alpha$ is trace preserving. Therefore, $\alpha\cdot \beta\in \text{Tp}$.

Proving that addition is associative: Let $\alpha, \beta, \gamma \in \text{Tp}$ and $\tau$ a translation. Then, $\tau^{(\alpha+\beta)+\gamma}=\tau^{\alpha+\beta}\tau^\gamma=(\tau^\alpha\tau^\beta)\tau^\gamma=\tau^\alpha(\tau^\beta\tau^\gamma)=\tau^\alpha\tau^{\beta+\gamma}=\tau^{\alpha+(\beta+\gamma)}$ by associativity of composition.

Proving that addition is commutative: Let $\alpha, \beta\in \text{Tp}$ and $\tau$ a translation. Then, $\tau^{\alpha+\beta}=\tau^\alpha\tau^\beta=\tau^\beta\tau^\alpha=\tau^{\beta+\alpha}$ because the group of translations is abelian.

$\alpha+0=\alpha$ because $\tau^{\alpha+0}=\tau^\alpha\tau^0=\tau^\alpha1=\tau^\alpha$.

$\alpha+(-1)\cdot\alpha=0$ because $\tau^{\alpha+(-1)\cdot\alpha}=\tau^\alpha\tau{(-1)\cdot\alpha}=\tau^\alpha(\tau^\alpha)^{-1}=1=\tau^0$.

$(\beta+\gamma)\alpha=\beta\alpha+\gamma\alpha$ because $\tau^{(\beta+\gamma)\alpha}=(\tau^\alpha)^{\beta+\gamma}=(\tau^\alpha)^\beta(\tau^\alpha)^\gamma=\tau^{\beta\alpha}\tau^{\gamma\alpha}=\tau^{\beta\alpha+\gamma\alpha}$.

$\alpha(\beta+\gamma)=\alpha\beta+\alpha\gamma$ because $\tau^{\alpha(\beta+\gamma)}=(\tau^{\beta+\gamma})^\alpha=(\tau^\beta\tau^\gamma)^\alpha=(\tau^\beta)^\alpha(\tau^\gamma)^\alpha=\tau^{\alpha\beta}\tau^{\alpha\gamma}=\tau^{\alpha\beta+\alpha\gamma}$.

$(\alpha\beta)\gamma=\alpha(\beta\gamma)$ because $\tau^{(\alpha\beta)\gamma}=(\tau^\gamma)^{\alpha\beta}=((\tau^\gamma)^\beta)^\alpha=(\tau^{\beta\gamma})^\alpha=\tau^{\alpha(\beta\gamma)}$.
\end{proof}

\begin{lem}
The ring $\text{Tp}$ satisfies that $\text{inv}(0)\vdash \bot$.
\end{lem}

\begin{proof}
Pick points $A$, $B$ of the affine plane that are apart from each other. $\tau_{AB}^0(A)=A$ while $\tau_{AB}^1(A)=B$. Hence, $0=1$ implies that $A=B$ and therefore is false.
\end{proof}

\begin{thrm} \label{thrmtracedilatation}
If $\alpha$ is a trace preserving homomorphism and $P$, $Q$ are points that are apart from each other and $\tau_{PQ}^\alpha(P)\#P$, then there exists a unique dilatation $\sigma$ which has $P$ as a fixed point and such that $\tau^{\alpha}=\sigma \tau \sigma^{-1}$ for all translations $\tau$.
\end{thrm}

\begin{proof}
Suppose such a $\sigma$ exists. Let $Q$ be a point. Then 
$$\tau_{PQ}^\alpha(P)=\sigma \tau_{PQ} \sigma^{-1}(P)=\sigma\tau_{PQ}(P)=\sigma(Q).$$
Therefore, if such a dilatation $\sigma$ exists it is unique and $\sigma(Q)=\tau_{PQ}^\alpha(P)$ for every point $Q$.

Given $\alpha$, $P$ and $Q$ as above, let $Q'$ be $\tau^\alpha_{PQ}(P)$. By Theorem \ref{thrmexistdil}, there exists a dilatation $\sigma$ with $P$ as a fixed point, sending $Q$ to $Q'$. Notice that $\tau_{PQ}^\alpha(P)=\sigma\tau_{PQ}\sigma^{-1}(P)$, therefore  $\tau_{PQ}^\alpha=\sigma\tau_{PQ}\sigma^{-1}$ because translations are uniquely determined by the image of a point. The map $\tau\mapsto\sigma\tau\sigma^{-1}$ is trace preserving homomorphism, therefore by Theorem \ref{thrmtraceunique}, $\sigma$ is the required dilatation.
\end{proof}

\begin{rmk}
Another way of phrasing the conclusion of the theorem is the following: there exists a unique dilatation $\sigma$ such that $\tau^\alpha_{PQ}=\tau_{P\sigma(Q)}$.
\end{rmk}

\begin{cor} \label{corinverse}
A trace preserving homomorphism $\alpha$ of the above form has an inverse.
\end{cor}

\begin{proof}
$\alpha(\tau)=\sigma\tau\sigma^{-1}$ for some dilatation $\sigma$ by the above theorem. Its inverse is the homomorphism sending a translation $\tau$ to $\sigma^{-1}\tau\sigma$.
\end{proof}

Notice that an element of $\text{Tp}$ is invertible as a homomorphism iff it is invertible as an element of the ring $\text{Tp}$.

\begin{prop} \label{propinv}
 Given a trace preserving homomorphism $\alpha$, the following are equivalent:
\begin{enumerate}
 \item $\alpha$ is invertible,
\item for some points $P$, $Q$ such that $P\#Q$, $\tau_{PQ}^\alpha(P)\#P$,
\item for any two points $P$, $Q$ such that $P\#Q$, $\tau_{PQ}^\alpha(P)\#P$.
\end{enumerate}

\end{prop}

\begin{proof}
 2 implies 1 by Corollary \ref{corinverse}.

3 implies 2 is clear.

\begin{tikzpicture}
\draw 
(0,0) node[left] {$P$} -- (6,3)
(0,0) -- (7,0);

\draw [fill] (4.75,1.125) circle [radius=0.03];

\draw 
(5,0) node[ below] {$Q'$}
--(4.5,2.25) node[above] {$A$}
 (2.5,0) node[ below] {$Q$}
--(2.25,1.125) node[above] {$B$}
(4.75, 1.125) node[right] {$C$};
\end{tikzpicture}

We now proceed to prove that 1 implies 3. Given a trace preserving homomorphism $\alpha$ which is invertible and $P$, $Q$ two points that are apart from each other, let $Q'=\tau_{PQ}^\alpha(P)$ (i.e. $\tau_{PQ}^\alpha=\tau_{PQ'}$) and let $\beta$ be $\alpha^{-1}$. We want to prove that $P\#Q'$. Let $A$ be a point of the affine plane lying outside $\ovl{PQ}$ and let $B$ be $\tau_{PA}^\beta(P)=B$ (i.e. $\tau_{PA}^\beta=\tau_{PB}$).

$A\notin \ovl{PQ}$ and $Q'\in \ovl{PQ}$, therefore $A\#Q'$, so we can consider the line $\ovl{AQ'}$.

$A\notin \ovl{PQ}$ therefore ($P\#A$ and) $Q\notin \ovl{PA}$. $B$ lies on $\ovl{PA}$, hence $Q\#B$ and we can consider the line $\ovl{BQ}$.

We claim that $\ovl{AQ'}$ is parallel to $\ovl{BQ}$.  $\tau_{PQ'}^\beta=\tau_{AQ'}^\beta\tau_{PA}^\beta$, therefore $\tau_{PQ}=\tau_{AQ'}^\beta\tau_{PB}$. Hence, $\tau_{AQ'}^\beta=\tau_{BQ}$. Let $C=\tau_{AQ'}^\beta(A)$, then $C\in \ovl{AQ'}$ because $\beta$ is trace preserving. $C$ is also equal to $\tau_{BQ}(A)$, hence by Lemma \ref{lemlinetranslation} $A\#C$ and $\ovl{BQ}\parallel \ovl{AC}$. $\ovl{AC}=\ovl{AQ'}$, hence $\ovl{BQ}\parallel \ovl{AQ'}$. 

$P\notin \ovl{AQ}$, therefore either $P\#Q'$ or $Q'\in \ovl{AQ}$. In the first case, our claim is proven, and in the latter we conclude that $Q\notin \ovl{AQ'}$. $Q$ lies on $\ovl{BQ}$, therefore $\ovl{BQ}\#\ovl{AQ'}$.
Since the lines $\ovl{BQ}$ and $\ovl{AQ'}$ are parallel and apart from each other $B$ lies outside one of them, hence $B\notin \ovl{AQ'}$. Hence, ($A\#B$ and) $Q'\notin \ovl{AB}$, and since $P\in\ovl{AB}$, we conclude that $P\#Q'$.
\end{proof}

\begin{thrm}
Given a trace preserving homomorphism $\alpha$, then $\alpha$ or $\alpha+1$ is invertible.
\end{thrm}

\begin{proof}
We choose points $P$ and $Q$ that are apart from each other, then by Lemma \ref{lemlinetranslation} for any point $A$, we know that $\tau_{PQ}(A)\#A$, therefore $\tau_{PQ}^{\alpha+1}(P)=\tau_{PQ}(\tau_{PQ}^\alpha(P))\#\tau_{PQ}^\alpha(P)$. Hence, at least one of the points $\tau_{PQ}^{\alpha+1}(P)$ and $\tau_{PQ}^\alpha(P)$  is apart from $P$. Therefore by Corollary \ref{corinverse}, in the first case $\alpha+1$ is invertible and in the second $\alpha$ is invertible.
\end{proof}

\begin{lem} \label{lemtrprs}
Let $P$, $Q$ and $Q'$ be three points such that $Q\#P\#Q'$ and $Q'\in\ovl{PQ}$. Then, there exists a (unique) trace preserving homomorphism $\alpha$, such that $\tau^\alpha_{PQ}=\tau_{PQ'}$.
\end{lem}

\begin{proof}
By Theorem \ref{thrmexistdil} there exists a unique dilatation $\sigma$ sending $P$ to $P$ and $Q$ to $Q'$. Let $\alpha$ be the trace preserving homomorphism sending a translation $\tau$ to $\sigma\tau\sigma^{-1}$. Then $\sigma\tau_{PQ}\sigma^{-1}(P)= \sigma(Q)=Q'$, i.e. $\tau_{PQ}^\alpha=\tau_{PQ'}$.

The uniqueness of $\alpha$ follows from Theorem \ref{thrmtraceunique}.
\end{proof}

\begin{rmk}
The trace preserving homomorphism constructed above is denoted by $\alpha_{PQQ'}$.
\end{rmk}

\begin{thrm} \label{thrmtr}
Let $P$, $Q$ and $Q'$ be any three points such that $P\#Q$ and $Q'\in\ovl{PQ}$. Then, there exists a (unique) trace preserving homomorphism $\alpha$, such that $\tau^\alpha_{PQ}=\tau_{PQ'}$.
\end{thrm}

\begin{proof}
$Q'\#\tau_{PQ}(Q')$, therefore $P\#Q'$ or $P\#\tau_{PQ}(Q')$. In the first case, we define $\alpha$ as above.

In the case where $P\#\tau_{PQ}(Q')$, we define $\beta$ to be the unique translation such that $\tau^\beta_{PQ}=\tau_{P\tau_{PQ}(Q')}$, and then define $\alpha$ to be $\beta-1$. Then $\tau^\alpha_{PQ}(P)=\tau^{-1}_{PQ}(\tau^\beta_{PQ}(P))=\tau^{-1}_{PQ}(\tau_{PQ}(Q'))=Q'$.

In both cases, the uniqueness of $\alpha$ follows from Theorem \ref{thrmtraceunique}.
\end{proof}

The following theorem gives a way of introducing coordinates.

\begin{thrm} \label{thrmafco}
Given two translations $\tau_1$ and $\tau_2$ and a point $P$ such that $P$, $\tau_1(P)$ and $\tau_2(P)$ are non-collinear, then for any translation $\tau$ there exist unique $\alpha, \beta\in \text{Tp}$ such that $\tau=\tau_1^\alpha\tau_2^\beta$.
\end{thrm}
 
\begin{proof}

\begin{tikzpicture}
\draw 
(0,0) node[left] {$P$} -- (4,0)
(0,0) -- (0,4)
(2,0) node[below] {$\tau_1(P)$}
(0,2) node[left] {$\tau_2(P)$}
(0.5,2.5) node[above] {$l$}
(1,1.25) node[right] {$k$};

\draw [fill] (0,2) circle [radius=0.03];

\draw [fill] (2,0) circle [radius=0.03];

\draw [dotted]
(0,2.5) node[left] {$P_2$}
-- (1,2.5) node[right] {$\tau(P)$}
--(1,0) node[below] {$P_1$};
\end{tikzpicture}

Let $k$ be the line through $\tau(P)$ parallel to $\ovl{P\tau_2(P)}$. $\ovl{P\tau_2(P)}\#\ovl{P\tau_1(P)}$ and they intersect, therefore $k\#\ovl{P\tau_1(P)}$ and they intersect. Let the unique intersection point of $k$ and $\ovl{P\tau_1(P)}$ be $P_1$. Let $\alpha$ be the unique trace preserving homomorphism sending $\tau_1$ to $\tau_{PP_1}$.

Let $l$ the line through $\tau(P)$ parallel to $\ovl{P\tau_1(P)}$. In a similar way as above, let $P_2$ be the unique intersection point of $l$ and $\ovl{P\tau_2(P)}$ and let $\beta$ be the unique trace preserving homomorphism such that $\tau_2^\beta=\tau_{PP_2}$.

$\tau(P)$ is the unique intersection of the lines $k$ and $l$ and so is the point $\tau_2^\beta(\tau_1^\alpha(P))$, so the two points are equal. A translation is uniquely determined by the image of a point, therefore $\tau=\tau_1^\alpha\tau_2^\beta$.

To see the uniqueness of $\alpha$ (and $\beta$), suppose that $\tau_1^\alpha\tau_2^\beta=\tau_1^\gamma\tau_2^\delta$. Then the point $\tau_1^{\alpha-\gamma}(\tau_2^\beta(P))=\tau_2^\delta(P)$ lies on $\ovl{P\tau_2(P)}$. $\tau_1^{\alpha-\gamma}(\tau_2^\beta(P))$ lies on the line through $\tau_2^\beta(P)$ parallel to $\ovl{P\tau_1(P)}$. The unique point lying on both these lines is $\tau_2^\beta(P)$. Therefore, $\alpha-\gamma=0$, or equivalently $\alpha=\gamma$, and by a symmetric argument $\beta=\delta$.
\end{proof}

\begin{rmk}
Note that an other way of phrasing the conclusion of the theorem is that for any point $A$ of the affine plane there exist unique $\alpha$ and $\beta$ such that $A=\tau_1^\alpha \tau_2^\beta (P)$ (and to prove this version we use the theorem for $\tau=\tau_{PA}$).
\end{rmk}

Notice that given translations $\tau_1$, $\tau_2$ such that there exists a point $Q$ for which $Q$, $\tau_1(Q)$ and $\tau_2(Q)$ are non-collinear, then by Proposition \ref{proptn} for any point $P$ the points $P$, $\tau_1(P)$, $\tau_2(P)$ are non-collinear.

The following lemma is going to be used when introducing coordinates to an affine plane in the next section of this chapter.

\begin{lem} \label{lemoutinv}
Let $P$ be a point of the affine plane $\ca A$ and let $\tau_1$, $\tau_2$ be two translations of the affine plane $\ca A$ such that there exists a point $Q$ for which $Q$, $\tau_1(Q)$, $\tau_2(Q)$ are non-collinear. Let $\alpha$ and $\beta$ be trace preserving homomorphisms, and let $A=\tau_1^{\alpha} \tau_2^{\beta}(P)$. Then, $A\notin \ovl{P\tau_2(P)}$ iff $\alpha$ is invertible.
\end{lem}

\begin{proof}
First notice that by the above comment, $P$, $\tau_1(P)$, $\tau_2(P)$ are non-collinear. Let $A_1=\tau_1^\alpha (P)$ and let $k$ be the line through $A_1$ and parallel to $\ovl{P\tau_2(P)}$ and notice that $A\in k$.

Suppose that $A\notin \ovl{P\tau_2(P)}$.  $A\in k$, therefore the lines $k$ and $\ovl{P\tau_2(P)}$ are apart from each other (and parallel), hence $A_1$ lies outside at least one of $k$ and $\ovl{P\tau_2(P)}$.  $A_1$ lies on $k$, therefore $A_1\notin \ovl{P\tau_2(P)}$ and $A_1\# P$. $A_1=\tau_1^\alpha (P)\# P$, hence $\alpha$ is invertible by Proposition \ref{propinv}.

Suppose that $\alpha$ is invertible. Then,  $A_1=\tau_1^\alpha (P)\# P$ by Proposition \ref{propinv}. The lines $\ovl{P\tau_1(P)}$ and $\ovl{P\tau_2(P)}$ are apart from each other. Given a pair of points that apart from each other and a pair of lines that are apart from each other at least one of the points lies outside from at least one of the two lines. In this case, at least one of $P$ and $A_1$ lies outside from at least one of $\ovl{P\tau_1(P)}$ and $\ovl{P\tau_2(P)}$. $P$ is the intersection of the two lines and $A_1\in \ovl{P\tau_1(P)}$, therefore $A_1\notin \ovl{P\tau_2(P)}$. $A_1\in k$, therefore $k\# \ovl{P\tau_2(P)}$. That combined with $A\in k$ and $k\parallel \ovl{P\tau_2(P)}$ implies that $A\notin \ovl{P\tau_2(P)}$.
\end{proof}

\begin{lem}
Let $P$, $Q$ and $R$ be three non-collinear points, and  let $\alpha$, $\beta$ be trace preserving homomorphisms such that $\tau_{PQ}^\alpha(P)\#P$ and $\tau_{PR}^\beta(P)\#P$. Then $\alpha\cdot \beta=\beta\cdot \alpha$.
\end{lem}

\begin{proof}
By Theorem \ref{thrmexistdil} there exists dilations $\sigma_\alpha$ and $\sigma_\beta$ such that $P$ is their fixed point, $\sigma_\alpha(Q)=Q'$, $\sigma_\beta(R)=R'$, and for any translation $\tau$, $\tau^\alpha=\sigma_\alpha\tau\sigma_\alpha^{-1}$ and $\tau^\beta=\sigma_\beta\tau\sigma_\beta^{-1}$. Hence, to prove that $\alpha\cdot\beta=\beta\cdot \alpha$ it is sufficient to show that $\sigma_\alpha\sigma_\beta=\sigma_\beta\sigma_\alpha$. 

Let $C=\sigma_\beta(Q')$. $R\#Q'$, hence $R'\#C$ and $\ovl{RQ'}\parallel\ovl{R'C}$. $\sigma_\beta$ fixes $P$, therefore it also fixes the line $\ovl{PQ}$. $Q'\in \ovl{PQ}$, therefore $C$ also lies on the line $\ovl{PQ}$. Also, $\ovl{PR}\#\ovl{RQ'}\parallel\ovl{R'C}$, hence $\ovl{PR}\#\ovl{R'C}$. $R'$ is the intersection of the two lines and $C$ is apart from it, therefore $C\notin \ovl{PR}$ and in particular $C$ is apart from $P$.

Let $D=\sigma_\alpha(R')$. Then by similar arguments as above, $Q'\#D$, $\ovl{QR'}\parallel\ovl{Q'D}$, $D\in \ovl{PR}$ and $D\#P$.

Hence, we can apply Pappus' axiom to the points $P$, $Q$, $Q'$, $R$, $R'$, $C$, $D$, to conclude that $\ovl{QR}\parallel\ovl{CD}$.

$\sigma_\beta\sigma_\alpha$ with $P$ as a fixed point, therefore $\sigma_\beta\sigma_\alpha(R)\in \ovl{PR}$ and $\sigma_\beta\sigma_\alpha(R)$ lies on the line $k$ through $C$ parallel to $\ovl{PQ}$. The line $k$ is apart from $\ovl{PR}$, therefore we have uniquely determined $\sigma_\beta\sigma_\alpha(R)$. $C$ also satisfies these conditions, therefore $C=\sigma_\beta\sigma_\alpha(R)$, i.e. $\sigma_\alpha\sigma_\beta(R)=\sigma_\beta\sigma_\alpha(R)$. Hence, the dilatations $\sigma_\beta\sigma_\alpha$ and $\sigma_\alpha\sigma_\beta$ are equal since they agree on the points $P$ and $R$ that are apart from each other.
\end{proof}

\begin{thrm}
Multiplication in $\text{Tp}$ is commutative.
\end{thrm}

\begin{proof}
Let $\alpha$ and $\beta$ be trace preserving homomorphisms.

Pick three points $P$, $Q$ and $R$ that are all apart from each other and such that $R\notin\ovl{PQ}$. Then, $\tau_{PQ}^\alpha(P)\#\tau_{PQ}^{\alpha+1}(P)$, hence $P\#\tau_{PQ}^\alpha(P)$ or $P\#\tau_{PQ}^{\alpha+1}(P)$. Also, $\tau_{PR}^\beta(P)\#\tau_{PR}^{\beta+1}(P)$, therefore $P\#\tau_{PR}^\beta(P)$ or $P\#\tau_{PR}^{\beta+1}(P)$.

If $P\#\tau_{PQ}^\alpha(P)$ and $P\#\tau_{PR}^\beta(P)$, then $\alpha\cdot \beta=\beta \cdot \alpha$ by the above lemma.

If $P\#\tau_{PQ}^{\alpha+1}(P)$ and $P\#\tau_{PR}^\beta(P)$, then
$$\alpha\cdot \beta=(\alpha+1)\cdot \beta -\beta=\beta \cdot (\alpha+1) -\beta=\beta \cdot \alpha.$$
A symmetric argument gives the result when $P\#\tau_{PQ}^\alpha(P)$ and $P\#\tau_{PR}^{\beta+1}(P)$.

If $P\#\tau_{PQ}^{\alpha+1}(P)$ and
$P\#\tau_{PR}^{\beta+1}(P)$, then
$$\alpha \cdot \beta=(\alpha+1)\cdot (\beta+1) -\alpha -\beta=(\beta+1)\cdot (\alpha+1) -\alpha -\beta=\beta\cdot \alpha.$$
Hence, $\alpha\cdot \beta=\beta \cdot \alpha$ in all the cases.
\end{proof}

Putting all the previous results together we have the following:

\begin{thrm}
$\text{Tp}$ is a local ring.
\end{thrm}

\section{Introducing coordinates to an affine plane} \label{secaffintr}

Given an affine plane $\ca A$ and three non-collinear points $O$, $X$ and $Y$, and $\text{Tp}$ the local ring constructed in the previous section, let $\phi$ be the morphism from the points of $\mathbb A(\text{Tp})$ to the points of $\ca A$ which maps the point $(\alpha,\beta)$ of $\mathbb A(\text{Tp})$ to the point
$$\phi(\alpha,\beta)=\tau_{OX}^\alpha \tau_{OY}^\beta.$$
This is a bijection by Theorem \ref{thrmafco}.

By Proposition \ref{proppreaffmor}, $\phi$ extends to a unique isomorphism of affine planes iff:
\begin{enumerate}
\item $\phi$ preserves and reflects $\#$,
\item given three points $A$, $B$, $C$ of $\mathbb A(\text{Tp})$, then $A\#B$ and $C\in \ovl{AB}$ iff $\phi(A)\#\phi(B)$ and $\phi(C)\in \ovl{\phi(A)\phi(B)}$,
\item three points are non-collinear iff their images under $\phi$ are non-collinear,
\item given four points $A$, $B$, $C$, $D$ of $\mathbb A(\text{Tp})$, such that $A\#B$ and $C\#D$, then $\ovl{AB}\parallel \ovl{CD}$ iff $\ovl{\phi(A)\phi(B)}\parallel \ovl{\phi(C)\phi(D)}$.
\end{enumerate}

We shall prove the four above points to show that the above isomorphism extends to an isomorphism of affine planes.

\begin{lem}
$(a_1,a_2)\#(b_1,b_2)$ in $\mathbb A (\text{Tp})$ iff $\phi(a_1,a_2)\#\phi(b_1,b_2)$ in $\ca A$. 
\end{lem}

\begin{proof}
Let $A=\phi(a_1,a_2)$ and $B=\phi(b_1,b_2)$. Suppose that $A\#B$. Then, $A\#\tau_{AB}(A)$, therefore $O\#\tau_{AB}(O)$. $\ovl{OX}$ and $\ovl{OY}$ are apart from each other and they intersect at $O$, therefore $\tau_{AB}(O)$ lies outside at least one of them. Notice that $\tau_{AB}=\tau_{OB}\tau_{AO}=\tau_{OX}^{b_1-a_1}\tau_{OY}^{b_2-a_2}$. $\tau_{AB}(O)=\tau_{OX}^{b_1-a_1}\tau_{OY}^{b_2-a_2}(O)\notin \ovl{OX}$ implies that $b_2-a_2$ is invertible by Lemma \ref{lemoutinv}. Similarly, $\tau_{AB}(O)\notin \ovl{OY}$ implies that $b_1-a_1$ is invertible. In both cases, we conclude that $(a_1,a_2)\#(b_1,b_2)$.

Suppose that $(a_1,a_2)\#(b_1,b_2)$ in $\mathbb A(\text{Tp})$,  then either $a_1-b_1$ is invertible or $a_2-b_2$ is invertible. By symmetry (since the group of translations is abelian), we may assume that $a_1-b_1$ is invertible. Then, $\tau_{OX}^{b_1-a_1}\tau_{OY}^{b_2-a_2}(O)\notin \ovl{OY}$ by Lemma \ref{lemoutinv}. Therefore, $\tau_{AB}(O)\notin \ovl{OY}$ and $\tau_{AB}(O)\# O$. Hence, $\tau_{AB}(A)\#A$ and therefore $B\#A$.
\end{proof}

\begin{lem}
Let $A=\phi(a_1,a_2)$,  $B=\phi(b_1,b_2)$ and $C=\phi(c_1,c_2)$, with $A\#B$ (and $(a_1,a_2)\#(b_1,b_2)$). We claim that $(c_1,c_2)$ lies on the line passing through $(a_1,a_2)$ and $(b_1,b_2)$ iff $C$ lies on $\ovl{AB}$.
\end{lem}

\begin{proof}
By Lemma \ref{lemprojcol} (restricted to the affine plane over $\text{Tp}$), we see that $(c_1,c_2)$ lies on the line passing through $(a_1,a_2)$ and $(b_1,b_2)$ iff there exist $x$ and $y$ in $\text{Tp}$ such that $\mathbf{c}=x\mathbf{a}+y\mathbf{b}$, and $x+y=1$.

$C=\tau_{OX}^{c_1}\tau_{OY}^{c_2}(O)= \tau_{OX}^{xa_1+yb_1}\tau_{OY}^{xa_2+yb_2}(O)=\tau_{OX}^{xa_1}\tau_{OY}^{xa_2}\tau_{OX}^{yb_1}\tau_{OY}^{yb_2}(O)=\tau_{OA}^x \tau_{OB}^y(O)$
Since $y=-x+1$, $C=\tau_{OA}^x\tau_{OB}^{-x}\tau_{OB}(O)$, hence $C=(\tau_{OA}\tau_{BO})^x(B)=\tau_{BA}^x(B)$. $x$ is a trace preserving homomorphism, therefore $C\in \ovl{AB}$.

Conversely,  given that $C\in \ovl{AB}$, there exists a unique $\alpha$ such that $\tau_{BA}^\alpha(B)=C$ by Theorem \ref{thrmtr}.  Therefore, $$C=(\tau_{OA}\tau_{BO})^\alpha(B)=\tau_{OA}^\alpha\tau_{OB}^{1-\alpha}\tau_{BO}^{-1}(B).$$
Hence, $C=\tau_{OA}^\alpha\tau_{OB}^{1-\alpha}(O)$, which means that $\mathbf{c}=\alpha \mathbf{a}+(1-\alpha) \mathbf{b}$, and therefore $(c_1,c_2)$ lies on the line that passes through $(a_1,a_2)$ and $(b_1,b_2)$.
\end{proof}

\begin{lem}
Let $A=\phi(a_1,a_2)$,  $B=\phi(b_1,b_2)$ and $C=\phi(c_1,c_2)$, with $A\#B$ (and $(a_1,a_2)\#(b_1,b_2)$). We claim that $(c_1,c_2)$ lies outside the line passing through $(a_1,a_2)$ and $(b_1,b_2)$ iff $C$ lies outside $\ovl{AB}$.
\end{lem}

\begin{proof}
The lines $\ovl{OX}$ and $\ovl{OY}$ are apart from each other, therefore $\ovl{O\tau_{AB}(O)}$ is apart from at least one of them. Suppose $\ovl{O\tau_{AB}(O)}\#\ovl{OY}$, then $b_1-a_1$ is invertible. Moreover, $O$, $\tau_{AB}(O)$ and $\tau_{OY}(O)$ are non-collinear, therefore $C=\tau_{AB}^\alpha \tau_{OY}^\beta (A)$, for unique $\alpha$, $\beta$ in $\text{Tp}$.

$\tau_{AB}=\tau_{AO}\tau_{OB}=\tau_{OX}^{b_1-a_1}\tau_{OY}^{b_2-a_2}$, hence $C=\tau_{OX}^{\alpha(b_1-a_1)} \tau_{OY}^{\alpha(b_2-a_2)+\beta} (A)$. Notice that $C=\tau_{OX}^{c_1}\tau_{OY}^{c_2}(O)=\tau_{OX}^{c_1-a_1}\tau_{OY}^{c_2-a_2}(A)$, hence
$$\tau_{OX}^{c_1-a_1}\tau_{OY}^{c_2-a_2}(A) = \tau_{OX}^{\alpha(b_1-a_1)} \tau_{OY}^{\alpha(b_2-a_2)+\beta} (A).$$
$b_1-a_1$ is invertible, therefore we have the following equations for $\alpha$ and $\beta$:
\begin{displaymath}
\begin{split}
\alpha & =(c_1-a_1)(b_1-a_1)^{-1}, \\
\beta & =-\alpha(b_2-a_2)+(c_2-a_2).
\end{split}
\end{displaymath}

By Lemma \ref{lemoutinv}, $C$ lies outside $\ovl{AB}$ iff $\beta$ is invertible, hence iff $(b_1-a_1)\beta$ is invertible. Observe that 
$$(b_1-a_1)\beta=-(c_1-a_1)(b_2-a_2)+(b_1-a_1)(c_2-a_2)=\det \begin{pmatrix} a_1 & b_1 & c_1
\\ a_2 & b_2 & c_2
\\ 1 & 1 & 1  \end{pmatrix}.$$
Thus $C\notin \ovl{AB}$ iff $(c_1,c_2)$ lies outside the line through $(a_0,a_1)$ and $(b_0,b_1)$.

The case where $\ovl{O\tau_{AB}(O)}\#\ovl{OX}$ is symmetric.
\end{proof}

\begin{lem}
Let $A=\phi(a_1,a_2)$,  $B=\phi(b_1,b_2)$, $C=\phi(c_1,c_2)$ and $D=\phi(d_1,d_2)$ such that $A\#B$ and $C\#D$. Then $\ovl{AB}\parallel \ovl{CD}$ iff $\ovl{(a_1,a_2)(b_1,b_2)}\parallel \ovl{(c_1,c_2)(d_1,d_2)}$.
\end{lem}

\begin{proof}
Let $A=\phi(a_1,a_2)$,  $B=\phi(b_1,b_2)$, $C=\phi(c_1,c_2)$ and $D=\phi(d_1,d_2)$ such that $A\#B$ and $C\#D$ then $\ovl{AB}\parallel \ovl{CD}$ iff there exists $\alpha$ in $\text{Tp}$ such that $\tau_{CD}=\tau_{AB}^\alpha$. In a similar way as before we rewrite the translations to see that $\tau_{CD}=\tau_{AB}^\alpha$ iff $\phi(d_1-c_1,d_2-c_2)=\phi(\alpha(b_1-a_1),\alpha(b_2-a_2))$. Hence, $\ovl{AB}\parallel \ovl{CD}$ iff there exists $\alpha$ in $\text{Tp}$ such that $(d_1-c_1,d_2-c_2)=(\alpha(b_1-a_1),\alpha(b_2-a_2))$ or equivalently when $\det \begin{pmatrix} d_1 - c_1 & b_1 - a_1
\\ d_2-c_2 & b_2 -a_2
\end{pmatrix}=0$, i.e. iff the corresponding lines in $\mathbb A(\text{Tp})$ are parallel.
\end{proof}

We are now ready to prove the following theorem:

\begin{thrm} \label{thrmafcoo}
Given an affine plane $\ca A$ with three points $O$, $X$ and $Y$ that are non-collinear, and $\text{Tp}$ the local ring constructed in the previous section, then there is an isomorphism of affine planes from $\mathbb A(\text{Tp})$ to $\ca A$ which sends $(0,0)$ to $O$, $(1,0)$ to $X$ and $(0,1)$ to $Y$.

Moreover, given two such isomorphisms $\phi_1, \phi_2: \mathbb A(\text{Tp})\rightrightarrows \ca A$, there exists a unique ring automorphism $f$ of $\text{Tp}$ such that $\phi_2=\phi_1\circ \mathbb A(f)$.
\end{thrm}

\begin{proof}
In the beginning of this section, we described an isomorphism $\phi$ from the points of $\mathbb A(\text{Tp})$ to the points of $\ca A$. By the above lemmas and Proposition \ref{proppreaffmor}, $\phi$ extends to a unique isomorphism of affine planes $\mathbb A(\text{Tp})\to \ca A$.

Given isomorphisms $\phi_1, \phi_2: \mathbb A(\text{Tp})\rightrightarrows \ca A$ as above, then $\phi_1^{-1}\phi_2: \mathbb A(\text{Tp}) \to \mathbb A(\text{Tp})$ is an isomorphism of affine planes which maps $(1,0)$, $(0,1)$ and $(0,0)$ to $(1,0)$, $(0,1)$ and $(0,0)$ respectively. Hence, by Lemma \ref{lemringpartaf} there exists a unique ring automorphism $f$ of $\text{Tp}$ such that $\mathbb A(f)=\phi_1^{-1}\phi_2$ or equivalently such that $\phi_1\circ \mathbb A(k)=\phi_2$.
\end{proof}

\begin{rmk}
Notice that the construction of the isomorphism of affine planes induced by $\phi$ only depends on the three non-collinear points $O$, $X$ and $Y$. Hence, even though the above theorem does not give a unique isomorphism of affine planes, there is a canonical one which we shall denote as $\phi_{XYO}$.
\end{rmk}

\begin{rmk}
Given an affine plane $\ca A$, let $\text{Tp}$ be the local ring of trace preserving homomorphisms of $\ca A$. There exist non-collinear points $O$, $X$, $Y$ in $\ca A$, and therefore by Theorem \ref{thrmafcoo} there exists an affine plane isomorphism from $\mathbb A(\text{Tp})$ to $\ca A$. Hence, $\ca A$ satisfies all the sentences satisfied by the affine plane over $\text{Tp}$. Therefore, the axioms of the theory of affine planes generate all the sentences satisfied by affine planes over local rings.
\end{rmk}

Let $G(\text{Tp})$ be the group of invertible matrices over $\text{Tp}$ of the form $\begin{pmatrix}
\alpha_0 & \beta_0 & \gamma_0 \\
\alpha_1 & \beta_1 & \gamma_1 \\
0 & 0 & 1
\end{pmatrix}$. Notice that an element $g$ of $G(\text{Tp})$ induces an automorphism of the affine plane $\mathbb A(\text{Tp})$ as in Lemma \ref{lemmatrixautoaf} which we shall also denote by $g$. This automorphism on points is left multiplication by the matrix $g$.

\begin{lem}
Let $O$, $X$, $Y$ be three non-collinear points of the affine plane $\ca A$ as above and let $\phi_{XYO}:\mathbb A(\text{Tp})\to \ca A$ be the morphism of affine planes which maps a point $(\alpha, \beta)$ to $\tau_{OX}^{\alpha} \tau_{OY}^{\beta}(O)$.

Let $A$, $B$, $C$ be three non-collinear points such that 
$$\phi^{-1}(A)=(a_0,a_1)\text{, }\phi^{-1}(B)=(b_0,b_1)\text{and }\phi^{-1}(C)=(c_0,c_1)$$
and let 
$\phi_{ABC}:\mathbb A(\text{Tp})\to \ca A$
be the morphism of affine planes which maps a point $(\alpha,\beta)$ to $\tau_{CA}^\alpha \tau_{CB}^\beta(C)$.

Let $g$ be the matrix $\begin{pmatrix}
a_0-c_0 & b_0-c_0 & c_0 \\
a_1-c_1 & b_1-c_1 & c_1 \\
0 & 0 & 1
\end{pmatrix}$. Then, $\phi_{ABC}=\phi_{XYO}\circ g$.
\end{lem}

\begin{proof}
Let $X$, $Y$, $O$, $A$, $B$, $C$ and $g$ be as above. Then, $A=\tau_{OX}^{a_0} \tau_{OY}^{a_1}(O)$, $B=\tau_{OX}^{b_0} \tau_{OY}^{b_1}(O)$ and $C=\tau_{OX}^{c_0} \tau_{OY}^{c_1}(O)$. 

Notice that $\tau_{OX}^{a_0} \tau_{OY}^{a_1}\tau_{OX}^{-c_0} \tau_{OY}^{-c_1}(C)=A$, hence 
$$\tau_{OX}^{a_0-c_0} \tau_{OY}^{a_1-c_1}=\tau_{CA}$$
and similarly 
$$\tau_{OX}^{b_0-c_0} \tau_{OY}^{b_1-c_1}=\tau_{CB}.$$
$$\psi(\alpha,\beta)=\tau_{CA}^\alpha \tau_{CB}^\beta(C)=\tau_{OX}^{\alpha(a_0-c_0)} \tau_{OY}^{\alpha (a_1 - c_1}\tau_{OX}^{\beta (a_0 - c_0)} \tau_{OY}^{\beta (a_1- c_1)}(C).$$
Hence by replacing $\tau_{OX}^{c_0} \tau_{OY}^{c_1}(O)$ for $C$ we see that
$$\psi(\alpha,\beta)=\tau_{OX}^{\alpha(a_0-c_0)+\beta(b_0-c_0)+c_0} \tau_{OY}^{\alpha(a_1-c_1)+\beta(b_1-c_1)+c_1}(O).$$

\begin{displaymath}
\begin{split}
\phi(g\cdot(\alpha,\beta)) & =\phi(\alpha(a_0-c_0)+\beta(b_0-c_0)+c_0,\alpha(a_1-c_1)+\beta(b_1-c_1)+c_1) \\
& = \tau_{OX}^{\alpha(a_0-c_0)+\beta(b_0-c_0)+c_0} \tau_{OY}^{\alpha(a_1-c_1)+\beta(b_1-c_1)+c_1}(O).
\end{split}
\end{displaymath}
Thus, $\phi_{ABC}=\phi_{XYO} \circ g$.
\end{proof}


Let $\omega$ be the object of triples of non-collinear points of the affine plane $\ca A$. Then, we have the following:

\begin{thrm} \label{thrmgtorsor}
$\omega$ is a right $G(\text{Tp})$-torsor via the action where an element
$\begin{pmatrix}
\alpha_0 & \beta_0 & \gamma_0 \\
\alpha_1 & \beta_1 & \gamma_1 \\
0 & 0 & 1
\end{pmatrix}
$
of  $G$
acts by sending a triple $(A,B,C)$ of non-collinear points to $$(\tau_{CA}^{\alpha_0+\gamma_0}\tau_{CB}^{\alpha_1+\gamma_1} (C), \tau_{CA}^{\beta_0+\gamma_0}\tau_{CB}^{\beta_1+\gamma_1}(C), \tau_{CA}^{\gamma_0}\tau_{CB}^{\gamma_1}(C)).$$
\end{thrm}

In the following proof we write $\begin{pmatrix}
x_0 & y_0 & z_0 \\
x_1 & y_1 & z_1 \\
1 & 1 & 1
\end{pmatrix}$ to denote the triple of non-collinear points $((x_0,x_1), (y_0,y_1), (z_0, z_1))$. 

\begin{proof}
We denote the right action of $g$ of $G$ on the triple $(A,B,C)$ of non-collinear points as $(A,B,C)\cdot g$.

Given $(X,Y,O)$, let $\phi_{XYO}$ be the isomorphism $\mathbb A(R)\to \ca A$ as constructed above. Given $g
=\begin{pmatrix}
\alpha_0 & \beta_0 & \gamma_0 \\
\alpha_1 & \beta_1 & \gamma_1 \\
0 & 0 & 1
\end{pmatrix}$ in $G$. Then, 
\begin{displaymath}
\begin{split}
\phi_{XYO}^{-1}((X,Y,O)\cdot g)
& =\phi_{XYO}^{-1}(\tau_{CA}^{\alpha_0+\gamma_0}\tau_{CB}^{\alpha_1+\gamma_1} (C), \tau_{CA}^{\beta_0+\gamma_0}\tau_{CB}^{\beta_1+\gamma_1}(C), \tau_{CA}^{\gamma_0}\tau_{CB}^{\gamma_1}(C)) \\
& =\begin{pmatrix}
\alpha_0+\gamma_0 & \beta_0 +\gamma_0 & \gamma_0 \\
\alpha_1+\gamma_1 & \beta_1 +\gamma_1 & \gamma_1 \\
1 & 1 & 1
\end{pmatrix} \\
& = g \begin{pmatrix}
1 & 0 & 0 \\
0 & 1 & 0 \\
1 & 1 &1
\end{pmatrix}.
\end{split}
\end{displaymath}
Therefore,
$$(X,Y,O)\cdot g=\phi_{XYO}(g \begin{pmatrix}
1 & 0 & 0 \\
0 & 1 & 0 \\
1 & 1 &1
\end{pmatrix}).$$

Let $(X',Y',O')$ be $(X,Y,O)\cdot g$. Then, $\phi_{X'Y'O'}=\phi\circ g$. Given $g'$ in $G$, then 
\begin{displaymath}
\begin{split}
((XYO)\cdot g)\cdot g' & =\phi_{X'Y'O'}(g' \begin{pmatrix}
1 & 0 & 0 \\
0 & 1 & 0 \\
1 & 1 &1
\end{pmatrix}) \\
& =\phi_{XYO}(gg' 
\begin{pmatrix}
1 & 0 & 0 \\
0 & 1 & 0 \\
1 & 1 &1
\end{pmatrix}) \\
&=(X,Y,O)\cdot (gg').
\end{split}
\end{displaymath}

Hence, the described morphism is a right group action.

Given $(X,Y,O)$ and $(X',Y',O')$ then $(XYO)\cdot g= (X',Y',O')$ iff 
$$\phi_{XYO}(g \begin{pmatrix}
1 & 0 & 0 \\
0 & 1 & 0 \\
1 & 1 &1
\end{pmatrix})=(X',Y',O')$$
or equivalently when
$$g \begin{pmatrix}
1 & 0 & 0 \\
0 & 1 & 0 \\
1 & 1 &1
\end{pmatrix}
= \phi_{XYO}^{-1}(X',Y',O').$$
By Theorem \ref{thrmleftGtors}, there is a unique such $g$ in $G(\text{Tp})$. Hence, this right $G(\text{Tp})$-action is a right $G(\text{Tp})$-torsor.
\end{proof}

\begin{rmk}
In the case where the affine plane is the affine plane over a local ring $R$, the action of an element
$\begin{pmatrix}
\alpha_0 & \beta_0 & \gamma_0 \\
\alpha_1 & \beta_1 & \gamma_1 \\
0 & 0 & 1
\end{pmatrix}$ of $G(R)$ on a triple of non-collinear points is right multiplication by
$$\begin{pmatrix}
1 & 0 & 0 \\
0 & 1 & 0 \\
-1 & -1 & 1
\end{pmatrix}
\begin{pmatrix}
\alpha_0 & \beta_0 & \gamma_0 \\
\alpha_1 & \beta_1 & \gamma_1 \\
0 & 0 & 1
\end{pmatrix}
\begin{pmatrix}
1 & 0 & 0 \\
0 & 1 & 0 \\
1 & 1 & 1
\end{pmatrix}.
$$
\end{rmk}

\section{Trace preserving homomorphisms and geometric morphisms} \label{seclocalpres}

We revisit the construction of a local ring from an affine plane to show that it is preserved by inverse images of geometric morphisms. 

Let us consider an affine plane in a topos $\ca E$ whose object of points is $\text{Pt}$ and whose object of lines is $\text{Li}$.

Let $\text{Tn}$ be the object of translations of this affine plane. The object of translations was defined as a subobject of $\text{Pt}^{\text{Pt}}$. Theorem \ref{thrmexistrans} shows that we have an epimorphism $\tau_{--}:\text{Pt}\times\text{Pt}\to \text{Tn}$ which sends a pair of points $(A,B)$ to the unique translation $\tau_{AB}$ which sends $A$ to $B$.

Let $\text{Tn}_{\#}$ be the object of translations $\tau$ such that $\tau(A)\#A$ for some (or equivalently any) point $A$. Then, $\text{Tn}_{\#}$ as a subobject of $\text{Tn}$ is isomorphic to the image of the composite $\#_{\text{Pt}}\rightarrowtail \text{Pt}\times\text{Pt}\twoheadrightarrow \text{Tn}$. The epimorphism $\#_{\text{Pt}}\to \text{Tn}_{\#}$ induces an equivalence relation on $\#_{\text{Pt}}$, which relates $(A,B)$ to $(A',B')$ iff $\tau_{AB}=\tau_{A'B'}$ or equivalently iff $\tau_{AB}(A')=B'$. 

\begin{lem}
The above equivalence relation is the transitive closure of $\sim_{\text{Tn}\#}$, where $(A,B)\sim_{\text{Tn}\#}(A',B')$ iff $(A'\notin \ovl{AB}) \wedge (B\# B') \wedge (\ovl{AB}\parallel \ovl{A'B'}) \wedge (\ovl{AA'}\parallel \ovl{BB'})$.
\end{lem}

\begin{proof}
Let $(A,B)$ and $(A',B')$ be such that $(A,B)\sim_{\text{Tn}\#}(A',B')$. There is a unique translation $\tau_{AB}$ sending $A$ to $B$. Let $C=\tau_{AB}(A')$. $A\#A'$, therefore $B\#C$ and $\ovl{AA'}\parallel \ovl{BC}$. $A\#B$, therefore $A'\#C$ and $\ovl{AB}\parallel \ovl{A'C}$ by Lemma \ref{lemlinetranslation}. Notice that $\ovl{A'C}\parallel \ovl{AB}\#\ovl{AA'}\parallel \ovl{B'C}$, hence $\ovl{A'C}\# \ovl{B'C}$. Therefore $C=\tau_{AB}(A')$ is uniquely determined as the unique intersection of the line through $A'$ parallel to $\ovl{AB}$ and the line through $B$ parallel to $\ovl{AA'}$. $(A,B)\sim_{\text{Tn}\#}(A',B')$ implies that $B'$ also satisfies this condition hence $B'=C$, and therefore $\tau_{AB}(A')=B'$.

Conversely, suppose $\tau_{AB}(A')=B'$. Let $C\notin\ovl{AB}$ and let $D=\tau_{AB}(C)$. Then, $(A,B)\sim_{\text{Tn}\#}(C,D)$ by the definition of translations. Notice that $\ovl{AB}$ and $\ovl{CD}$ are parallel and apart from each other. $\ovl{A'B'}$ is apart from a least one of the lines $\ovl{AB}$ and $\ovl{CD}$. In the first case, $(A,B)\sim_{\text{Tn}\#}(A',B')$ and in the second case $(C,D)\sim_{\text{Tn}\#}(A',B')$ both by the definition of translations. In both cases, the transitive closure of $\sim_{\text{Tn}\#}$ relates $(A,B)$ to $(A',B')$.
\end{proof}

Thus, the $\text{Tn}_{\#}$ is the quotient of $\#_{\text{Pt}}$ by a relation written in the language of affine planes and therefore it is preserved by inverse images of geometric morphisms.

The epimorphism $\tau_{--}:\text{Pt}\times\text{Pt}\to \text{Tn}$ is also induced by an equivalence relation on $\text{Pt}\times \text{Pt}$, which relates $(A,B)$ to $(A',B')$ iff $\tau_{AB}=\tau_{A'B'}$ or equivalently iff $\tau_{AB}(A')=B'$. 

\begin{lem}
The above relation relates $(A,B)$ and $(A',B')$ iff either $A\#B$, $A'\#B'$ and $\tau_{AB}(A')=B'$ or there exist points $C$, $D$ such that $A\#C\#B$, $A'\#D\#B'$ and $\tau_{AC}(A')=D$, $\tau_{CB}(D)=B'$.
\end{lem}

\begin{proof}
Suppose that $\tau_{AB}(A')=B'$. Then there exists a point $C$ such that $A\#C$ and either $A\#B$ or $C\#B$. In the first case $\tau_{AB}(A')=B'$  and the condition holds. In the second case, let $D=\tau_{AC}(A')$. Then, $\tau_{CB}=\tau_{AB}\circ \tau_{CA}=\tau_{A'B'}\circ\tau_{DA'}$, and therefore $\tau_{CB}(D)=B'$.

Conversely suppose that $(A,B)$ and $(A',B')$ satisfy the condition in the statement of the lemma. If $A\#B$ and $A'\#B'$ then $\tau_{AB}(A')=B'$. Let us suppose that there exist points $C$, $D$ such that $A\#C\#B$, $A'\#D\#B'$ and $\tau_{AC}(A')=D$, $\tau_{CB}(D)=B'$. Then $\tau_{AB}=\tau_{CB}\circ \tau_{AC}=\tau_{DB'} \circ \tau_{A'B'}$ and therefore $\tau_{AB}(A')=B'$ as required.
\end{proof}

Notice that by combining the above two lemmas, we prove that the relation which relates $(A,B)$ and $(A',B')$ when $\tau_{AB}(A')=B'$ is generated by a relation that can be written in the language of affine planes. Hence, $\text{Tn}$ which is isomorphic to the quotient of $\text{Pt}\times \text{Pt}$ by this relation is preserved by inverse images of geometric morphisms.

Let $\Delta:\text{Pt}\to \text{Pt}\times \text{Pt}$ be the diagonal map. Then, the image of the composite $\text{Pt}\xrightarrow{\Delta} \text{Pt}\times \text{Pt}\xrightarrow{\tau_{--}} \text{Tn}$ is isomorphic to $1\to \text{Tn}$ which is the map to the identity translation. Hence the group unit of $\text{Tn}$ is preserved by inverse images of geometric morphisms.

Consider the epimorphism $\text{Pt}\times\text{Pt} \times \text{Pt}\twoheadrightarrow \text{Tn}\times \text{Tn}$ which sends $(A,B,C)$ to $(\tau_{BC},\tau_{AB})$. The map $\text{Pt}\times\text{Pt} \times \text{Pt}\to \text{Pt}\times \text{Pt}$ sending $(A,B,C)$ to $(A,C)$ makes the diagram
\begin{displaymath}
\xymatrix{
\text{Pt}\times\text{Pt} \times \text{Pt} \ar @{->>}[d] \ar[r] & \text{Pt}\times\text{Pt} \ar @{->>}[d]^{\tau_{--}}
\\
\text{Tn}\times \text{Tn} \ar[r]^{-\circ-} & \text{Tn}
}
\end{displaymath}
commute. Hence, the composition of translations is also preserved by inverse images of geometric morphisms, and therefore the group structure of the object of translations is preserved by inverse images of geometric morphisms. Note that the composition morphism $\circ:\text{Tn}\times\text{Tn}\to \text{Tn}$ restricts to an epimorphism $\circ:\text{Tn}_{\#}\times\text{Tn}_{\#}\twoheadrightarrow \text{Tn}$.

Consider the epimorphism $(\tau_{--},\pi_1):\text{Pt} \times \text{Pt}\twoheadrightarrow \text{Tn}\times \text{Pt}$ which sends $(A,B)$ to $(\tau_{AB},A)$. There is a left $\text{Tn}$-action $a_{\text{Pt}}:\text{Tn}\times \text{Pt}\to \text{Pt}$ induced by the monomorphism $\text{Tn}\to \text{Pt}^{\text{Pt}}$. Then the following diagram
\begin{displaymath}
\xymatrix{
\text{Pt}\times\text{Pt} \ar @{->>}[d] \ar[r]^{\pi_2} & \text{Pt} \ar @{->>}[d]^{1_{\text{Pt}}}
\\
\text{Tn}\times \text{Pt} \ar[r]^{a_{\text{Pt}}} & \text{Pt}
}
\end{displaymath}
commutes. Hence the action $a_{\text{Pt}}$ is preserved by inverse images of geometric morphisms.

Let $\text{Tp}$ be the object of trace preserving homomorphisms as defined in Section \ref{secaffring}. $\text{Tp}$ was defined as a subobject of $\text{Tn}^{\text{Tn}}$. By Theorem \ref{thrmtr}, we have an epimorphism $\{(A,B,C):\text{Pt}^3| (A\#B) \wedge (C\in\ovl{AB})\}\to \text{Tp}$ which sends a triple of points $(A,B,C)$ to the unique trace preserving homomorphism sending $\tau_{AB}$ to $\tau_{AC}$ which we denote as $\alpha_{ABC}$. Therefore, $\text{Tp}$ is the quotient of $\{(A,B,C):\text{Pt}^3| (A\#B) \wedge (C\in\ovl{AB})\}$ by the equivalence relation which relates $(A,B,C)$ to $(A',B',C')$ iff $\alpha_{ABC}=\alpha_{A'B'C'}$ or equivalently iff $\tau_{A'B'}^{\alpha_{ABC}}=\tau_{A'C'}$.

\begin{lem} \label{lemsimtp}
The above equivalence relation is generated by the relation $\sim_{\text{Tp}}$, where $(A,B,C)\sim_{\text{Tp}}(A',B',C')$ when $\tau_{A'B'}(A)\notin \ovl{AB}$ and $(\exists l). (l\parallel \ovl{B\tau_{A'B'}(A)})\wedge (C,\tau_{A'C'}(A) \in l)$.
\end{lem}

\begin{proof}
Suppose that $(A,B,C)\sim_{\text{Tp}}(A',B',C')$ and let $Y=\tau_{A'B'}(A)$ and $Z=\tau_{A'C'}(A)$ and notice that $\tau_{A'B'}=\tau_{AY}$, $\tau_{A'C'}=\tau_{AZ}$ and  $Z\in \ovl{AY}$. The above condition says that $Z$ lies on the line $l$ through $C$ parallel to $\ovl{BY}$ and we need to show that $\tau_{AY}^{\alpha_{ABC}}(A)=Z$. $\tau_{AY}^{\alpha_{ABC}}(A)$ lies on $\ovl{AY}$. $\tau_{AY}=\tau_{BY}\circ \tau_{AB}$, therefore $\tau_{AY}^{\alpha_{ABC}}(A)=(\tau_{BY}\circ \tau_{AB})^{\alpha_{ABC}}(A)=\tau_{BY}^{\alpha_{ABC}}\circ \tau_{AC}(A)=\tau_{BY}^{\alpha_{ABC}}(C)$. Therefore $\tau_{AY}^{\alpha_{ABC}}(A)$ lies on the line $l$. Both $\tau_{AY}^{\alpha_{ABC}}(A)$ and $Z$ lie on both of the lines $\ovl{AY}$ and the line $l$. These two lines are apart from each other because they are parallel to the lines $\ovl{AY}$ and $\ovl{BY}$ which are apart from each other and intersect. Hence, $\ovl{AY}$ and $l$ have a unique intersection point and therefore  $\tau_{AY}^{\alpha_{ABC}}(A)=Z$.

Conversely, suppose $A$, $B$, $C$, $A'$, $B'$, $C'$ are points such that $A\#B$, $A'\#B'$, $C\in \ovl{AB}$ and $C'\in \ovl{A'B'}$ and such that $\tau_{A'B'}^{\alpha_{ABC}}=\tau_{A'C'}$.

Let us first consider the case where $\tau_{A'B'}(A)\notin \ovl{AB}$. Let $Y=\tau_{A'B'}(A)$ and $Z=\tau_{A'C'}(A)$. Then, as before $\tau_{A'B'}=\tau_{AY}$, $\tau_{A'C'}=\tau_{AZ}$ and $Z\in \ovl{AY}$. Hence, $Z=\tau_{AY}^{\alpha_{ABC}}(A)=(\tau_{BY}\circ \tau_{AB})^{\alpha_{ABC}}(A)=\tau_{BY}^{\alpha_{ABC}}\circ \tau_{AC}(A)=\tau_{BY}^{\alpha_{ABC}}(C)$ and therefore $Z$ lies on the line $l$ through $C$ parallel to $\ovl{BY}$. Hence, $(A,B,C)\sim_{\text{Tp}}(A',B',C')$.

Let us now consider the general case. There exists point $Y$ lying outside $\ovl{AB}$. Let $Z=\tau_{AY}^{\alpha_{ABC}}(A)$. The point $\tau_{A'B'}(A)$ is apart from $A$, therefore it lies outside at least one of the lines $\ovl{AB}$ and $\ovl{AY}$. In the first case, $(A,B,C)\sim_{\text{Tp}} (A',B',C')$ as we show above. In the second case, $(A,B,C)\sim_{\text{Tp}}(A,Y,Z)$ and $(A,Y,Z)\sim_{\text{Tp}}(A',B',C')$. Hence $(A,B,C)$ is related to $(A',B',C')$ in the transitive closure of the relation $\sim_{\text{Tp}}$.
\end{proof}

Therefore, the object of trace preserving homomorphisms is also preserved by inverse images of geometric morphisms. We now need to show that the ring structure of $\text{Tp}$ is also preserved by inverse images.

Let $\ovl{1}:1\to \text{Tp}$ be the exponential transpose of the identity on translations $1_{\text{Tn}}:\text{Tn}\to \text{Tn}$. Then, the morphism $\ovl{1}$ is isomorphic to the image of the morphism $\{(A,B):\text{Pt}^2|A\#B\}\to \{(A,B,C):\text{Pt}^3| (A\#B) \wedge (C\in\ovl{AB})\}$ which sends $(A,B)$ to $(A,B,B)$. Hence, $\ovl{1}$ is preserved by inverse images of geometric morphisms.

Let $\ovl{0}:1\to \text{Tp}$ be the exponential transpose of $0:\text{Tn}\to \text{Tn}$ as defined earlier, i.e. the map which sends all translations to the identity translation. Then, $\ovl{0}$ is the image of the morphism $\{(A,B):\text{Pt}^2|A\#B\}\to \{(A,B,C):\text{Pt}^3| (A\#B) \wedge (C\in\ovl{AB})\}$ which sends $(A,B)$ to $(A,B,A)$. Hence, $\ovl{0}$ is preserved by inverse images of geometric morphisms.

Let $+:\text{Tp}\times\text{Tp}\to \text{Tp}$ be addition as defined earlier. Consider the epimorphism $\{(A,B,C,D):\text{Pt}^4| (A\#B) \wedge (C,D\in\ovl{AB})\}\twoheadrightarrow \text{Tp}\times \text{Tp}$ which sends $(A,B,C,D)$ to $(\alpha_{ABC},\alpha_{ABD})$. Then, the morphism $$\{(A,B,C,D):\text{Pt}^4| (A\#B) \wedge (C,D\in\ovl{AB})\}\to \{(A,B,C):\text{Pt}^3| (A\#B) \wedge (C\in\ovl{AB})\}$$
which sends $(A,B,C,D)$ to $(A,B,\tau_{AC}(D)$ makes the square
\begin{displaymath}
\xymatrix{
\{(A,B,C,D):\text{Pt}^4| (A\#B) \wedge (C,D\in\ovl{AB})\} \ar @{->>}[d] \ar[r] & \{(A,B,C):\text{Pt}^3| (A\#B) \wedge (C\in\ovl{AB})\} \ar @{->>}[d]
\\
\text{Tp}\times \text{Tp} \ar[r]^+ & \text{Tp}
}
\end{displaymath}
commute. Hence, $+$ is preserved by inverse images of geometric morphisms.

Let $\circ :\text{Tp}\times\text{Tp}\to \text{Tp}$ be composition of trace preserving maps or equivalently multiplication of the ring structure on $\text{Tp}$ as defined in Section \ref{secaffring}. Let $P_5$ be the set of quintuples of points $(A,B,C,B',C')$ such that 
$$(A\#B) \wedge (C\in\ovl{AB})\wedge(A'\#B') \wedge (C'\in\ovl{A'B'})\wedge(B'\notin\ovl{AB}).$$
Consider the epimorphism 
$$P_5 \to \text{Tp}\times\text{Tp}$$
which sends $(A,B,C,B',C')$ to $(\alpha_{ABC},\alpha_{AB'C'})$. Now consider the morphism
$$P_5\to \{(A,B,C):\text{Pt}^3| (A\#B) \wedge (C\in\ovl{AB})\}$$
which sends $(A,B,C,B',C')$ to $(A,B',D)$ where $D$ is defined in the following way: From the above conditions we conclude that $B\#C'$. Let $l$ be the line through $C$ parallel to $\ovl{BC'}$. The lines $l$ and $\ovl{AB'}$ are apart from each other and they intersect and we define $D$ to be their intersection as in the picture:
\begin{center}
\begin{tikzpicture}
\draw
(0,0) node[left] {$A$} -- (5,3)
(0,0)--(5,-2);

\draw
(1,0.6) node[above] {$B$}--(2,-0.8)  node[below] {$C'$}
(2,1.2)  node[above] {$C$} --(4,-1.6) node[below] {$D$}
(3.1,-0.2) node[right] {$l$};

\draw
(2,1.2)--(1,-0.4) node[below] {$B'$};
\end{tikzpicture}
\end{center}
Hence, we have the following commutative diagram
\begin{displaymath}
\xymatrix{
P_5 \ar @{->>}[d] \ar[r] & \{(A,B,C):\text{Pt}^3| (A\#B) \wedge (C\in\ovl{AB})\}\ar @{->>}[d] \\
\text{Tp}\times \text{Tp} \ar[r]^{-\circ-} & \text{Tp}
}
\end{displaymath}
and therefore the multiplication map on $\text{Tp}$ is preserved by inverse images of geometric morphisms.

From the above results we conclude the following:

\begin{thrm}
The local ring of trace preserving homomorphisms is preserved by inverse images of geometric morphisms.
\end{thrm}

\begin{rmk}
$\text{Tp}$ acts on $\text{Tn}$ via the exponential transpose of the monomorphism $\text{Tp}\to \text{Tn}^{\text{Tn}}$. This action restricts to a morphism $\text{Tp}\times \text{Tn}_{\#}\to \text{Tn}$. Consider the epimorphism
$$\{(A,B,C):\text{Pt}^3| (A\#B) \wedge (C\in\ovl{AB})\}\twoheadrightarrow \text{Tp}\times \text{Tn}_{\#}$$
which sends $(A,B,C)$ to $(\alpha_{ABC}, \tau_{AB})$ and the morphism 
$$\{(A,B,C):\text{Pt}^3| (A\#B) \wedge (C\in\ovl{AB})\}\to \text{Tn}$$
which sends $(A,B,C)$ to $\tau_{AC}$. Then the following diagram 
\begin{displaymath}
\xymatrix{
\{(A,B,C):\text{Pt}^3| (A\#B) \wedge (C\in\ovl{AB})\}\ar @{->>}[d] \ar @{->>}[dr] \\
\text{Tp}\times \text{Tn}_{\#} \ar[r] & \text{Tn}
}
\end{displaymath}
commutes, where the bottom arrow is the action of $\text{Tp}$ on $\text{Tn}_{\#}$. Consider the composition of translations which restricts to an epimorphism $\text{Tn}_{\#}\times \text{Tn}_{\#}\twoheadrightarrow \text{Tn}$, hence we have the following commutative diagram:
\begin{displaymath}
\xymatrix{
\text{Tp}\times \text{Tn}_{\#}\times \text{Tn}_{\#} \ar @{->>}[d] \ar[r] & \text{Tn}_{\#}\times \text{Tn}_{\#} \ar @{->>}[d] \\
\text{Tp} \times \text{Tn} \ar[r] & \text{Tn},
}
\end{displaymath}
where the top arrow maps $(\alpha,\tau_1,\tau_2)$ to $(\tau_1^{\alpha},\tau_2^{\alpha})$, the left arrow maps $(\alpha,\tau_1,\tau_2)$ to $(\alpha,\tau_1\tau_2)$ and the bottom arrow is the action of $\text{Tp}$ on $\text{Tn}$. Therefore, the action of $\text{Tp}$ on $\text{Tn}$ is also preserved by inverse images of geometric morphisms.
\end{rmk}


\section{Alternative construction of the local ring}

In \cite[Chapter 3]{Seidenberg}, given an affine plane (in the classical sense) $\ca A$ and three non-collinear points, a field is constructed. This construction can be modified to give a local ring from a constructive affine plane with three non-collinear points $X$, $Y$, $O$. The local ring $R_{XYO}$ is constructed in the following way. The underlying set of $R_{XYO}$ is the set of points $\{A\in \ovl{OX}\}$. The ring operations are defined via geometric constructions. Moreover, there is a canonical isomorphism of affine planes $\rho_{XYO}: \mathbb A(R_{XYO})\to \ca A$, sending $(1,0)$, $(0,1)$ and $(0,0)$ to $X$, $Y$ and $O$ respectively.


Let $\omega$ be the object of triples of non-collinear points of the affine plane $\ca A$.

\begin{lem} \label{lemafext}
Given an affine plane $\ca A$, let $F$ be a functor from the discrete category $\omega$ to the category of local rings. Suppose that for each $(X,Y,O)$ in $\omega$, we have a choice of an isomorphism of affine planes $\sigma_{XYO}: \mathbb A(F(X,Y,O))\to \ca A$ which maps the points $(1,0)$, $(0,1)$ and $(0,0)$ to $X$, $Y$ and $O$ respectively. Then, $F$ can be canonically extended to a functor from $\text{Ind}(\omega)$ to the category of local rings, where $\text{Ind}(\omega)$ is the total preorder on $\omega$ (i.e. the category with object of objects $\omega$ and a unique morphism between any two objects).
\end{lem}

\begin{proof}
Let $F$ be as in the statement of the lemma. We shall extend $F$ to a functor $\ovl{F}$ whose domain is $\text{Ind}(\omega)$

Let $(X,Y,O)$ and $(X',Y',O')$ be in $\omega$. Let $g$ be the unique element of the group $G(F(X,Y,O))$ which sends $(1,0)$, $(0,1)$, $(0,0)$ to $\sigma_{XYO}^{-1}(X')$, $\sigma_{XYO}^{-1}(Y')$, $\sigma_{XYO}^{-1}(O')$ respectively. Then, the composite
$$\sigma_{X'Y'O'}^{-1} \circ \sigma_{XYO} \circ g:\mathbb A(F(X,Y,O)\to \mathbb A(F(X',Y',O'))$$
sends $(1,0)$, $(0,1)$, $(0,0)$ to $(1,0)$, $(0,1)$, $(0,0)$ respectively. Hence, by Lemma \ref{lemringpartaf} there is a unique ring isomorphism 
$$f:F(X,Y,O)\to F(X',Y',O')$$
such that $\mathbb A(f)=\sigma_{X'Y'O'}^{-1} \circ \sigma_{XYO} \circ g$. We define $\ovl{F}$ of the unique morphism from $(X,Y,O)$ to $(X',Y',O')$ to be the ring isomorphism $f$. Notice that $f$ is the identity when $(X,Y,O)=(X',Y',O')$.

Given a third triple of non-collinear points $(O'',X'',Y'')$ in $\omega$, let $g'$ be the unique element of $G(F(X',Y',O'))$ which sends $(1,0)$, $(0,1)$, $(0,0)$ to $\sigma_{X'Y'O'}^{-1}(X'')$, $\sigma_{X'Y'O'}^{-1}(Y'')$, $\sigma_{X'Y'O'}^{-1}(O'')$ respectively. Let $f':F(X',Y',O')\to F(X'',Y'',O'')$ be the unique ring automorphism such that $\mathbb A(f')= \sigma_{X''Y''O''}^{-1} \circ \sigma_{X'Y'O'} \circ g'$. Let $f^{-1}(g')$ be the matrix we get when we apply $f^{-1}$ to each of the components of matrix $g'$. Then, the following diagram:
\begin{displaymath}
\xymatrix{
\mathbb A(F(X,Y,O)) \ar[r]^{f^{-1}(g')} \ar[dr]^{\mathbb A(f)} & 
\mathbb A(F(X,Y,O)) \ar[r]^g \ar[dr]^{\mathbb A(f)} &
\mathbb A(F(X,Y,O)) \ar[rr]^-{\sigma_{XYO}} &
&
\ca A \ar[d]^{1_{\ca A}} \\
&
\mathbb A(F(X',Y',O')) \ar[r]^{g'} \ar[dr]^{\mathbb A(f')} & 
\mathbb A(F(X',Y',O')) \ar[rr]^-{\sigma_{X'Y'O'}} &
&
\ca A \ar[d]^{1_{\ca A}} \\
&
&
\mathbb A(F(X'',Y'',O'')) \ar[rr]^-{\sigma_{X''Y''O''}} &
&
\ca A }
\end{displaymath}
commutes. Notice that $g\circ f^{-1}(g')$ is the unique element of $G(F(X,Y,O))$ which sends $(1,0)$, $(0,1)$ and $(0,0)$ to $\sigma_{XYO}^{-1}(X'')$, $\sigma_{XYO}^{-1}(Y'')$ and $\sigma_{XYO}^{-1}(O'')$ respectively.  $f'f:F(X,Y,O)\to F(X'',Y'',O'')$ is the unique ring isomorphism for which $\mathbb A(f'f)=\sigma_{X''Y''O''}^{-1} \circ \sigma_{XYO} \circ (g\circ f^{-1}(g'))$. Therefore, $f'f:F(X,Y,O)\to F(X'',Y'',O'')$ is the image under $\ovl{F}$ of the unique morphism from $(X,Y,O)$ to $(X'',Y'',O'')$. Hence, $\ovl{F}$ respects composition and therefore is a functor.
\end{proof}

Let $F$ be as in the statement of the above lemma, and let $\ovl{F}$ be its extension to a functor with domain $\text{Ind}(\omega)$. Then, both $F$ and the extension of $F$ are naturally isomorphic to the constant functor to the colimit of $\ovl{F}$. Hence, we have constructed a local ring $R_F$ without a choice of an element of $\omega$.

For each triple $(X,Y,O)$ in $\omega$, there exists a ring isomorphism
$$f_{XYO}: R_F\to F(X,Y,O)$$
which induces an isomorphism of affine planes 
$$\mathbb A(f_{XYO}): \mathbb A(R_F) \to \mathbb A(F(X,Y,O)).$$
Hence, we have an isomorphism of affine planes 
$$\ovl{\sigma}_{XYO}:\mathbb A(R_F)\xrightarrow{\mathbb A(f_{XYO})}\mathbb A(F(X,Y,O))\xrightarrow{\sigma_{XYO}} \ca A.$$
Given a second triple $(X',Y',O')$ in $\omega$ and the isomorphism $\ovl{\sigma}_{X'Y'O'}:\mathbb A(R_F)\to \ca A$ let $g$ be the unique element of $G(R_F)$ which maps the points $(1,0)$, $(0,1)$, $(0,0)$ to the points $\ovl{\sigma}_{X'Y'O'}^{-1}(X'')$, $\ovl{\sigma}_{X'Y'O'}^{-1}(Y'')$, $\ovl{\sigma}_{X'Y'O'}^{-1}(O'')$ respectively. Then, $\sigma_{X'Y'O'}=\sigma_{XYO}\circ g$.

\begin{defn} \label{defnaffring}
Let $R$ be a local ring and suppose that we have an $\omega$-indexed family of affine plane isomorphisms $\mathbb A(R)\to \ca A$:
$$\sigma:\omega \times \mathbb A(R)\to \ca A,$$
$$\ovl{\sigma}:\omega \times \ca A\to \mathbb A(R).$$
Suppose that $\sigma_{XYO}: \mathbb A(R)\to \ca A$ maps the points $(1,0)$, $(0,1)$, $(0,0)$ to the points $X$, $Y$, $O$ respectively. Suppose further that given two triples $(X,Y,O)$ and $(X',Y',O')$ in $\omega$ the  induced isomorphism of affine planes $\ovl{\sigma}_{X'Y'O'}\sigma_{XYO}$ is induced by a (unique) element $g$ of the group of affine transformations $G(R)$. A local ring $R$ with an $\omega$-indexed affine plane isomorphism $\sigma$ satisfying the above properties is called \emph{a coordinate ring of $\ca A$}.
\end{defn}

Notice that the local ring $R_F$ described above is a coordinate ring for $\ca A$ via the isomorphisms $\sigma_-$.

\begin{lem} \label{lemaffunique}
Let $R$ and $R'$ both be coordinate rings of an affine plane $\ca A$ via affine plane isomorphisms $\sigma_-$ and $\sigma'_-$ respectively. Then, $R$ is isomorphic to $R'$.
\end{lem}

\begin{proof}
Pick $(X,Y,O)$ in $\omega$. Then, we have isomorphisms $\sigma_{XYO}: \mathbb A(R)\to \ca A$ and $\sigma'_{XYO}: \mathbb A(R')\to \ca A$. $\sigma'^{-1}_{XYO}\sigma_{XYO}$ maps $(1,0)$, $(0,1)$, $(0,0)$  to $(1,0)$, $(0,1)$, $(0,0)$  respectively. Hence, by Lemma \ref{lemringpart} there exists a unique ring isomorphism $f:R\to R'$ such that $\mathbb A(f)=\sigma'^{-1}_{XYO}\sigma_{XYO}$.

Given a second triple of points $(X',Y',O')$ in $\omega$. Let $g$ be the unique element of $G(R)$ which sends the points $(1,0)$, $(0,1)$, $(0,0)$  to $\sigma^{-1}_{XYO}(A')$, $\sigma^{-1}_{XYO}(B')$, $\sigma^{-1}_{XYO}(O')$ and $\sigma^{-1}_{XYO}(I')$ respectively. Let $f(g)$ be the element of $G(R')$ represented by the matrix we get when we apply $f$ to each of the components of a matrix representing $g$ (and notice that this morphism does not depend on the choice of representative of $g$). Then, $f(g)$ is the unique element of $G(R')$ sending $(1,0)$, $(0,1)$, $(0,0)$  to $\sigma'^{-1}_{XYO}(X')$, $\sigma'^{-1}_{XYO}(Y')$, $\sigma'^{-1}_{XYO}(O')$ respectively. Hence, we have the following commutative diagram
\begin{displaymath}
\xymatrix{
\mathbb A(R) \ar[r]^h \ar[d]^{\mathbb A(f)} & \mathbb A(R) \ar[rr]^{\sigma_{XYO}} \ar[d]^{\mathbb A(f)} && \ca A \ar[d]^{1_{\ca A}} \\
\mathbb A(R') \ar[r]^{f(g)} & \mathbb A(R') \ar[rr]^{\sigma'_{XYO}} && \ca A
}
\end{displaymath}
where the top row composes to $\sigma_{X'Y'O'}$ and bottom row composes to $\sigma'_{X'Y'O'}$. Hence $f$ is the unique ring homomorphism such that $\mathbb A(f)=\sigma'^{-1}_{X'Y'O'}\sigma'_{X'Y'O'}$. Therefore the ring isomorphism $f:R\to R'$ does not depend on the choice of element of $\omega$.
\end{proof}

Given an affine plane $\ca A$, using the construction of \cite[Chapter 3]{Seidenberg} we construct a functor from $\omega$ to the category of local rings. By Lemma \ref{lemafext} we extend this functor to one from $\text{Ind}(\omega)$ and we construct a local ring $R_{\ca A}$ as the colimit of this diagram (since $\omega$ is well-supported). The local ring $R_{\ca A}$ satisfies the properties of the above lemma. Recall the construction of the local ring $\text{Tp}$ from the affine plane $\ca A$ using trace preserving homomorphisms. By the results of Section \ref{secaffintr}, for each triple $(X,Y,O)$ of non-collinear points we have an isomorphism of affine planes $\phi_{XYO}:\mathbb A(\text{Tp}) \to \ca A$  and the local ring $\text{Tp}$ is a coordinate ring of $\ca A$ via the isomorphisms $\phi_-$. Hence, by Lemma \ref{lemaffunique} the local rings $R_{\ca A}$ and $\text{Tp}$ are isomorphic.

The properties required for the local ring in Definition \ref{defnaffring} are preserved under inverse images of geometric morphisms. The local ring of trace preserving homomorphisms satisfies these properties and therefore so does its inverse image under a geometric morphism. By Lemma \ref{lemaffunique} there is a unique such local ring up to isomorphism. Hence, we have proved again that the construction of the local ring of trace preserving homomorphisms of an affine plane is preserved under inverse images of geometric morphisms.


\section{Revisiting Desargues' theorem} \label{secfurdes}

In this section, we prove Theorem \ref{thrmpoly} which is another version of Desargues' theorem on the affine plane. Then, using Theorem \ref{thrmpoly} we prove Theorem \ref{thrmdes5} which was stated in Section \ref{secfurthdes}. We could have presented the following proofs right after Theorem \ref{thrmexistrans} and that is why we used Theorem \ref{thrmdes5} in the proof of Theorem \ref{thrmexistdil}.

\begin{thrm} \label{thrmpoly}
Let $\ca A$ be a preaffine plane satisfying big and small Desargues' axioms. Let $P$, $A_1$, $A_2$,\ldots $A_n$, $P'$, $A'_1$, $A'_2$,\ldots $A'_n$ be points, where $n\geq 3$. If the following are true:
\begin{enumerate}
\item $P$ is apart from all the points $A_1$, $A_2$,\ldots $A_n$,
\item $P'$ is apart from all the points $A'_1$, $A'_2$,\ldots $A'_n$,
\item for each $1\leq i \leq n$, $\ovl{PA_i}\parallel \ovl{P'A'_i}$,
\item \label{itempolygon} $A_i\# A_{i+1}$ for each $1\leq i \leq n-1$, and $A_n\#A_1$,
\item for each $1\leq i \leq n-1$, $\ovl{A_iA_{i+1}}\parallel \ovl{A'_i A'_{i+1}}$.
\end{enumerate}
Then  $\ovl{A_nA_1}$ is parallel to $\ovl{A'_nA'_1}$.
\end{thrm}

Notice that in the presence of the other conditions, condition \ref{itempolygon} above is equivalent to  $\ovl{PA_i}\#\ovl{PA_{i+1}}$ for each $1\leq i \leq n-1$, and $\ovl{PA_n}\#\ovl{PA_1}$. It is also equivalent to  $\ovl{P'A'_i}\#\ovl{P'A'_{i+1}}$ for each $1\leq i \leq n-1$ and $\ovl{P'A'_n}\#\ovl{P'A'_1}$.

\begin{proof}
The proof is by induction on $n$. We first prove the result for $n=3$ and $n=4$.

Given the points and lines in the statement of Theorem \ref{thrmpoly} for $n=3$, let $\tau$ be the translation which sends $P$ to $P'$. $\tau$ is a translation and therefore sends lines to parallel lines therefore it sends the line $\ovl{PA}$ to a parallel line. $\ovl{P'A'_1}$ is the line through $P'=\tau(P)$ and parallel to $\ovl{PA_1}$. Therefore, $\tau$ sends $\ovl{PA_1}$ to $\ovl{P'A'_1}$, and $\tau(A_1)$ lies on $\ovl{P'A'_1}$. Similarly, $\tau(A_2)\in \ovl{P'A'_2}$ and $\tau(A_3)\in \ovl{P'A'_3}$. $\ovl{A'_1A'_2}\parallel \ovl{A_1A_2}\parallel \ovl{\tau(A_1)\tau(A_2)}$, hence $\ovl{A'_1A'_2}\parallel \ovl{\tau(A_1)\tau(A_2)}$. Similarly, $\ovl{A'_2A'_3}\parallel \ovl{\tau(A_2)\tau(A_3)}$. Hence, we have the following picture:
\begin{center}
\begin{tikzpicture}
\draw 
(0,0) node[left] {$P'$} -- (6,3)
(0,0) -- (7,0) 
(0,0) -- (6,-1.5);

\draw 
(3, -0.75) node[ below] {$A'_3$} 
--(5,0) node[ below] {$A'_2$}
(1.5,-0.375)  node[below] {$\tau(A_3)$} 
--(2.5,0) node[ below] {$\tau(A_2)$};

\draw 
(5,0)--(4.5,2.25) node[above] {$A'_1$}
 (2.5,0)--(2.25,1.125) node[above] {$\tau(A_1)$};

\draw[red, thick, densely dotted] 
(4.5,2.25)--(3,-0.75) 
(2.25,1.125)--(1.5,-0.375);
\end{tikzpicture}
\end{center}
The conditions of Theorem \ref{thrmdesbig} are satisfied, hence $\ovl{A'_1A'_3}$ is parallel to $\ovl{\tau(A_1)\tau(A_3)}$. $\tau$ is a translation, therefore $\ovl{\tau(A_1)\tau(A_3)}$ is parallel to $\ovl{A_1A_3}$. Hence, $\ovl{A'_1A'_3}$ is parallel to $\ovl{A_1A_3}$.

Given the points and lines in the statement of Theorem for $n=4$, let $\tau$ be the translation which sends $P$ to $P'$. $\tau$ is a translation and therefore sends lines to parallel lines therefore it sends the line $\ovl{PA}$ to a parallel line. $\ovl{P'A'_1}$ is the line through $P'=\tau(P)$ and parallel to $\ovl{PA_1}$. Therefore, $\tau$ sends $\ovl{PA_1}$ to $\ovl{P'A'_1}$, and $\tau(A_1)$ lies on $\ovl{P'A'_1}$. Similarly, $\tau(A_2)\in \ovl{P'A'_2}$, $\tau(A_3)\in \ovl{P'A'_3}$ and $\tau(A_4)\in \ovl{P'A'_4}$. $\ovl{A'_1A'_2}\parallel \ovl{A_1A_2}\parallel \ovl{\tau(A_1)\tau(A_2)}$, hence $\ovl{A'_1A'_2}\parallel \ovl{\tau(A_1)\tau(A_2)}$. Similarly, $\ovl{A'_2A'_3}\parallel \ovl{\tau(A_2)\tau(A_3)}$ and $\ovl{A'_3A'_4}\parallel\ovl{\tau(A_3)\tau(A_4)}$. Hence, we have the following picture:
\begin{center}
\begin{tikzpicture}
\draw 
(0,0) node[left] {$P'$} -- (7,3.5)
(0,0) -- (7,0) 
(0,0) -- (6,-1.5)  
(0,0) -- (5,6.25) ;

\draw 
(3, -0.75) node[ below] {$A'_3$} 
--(5,0) node[ below] {$A'_2$}
(1.5,-0.375)  node[below] {$\tau(A_3)$} 
--(2.5,0) node[ below] {$\tau(A_2)$};

\draw 
(5,0)--(6,3) node[above] {$A'_1$}
 (2.5,0)--(3,1.5) node[above] {$\tau(A_1)$};

\draw
(3, -0.75)--(4,5) node[above] {$A'_4$}
(1.5,-0.375)--(2,2.5) node[above] {$\tau(A_4)$};

\draw[red, thick, densely dotted] 
(6,3)--(4,5)
(3,1.5)--(2,2.5);
\end{tikzpicture}
\end{center}
The conditions of Theorem \ref{thrmdesbig4} are satisfied, hence $\ovl{A'_1A'_4}$ is parallel to $\ovl{\tau(A_1)\tau(A_4)}$. $\tau$ is a translation, therefore $\ovl{\tau(A_1)\tau(A_4)}$ is parallel to $\ovl{A_1A_4}$. Hence, $\ovl{A'_1A'_4}$ is parallel to $\ovl{A_1A_4}$.

Let us now consider the case where $n>5$. $A_{n-2}\#A_{n-1}$, therefore at least one of $A_{n-2}$ and $A_{n-1}$ is apart from $A_1$.

In the case where $A_{n-2}\#A_1$, by the induction hypothesis $\ovl{A_{n-2}A_1}\parallel \ovl{A'_{n-2}A'_1}$. Hence, we apply the theorem for the case $n=4$ proved above on the points $P$, $A_1$, $A_{n-2}$, $A_{n-1}$, $A_n$, $P'$, $A'_1$, $A'_{n-2}$, $A'_{n-1}$, $A'_n$ to conclude that $\ovl{A_nA_1}\parallel \ovl{A'_nA'_1}$.

In the case where $A_{n-1}\#A_1$, by the induction hypothesis $\ovl{A_{n-1}A_1}\parallel \ovl{A'_{n-1}A'_1}$. Hence, we apply the theorem for the case $n=3$ proved above on the points $P$, $A_1$, $A_{n-1}$, $A_n$, $P'$, $A'_1$, $A'_{n-1}$, $A'_n$ to conclude that $\ovl{A_nA_1}\parallel \ovl{A'_nA'_1}$.
\end{proof}

\noindent {\bf Proof of Theorem \ref{thrmdes5} using Theorem \ref{thrmpoly} for $n=4$:}
\vspace{0.05 in}

\noindent Recall Theorem \ref{thrmdes5}: Let $\ca A$ be a preaffine plane satisfying Desargues' big and small axioms.
Let $k$, $l$, $m$ be lines of $\ca A$ and let $Q$, $A$, $A'$, $B$, $B'$, $C$, $C'$,  $D$, $D'$ be points of $\ca A$. If the following are true:
\begin{enumerate}
\item $Q$ lies on all three lines $k,l,m$,
\item $k\#l\#m$,
\item $Q$ is apart from the points $A$, $A'$, $B$, $B'$, $C$ and $C'$,
\item $A, A'\in k$,
\item $B, B'\in l$,
\item $C, C'\in m$,
\item $D$ lies outside the lines $\ovl{AB}$ and $\ovl{BC}$,
\item $D'$ lies outside the lines $\ovl{A'B'}$ and $\ovl{B'C'}$,
\end{enumerate}
If in addition $\ovl{AB}\parallel\ovl{A'B'}$, $\ovl{BC}\parallel \ovl{B'C'}$, $\ovl{BD}\parallel \ovl{B'D'}$ and $\ovl{CD}\parallel \ovl{C'D'}$, then $\ovl{AD}$ is parallel to $\ovl{A'D'}$.
\begin{center}
 \begin{tikzpicture}
\draw 
(0,0) node[left] {$Q$} -- (7,3.5) node[right] {$k$} 
(0,0) -- (7,0) node[right] {$l$} 
(0,0) -- (6,-1.5)  node[right] {$m$};

\draw 
(3, -0.75) node[ below] {$C'$} 
--(5,0) node[ below] {$B'$}
(1.5,-0.375)  node[below] {$C$} 
--(2.5,0) node[ below] {$B$};

\draw 
(5,0)--(6,3) node[above] {$A'$}
 (2.5,0)--(3,1.5) node[above] {$A$};

\draw
(3, -0.75)--(4,5) node(D')[above] {$D'$}
(1.5,-0.375)--(2,2.5) node[above] {$D$};

\draw 
(5,0)--(4,5)
(2.5,0)--(2,2.5);

\draw[red, thick, densely dotted] 
(6,3)--(4,5)
(3,1.5)--(2,2.5);
\end{tikzpicture}
\end{center}

\begin{proof}
Let $P=B$, $A_1=A$, $A_2=Q$, $A_3=C$, $A_4=D$, $P'=B'$, $A'_1=A'$, $A'_2=Q$, $A'_3=C'$, $A'_4=D'$ as in the following picture.
\begin{center}
\begin{tikzpicture}
\draw 
(0,0) node[left] {$A_2=A'_2$} -- (7,3.5) 
(0,0) -- (7,0) 
(0,0) -- (6,-1.5) ;

\draw 
(3, -0.75) node[ below] {$A'_3$} 
--(5,0) node[ below] {$P'$}
(1.5,-0.375)  node[below] {$A_3$} 
--(2.5,0) node[ below] {$P$};

\draw 
(5,0)--(6,3) node[above] {$A'_1$}
 (2.5,0)--(3,1.5) node[above] {$A_1$};

\draw
(3, -0.75)--(4,5) node[above] {$A'_4$}
(1.5,-0.375)--(2,2.5) node[above] {$A_4$};

\draw 
(5,0)--(4,5)
(2.5,0)--(2,2.5);

\draw[red, thick, densely dotted] 
(6,3)--(4,5)
(3,1.5)--(2,2.5);
\end{tikzpicture}
\end{center}
Apply Theorem \ref{thrmpoly} for $n=4$ to the points $P$, $A_1$, $A_2$, $A_3$, $A_4$, $P'$, $A'_1$, $A'_2$, $A'_3$, $A'_4$. The conditions of Theorem \ref{thrmpoly} hold for this configuration, hence $\ovl{A_4A_1}\parallel \ovl{A'_4A'_1}$. Therefore, $\ovl{AD}\parallel \ovl{A'D'}$.
\end{proof}

\chapter{Introducing coordinates to a Projective Plane} \label{chaprojcoo}

In this chapter, we first construct a local ring from a projective plane $\ca P$ with a choice four points in general position and we then introduce coordinates to the projective plane using these four points. We then show how to construct a local ring without a choice of such points and how any such local ring when it satisfies certain properties is unique up to isomorphism. Given such a local ring $R_{\ca P}$, we also see the construction of a right $H(R_{\ca P})$-torsor. 

\section{The local ring of a projective plane}

Given a projective plane $\ca P$ and a line $l_\infty$, we construct the affine plane $\mathfrak A(\ca P, l_\infty)$ as in Section \ref{secprojtoaff}. We construct  the local ring of trace preserving homomorphisms of the affine plane $\mathfrak A(\ca P, l_\infty)$ as in Section \ref{secaffring}.

\begin{thrm} \label{thrmprojcor1}
Let $\ca P$ be a projective plane, let $l_\infty$ be a line of $\ca P$ and let $O$, $X$ and $Y$ be three non-collinear points of $\ca P$ which lie outside $l_\infty$. Let $\text{Tp}$ be the local ring constructed as above, then there exists an isomorphism of projective planes $\psi: \mathbb P(\text{Tp})\to \ca P$ (from the projective plane over $\text{Tp}$ to $\ca P$) which maps the points $(0,0,1)$, $(1,0,1)$, $(0,1,1)$ and the line $(0,0,1)$ to the points $O$, $X$, $Y$ and the line $l_\infty$ respectively.

Moreover, given two such isomorphisms $\psi_1, \psi_2:\mathbb P(\text{Tp})\to \ca P$, there exists a unique ring automorphism $f$ of $\text{Tp}$ such that $\psi_2=\psi_1 \circ \mathbb P(f)$. 
\end{thrm}

\begin{proof}
By Theorem \ref{thrmafcoo}, there exists a canonical isomorphism $\phi_{XYO}$ from the affine plane over $\text{Tp}$ to the affine plane $\mathfrak A(\ca P, l_\infty)$ which sends the points $(0,0,1)$, $(1,0,1)$, $(0,1,1)$ to the points $O$, $X$, $Y$. By Theorem \ref{thrmaffprojmorph}, we uniquely extend $\phi_{XYO}$ to a morphism $\psi$ from the projective plane over $\text{Tp}$ to the projective plane $\ca P$. Notice that $\psi$ necessarily sends the line $(0,0,1)$ to the line $l_\infty$.

By extending the inverse of $\phi_{XYO}$ to a morphism from $\ca P$ to the projective plane over $\text{Tp}$, we see that $\psi$ is an isomorphism.

Given isomorphisms $\psi_1, \psi_2: \mathbb P(\text{Tp})\rightrightarrows \ca P$ as above, then $\psi_1^{-1}\psi_2: \mathbb P(\text{Tp}) \to \mathbb P(\text{Tp})$ is an isomorphism of projective planes which maps the points $(1,0,0)$, $(0,1,0)$, $(0,0,1)$ and $(1,1,1)$ to $(1,0,0)$, $(0,1,0)$, $(0,0,1)$ and $(1,1,1)$, respectively. Hence, by Lemma \ref{lemringpart} there exists a unique ring automorphism $f$ of $\text{Tp}$ such that $\mathbb P(f)=\psi_1^{-1}\psi_2$ or equivalently such that $\psi_2=\psi_1 \circ \mathbb P(f)$.
\end{proof}

Let $A$, $B$, $O$ and $I$ be points in general position of a projective plane $\ca P$. There exist such four points by Lemma \ref{lemomega4}. Then, $B\notin \ovl{OA}$, therefore $\ovl{IB}\#\ovl{OA}$. Let $X=\ovl{IB}\cap \ovl{OA}$. Similarly $\ovl{IA}$ is apart from $\ovl{OB}$ and we define $Y$ be their intersection point.

\begin{lem}
Let $A$, $B$, $O$, $I$, $X$, $Y$ be points of a projective plane $\ca P$ as above. Then $O$, $X$ and $Y$ are non-collinear and they lie outside $\ovl{AB}$.
\end{lem}

\begin{proof}
$A\notin \ovl{BI}$ and $X\in \ovl{BI}$, therefore $A\#X$. Hence, $\ovl{XA}=\ovl{OA}$. $B\notin \ovl{OA}$, therefore $B\notin \ovl{XA}$ and therefore $X\notin\ovl{AB}$. Similarly, $Y\notin \ovl{AB}$.

$O\notin\ovl{IB}$ and $X\in\ovl{IB}$, hence $O\#X$, and similarly $O\#Y$. $\ovl{OX}=\ovl{OA}\#\ovl{OB}=\ovl{OY}$, hence $\ovl{OX}\#\ovl{OY}$ and the points $O$, $X$, $Y$ are non-collinear.
\end{proof}

\begin{rmk}
In a projective plane over a local ring, let $A=(1,0,0)$, $B=(0,1,0)$, $O=(0,0,1)$, $I=(1,1,1)$. Then, the above construction gives $X=(1,0,1)$ and $Y=(0,1,1)$.
\end{rmk}

Given a projective plane $\ca P$, let $\omega_4$ be the set of quadruples of points in general position of $\ca P$. We can now state Theorem \ref{thrmprojcor1} in the following form:

\begin{thrm} \label{thrmprojcor2}
Given a projective plane $\ca P$, and $(A,B,O,I)$ in $\omega_4$. Then, we can construct a local ring $\text{Tp}_{\ovl{AB}}$ and an isomorphism of projective planes $\psi_{ABOI}:$ $\mathbb P(\text{Tp}_{\ovl{AB}})\to \ca P$, such that $\psi_{ABOI}$ maps the points $(1,0,0)$, $(0,1,0)$, $(0,0,1)$ and $(1,1,1)$ to $A$, $B$, $O$ and $I$ respectively.
\end{thrm}

\begin{proof}
Given $(A,B,O,I)$ in $\omega_4$, construct $X$ and $Y$ as above and apply Theorem \ref{thrmprojcor1} to $\ca P$ with $l_\infty=\ovl{AB}$ and $O$, $X$ and $Y$. Let $\text{Tp}_{\ovl{AB}}$ be the local ring constructed from the affine plane $\mathfrak A(\ca P,\ovl{AB})$, and let $\psi_{ABOI}$ be the canonical isomorphism $\mathbb P(\text{Tp})\to \ca P$ described in the proof of Theorem \ref{thrmprojcor1}. Notice that $\psi_{ABOI}$ sends the points $(1,0,0)$, $(0,1,0)$, $(0,0,1)$ and $(1,1,1)$ to the points $A$, $B$, $O$ and $I$ respectively.
\end{proof}

\begin{rmk}
Given a projective plane $\ca P$, there exists a quadruple of points $(A,B,O,I)$ in $\omega_4$ by Lemma \ref{lemomega4}. Hence, by the above theorem there exists a projective plane isomorphism from $\mathbb P(\text{Tp}_{\ovl{AB}})$ to $\ca P$. Hence, $\ca P$ satisfies all the sentences satisfied by the projective plane over $\text{Tp}_{\ovl{AB}}$. Therefore, the axioms of the theory of projective planes generate all the sentences satisfied by projective planes over local rings.
\end{rmk}

Given a projective plane $\ca P$, Theorem \ref{thrmprojcor2} constructs a functor from the discrete category $\omega_4$ to the category of local rings.

\begin{lem} \label{lemindcat}
Given a projective plane $\ca P$, let $F$ be a functor from the discrete category $\omega_4$ to the category of local rings. Suppose that for each $(A,B,O,I)$ in $\omega_4$, we have a choice of an isomorphism of projective planes $\chi_{ABOI}:\mathbb P(F(A,B,O,I))\to \ca P$ which maps the points $(1,0,0)$, $(0,1,0)$, $(0,0,1)$ and $(1,1,1)$ to $A$, $B$, $O$ and $I$ respectively. Then, $F$ can be extended canonically to a functor from $\text{Ind}(\omega_4)$ to the category of local rings, where $\text{Ind}(\omega_4)$ is the total preorder on $\omega_4$ (i.e. the category with object of objects $\omega_4$ and a unique morphism between any two objects).
\end{lem}

\begin{proof}
Let $F$ be as in the statement of the lemma. We shall extend $F$ to a functor $\ovl{F}$ whose domain is $\text{Ind}(\omega_4)$.

Let $(A,B,O,I)$ and $(A',B',O',I')$ be in $\omega_4$. By Lemma \ref{lemmatrixauto}, there exists a unique $h$ in the projective general linear group $H(F(A,B,O,I))$ such that the corresponding automorphism of $\mathbb P(F(A,B,O,I))$ sends $(1,0,0)$, $(0,1,0)$, $(0,0,1)$ and $(1,1,1)$ to $\chi_{ABOI}^{-1}(A')$, $\chi_{ABOI}^{-1}(B')$, $\chi_{ABOI}^{-1}(O')$ and $\chi_{ABOI}^{-1}(I')$ respectively. Hence, the composite $\chi_{A'B'O'I'}^{-1}\circ \chi_{ABOI} \circ h:\mathbb P(F(A,B,O,I))\to \mathbb P(F(A',B',O',I'))$ maps the points $(1,0,0)$, $(0,1,0)$, $(0,0,1)$, $(1,1,1)$ to the points $(1,0,0)$, $(0,1,0)$, $(0,0,1)$, $(1,1,1)$ respectively. Therefore, by Lemma \ref{lemringpart} there is a unique ring isomorphism $f:F(A,B,O,I)\to F(A',B',O',I')$ such that $\mathbb P(f)= \chi_{A'B'O'I'}^{-1}\circ \chi_{ABOI} \circ h$. We define $\ovl{F}$ of the unique morphism from $(A,B,O,I)$ to $(A',B',O',I')$ to be the ring isomorphism $f$. Evidently, $f$ is the identity when $(A,B,O,I)=(A',B',O',I')$.

Given a third quadruple $(A'',B'',O'',I'')$ in $\omega_4$, let $h'$ be the unique element of $H(R_{O'X'Y'})$ which sends the points $(1,0,0)$, $(0,1,0)$, $(0,0,1)$ and $(1,1,1)$ to the points $\chi_{A'B'O'I'}^{-1}(A'')$, $\chi_{A'B'O'I'}^{-1}(B'')$, $\chi_{A'B'O'I'}^{-1}(O'')$ and $\chi_{A'B'O'I'}^{-1}(I'')$ respectively. Let $g:F(A',B',O',I')\to F(A'',B'',O'',I'')$ be the unique ring automorphism such that $\mathbb P(g)= \chi_{A''B''O''I''}^{-1} \circ \chi_{A'B'O'I'} \circ h'$. Let $f^{-1}(h')$ be the matrix we get when we apply $f^{-1}$ to each of the components of the matrix $h'$. Then, the following diagram:
\begin{displaymath}
\xymatrix{
\mathbb P(F(A,B,O,I)) \ar[r]^{f^{-1}(h')} \ar[dr]^{\mathbb P(f)} & 
\mathbb P(F(A,B,O,I)) \ar[r]^h \ar[dr]^{\mathbb P(f)} &
\mathbb P(F(A,B,O,I)) \ar[rr]^-{\chi_{ABOI}} &
&
\ca P \ar[d]^{1_{\ca P}} \\
&
\mathbb P(F(A',B',O',I')) \ar[r]^{h'} \ar[dr]^{\mathbb P(g)} & 
\mathbb P(F(A',B',O',I')) \ar[rr]^-{\chi_{A'B'O'I'}} &
&
\ca P \ar[d]^{1_{\ca P}} \\
&
&
\mathbb P(F(A'',B'',O'',I'')) \ar[rr]^-{\chi_{A''B''O''I''}} &
&
\ca P }
\end{displaymath}
commutes. Notice that $h\circ f^{-1}(h')$ is the unique element of $H(F(A,B,O,I))$ which sends $(1,0,0)$, $(0,1,0)$, $(0,0,1)$ and $(1,1,1)$ to $\chi_{ABOI}^{-1}(A'')$, $\chi_{ABOI}^{-1}(B'')$, $\chi_{ABOI}^{-1}(O'')$ and $\chi_{ABOI}^{-1}(I'')$ respectively. $gf:F(A,B,O,I)\to F(A'',B'',O'',I'')$ is the unique ring automorphism such that $\mathbb P(gf)=\chi_{A''B''O''I''}^{-1} \circ \chi_{ABOI} \circ (h\circ f^{-1}(h'))$. Therefore, $gf:F(A,B,O,I)\to F(A'',B'',O'',I'')$ is the image under $\ovl{F}$ of the unique morphism $(A,B,O,I)$ to $(A'',B'',O'',I'')$. Hence, $\ovl{F}$ respects composition and therefore it is a functor.
\end{proof}

Let $F$ be a functor as in the statement of the above lemma and let $\ovl{F}$ be its extension to a functor with domain $\text{Ind}(\omega_4)$. Then, both $F$ and the extension of $F$ are naturally isomorphic to the constant functor to the colimit of $\ovl{F}$. Hence, we have constructed a local ring $R_F$ without choosing an element of $\omega_4$.

For each quadruple of points $(A,B,O,I)$ in general position, there exists a ring isomorphism $f_{ABOI}: R_F\to F(A,B,O,I)$ which induces an isomorphism of projective planes $\mathbb P(f_{ABOI}): \mathbb P(R_F) \to \mathbb P(F(A,B,O,I))$. Hence, we have an isomorphism of projective planes $\ovl{\chi}_{ABOI}:\mathbb P(R_F)\xrightarrow{\mathbb P(f_{ABOI})}\mathbb P(F(A,B,O,I))\xrightarrow{\chi_{ABOI}} \ca P$. Given a second quadruple $(A',B',O',I')$ in $\omega_4$ and the isomorphism $\ovl{\chi}_{A'B'O'I'}:\mathbb P(R_F)\to \ca P$ let $h$ be the unique element of $H(R_F)$ which maps the points $(1,0,0)$, $(0,1,0)$, $(0,0,1)$ and $(1,1,1)$ to the points $\ovl{\chi}_{A'B'O'I'}^{-1}(A'')$, $\ovl{\chi}_{A'B'O'I'}^{-1}(B'')$, $\ovl{\chi}_{A'B'O'I'}^{-1}(O'')$ and $\ovl{\chi}_{A'B'O'I'}^{-1}(I'')$ respectively. Then, $\ovl{\chi}_{ABOI}h=\ovl{\chi}_{A'B'O'I'}$.

\section{The uniqueness of the local ring}

\begin{defn} \label{defnprojring}
Let $R$ be a local ring and suppose that we have an $\omega_4$-indexed family of projective plane isomorphisms $\mathbb P(R)\to \ca P$:
$$\psi:\omega_4\times \mathbb P(R)\to \ca P,$$
$$\ovl{\psi}:\omega_4\times \ca P\to \mathbb P(R).$$
Suppose that $\psi_{ABOI}: \mathbb P(R)\to \ca P$ maps the points $(1,0,0)$, $(0,1,0)$, $(0,0,1)$ and $(1,1,1)$ to the points $A$, $B$, $O$ and $I$ respectively. Suppose further that given two quadruples $(A,B,O,I)$ and $(A',B',O',I')$ in $\omega_4$ the  induced isomorphism of projective planes $\ovl{\psi}_{A'B'O'I'}\psi_{ABOI}$ is induced by a (unique) element $h$ of the projective linear group $H(R)$. A local ring $R$ with an $\omega_4$-indexed plane isomorphism $\psi$ satisfying the above properties is called \emph{a coordinate ring of $\ca P$}.
\end{defn}

Notice that the local ring $R_F$ described at the end of the previous section is a coordinate ring of $\ca P$ via the isomorphisms $\ovl{\chi}_-$.

Let us give an alternative description of the last condition of the above definition. Consider the $\omega_4\times \omega_4$-indexed projective plane isomorphism $\mathbb P(R)\to \mathbb P(R)$ given by the morphism
$$\tilde{\psi}:\omega_4\times \omega_4\times \mathbb P(R)\to \mathbb P(R),$$
which maps the triple $(A',B',O',I')$, $(A,B,O,I)$, $X$ to $\ovl{\psi}_{A'B'O'I'}\psi_{ABOI}(X)$. The last condition is equivalent to $\tilde{\psi}$ factoring through the left $H(R)$-indexed isomorphism of projective planes $a_{\mathbb P}:H(R)\times \mathbb P(R) \to \mathbb P(R)$ which is the right action of $H(R)$ on $\mathbb P(R)$, i.e. there exists a morphism $t:\omega_4\times \omega_4\to H(R)$ such that the triangle
\begin{displaymath}
\xymatrix{
\omega_4\times \omega_4\times \mathbb P(R) \ar[r]^-{\tilde{\psi}} \ar[d]_{t\times 1_{\mathbb P(R)}} & \mathbb P(R) \\
H(R)\times \mathbb P(R) \ar[ur]_{a_{\mathbb P}}
}
\end{displaymath}
commutes.

\begin{lem} \label{lemprojunique}
Let $R$ and $R'$ both be coordinate rings for a projective plane $\ca P$ via projective plane isomorphisms $\psi_-$ and $\psi'_-$ respectively. Then, $R$ is isomorphic to $R'$.
\end{lem}

\begin{proof}
Pick $(A,B,O,I)$ in $\omega_4$. Then, we have isomorphisms $\psi_{ABOI}: \mathbb P(R)\to \ca P$ and $\psi'_{ABOI}: \mathbb P(R')\to \ca P$. $\psi'^{-1}_{ABOI}\psi_{ABOI}$ maps $(1,0,0)$, $(0,1,0)$, $(0,0,1)$ and $(1,1,1)$ to $(1,0,0)$, $(0,1,0)$, $(0,0,1)$ and $(1,1,1)$ respectively. Hence, by Lemma \ref{lemringpart} there exists a unique ring isomorphism $f:R\to R'$ such that $\mathbb P(f)=\psi'^{-1}_{ABOI}\psi_{ABOI}$.

Given a second quadruple of points $(A',B',O',I')$ in $\omega_4$. Let $h$ be the unique element of $H(R)$ which sends the points $(1,0,0)$, $(0,1,0)$, $(0,0,1)$ and $(1,1,1)$ to $\psi^{-1}_{ABOI}(A')$, $\psi^{-1}_{ABOI}(B')$, $\psi^{-1}_{ABOI}(O')$ and $\psi^{-1}_{ABOI}(I')$ respectively. Let $f(h)$ be the element of $H(R')$ represented by the matrix we get when we apply $f$ to each of the components of a matrix representing $h$ (and notice that this morphism does not depend on the choice of representative of $h$). Then $f(h)$ is the unique element of $H(R')$ sending $(1,0,0)$, $(0,1,0)$, $(0,0,1)$ and $(1,1,1)$ to $\psi'^{-1}_{ABOI}(A')$, $\psi'^{-1}_{ABOI}(B')$, $\psi'^{-1}_{ABOI}(O')$ and $\psi'^{-1}_{ABOI}(I')$ respectively. Hence we have the following commutative diagram
\begin{displaymath}
\xymatrix{
\mathbb P(R) \ar[r]^h \ar[d]^{\mathbb P(f)} & \mathbb P(R) \ar[rr]^{\psi_{ABOI}} \ar[d]^{\mathbb P(f)} && \ca P \ar[d]^{1_{\ca P}} \\
\mathbb P(R') \ar[r]^{f(h)} & \mathbb P(R') \ar[rr]^{\psi'_{ABOI}} && \ca P
}
\end{displaymath}
where the top row composes to $\psi_{A'B'O'I'}$ and bottom row composes to $\psi'_{A'B'O'I'}$. Hence, $f$ is the unique ring homomorphism such that $\mathbb P(f)=\psi'^{-1}_{A'B'O'I'}\psi'_{A'B'O'I'}$. Therefore the ring isomorphism $f:R\to R'$ does not depend on the choice of element of $\omega_4$.
\end{proof}

\begin{rmk}
Let $\ca P$ be a projective plane in a topos $\ca E$ and let $R$ and $R'$ be local rings in $\ca E$. Then $(\ca P,\pi_2)$  is a projective plane in $\ca E/\omega_4$ and $(R,\pi_2)$ and $(R',\pi_2)$ are local rings in $\ca E/\omega_4$. Suppose that in $\ca E$, for each $(A,B,O,I)$ in $\omega_4$, we have a projective plane isomorphism $\psi_{ABOI}: \mathbb P(R)\to \ca P$ which maps the points $(1,0,0)$, $(0,1,0)$, $(0,0,1)$ and $(1,1,1)$ to the points $A$, $B$, $O$ and $I$ respectively. This is an isomorphism of projective planes in $\ca E/\omega_4$ from $\mathbb P(R,\pi_2)$ to $(\ca P,\pi_2)$. Suppose that we also have such an isomorphism of projective planes from $\mathbb P(R',\pi_2)$ to $(\ca P,\pi_2)$. Then, we have an isomorphism of projective planes from $\mathbb P(R,\pi_2)$ to $\mathbb P(R',\pi_2)$ which maps the points $(1,0,0)$, $(0,1,0)$, $(0,0,1)$ and $(1,1,1)$ to the points $(1,0,0)$, $(0,1,0)$, $(0,0,1)$ and $(1,1,1)$ respectively. Hence, by Lemma \ref{lemringpart} it is induced by a unique ring isomorphism $(R,\pi_2)\to (R',\pi_2)$ in $\ca E/\omega_4$. The final condition of Definition \ref{defnprojring} (which in the notation used there states that $\ovl{\psi}_{A'B'O'I'}\psi_{ABOI}$ is in $H(R)$) makes sure that this ring isomorphism is induced by a ring isomorphism in $\ca E$.
\end{rmk}

\section{The $H$-torsor}

We have an isomorphism from the underlying object of the group $H(R)$ to $\omega_4(R)$ which sends an element $h$ of $H(R)$ to the quadruple of the points $h \begin{pmatrix} 1 \\ 0 \\ 0 \end{pmatrix}$, $ h\begin{pmatrix} 0 \\ 1 \\ 0 \end{pmatrix}$, $h\begin{pmatrix} 0 \\ 0 \\1 \end{pmatrix}$ and $h\begin{pmatrix} 1 \\ 1 \\ 1 \end{pmatrix}$. Let $R$ be a coordinate ring for $\ca P$ via projective isomorphisms $\psi_-$. Consider the morphism $\omega_4\times H(R)\to \omega_4$, sending the pair of $(A,B,O,I)$ in $\omega_4$ and $h$ in $H(R)$ to $\psi_{ABOI}(h)$ (where we view $h$ as an element $\omega_4(R)$).

\begin{lem} \label{lemprojtors}
The morphism $\omega_4\times H(R)\to \omega_4$ described above is a right $H(R)$-torsor.
\end{lem}

\begin{proof}
Let $(A,B,O,I)$ be in $\omega_4$ and let $h$ be in $H(R)$. Then, $\psi_{ABOI}(h)$ is the quadruple of the points $A'=\psi_{ABOI}(h\begin{pmatrix} 1 \\ 0 \\ 0 \end{pmatrix})$,
$B'=\psi_{ABOI}(h\begin{pmatrix} 0 \\ 1 \\ 0 \end{pmatrix})$,
\newline
$O'=\psi_{ABOI}(h\begin{pmatrix} 0 \\ 0 \\1 \end{pmatrix})$ and
$I'=\psi_{ABOI}(h\begin{pmatrix} 1 \\ 1 \\ 1 \end{pmatrix})$. Notice that when $h$ is the identity of the group, then $\psi_{ABOI}(h)=(A,B,O,I)$. Also given $h'$ in $H(R)$, then \begin{displaymath}
\begin{split}
((A,B,O,I)\cdot h)\cdot h' & =(A',B',O',I')\cdot h' \\
& =\psi_{A'B'O'I'}(h') \\
& =\psi_{ABOI}(h(h')) \\
& =\psi_{ABOI}(hh') \\
& =(A,B,O,I)\cdot (hh').
\end{split}
\end{displaymath}
Hence, the described morphism is a right $H(R)$-action on $\omega_4$.

Let $(A,B,O,I)$ and $(A',B',O',I')$ be in $\omega_4$ and let $h$ be in $H(R)$. Then,
\begin{displaymath}
\begin{split}
(A,B,O,I)\cdot h &=(A',B',O',I') \text{ iff} \\
\psi_{ABOI}\circ h &=\psi_{A'B'O'I'}
\end{split}
\end{displaymath}
which is true iff $h$ is the element of $H(R)$ which maps  $(1,0,0)$, $(0,1,0)$, $(0,0,1)$ and $(1,1,1)$ to $\psi_{ABOI}^{-1}(A')$, $\psi_{ABOI}^{-1}(B')$, $\psi_{ABOI}^{-1}(O')$ and $\psi_{ABOI}^{-1}(I')$ respectively. By Lemma \ref{lemmatrixauto} there exists unique $h$ in $H(R)$, hence the described action is a right $H(R)$-torsor.
\end{proof}

\begin{rmk}
Let $R$ and $R'$ both be coordinate rings for a projective plane $\ca P$ via projective plane isomorphisms $\psi_-$ and $\psi'_-$ respectively. The ring isomorphism $R\to R'$ constructed in Lemma \ref{lemprojunique} induces a group isomorphism $H(R)\to H(R')$. Then, the right $H(R')$-action on $\omega_4$ described in the above lemma induces a new right $H(R)$-action on $\omega_4$ via this isomorphism $H(R)\to H(R')$. By the properties of $R$, $R'$, $\psi_-$ and $\psi'_-$ , this new $H(R)$-action is isomorphic to the one constructed in the above lemma.
\end{rmk}

Let us summarize. Let $\ca P$ be a projective plane in a topos $\ca E$. By Theorem \ref{thrmprojcor2}, there exists a functor from the discrete category $\omega_4$ to the category $\mathbf{LocRing}(\ca E)$. By Lemma \ref{lemindcat}, there exists an extension of this functor to a functor from $\text{Ind}(\omega_4)$ to $\mathbf{LocRing}(\ca E)$. $\omega_4$ is well-supported by Lemma \ref{lemomega4}, hence we can define a local ring without a choice of an element of $\omega_4$ by taking the colimit of the underlying internal diagram $\text{Ind}(\omega_4) \to \ca E$. We denote this local ring as $R_{\ca P}$ and we know that it is a coordinate ring for $\ca P$. By Lemma \ref{lemprojunique}, the coordinate ring of $\ca P$ is unique up to a chosen isomorphism. Since these properties which make $R_{\ca P}$ a coordinate ring are preserved by inverse images of geometric morphisms, we see that the construction of the local ring $R_{\ca P}$ is preserved by inverse images of geometric morphisms. Also, by Lemma \ref{lemprojtors}, we have a right $H(R_{\ca P})$-action on $\omega_4$ which makes $\omega_4$ an $H(R_{\ca P})$-torsor.

\part{Classifying toposes}




\chapter{Discussion on Diaconescu's theorem} \label{chaDiac}

Let $\mathbb C$ be an internal category in a topos $\ca S= \set[\mathbb T]$ (the classifying topos of a geometric theory $\mathbb T$). Let $[\mathbb C,\ca S]$ be the topos of internal diagrams in $\ca S$. The goal of this chapter is to give a geometric theory classified by the topos $[\mathbb C, \ca S]$. 

\section{Background and notation}

We consider a coherent theory of \emph{categories} as in \cite{Lawverecategories}. It consists of:
\begin{enumerate}
\item a sort $O$ (objects),
\item a sort $M$ (morphisms),
\item  two functions $\text{dom}, \text{cod}:M\to O$,
\item a function $\text{id}:O\to M$,
\item a ternary relation $T$ on $M$ (to be interpreted as $T(x,y,z)$ when $x$, $y$ are composable and $z$ is the composite "$y\circ x$").
\end{enumerate}

The axioms are:
\begin{enumerate}
\item $\top\vdash_a \text{dom}(\text{id}(a))=a$,
\item $\top\vdash_a \text{cod}(\text{id}(a))=a$,
\item $T(x,y,z)\vdash_{x,y,z} \text{dom}(x)=\text{cod}(y)$,
\item $T(x,y,z)\vdash_{x,y,z} \text{cod}(x)=\text{cod}(z)$,
\item $T(x,y,z)\vdash_{x,y,z} \text{dom}(y)=\text{cod}(z)$,
\item $T(x,y,z)\wedge T(x,y,w) \vdash_{x,y,z,w} (z=w)$,
\item $\text{dom}(x) = \text{cod}(y)\vdash_{x,y}(\exists z)T(x,y,z)$,
\item $\top\vdash_{x} T(x,\text{id}(\text{dom}(x)),x)$,
\item $\top\vdash_{x} T(\text{id}(\text{cod}(x)),x,x)$,
\item $T(x,y,p)\wedge T(y,z,q) \wedge T(p,z,r) \wedge T(x,p,s) \vdash_{x,y,z,p,q,r,s} r=s$.
\end{enumerate}

Notice that in a topos $\ca S$, an internal category is exactly a model of the theory of categories, and an internal functor is exactly a morphism of models.

Let us consider an internal category $\mathbb C$ in a topos $\ca S$ and fix the notation of this chapter. Following the notation from \cite[2.1]{PTJtopos}, $\mathbb C$ consists of the objects:
\begin{itemize}
 \item  $C_0$ as its object of objects,
\item $C_1$ as its object of morphisms,
\item $C_2=C_1\times_{C_0} C_1$ as its object of composable pairs of morphisms (i.e. the following square:
\begin{displaymath}
 \xymatrix{
C_2 \ar[d]^{\pi_1} \ar[r]^-{\pi_2} & C_1 \ar[d]^{d_0} \\
C_1 \ar[r]^{d_1} & C_0
}
\end{displaymath}
is a pullback),
\end{itemize}
and the morphisms:
\begin{itemize}
 \item  $d_0, d_1:C_1\to C_0$ as codomain and domain,
\item $i:C_0\to C_1$ as the inclusion of identities,
\item $m:C_2\to C_1$ as the composite of a composable pair.
\end{itemize}


Given an internal category $\mathbb C$ as above, we obtain the dual category $\mathbb C^{\text{op}}$ by interchanging $d_0$ with $d_1$.

We use the notion of filtered categories as defined in \cite{PTJtopos}.
\begin{defn} \label{defnfiltered}
An internal category $\mathbb C$ in a topos $\ca S$ is \emph{filtered} when the following conditions are satisfied:
\begin{enumerate}
\item $C_0\to 1$ is a cover,
\item the map $(d_0\pi_1,d_0\pi_2):P\to C_0\times C_0$ is a cover, where $P$ is the pullback
\begin{displaymath}
 \xymatrix{
P \ar[d]^{\pi_1} \ar[r]^-{\pi_2} & C_1 \ar[d]^{d_1} \\
C_1 \ar[r]^{d_1} & C_0 ,
}
\end{displaymath}
\item the map $(\pi_1\pi_1, \pi_1\pi_2):T\to R$ is a cover, where $R$ and $T$ are the pullbacks
\begin{displaymath}
 \xymatrix{
R \ar[d]^{\pi_1} \ar[r]^-{\pi_2} & C_1 \ar[d]^{(d_0,d_1)} \\
C_1 \ar[r]^{(d_0,d_1)} & C_0\times C_0
}
\text{and }
\xymatrix{
T \ar[d]^{\pi_1} \ar[r]^-{\pi_2} & C_2 \ar[d]^{(\pi_2,m)} \\
C_2 \ar[r]^-{(\pi_2,m)} & C_1 \times C_1.
}
\end{displaymath}
\end{enumerate}
\end{defn}

The theory of \emph{filtered categories} is a quotient of the theory of categories given above, with the following three extra axioms:
\begin{enumerate}
\item $\top\vdash (\exists a:O)\top$,
\item $\top \vdash_{a,b} (\exists x,y) (\text{dom}(x)=a)\wedge(\text{dom}(y)=b)\wedge(\text{cod}(x)=\text{cod}(y))$,
\item $(\text{dom}(x)=\text{dom}(y))\wedge (\text{cod}(x)=\text{cod}(y))\vdash_{x,y}(\exists z,w) T(x,z,w) \wedge T(y,z,w)$.
\end{enumerate}

It is easily verified that the internal filtered categories correspond to models of the theory of filtered categories via the correspondence between internal categories in a topos and models of the theory of categories.

An internal diagram $(f,\ovl{\phi})$ on $\mathbb C$ consists of a morphism $f:F\to C_0$ in $\ca S$ and an arrow $\ovl{\phi} :C_1\times_{C_0} F\to F$, where
\begin{displaymath}
 \xymatrix{
C_1\times_{C_0} F \ar[d]^{\pi_1} \ar[r]^-{\pi_2} & F \ar[d]^f \\
C_1 \ar[r]^{d_1} & C_0
}
\end{displaymath}
is a pullback, the square
\begin{displaymath}
 \xymatrix{
C_1\times_{C_0} F \ar[d]^{\pi_1} \ar[r]^-{\ovl{\phi}} & F \ar[d]^f \\
C_1 \ar[r]^{d_0} & C_0
}
\end{displaymath}
commutes,
\begin{displaymath}
 \xymatrix{
F \ar[r]^-{(s_0f,1_F)} & C_1\times_{C_0} F \ar[r]^-{\ovl{\phi}} & F
}
\end{displaymath}
is the identity and the two composites
\begin{displaymath}
 \xymatrix{
C_1\times_{C_0} C_1 \times_{C_0}  F \ar@<-.5ex>[r]_-{d_1\times 1} \ar@<.5ex>[r]^-{1\times \ovl{\phi}}& C_1\times_{C_0} F \ar[r]^---{\ovl{\phi}} & F
}
\end{displaymath}
are equal.

An internal diagram on $\mathbb C^{\text{op}}$ is defined in the same way as above but we interchange $d_0$ with $d_1$.

A natural transformation from a $\mathbb C$-diagram $(f:F\to C_0,\ovl{\phi} :C_1\times_{C_0} F\to F)$ to a $\mathbb C$-diagram $(f':F'\to C_0,\ovl{\phi'} :C_1\times_{C_0} F'\to F')$ is a morphism $a:F\to F'$ such that the following two triangles
\begin{displaymath}
\xymatrix{
F \ar[r]^{a} \ar[d]^f & F' \ar[dl]^{f'} \\
C_0 & &, & 
}
\xymatrix{
C_1\times_{C_0} F \ar[r]^{{C_1}\times_{C_0} a} \ar[d]^f & C_1\times_{C_0} F'  \ar[d]^{f'}\\
F \ar[r]^a & F'
}
\end{displaymath}
commute.

The category of $\mathbb C$-diagrams in $\ca S$ and natural transformations of diagrams is denoted by $[\mathbb C,\ca S]$, and it is a topos (see \cite[2.33]{PTJtopos}).

\begin{defn}
An internal functor $f:\mathbb G\to \mathbb H$ in a topos is a \emph{discrete opfibration} when the square
\begin{displaymath}
 \xymatrix{
G_1 \ar[r]^{d_1} \ar[d]^{f_1} & G_0 \ar[d]^{f_0} \\
H_1 \ar[r]^{d_1} & H_0
}
\end{displaymath}
is a pullback. $f$ is a \emph{discrete fibration} when the square
\begin{displaymath}
 \xymatrix{
G_1 \ar[r]^{d_0} \ar[d]^{f_1} & G_0 \ar[d]^{f_0} \\
H_1 \ar[r]^{d_0} & H_0
}
\end{displaymath}
is a pullback.
\end{defn}

We denote the category of internal categories and discrete opfibrations between them as  $\mathbf{doFib}(\ca S)$. Then, for any internal category $\mathbb C$ in $\ca S$, the slice category $\mathbf{doFib}(\ca S)/{\mathbb C}$ is equivalent to $[\mathbb C,\ca S]$ as described in \cite[B2.5.3]{Elephant1}. Let $\mathbf{dFib}(\ca S)$ be the category of internal categories and discrete fibrations between them. Then, the category $\mathbf{dFib}(\ca S)/{\mathbb C}$ is equivalent to $[\mathbb C^{\text{op}},\ca S]$.

\begin{defn}
Given an internal category $\mathbb C$ in a topos $\ca S$, a \emph{$\mathbb C$-torsor} in $\ca S$ is a $\mathbb C^{\text{op}}$-diagram in $\ca S$ which corresponds via the above equivalence to a discrete fibration whose source is a filtered category.
\end{defn}

The category $\mathbf{Tors}(\mathbb C, \ca S)$ is the full subcategory of $[\mathbb C^{\text{op}},\ca S]$ whose objects are $\mathbb C$-torsors. Notice that $\mathbf{Tors}(\mathbb C, \ca S)$ is equivalent to the full subcategory of $\mathbf{dFib}(\ca S)/{\mathbb C}$ whose objects are discrete fibrations $\mathbb F\to \mathbb C$ such that $\mathbb F$ is filtered as an internal category in $\ca S$. We denote this subcategory of $\mathbf{dFib}(\ca S)/{\mathbb C}$ as $\mathbf{fidFib}(\mathbb C, \ca S)$. 

Notice that we have constructed the object part of a pseudofunctor $\mathbf{Cat}(\ca S)\to \mathfrak{Top}$ which sends an internal category $\mathbb C$ to the topos $[\mathbb C,\ca S]$ or equivalently to the topos $\mathbf{dFib}(\ca S)/\mathbb C$. In this chapter, we prefer to use discrete opfibrations instead of internal diagrams, hence we shall express the action of the functor on morphisms of discrete opfibrations. Given an internal functor $F:\mathbb C\to \mathbb D$, the pullback functor $F^*:\mathbf{Cat}(\ca S)/\mathbb D\to \mathbf{Cat}(\ca S)/\mathbb C$  sends discrete opfibrations to discrete opfibrations, hence we have a functor $F^*:\mathbf{dFib}(\ca S)/\mathbb D \to \mathbf{dFib}(\ca S)/\mathbb C$. This functor has both a left and a right adjoint as explained in \cite[B2.5]{Elephant1}. Hence $F^*$ is the inverse image of a geometric morphism. By considering the inverse images of the geometric morphisms it is clear that this assignment is pseudofunctorial, hence we have constructed a pseudofunctor $\mathbf{Cat}(\ca S)\to \mathfrak{Top}$.

There is a unique (internal) functor from $\mathbb C$ to the terminal (internal) category $\mathbf 1$ of $\ca S$. This functor induces a geometric morphism $p_{\mathbb C}: [\mathbb C, \ca S]\to \ca S$ by identifying $[\mathbf 1, \ca S]$ with $\ca S$. Hence, we can view the above pseudofunctor as a pseudofunctor $\mathbf{Cat}(\ca S)\to \mathfrak{Top}/\ca S$.

\section{Diaconescu's theorem}
Diaconescu's theorem as given in \cite[B3.2.7]{Elephant1} states the following:

\begin{thrm} \label{thrmDiac1}
 Let $\ca S$ be a topos, $\mathbb C$ an internal category in $\ca S$ and $f:\ca E \to \ca S$ a geometric morphism. Then there is an equivalence of categories
$$ \mathfrak{Top}/\ca S(\ca E,[\mathbb C,\ca S])\simeq \mathbf{Tors}(f^*(\mathbb C), \ca E),$$
which is natural in $\ca E$, in the sense that, if $g:\ca F\to \ca E$ is a geometric morphism over $\ca S$, then the square
\begin{displaymath}
\xymatrix{
\mathfrak{Top}/\ca S(\ca E, [\mathbb C, \ca S]) \ar[r]^{(-)\circ g} \ar[d]^{\simeq} & \mathfrak{Top}/\ca S(\ca F, [\mathbb C, \ca S]) \ar[d]^{\simeq} \\
\mathbf{Tors}(f^*(\mathbb C), \ca E) \ar[r]^-{g^*} & \mathbf{Tors}(g^*(f^*(\mathbb C)),\ca F)
}
\end{displaymath}
commutes up to coherent natural isomorphism.
\end{thrm}

Let $Y(\mathbb C)$ be the Yoneda profunctor $\mathbb C \looparrowright \mathbb C$ in $\ca S$ viewed as a diagram of shape $\mathbb C^{\text{op}}$ in $[\mathbb C,\ca S]$, as described in \cite[B2.7.2]{Elephant1}. $Y(\mathbb C)$ is a $p^*_{\mathbb C}(\mathbb C)$-torsor in $[\mathbb C,\ca S]$ and the above equivalence sends a geometric morphism $g: \ca E \to [\mathbb C,\ca S]$ over $\ca S$ to the $g^*p_{\mathbb C}^*(\mathbb C)$-torsor $g^*(Y(\mathbb C))$ in $\ca E$.
The $p^*_{\mathbb C}(\mathbb C)$-torsor $Y(\mathbb C)$ corresponds to a discrete fibration. We will denote this discrete fibration by the internal functor $\mathbb F\to p^*_{\mathbb C}(\mathbb C)$ (where $\mathbb F$ is filtered). We will denote $\mathbb F$'s object of objects as $F_0$, its object of morphisms as $F_1$ and the components of the functor $r:\mathbb F \to p^*_{\mathbb C}(\mathbb C)$ as $r_0:F_0\to p^*_{\mathbb C}(C_0)$ and $r_1:F_1\to p^*_{\mathbb C}(C_1)$.

We have mentioned the equivalence $\mathbf{Tors}(f^*(\mathbb C), \ca E)\simeq \mathbf{fidFib}(f^*(\mathbb C), \ca E)$. In this chapter it is more convenient to work with discrete fibrations whose source is a filtered category instead of torsors. Hence, we restate Diaconescu's theorem in the following form:

\begin{thrm} \label{thrmDiac2}
Let $\ca S$ be a topos, $\mathbb C$ an internal category in $\ca S$ and $f:\ca E \to \ca S$ a geometric morphism. Then there is an equivalence of categories
$$ \mathfrak{Top}/\ca S(\ca E,[\mathbb C,\ca S])\simeq \mathbf{fidFib}(f^*(\mathbb C), \ca E),$$
which is natural in $\ca E$, in the sense that, if $g:\ca F\to \ca E$ is a geometric morphism over $\ca S$, then the square
\begin{displaymath}
\xymatrix{
\mathfrak{Top}/\ca S(\ca E, [\mathbb C, \ca S]) \ar[r]^{(-)\circ g} \ar[d]^{\simeq} & \mathfrak{Top}/\ca S(\ca F, [\mathbb C, \ca S]) \ar[d]^{\simeq} \\
\mathbf{fidFib}(f^*(\mathbb C), \ca E) \ar[r]^-{g^*} & \mathbf{fidFib}(g^*(f^*(\mathbb C)),\ca F)
}
\end{displaymath}
commutes up to coherent natural isomorphism.
\end{thrm}

The equivalence in the above form sends a geometric morphism $g: \ca E \to [\mathbb C,\ca S]$ over $\ca S$ to the discrete fibration $g^*(r):g^*(\mathbb F) \to g^*(p^*_{\mathbb C}(\mathbb C))$ where $g^*(\mathbb F)$ is filtered because $\mathbb F$ is filtered.

Let $\ca S$ be a topos, $\mathbb C$ an internal category in $\ca S$ and $f:\ca E \to \ca S$ a geometric morphism. Let
$$\Phi: \mathfrak{Top}/\ca S(\ca E,[\mathbb C,\ca S])\to \mathbf{fidFib}(f^*(\mathbb C), \ca E)$$
be the functor that sends a geometric morphism $g: \ca E \to [\mathbb C,\ca S]$ over $\ca S$ to the discrete fibration $g^*(r):g^*(\mathbb F) \to g^*(p^*_{\mathbb C}(\mathbb C))$. Then, $\Phi$ is one half of an equivalence. An explicit description of the second half of the equivalence is given in the proof of Diaconescu's theorem in \cite{Elephant1} and we denote it by
$$\Psi: \mathbf{fidFib}(f^*(\mathbb C), \ca E)\to \mathfrak{Top}/\ca S(\ca E,[\mathbb C,\ca S]).$$

In both of the above statements of Diaconescu's theorem, we work in $\mathfrak{Top}/S$.
We are also interested in the category $\mathfrak{Top}(\ca E,[\mathbb C,\ca S])$ where $\ca E$ and $[\mathbb C,\ca S]$ are not viewed as $\ca S$-toposes. Theorem \ref{thrmDiac2} already gives an explicit presentation of the objects of this category but not of the morphisms. In particular, we need to consider geometric transformations between pairs of geometric morphisms from $\ca E$ to $[\mathbb C,\ca S]$ which do not necessarily give isomorphic geometric morphisms $\ca E\to \ca S$ when postcomposed with the geometric morphism $p_{\mathbb C}:[\mathbb C,\ca S]\to \ca S$.

For $\mathbb C$ an internal category in a topos $\ca S$ and $\ca E$ a second topos we define the category $\mathbf{FDFIB}(\mathbb C, \ca E)$ in the following way.
\begin{itemize}
 \item An object of $\mathbf{FDFIB}(\mathbb C, \ca E)$ is a pair $(f,s:\mathbb G\to f^*(\mathbb C))$ where $f:\ca E \to \ca S$ is a geometric morphism, $\mathbb G$ is a filtered category in $\ca E$ and $s$ is a discrete fibration in $\ca E$.
\item A morphism from $(f,s:\mathbb G\to f^*(\mathbb C) )$ to $(f',s':\mathbb G'\to f'^*(\mathbb C))$ is a pair $(\alpha,q)$ where $\alpha: f^*\to f'^*$ is a geometric transformation and $q:\mathbb G \to \mathbb G'$ is an internal functor in $\ca E$ such that the square
\begin{displaymath}
 \xymatrix{
\mathbb G \ar[r]^q \ar[d]^s & \mathbb G' \ar[d]^{s'} \\
f^*(\mathbb C) \ar[r]^{\alpha_{\mathbb C}} & f'^*(\mathbb C)
}
\end{displaymath}
commutes where $\alpha_{\mathbb C}: f^*(\mathbb C) \to f'^*(\mathbb C)$ is the internal functor induced by the geometric transformation $\alpha$.
\end{itemize}

A geometric morphism $g:\ca F \to \ca E$ induces a functor 
$$\mathbf{FDFIB}(\mathbb C, \ca E) \to \mathbf{FDFIB}(\mathbb C,\ca F)$$
which sends an object $(f,s:\mathbb G \to f^*(\mathbb C))$ to the object $(fg,g^*(s):g^*(\mathbb G) \to g^*(f^*(\mathbb C)))$ and a morphism $(\alpha,q)$ to $(\alpha\circ g, g^*(q))$. We denote this functor by $g^*$.

In the proof of the following lemma, we consider the topos $[\mathbf 2,\ca E]$ which is the arrow category of $\ca E$. In the 2-category $\mathfrak{Top}$, $[\mathbf 2,\ca E]$ is the cocomma object of the identities on $\ca E$. Hence, it is equipped with a pair of geometric morphisms and a natural transformation as in the following diagram
\begin{displaymath}
 \xymatrix@C=.6in{\ca E \rtwocell^{d_0}_{d_1}{\alpha^{\ca E}} & [\mathbf 2,\ca E]}
\end{displaymath}
where $d_0^*(A\stackrel{h}\to B)=A$, $d_1^*(A\stackrel{h}\to B)=B$ and $\alpha^{\ca E}_h=h:A\to B$.

Given a geometric transformation
\begin{displaymath}
 \xymatrix@C=.6in{\ca E \rtwocell^{f}_{g}{\alpha} & \ca F}
\end{displaymath}
there exists a (unique up to isomorphism) geometric morphism $k:[\mathbf 2,\ca E]\to \ca F$, such that $\alpha$ is isomorphic to the composite
\begin{displaymath}
\xymatrix@C=.6in{\ca E \rtwocell^{d_0}_{d_1}{\alpha^{\ca E}} & [\mathbf 2,\ca E] \ar[r]^k & \ca F .}
\end{displaymath}
In particular, $k$ can be chosen so that $\alpha^{\ca E}\circ k^*=\alpha:f^*\Rightarrow g^*$ by defining $k^*(A)$ to be $\alpha_A:f^*(A)\to g^*(A)$ and in the obvious way on morphisms.

An internal category in $[\mathbf 2,\ca E]$ is an internal functor $r:\mathbb D \to \mathbb E$ between internal categories in $\ca E$. Given a second internal category $r':\mathbb D' \to \mathbb E'$ in $[\mathbf 2,\ca E]$, an internal functor from $r'$ to $r$ is a pair $(s,t)$ of internal functors in $\ca E$ such that the square
\begin{displaymath}
\xymatrix{
\mathbb D' \ar[d]^-s \ar[r]^{r'} & \mathbb E' \ar[d]^-t \\
\mathbb D \ar[r]^-r & \mathbb E
}
\end{displaymath}
commutes. The internal functor $(s,t)$ is a discrete fibration iff both $s$ and $t$ are discrete fibrations. Also, the internal category $r'$ is filtered iff both $\mathbb D'$ and $\mathbb E'$ are filtered.

Let us consider the functor
$$\Phi':\mathfrak{Top}(\ca E,[\mathbb C,\ca S])\to \mathbf{FDFIB}(\mathbb C, \ca E).$$
$\Phi'$ sends a geometric morphism $f:\ca E\to [\mathbb C,\ca S]$ to the pair of the geometric morphism $p_{\mathbb C} \circ f:\ca E\to \ca S$ and the discrete fibration $\Phi(f)=f^*(r): f^*(\mathbb F)\to f^*p_{\mathbb C}^*(\mathbb C)$.
Given a geometric transformation $\alpha$ as in the diagram:
\begin{displaymath}
\xymatrix@C=.6in{\ca E \rtwocell^f_g{\alpha}& [\mathbb C,\ca S]},
\end{displaymath}
 we define $\Phi'(\alpha)$ to be $(p_{\mathbb C} \circ \alpha, \alpha_{\mathbb F})$. This is a morphism in $\mathbf{FDFIB}$ because the square
\begin{displaymath}
 \xymatrix{
f^*(\mathbb F) \ar[d]^{f^*(r)} \ar[r]^{\alpha_{\mathbb F}} & g^*(\mathbb F) \ar[d]^{g^*(r)} \\
f^*p_{\mathbb C}^*(\mathbb C) \ar[r]^{\alpha_{p_{\mathbb C}^*(\mathbb C)}} & g^*p_{\mathbb C}^*(\mathbb C)
}
\end{displaymath}
commutes by the naturality of $\alpha$.

\begin{lem} \label{lemECS}
Let $\ca S$ and $\ca E$ be toposes and $\mathbb C$ an internal category in $\ca S$. Then, the functor
$$\Phi':\mathfrak{Top}(\ca E,[\mathbb C,\ca S])\to \mathbf{FDFIB}(\mathbb C, \ca E)$$
 as defined above is one half of an equivalence of categories.
\end{lem}

\begin{proof}
Consider the functor 
$$\Psi':\mathbf{FDFIB}(\mathbb C, \ca E) \to \mathfrak{Top}(\ca E,[\mathbb C,\ca S]).$$
Given an object $(f:\ca E\to \ca S,s:\mathbb G\to f^*(\mathbb C))$ of $\mathbf{FDFIB}(\mathbb C, \ca E)$, by Diaconescu's theorem the discrete fibration $s$ corresponds to a (unique up to isomorphism) geometric morphism $\Psi(s):\ca E\to [\mathbb C,\ca S]$ in such a way so that $\Phi'(\Psi(s))$ is isomorphic to $(f,s)$, hence we define $\Psi'(f,s)$ to be $\Psi(s)$.

Let $(\alpha,q):(f,s:\mathbb G\to f^*(\mathbb C))\to (g,t:\mathbb H\to g^*(\mathbb C))$ be a morphism in $\mathbf{FDFIB}(\mathbb C, \ca E)$. $[\mathbf 2,\ca E]$ is the cocomma object of the identities on $\ca E$, therefore $\alpha$ induces a (unique up to isomorphism) geometric morphism $k:[\mathbf 2,\ca E]\to \ca S$, such that $\alpha$ is isomorphic to the composite
\begin{displaymath}
\xymatrix@C=.6in{\ca E \rtwocell^{d_0}_{d_1}{\alpha^{\ca E}} & [\mathbf 2,\ca E] \ar[r]^k & \ca S}
\end{displaymath}
and as shown before we choose $k$ so that $\alpha^{\ca E}\circ k^*=\alpha:f^*\Rightarrow g^*$. As mentioned above, an internal category in $[\mathbf 2,\ca E]$ corresponds to an internal functor of internal categories in $\ca E$ and in that notation $k^*(\mathbb C)$ is the functor $\alpha_{\mathbb C}: f^*(\mathbb C)\to g^*(\mathbb C)$. $s:\mathbb G\to f^*(\mathbb C)$ and $t:\mathbb H\to g^*(\mathbb C)$ are discrete fibrations and the square
\begin{displaymath}
\xymatrix{
 \mathbb G \ar[r]^{q} \ar[d]^-s &  \mathbb H \ar[d]^-t \\
f^*(\mathbb C) \ar[r]^{\alpha_{\mathbb C}} &  g^*(\mathbb C)
}
\end{displaymath}
commutes. Hence $(s,t)$ is a discrete fibration over $k^*(\mathbb C)$. The source of the fibration is filtered because both $\mathbb G$ and $\mathbb H$ are filtered. Therefore, by Diaconescu's theorem this discrete fibration corresponds to the geometric morphism $\Psi(s,t):[2,\ca E]\to [\mathbb C,\ca S]$ which is such that $p_{\mathbb C}\circ \Psi(s,t) \cong k$ and such that the discrete fibration $\Phi(\Psi(s,t))=\Psi(s,t)^*(r)$ is isomorphic to the discrete fibration $(s,t)$ in $[\mathbf 2, \ca E]$ described above. Therefore, the discrete fibration $\Phi (\Psi(s,t)\circ d_0)= d_0^*(\Phi(\Psi(s,t)))$ in $\ca E$ is isomorphic to the discrete fibration $d_0^*(s,t)=s$ and the discrete fibration $\Phi (\Psi(s,t)\circ d_1)= d_1^*(\Phi(\Psi(s,t)))$ in $\ca E$ is isomorphic to the discrete fibration $d_1^*(s,t)=t$. By applying $\Psi$ to the two described isomorphisms (and using the fact that $\Psi\Phi$ is naturally isomorphic to the identity) we construct natural isomorphisms $\Psi(s,t)\circ d_0 \cong \Psi(s)$ and $\Psi(s,t)\circ d_1 \cong \Psi(t)$. Hence, we define $\Psi'(\alpha,q):\Psi(s)\to \Psi(t)$ to be the composite of these two isomorphisms with 
\begin{displaymath}
  \xymatrix@C=.6in{\ca E \rtwocell^{d_0}_{d_1}{\alpha^{\ca E}} & [\mathbf 2,\ca E] \ar[r]^{\Psi(s,t)} & [\mathbb C,\ca S]}.
\end{displaymath}

Given an object $(f:\ca E\to \ca S,s:\mathbb G\to f^*(\mathbb C))$ of $\mathbf{FDFIB}(\mathbb C, \ca E)$, then 
\begin{displaymath}
\begin{split}
\Phi'\Psi'(f:\ca E\to \ca S, s:\mathbb G\to f^*(\mathbb C))
& =\Phi'(\Psi(s):\ca E \to [\mathbb C,\ca S]) \\
& =(p_{\mathbb C}\circ (\Psi(s)), \Phi(\Psi(s)))
\cong (f,s).
\end{split}
\end{displaymath}
We show that this isomorphism is natural in $(f,s)$ by spelling out the definitions of $\Phi'$ and $\Psi'$ on morphisms and using the equivalence $(\Phi,\Psi)$. Given a geometric morphism $f:\ca E\to [\mathbb C, \ca S]$, then
\begin{displaymath}
\begin{split}
\Psi'\Phi'(f) & =\Psi'(p_{\mathbb C}\circ f,\Phi(f)) \\
 & =\Psi(\Phi(f))\cong f.
\end{split}
\end{displaymath}
We prove that this isomorphism is natural in $f$ by spelling out the definitions of $\Phi'$ and $\Psi'$ on morphisms and using the equivalence $(\Phi,\Psi)$.
Therefore, $(\Phi',\Psi')$ is an equivalence of categories.
\end{proof}

\section{$[\mathbb C, \ca S]$ as a classifying topos}

Let $\ca E$ be the classifying topos of a geometric theory $\mathbb S$, and let $\mathbb C$ be an internal category in $\ca E$. Then, we can construct a geometric theory $\mathbb S_{\mathbb C}$ which is Morita equivalent to $\mathbb S$ and written in an extension of the signature of $\mathbb S$ which contains sorts, functions and a relation whose interpretation in the generic model of $\mathbb S_{\mathbb C}$ in $\ca E$ is isomorphic to the internal category $\mathbb C$. This can be thought as an extension of Giraud's theorem as presented in \cite[C2.2.8]{Elephant2}.

Let $\ca S$ be the classifying topos for a geometric theory $\mathbb T$ and let $\mathbb C$ be an internal category in $\ca S$. Let us assume that the theory $\mathbb T$ is over a language $\Sigma$ which includes two sorts $O$ and $M$, functions $\text{dom}, \text{cod}:M\to O$ and $\text{id}:O\to M$, and a ternary relation $T$ on $M$. Let us also assume that the interpretation of the above sorts, functions and relation in the generic model of $\mathbb T$ in $\ca S$ gives an internal category in $\ca S$ isomorphic to $\mathbb C$. If the theory $\mathbb T$ does not satisfy these conditions, then by the previous paragraph we can construct a Morita equivalent theory which does. Notice that the axioms of categories for $O$, $M$, $\text{dom}$, $\text{cod}$, $\text{id}$ and $T$ are derivable from $\mathbb T$ because they are satisfied in the generic model of $\mathbb T$.

Let  $\mathbb T'$ be the theory over the language $\Sigma'$ which is  the extension of $\Sigma$ by sorts $O'$ and $M'$, functions $\text{dom}', \text{cod}':M'\to O'$ and $\text{id}':O'\to M'$, a ternary relation $T'$ on $M'$, and functions $f_O:O'\to O$ and $f_M:M'\to M$. The theory $\mathbb T'$ contains the axioms of the theory $\mathbb T$ and the additional following axioms:
\begin{enumerate}
\item Axioms for $(O', M', \text{dom}', \text{cod}', \text{id}', T')$ to be a filtered category.
\item Axioms for $f_O:O'\to O$, $f_M:M'\to M$ to be a functor:
\begin{itemize}
\item $\top\vdash_x f_O(\text{dom}'(x))=\text{dom}(f_M(x))$,
\item $\top\vdash_x f_O(\text{cod}'(x))=\text{cod}(f_M(x))$,
\item $\top\vdash_a f_M(\text{id}'(a))=\text{id}(f_O(a))$ and
\item $T'(x,y,z) \vdash_{x,y,z} T(f_M(x),f_M(y),f_M(z))$.
\end{itemize}
\item Axioms for the above functor to be a discrete fibration:
\begin{itemize}
\item $(\text{cod}'(x)=\text{cod}'(y)) \wedge (f_M(x)=f_M(y)) \vdash_{x,y} x=y$,
\item $f_O(a)=\text{cod}(x) \vdash_{a,x} (\exists x') ((\text{cod}'(x')=a) \wedge (f_M(x')=x))$.
\end{itemize}
\end{enumerate}

\begin{lem} \label{lemTprime}
Given $\mathbb T$, $\mathbb T'$, $\ca S$ and $\mathbb C$  as above and a $\set$-topos $\ca E$, we have an equivalence
$$\mathbf{FDFIB}(\ca E, \mathbb C)\simeq  \mathbb T'-\mathbf{Mod}(\ca E),$$
which is natural in $\ca E$, in the sense that, if $k:\ca F\to \ca E$ is a geometric morphism, then the square
\begin{displaymath}
\xymatrix{
\mathbf{FDFIB}(\mathbb C, \ca E) \ar[r]^{k^*} \ar[d]^{\simeq} &\mathbf{FDFIB}(\mathbb C,\ca F) \ar[d]^{\simeq} \\
\mathbb T'\text{-}\mathbf{Mod}(\ca E) \ar[r]^{k^*} & \mathbb T'\text{-}\mathbf{Mod}(\ca F)
}
\end{displaymath}
commutes up to natural isomorphism.
\end{lem}

\begin{proof}
Let $\ca M$ be the generic model of $\mathbb T$ in $\ca S$.

Given a geometric morphism $g:\ca E\to S$, a filtered category $\mathbb G$ in $\ca E$ and a discrete fibration $s: \mathbb G\to g^*(\mathbb C)$ in $\ca E$, we construct a $\mathbb T'$-model in the following way: $g^*(\ca M)$ gives an interpretation of the sorts, functions and relations that are in $\Sigma$, and moreover the interpretation of $(O, M, \text{dom}, \text{cod}, \text{id}, T)$ is the internal category $f^*(\mathbb C)$. The internal category $\mathbb G$ gives us a way of interpreting $(O', M', \text{dom}', \text{cod}', \text{id}', T')$ and $f_O$  and $f_M$ are interpreted as $s_0:G_0\to g^*(C_0)$ and $s_1:G_1\to g^*(C_1)$. The axioms of $\mathbb T$ are satisfied because $g^*(\ca M)$ satisfies them, and the additional axioms are satisfied because $\mathbb G$ is a filtered category, $s_0$, $s_1$ are the components of the functor $s:\mathbb G\to g^*(\mathbb C)$, and $s$ is a discrete fibration. Hence, let us consider the functor
$$\Phi'':\mathbf{FDFIB}(\ca E, \mathbb C)\to   \mathbb T'\text{-}\mathbf{Mod}(\ca E),$$
which sends an object $(g,s:\mathbb G\to f^*(\mathbb C))$ to the $\mathbb T'$-model described above and which sends a morphism $(\alpha,q):(g,s:\mathbb G\to g^*(\mathbb C) ) \to (g',s':\mathbb G'\to g'^*(\mathbb C))$ to the morphism $\alpha_{\ca M}$ on the sorts that are contained in $\Sigma$. The morphisms on the interpretation of the sorts $O$ and $M$ are $q_0$ and $q_1$ respectively, where $q_0$ and $q_1$ are the components of the functor $q: \mathbb G \to \mathbb G'$. It is easily verified that the described morphisms give a morphism of $\mathbb T'$-models.

Conversely, given a $\mathbb T'$-model $\ca N$ in $\ca E$, by restricting it to the interpretation $\Sigma$ we have a $\mathbb T$-model $\ca N_\Sigma$. $\ca S$ is the classifying topos for $\mathbb T$, therefore the $\mathbb T$-model $\ca N_\Sigma$ corresponds to a geometric morphism $g:\ca E\to \ca S$ so that $g^*(\ca M)$ is isomorphic to $\ca N_\Sigma$. Hence, the interpretation of $(O, M, \text{dom}, \text{cod}, \text{id}, T)$ in $\ca N$ is an internal category $\mathbb D$ isomorphic to the internal category $g^*(\mathbb C)$ of $\ca E$. The interpretation of $(O', M', \text{dom}', \text{cod}', \text{id}', T')$ gives an internal filtered category $\mathbb G$ in $\ca E$. The interpretation of $f_O$, $f_M$ give an internal functor $\mathbb G\to \mathbb D$ which is a discrete fibration and using the isomorphism $\mathbb D\to g^*(\mathbb C)$ we have a discrete fibration $s:\mathbb G\to g^*(\mathbb C)$. Hence, let us consider the functor
$$\Psi'': \mathbb T'\text{-}\mathbf{Mod}(\ca E)\to \mathbf{FDFIB}(\ca E, \mathbb C),$$
which sends a $\mathbb T'$-model $\ca N$ to $(g,s:\mathbb G\to g^*(\mathbb C))$ as described above. Let $\ca N'$ also be a $\mathbb T'$-model and let $\Psi(\ca N')$ be $(g',s':\mathbb G'\to g'^*(\mathbb C'))$. Given a morphism of $\mathbb T'$-models $\ca N \to \ca N'$, it restricts to a morphism of the $\mathbb T$-models $u:\ca N_\Sigma\to \ca N'_\Sigma$, and therefore a geometric transformation $\alpha^u: g\Rightarrow g'$. The morphisms between the interpretations of $\mathbb O'$ and between the interpretations of $\mathbb M'$ give an internal functor $q^u:\mathbb G\to \mathbb G'$, so that $(\alpha^u, q^u): (g,s:\mathbb G\to g^*(\mathbb C))\to (g',s':\mathbb G'\to g'^*(\mathbb C))$ is an arrow in $\mathbf{FDFIB}(\ca E, \mathbb C)$.

It is clear from the above constructions that given an object  $(g,s:\mathbb G\to f^*(\mathbb C)$ of $\mathbf{FDFIB}(\ca E, \mathbb C)$, $\Psi''\circ\Phi''(g,s)$ is isomorphic to $(g,s)$. Also, given $\ca N$ a $\mathbb T'$-model in $\ca E$, then $\Phi''\circ \Psi''(\ca N)$ is isomorphic to $\ca N$. Thus, the equivalence of the lemma holds. Also, the naturality is clear from the above constructions.
\end{proof}

By combining Lemma \ref{lemECS} and Lemma \ref{lemTprime} we conclude the following theorem:
\begin{thrm}
Given $\mathbb T$, $\mathbb T'$, $\ca S$ and $\mathbb C$  as above, then $[\mathbb C,\ca S]$ is the classifying topos for the theory $\mathbb T'$.
\end{thrm}

\chapter{Results about $\ca E[G]$} \label{chaEG}

In this chapter, we present some background material about toposes of the form $\ca E[G]$ where $G$ is a group in $\ca E$ and about geometric morphisms $\ca E[G]\to \ca E[H]$ induced by group homomorphisms $G\to H$ in $\ca E$. The main goal of the chapter is to show that a group monomorphism $G\to H$ in a topos $\ca E$ induces a local homeomorphism $\ca E[G]\to \ca E[H]$.

\section{$\mathbf{Group}(\ca E)\to \mathfrak{Top}$}

Let $\ca E$ be a topos and $G$ a group object in $\ca E$ with group multiplication 
$$m_G:G\times G\to G,$$ unit 
$$\text{id}_G:1\to G,$$ and the inverse morphism 
$$\text{inv}_G:G\to G.$$

\begin{defn}
 A left \emph{$G$-action} on an object $A$ of $\ca E$ is a morphism $a_A:G\times A\to A$  such that the following diagrams:
\begin{displaymath}
 \xymatrix{
1\times A \ar[dr]^{\pi_2} \ar[d]_{\text{id}_G\times 1_A} \\
G\times A \ar[r]^{a_A} & A}
\text{, }
 \xymatrix{
G\times G \times A \ar[d]_{m_G\times 1_A} \ar[r]^-{1_G\times a_A} & G\times A \ar[d]^{a_A} \\
G\times A \ar[r]^{a_A} & A}
\end{displaymath}
 commute.

A (left) \emph{$G$-object} is an object $A$ with a left $G$-action $a_A$, denoted by $(A,a_A)$.

A \emph{$G$-morphism} from a $G$-object $(A,a_A)$ to a $G$-object $(B,a_B)$ is a morphism $f:A\to B$ in $\ca E$ such that the following diagram:
\begin{displaymath}
 \xymatrix{
G\times A \ar[d]^{a_A} \ar[r]^{1_G\times f} & G\times B \ar[d]^{a_B} \\
A \ar[r]^f & B}
\end{displaymath}
commutes.
\end{defn}

We denote by $\ca E[G]$ the category of $G$-objects and $G$-morphisms.
Note that $\ca E[G]$ can be thought as the category of algebras of the monad 
$$G\times -:\ca E \to \ca E,$$ with multiplication 
$$m_G\times 1_A: G\times G\times A\to G\times A,$$ and unit 
$$A\xrightarrow{(!,1_A)}1\times A\xrightarrow{\text{id}_G\times 1_A} G\times A.$$

The underlying functor of this monad has a right adjoint $(-)^G:\ca E\to \ca E$ and the adjunction induces a comonad structure on this right adjoint by \cite{Adjointtriples}, so that the category of coalgebras of this comonad is equivalent to the category of algebras of the monad. Hence, the forgetful functor $\ca E[G]\to \ca E$ has left and right adjoints, with the left adjunction being monadic and the right one comonadic.

The comonad is cartesian, hence the category of its coalgebras (which is equivalent to $\ca E[G]$) is a topos by \cite[A4.2.1]{Elephant1}. Moreover, the forgetful functor $\ca E[G]\to \ca E$ is the inverse image of a surjective essential geometric morphism $\ca E\to \ca E[G]$. The left adjoint to the forgetful functor sends an object $A$ of $\ca E$ to 
$$(G\times A,m_G\times 1_A: G \times G \times A \to G\times A).$$ And the direct image of the geometric morphism sends an object $A$ of $\ca E$ to 
$$(A^G, a_{A^G}: G\times A^G\to A^G)$$ where $a_{A^G}$ is the transpose of the morphism 
$$A^G\xrightarrow{A^{m_G}} A^{G\times G}\simeq (A^G)^G.$$

The forgetful functor $\ca E[G]\to \ca E$ is both monadic and comonadic, therefore it creates all limits and colimits which exist in $\ca E$.

The subobject classifier of $\ca E[G]$ is 
$$(\Omega, \pi_2: G\times \Omega \to \Omega),$$ where $\Omega$ is the subobject classifier of $\ca E$.

Given $(A,a_A)$ and $(B,b_B)$ in $\ca E[G]$, the exponential $(A,a_A)^{(B,a_B)}$ is $(A^B,a_{A^B})$, where the left $G$-action $a_{A^B}: G\times A^B \to A^B$, where $a_{A^B}$ is the transpose of the composite 
$$G\times A^B \times B\xrightarrow{(\pi_1,\pi_2,\text{inv}{\scriptscriptstyle{\circ}}\pi_1,\pi_3)} G\times A^B\times G \times B \xrightarrow{1_{G\times A^B}\times a_B} G\times A^B \times B \xrightarrow{1_G\times \text{ev}}G\times A \xrightarrow{a_A} A.$$
Using generalized elements, we can describe the $G$-action on $A^B$ to be the one sending an element $(g,u)$ of $G\times A^B$ to the element $a_{A^B}(g,u)$ of $A^B$, such that for $b$ an element of $B$, 
$$a_{A^B}(g,u)(b)=a_A(g,u(a_B(\text{inv}_G(g),b))).$$
To make this more readable we can denote the actions $a_A$, $a_B$ and $a_{A^B}$ by $*$, and the morphism $\text{inv}_G$ by $(-)^{-1}$, and then 
$$(g*u)(b)= g*(u(g^{-1}*b)).$$

Let $H$ be a second group object of $\ca E$ with group multiplication $m_H:H\times H\to H$, unit $\text{id}_H:1\to H$ and the inverse morphism $\text{inv}_H:H\to H$, and let $\theta: G\to H$ be a group homomorphism in $\ca E$.

In \cite[VII 3.1]{Sheaves}, there is a description of the essential geometric morphism $\theta:\ca E[G]\to \ca E[H]$, whose inverse image $\theta^*$ sends an object $(A,a_A)$ of $\ca E[H]$ to the object $(A, a_A(\theta\times 1_A))$. The left adjoint $\theta_{!}$ of $\theta^*$ sends an object $(B,a_B)$ of $\ca E[G]$ to $(B',a_{B'})$ where $(H\times B,m_H\times 1_B) \stackrel{c_B}\twoheadrightarrow (B',a_{B'})$ is the coequalizer of the following diagram:
$(H\times G\times B,m_H\times 1_{G\times B}) \rightrightarrows (H\times B,m_H\times 1_B)$. The two arrows we are coequalizing are $m_H(1_H\times \theta)\times 1_B$ and $1_H\times a_B$ and they both commute with the left $H$-actions. The inverse image of the induced geometric morphism is faithful, therefore the geometric morphism is a geometric surjection. 

Notice that given another group homomorphism $\theta': H \to K$ in $\ca E$, the geometric morphism induced by the composite group homomorphism $\theta'\theta:G\to K$  is isomorphic to the composite of the two induced geometric morphisms $\theta: \ca E[G]\to \ca E[H]$ and $\theta': \ca E[H]\to \ca E[K]$ (by comparing the inverse image part of the two geometric morphisms). Therefore, (up to isomorphism) there is no ambiguity when we are talking about the geometric morphism $\theta'\theta$.

Thus, we have a pseudofunctor 
$$\mathbf{Group}(\ca E)\to \mathfrak{Top}.$$ Also, notice that the trivial group object $1$ (where $1$ is the terminal object of $\ca E$) is the terminal object of the category $\mathbf{Group}(\ca E)$. $\ca E[1]$ is equivalent to $\ca E$, therefore we also have a pseudofunctor  
$$\mathbf{Group}(\ca E)\to \mathfrak{Top}/\ca E.$$
More precisely, given a group object $G$, the $\ca E$-topos given by this pseudofunctor is the geometric morphism $\ca E[G]\to \ca E$ whose inverse image sends an object $A$ of $\ca E$ to $(A,\pi_2: G\times A\to A)$, i.e. the object $A$ with a trivial $G$-action.

Notice that given a group $G$ we also have a group homomorphism $\text{id}_G:1\to G$ which induces a geometric morphism $\ca E\to \ca E[G]$, which is the one induced by the monads and comonads mentioned above and whose inverse image is the forgetful functor.

Let us go back to the essential geometric morphism $\theta: \ca E[G]\to \ca E[H]$. Let $(B,a_B)$ be in $\ca E[G]$ and let $\theta_!(B,a_B)=(B',a_{B'})$ and $c_B:H\times B\to B'$ be as in the description of $\theta_!$ above.  Then, the unit $\eta$ of the adjunction $\theta_{!}\dashv \theta^*$ at $(B,a_B)$ is given by the morphism
$$\eta_B:B\cong 1\times B \stackrel{\text{id}_H\times 1_B}\longrightarrow H\times B \stackrel{c_B}\twoheadrightarrow B'.$$
This arrow commutes with the left $G$-actions $a_B$ and $a_{B'}(\theta\times 1_{B'})$, hence it is indeed a morphism in $\ca E[G]$ from $(B,a_B)$ to $(B',a_{B'}(\theta\times 1_{B'}))=\theta^*\theta_!(B,a_B)$.

\begin{defn} \label{defnorbits}
Let $a_A:A\times G'\to A$ be a right action of a group $G'$ on a set $A$. Then the \emph{object of orbits} of $a_A$ is the coequalizer of the morphisms 
$$a_A, \pi_1: A\times G' \rightrightarrows A.$$
\end{defn}

In $\ca E[H]$, the group $(H,\pi_2)$ acts on the right of $(H,m_H)$ via multiplication. The group homomorphism $(G,\pi_2)\to (H,\pi_2)$ induces a right $(G,\pi_2)$-action on $(H,m_H)$. Notice that $\theta_{!}(1)$ is the object of orbits of the right $(G,\pi_2)$-action on $(H,m_H)$. We denote $\theta_{!}(1)$ by $(L,a_L)$.

\begin{lem}
 $\theta: \ca E[G]\to \ca E[H]$ is an atomic geometric morphism.
\end{lem}

\begin{proof}
A geometric morphism is atomic iff its inverse image preserves the subobject classifier and exponentials.

The subobject classifier of $\ca E[H]$ is $(\Omega, \pi_2:H\times \Omega\to \Omega)$ where $\Omega$ is the subobject classifier of $\ca E$. $\theta^*(\Omega,\pi_2:H\times \Omega\to \Omega)$ is $(\Omega, \pi_2:G\times \Omega\to \Omega)$, the subobject classifier of $\ca E[G]$.

Given $(A,a_A)$ and $(B,a_B)$ in $\ca E[H]$, the exponential $(A,a_A)^{(B,a_B)}$ is $(A^B,a_{A^B})$, where the action $a_{A^B}:H\times A^B\to A^B$ as described above in terms of elements sends $(h,u)$ to $h*u$ in $A^B$ such that for $b$ an element of $B$, 
$$(h*u)(b)= h*(u(h^{-1}*b)).$$

$\theta^*(A,a_A)^{\theta^*(B,a_B)}$ is $(A,a_A(\theta\times 1_A))^{(B,a_B(\theta \times 1_B))}$. This is the object $A^B$ with the left $G$-action $b_{A^B}:G\times A^B \to A^B$ sending an element $(g,u)$ to $b_{A^B}(g,u)$ such that for an element $b$ of $B$, 
$$b_{A^B}(g,u)(b)=\theta(g)*(u(\theta(g)^{-1}*b))=\theta(g)*(u(\theta(g^{-1})*b)).$$

$\theta^*(A^B,a_{A^B})$ is $(A^B, a_{A^B}(\theta \times 1_{A^B}))$, where the $G$-action $a_{A^B}(\theta \times 1_{A^B}):G\times A^B\to A^B$ sends an element $(g,u)$ of $G\times A^B$ to $a_{A^B}(\theta \times 1_{A^B})(g,u)$ so that for $b$ an element of $B$, 
$$(a_{A^B}(\theta \times 1_{A^B})(g,u))(b)=(a_{A^B}(\theta(g), u))(b)=\theta(g)*(u(\theta(g^{-1})*b)).$$
Hence, $a_{A^B}(\theta \times 1_{A^B})= b_{A^B}$, and therefore $\theta^*$ preserves exponentials.

$\theta^*$  preserves exponentials and the subobject classifier, therefore $\theta$ is an atomic geometric morphism.
\end{proof}

\begin{rmk}
Notice that the above lemma does not generalize to the case where $G$ and $H$ are monoids.
\end{rmk}

If $\theta$ is a monomorphism $\theta^* (L,a_L)=(L,\pi_2)$ and the unit on $1$ is the morphism $1\xrightarrow{\text{id}_H} H \xrightarrow{c_1} L$ where $c_1:H\to L$ is the coequalizer given in the description of $\theta_!(B,a_B)$ for $(B,a_B)=(1,!)$. It commutes with the appropriate $G$-actions, therefore it defines a morphism $\eta_1:(1,!)\to (L,\pi_2)$.

\begin{lem} \label{lemifflocalic}
 $\theta: \ca E[G]\to \ca E[H]$ is a localic geometric morphism iff $\theta: G\to H$ is a monomorphism.
\end{lem}

\begin{proof}
Suppose that $\theta: G\to H$ is a monomorphism. Given an object $(A,a_A)$ in $\ca E[G]$, consider the epimorphism 
$$a_A: (G\times A, m_G\times 1_A)\twoheadrightarrow (A,a_A)$$ and the monomorphism 
$$\theta\times 1_A: (G\times A, m_G\times 1_A) \rightarrowtail (H\times A, (m_H(\theta\times 1_H))\times 1_A)= \theta^*(H\times A, m_H\times 1_A).$$
Therefore, $(A,a_A)$ is a subquotient of $\theta^*(H\times A, m_H\times 1_A)$ and the geometric morphism $\theta$ is localic.

For the direct implication let us suppose that $\theta: \ca E[G] \to \ca E[H]$ is localic. The object $(G,m_G)$ is a subquotient of some $\theta^*(A,a_A)=(A,a_A(\theta \times 1_A)$ for $(A,a_A)$ an object of $\ca E[H]$, i.e. there exists an object $(B,a_B)$ of $\ca E[G]$, a monomorphism $f:(G,m_G)\rightarrowtail (B,a_B)$ and an epimorphism $k:(A,a_A(\theta\times 1_A)\twoheadrightarrow (B,a_B)$. Let $g$, $g'$ be in $G$ and suppose that $\theta(g)=\theta(g')$. $k$ is an epimorphism so there exists an element $x$ of $A$ such that $f(1_G)=k(x)$ (and the final result does not depend on the choice of $x$).
\begin{displaymath}
\begin{split}
f(g) =a_B(g,f(1_G)) & =a_B(g,k(x))=k(a_A(\theta(g),x)) = \\
& = k(a_A(\theta(g'),x))=a_B(g',k(x))=a_B(g',f(1_G))=f(g').
\end{split}
\end{displaymath}
$f$ is a monomorphism therefore $g=g'$, hence $\theta$ is also a monomorphism.
\end{proof}

\begin{rmk}
Notice that the above lemma is a special case of the fact that for an internal functor $f: \mathbb C\to \mathbb D$ in a topos $\ca E$, the induced geometric morphism $\ca E[\mathbb C] \to \ca E[\mathbb D]$ is localic iff $f$ is a faithful functor. This is proved for $\ca E = \set$ in \cite[A4.6.2(c)]{Elephant1}.
\end{rmk}



\begin{thrm} \label{thrmlocal}
 If $\theta:G\to H$ is a group monomorphism in a topos $\ca E$, then $\theta: \ca E[G]\to \ca E[H]$ is a local homeomorphism, and in particular $\ca E[G]$ is equivalent to $\ca E[H]/(L,a_L)$.
\end{thrm}

\begin{proof}
By \cite[C3.5.4(iii)]{Elephant2} an atomic morphism that is also localic is a local homeomorphism.

By following the proof, $\theta$ is isomorphic to the geometric morphism 
$$\ca E[G]\xrightarrow{\eta_1} \ca E[G]/(L,\pi_2)\xrightarrow{\theta/(L,a_L)}\ca E[H]/(L,a_L)\xrightarrow{p} \ca E[H]$$
where the geometric morphism $p$ is the local homeomorphism induced by the unique morphism $(L,a_L)\to (1,!)$ in $\ca E[H]$. In the same proof it is shown that 
$$\ca E[G]\xrightarrow{\eta_1} \ca E[G]/(L,\pi_2)\xrightarrow{\theta/(L,a_L)} \ca E[H]/(L,a_L)$$ is an equivalence.
\end{proof}

\begin{rmk}
 Note that the converse of the theorem is also true. A local homeomorphism is always localic. Therefore, by Lemma \ref{lemifflocalic} $\theta: \ca E[G]\to \ca E[H]$  being a local homeomorphism implies that $\theta: G\to H$ is a monomorphism.
\end{rmk}

\section{The $H$-endomorphisms of $L$}

Theorem \ref{thrmlocal} says that given a group monomorphism $\theta:G\to H$ in a topos $\ca E$, the induced geometric morphism $\theta:\ca E[G]\to \ca E[H]$  induces a local homeomorphism and in particular $\ca E[G]$ factors through the local homeomorphism $\ca E[H]/(L,a_L)\to \ca E[H]$, where $(L,a_L)$ is $\theta_{!}(1)$. In this section, we are interested in all the factorizations of $\theta$ through $\ca E[H]/(L,a_L)\to \ca E[H]$.

As mentioned in \cite[B3.2.8(b)]{Elephant1}, for a topos $\ca F$, the functor $\ca F\to \mathfrak{Top}/\ca F$ which sends $I$ to $\ca F/I$, is a (2-categorical) full embedding. Given an object $X$ of $\ca F$, we spell out the details of this statement for the endomorphisms of the local homeomorphism $\ca F/X\to \ca F$ (when viewed as an object of $\mathfrak{Top}/\ca F$).

Given a topos $\ca F$ and an object $X$ of $\ca F$, let $p: \ca F/X\to \ca F$ be the local homeomorphism. By \cite[B3.2.8(b)]{Elephant1}, $$\mathfrak{Top}/\ca F(p,p)$$ is equivalent to the discrete category whose objects are the morphisms in $\ca F/X$ from the terminal object $1_X$ to $p^*(X)=(X\times X\xrightarrow{\pi_2} X)$. This equivalence sends a geometric morphism $g:\ca F/X\to \ca F/X$ to the inverse image of the diagonal map $\Delta_X: 1_X\to p^*(X)$.

Conversely, suppose we are given a morphism $1_X\to p^*(X)$, i.e. a morphism $X\to X\times X$ which makes the following diagram:
\begin{displaymath}
 \xymatrix{
X\ar[d]^{1_X} \ar[r] & X\times X \ar[dl]^{\pi_2} \\
X
}
\end{displaymath}
commute. This morphism is of the form $(\phi,1_X):X\to X\times X$. We claim that the geometric morphism that corresponds to it is $\phi:\ca F/X\to \ca F/X$ (whose inverse image is pullback by $\phi$). $\phi^*$ sends the diagonal map $\Delta_X:1_X\to p^*(X)$ to $(\phi,1_X)$ because in the following diagram
\begin{displaymath}
 \xymatrix{
X \ar[d]^{(\phi,1_X)} \ar[r]^{\phi} & X \ar[d]^{\Delta_X} \\
X\times X \ar[r]^{1_X\times \phi} \ar[d]^{\pi_2} &  X\times X \ar[d]^{\pi_2} \\
X \ar[r]^{\phi} & X
}
\end{displaymath}  
the top square is a pullback (since the bottom one and the whole rectangle are).

Given a group monomorphism $\theta:G\to H$ in a topos $\ca E$ as in the previous section, let $\theta$ be the induced geometric morphism $\ca E[G]\to \ca E[H]$. Let $p:\ca E[H]/(L,a_L)\to \ca E[H]$ be the local homeomorphism of the previous section. $G$ is a subgroup of $H$ via the monomorphism $\theta$. We say that $G$ is a \emph{normal subgroup} of $H$ when for any $g$ in $G$ and $h$ in $H$, $hgh^{-1}$ is in $G$. Then, the following two lemmas hold.

\begin{lem}
The category of endomorphisms of $p: \ca E[H]/(L,a_L)\to \ca E[H]$ (when $p$ is viewed as an object of  $\mathfrak{Top}/\ca E[H]$) is equivalent to the discrete category whose objects are (external) endomorphisms of $(L,a_L)$ in $\ca E[H]$. Furthermore, any factorization of $p$ through $p$ is a geometric surjection, and if $G$ is a normal subgroup of $H$ then any factorization of $p$ through $p$ is an equivalence.
\end{lem}

\begin{proof}
By the above discussion, an endomorphism of $p$ is induced by a morphism $(L,a_L)\to (L,a_L)$ in $\ca E[H]$, i.e. a morphism $f:L\to L$ in $\ca E$ which commutes with the $H$-action $a_L$.

Recall the definition of $(L,a_L)$. $L$ is the quotient $H/\sim$, where $h\sim h'$ iff there exists $g$ in $G$ such that $hg=h'$. This quotient of $H$ inherits the left $H$-action of $H$ via left multiplication. Let $f:(L,a_L)\to (L,a_L)$ be a morphism in $\ca E[H]$. $L$ contains the equivalence class of the unit of $H$ as an element which we shall denote by $x$. Given $y$ in $L$, there exists $h$ in $H$ such that $a_L(h,f(x))=y$, and therefore $y=f(a_L(h,x))$. Hence, $f$ is an epimorphism, and therefore the induced geometric morphism $f:\ca E[H]/(L,a_L)\to \ca E[H]/(L,a_L)$ is a geometric surjection.

Suppose that $G$ is a normal subgroup of $H$, and let $f:(L,a_L)\to (L,a_L)$ be a morphism in $\ca E[H]$. Let $y$, $y'$ be elements of $H$ representing elements of $L$ and suppose that $f(y)=f(y')$ represented by the element $k$ of $H$. Let $h$ be in $H$ such that $hy=y'$. Then, $a_L(h,f(y))=f(y')$ or equivalently $hk\sim k$ which implies that $k^{-1}hk$ is an element of $G$. $G$ is a normal subgroup of $H$, therefore $h$ is also in $G$, and therefore $y$ and $y'$ represent the same element of $L$. Hence, $f$ is a monomorphism, and therefore it is an isomorphism since we have already shown that it is an epimorphism. Therefore, the induced geometric morphism $f: \ca E[H]/(L,a_L)\to \ca E[H]/(L,a_L)$ is an equivalence.
\end{proof}

\begin{lem}
 $\mathfrak{Top}/\ca E[H](\theta,p)$ is equivalent to the discrete category whose objects are the (external) endomorphisms of $(L,a_L)$ in $\ca E[H]$. Furthermore, any factorization of $\theta$ through $p$ is a geometric surjection, and if $G$ is a normal subgroup of $H$ then any factorization of $\theta$ through $p$ is an equivalence.
\end{lem}

\begin{proof}
By Theorem \ref{thrmlocal}, we know that $\theta:\ca E[G] \to \ca E[H]$ and the local homeomorphism $p:\ca E[H]/(L,a_L)\to \ca E[H]$ are isomorphic as objects of $\mathfrak{Top}/\ca E[H]$. Hence,
$$\mathfrak{Top}/\ca E[H](\theta,p)\simeq\mathfrak{Top}/\ca E[H](p,p).$$
Therefore, by the above lemma $\mathfrak{Top}/\ca E[H](\theta,p)$ is equivalent to the discrete category whose objects are endomorphisms of $(L,a_L)$ in $\ca E[H]$.

Thus, by the above lemma all factorizations $\ca E[G]\to \ca E[H]/(L,a_L)$ of $\theta$ through $p$ are geometric surjections and when $G$ is a normal subgroup they are  all equivalences.
\end{proof}

\chapter{The classifying topos for affine planes} \label{chaZG}

The main result of this chapter is the identification of the classifying topos for the theory of affine planes as the topos $\ca Z[G]$, where $\ca Z$ is the Zariski topos and $G$ is the group of affine transformations of the generic local ring. The proof of this result contains an explicit construction of an affine plane from a pair of a local ring $R$ and a $G(R)$-torsor. The proof also involves applying some of the constructions of Chapter \ref{chalocal} to an affine plane of the topos $\ca Z[G]$. In the beginning of the chapter, we prove that the internal group of automorphisms of the generic local ring in the Zariski topos is the trivial group which enables us to view $G$ as the group of automorphisms of the affine plane over the generic local ring.

\section{The internal automorphism group of the generic local ring}

Let $M$ be the generic local ring in the Zariski topos $\ca Z$. $\ca Z$ is the Grothendieck topos $\mathbf{Sh}(\mathbf{Ring}_{\text{fp}}^{\text{op}}, J_{\text{Zar}})$, where $\mathbf{Ring}_{\text{fp}}$ is the category of finitely presented rings and $J_{\text{Zar}}$ is the Zariski topology. The Zariski topology is subcanonical. The underlying object of the generic local ring which we also denote by $M$ is the sheaf $\mathbf{Ring}_{\text{fp}}(\mathbb Z[X],-)$ (or equivalently the forgetful functor which sends a finitely presented ring to its underlying set). The exponential object $M^M$ is the sheaf sending a ring $A$ to the polynomial ring $A[X]$ and a morphism $f:A\to B$ to the morphism $A[X]\to B[X]$ which sends a polynomial $P=a_0+a_1 X+\ldots+a_n X^n$ to the polynomial $f(P)=f(a_0)+f(a_1) X+\ldots+f(a_n) X^n$. Moreover, the evaluation morphism $M^M\times M\to M$ is the natural transformation which at the object $A$ is the morphism $(A[X],A)\to A$ sending $(P,a)$ to the evaluation of the polynomial $P$ at $a$.

We define the automorphism group $\text{Aut}(M)$ of the generic local ring to be the subobject of $M^M$ described in the following way $\{P:M^M|(P(0)=0)\wedge(P(1)=1)\wedge((\forall a,b).(P(a+b)=P(a)+P(b)))\wedge((\forall a,b).(P(a\cdot b)=P(a)\cdot f(b))) \wedge ((\forall a,b).(P(a)=P(b)\Rightarrow a=b)) \wedge ((\forall a)(\exists x).(P(x)=a))\}$. This subobject of $M^M$ in the Zariski topos can be described as the sheaf sending a finitely presented ring $A$ to the subset of $A[X]$ consisting of polynomials $P$ such that for every morphism $f:A\to B$, the polynomial $f(P)$ (viewed as a morphism $B\to B$) is a ring isomorphism.

\begin{thrm} \label{thrmautgen}
$\text{Aut}(M)$ is the trivial group.
\end{thrm}

\begin{proof}
Suppose that $P=a_0+a_1 X+\ldots+a_n X^n$ is in $\text{Aut}(M)(A)$. Then, $P$ viewed as a morphism of $A$ is a ring homomorphism. Therefore, $P(0)=0$ which implies that $a_0=0$, and $P(1)=1$ which implies that $a_1+a_2+\ldots+a_n=1$.

Consider the inclusion $g:A\to A[Y,Z]$. Then, $g(P)$ is a ring homomorphism of $A[Y,Z]$ and therefore $P(Y)\cdot P(Z)= P(Y\cdot Z)$. Hence, $(a_1Y+a_2Y^2+\ldots+a_nY^n)\cdot (a_1Z+a_2Z^2+\ldots+a_nZ^n)=(a_1YZ+a_2Y^2Z^2+\ldots+a_nY^nZ^N)$. By comparing coefficients, we can see that $a_i^2=a_i$ for $1\leq i\leq n$, and $a_ia_j=0$ when $1\leq i<j\leq n$.

Let us also consider the inclusion $h:A\to A[Y]$. Then, $h(P)$ is an ring automorphism of $A[Y]$ and in particular $P(b_0+b_1 Y+\ldots +b_m Y^m)= Y$ for some $b_0,b_1,\ldots,b_m$ in $A$. $P(b_0+b_1 Y+\ldots b_m Y^m)=P(b_0)+P(b_1Y)+\ldots+P(b_mY^m)$ and therefore the coefficient of $Y$ is $a_1b_1$. Hence, $a_1b_1=1$ which implies that $a_1$ is invertible. $a_1^2=a_1$, therefore $a_1=1$. For $2\leq i\leq n$, $a_1a_i=0$, therefore $a_i=0$. Hence $P=X$, and therefore $\text{Aut}(M)(A)$ is the singleton. Therefore, $\text{Aut}(M)$ is the trivial group.
\end{proof}

\begin{rmk}
The automorphism group of the generic ring in $[\mathbf{Ring}_{\text{fp}},\set]$ is also trivial. The proof is identical to the one above.
\end{rmk}

Notice that we do not use the above result in any of our proofs. We still present it because it enables us to prove that the group of automorphisms of the affine plane over the generic local ring is the group $G$ (the group of affine transformations of the generic local ring). We also use it in the next chapter to show that the group of automorphisms of the projective plane over the generic local ring is the group $H$ (the projective general linear group of the generic local ring).

\section{Affine planes in $\ca Z$ and $\ca Z[G]$}

In Chapter \ref{chaaffplanes}, we construct an affine plane over a given local ring in a topos. Hence, we construct an affine plane over the generic local ring $M$ of $\ca Z$. We denote this affine plane as $\mathbb A(M)$.

Let $G$ be the group of affine transformations over the generic local ring, i.e. $G$ is the group of invertible matrices over the generic local ring of the form 
\begin{displaymath}
\begin{pmatrix}
\alpha_0 & \beta_0 & \gamma_0 \\
\alpha_1 & \beta_1 & \gamma_1 \\
0 & 0 & 1
\end{pmatrix} .
\end{displaymath}

By Theorem \ref{thrmaffmorph}, an automorphism of the affine plane over the generic local ring is uniquely of the form $g\circ \mathbb A(\alpha)$ where $g$ is in $G$ and $\alpha$ is an automorphism of the generic local ring. In fact, by Theorem \ref{thrmautgen} the unique automorphism of the generic local ring is the identity. Hence, the group of automorphisms of $\mathbb A(M)$ is $G$. However, all we need to know is that the affine plane $\mathbb A(M)$ has a left $G$-action.

$G$ acts on points via left matrix multiplication (when viewing a point $(a,b)$ as $(a,b,1)$). We denote this action by $a_{\text{pt}}:G\times\mathbb A_{\text{pt}}(M)\to \mathbb A_{\text{pt}}(M)$.

An element $g$ of $G$ acts on a line represented by $\pt{\lambda}$ via left matrix multiplication by  $(g^{-1})^T $. We denote this action by $a_{\text{li}}:G\times\mathbb A_{\text{li}}(M)\to \mathbb A_{\text{li}}(M)$.

We shall denote the described $G$-action on $\mathbb A(M)$ by $a_{\mathbb A}$. $a_{\mathbb A}$ is the $G$-action induced by the fact that $G$ is the group of automorphisms of $\mathbb A(M)$. Hence, it preserves the subobjects $\#_{\text{pt}}$, $\#_{\text{li}}$, $\parallel$, $\in$, $\notin$ of $\mathbb A_{\text{pt}}(M)\times \mathbb A_{\text{pt}}(M)$, $\mathbb A_{\text{li}}(M)\times \mathbb A_{\text{li}}(M)$, $\mathbb A_{\text{li}}(M)\times \mathbb A_{\text{li}}(M)$, $\mathbb A_{\text{pt}}(M) \times \mathbb A_{\text{li}}(M)$, $\mathbb A_{\text{pt}}(M)\times \mathbb A_{\text{li}}(M)$ respectively.

This structure in $\ca Z[G]$ satisfies the axioms of affine planes because the forgetful functor $\ca Z[G]\to \ca Z$ reflects monomorphisms and it maps this structure to the affine plane over the generic local ring in $\ca Z$ (which satisfies the axioms for affine planes). We denote this affine plane of $\ca Z[G]$ as $(\mathbb A(M), a_{\mathbb A})$. Let $\mathbf{Aff}$ be the classifying topos of the theory of affine planes and let $\ca Ag$ be its generic model. Then, the affine plane $(\mathbb A(M), a_{\mathbb A})$ of $\ca Z[G]$ corresponds to a geometric morphism 
$$f_{\mathbb A}: \ca Z[G]\to \mathbf{Aff}$$
such that the $f_{\mathbb A}^*(\ca Ag)$ is isomorphic to $(\mathbb A(M), a_{\mathbb A})$.


\section{Explicit construction of an affine plane from a $G$-torsor} \label{secGtoraff}

Let $R$ be a local ring in a topos $\ca E$, and let $a_A:A\times G(R)\to G(R)$ be a right $G(R)$-torsor in $\ca E$. We shall describe a structure in $\ca E$ in the language of affine planes (where instead of a relation on lines that are parallel we have a relation on lines that are parallel and apart from each other). We shall prove later in this chapter that this structure is an affine plane. The intentional interpretation of the objects and morphisms that will follow is the following: $A_{\text{pt}1}$, $A_{\text{pt}2}$, $A_{\text{pt}3}$ are all isomorphic and will be the object of points of the affine plane. $A_{\text{li}}$ and $A_{\text{li}'}$ are isomorphic and they will be the object of lines of the affine plane. $A_{\#\text{pt}}\to A_{\text{pt}2}\times A_{\text{pt}3}$ is a monomorphism and it will be the subobject of pairs of points that are apart from each other. $A_{\#\text{li}}\to A_{\text{li}}\times A_{\text{li}'}$ is a monomorphism and it will be the subobject of pairs of lines that are apart from each other. $A_{\in}\to A_{\text{pt}3}\times A_{\text{li}}$ is a monomorphism and it will be the subobject of pairs $(P,l)$ of a point $P$ lying on a line $l$. $A_{\notin}\to A_{\text{pt}1}\times A_{\text{li}}$ is a monomorphism and it will be the subobject of pairs $(P,l)$ of a point $P$ lying outside a line $l$. $A_{\parallel \#}\to A_{\text{li}'}\times A_{\text{li}}$ is a monomorphism and it will be the subobject of pairs of lines that are parallel and apart from each other.

Notice that the structure we are describing has a binary relation on lines whose interpretation is as pairs of lines which are parallel and apart from each other. In an affine plane, two lines $k$ and $l$ are parallel iff they are parallel and apart from each other or there exists a third line $m$ such that $k$ and $l$ are both parallel and apart from $m$. Hence, the theory of affine planes can also be formulated in a language where the parallel relation is replaced by this new relation.

Recall Definition \ref{defnorbits} of orbits of a group action. Let $G'$ be a group and $a_X:X\times G'\to X$ a right $G'$-action. Notice that a group homomorphism $k:G''\to G'$ induces a right $G''$-action $a_X(1_X\times k)$ on $X$. The epimorphism from $X$ to the object of $G'$-orbits coequalizes the arrows $a_X(1_X\times k), \pi_2: X\times G''\rightrightarrows X$, and therefore it induces a morphism from the object of $G''$-orbits to the object of $G'$-orbits. We shall use this fact to define morphisms between objects of orbits.

Let $G_3(R)$ be the group of invertible matrices over the local ring $R$ of the form $$\begin{pmatrix}
\alpha_0 & \beta_0 & 0 \\
\alpha_1 & \beta_1 & 0 \\
0 & 0 & 1
\end{pmatrix}.$$
Notice that $G_3(R)$ is a subgroup of $G(R)$. Hence, the right $G(R)$-action $a_A$ on $A$ induces a right $G_3(R)$-action on $A$. Let $A_{\text{pt}3}$ be the object of $G_3(R)$-orbits of $a_A$.

Let $G_1(R)$ be the group of invertible matrices over the local ring $R$ of the form
$$\begin{pmatrix} 1-\gamma_0 & \beta_0 & \gamma_0 \\ -\gamma_1 & \beta_1 & \gamma_1 \\ 0 & 0 & 1 \end{pmatrix}.$$
Let $A_{\text{pt}1}$ be the object of $G_1(R)$-orbits of $a_A$

Let $G_2(R)$ be the group of invertible matrices over the local ring $R$ of the form 
$$\begin{pmatrix} \alpha_0 & -\gamma_0 & \gamma_0 \\ \alpha_1 & 1-\gamma_1 & \gamma_1 \\ 0 & 0 & 1 \end{pmatrix}.$$
Let $A_{\text{pt}2}(M),a_{\text{pt}})$ be the object of $G_2(R)$-orbits of $a_A$

Let $G_{23}(R)$ be the group of invertible matrices over the local ring $R$ of the form
$$\begin{pmatrix}
\alpha_0 & 0 & 0 \\
\alpha_1 & 1 & 0 \\
0 & 0 & 1
\end{pmatrix}.$$
Let $A_{\#\text{pt}}$ be the object of $G_{23}(R)$-orbits of $a_A$. Notice that the group monomorphism $G_{23}(R)\to G_2(R)$ and $G_{23}(R)\to G_3(R)$ induce morphisms $A_{\#\text{pt}}\to A_{\text{pt}2}$ and $A_{\#\text{pt}}\to A_{\text{pt}3}$ respectively. Hence a morphism $A_{\#\text{pt}}\to A_{\text{pt}2}\times A_{\text{pt}3}$ is induced.

Let $G_\text{li}(R)$ be the group of invertible matrices over the local ring $R$ of the form
$$\begin{pmatrix}
\alpha_0 & 0 & 0 \\
\alpha_1 & \beta_1 & \gamma_1 \\
0 & 0 & 1
\end{pmatrix}.$$
Let $A_{\text{li}}$ be the object of $G_{\text{li}}(R)$-orbits of $a_A$.

Let $G_{\text{li}'}(R)$ be the group of invertible matrices over the local ring $R$ of the form
$$\begin{pmatrix}
1-\gamma_0 & 0 & \gamma_0 \\
\alpha_1 & \beta_1 & \gamma_1 \\
0 & 0 & 1
\end{pmatrix}.$$
Let $A_{\text{li}'}$ be the object of $G_{\text{li}'}(R)$-orbits of $a_A$.

Let $G_{\#\text{li}}(R)$ be the group of invertible matrices over the local ring $R$ of the form
$$\begin{pmatrix}
1 & 0 & 0 \\
\alpha_1 & \beta_1 & \gamma_1 \\
0 & 0 & 1
\end{pmatrix}.$$
Let $A_{\#\text{li}}$ be the object of $G_{\#\text{li}}(R)$-orbits of $a_A$. The group monomorphisms $G_{\#\text{li}}(R)\to G_{\text{li}}(R)$ and $G_{\#\text{li}}(R)\to G_{\text{li}'}(R)$ induce morphisms $A_{\#\text{li}}\to A_{\text{li}}$ and $A_{\#\text{li}}\to A_{\text{li}'}$ respectively. Hence, a morphism $A_{\#\text{li}}\to A_{\text{li}}\times A_{\text{li}'}$ is induced.

Let $G_{\in}(R)$ be the group of invertible matrices over the local ring $R$ of the form
$$\begin{pmatrix}
\alpha_0 & 0 & 0 \\
\alpha_1 & \beta_1 & 0 \\
0 & 0 & 1
\end{pmatrix}.$$
Let $A_{\in}$ be the object of $G_{\in}(R)$-orbits of $a_A$. The group monomorphism $G_{\in}(R)\to G_3(R)$ induces a morphism $A_{\in}\to A_{\text{pt}3}$. The group monomorphism $G_{\in}(R)\to G_{\text{li}}(R)$ induces a morphism $A_{\in}\to A_{\text{li}}$. Hence, a morphism $A_{\in}\to A_{\text{pt}3}\times A_{\text{li}}$ is induced.

Let $G_{\notin}(R)$ be the group of invertible matrices over the local ring $R$ of the form
$$\begin{pmatrix}
1 & 0 & 0 \\
-\gamma_1 & \beta_1 & \gamma_1 \\
0 & 0 & 1
\end{pmatrix}.$$
Let $A_{\notin}$ be the object of $G_{\notin}(R)$-orbits of $a_A$. The group monomorphism $G_{\notin}(R)\to G_1(R)$ induces a morphism $A_{\in}\to A_{\text{pt}1}$. The group monomorphism $G_{\notin}(R)\to G_{\text{li}}(R)$ induces a morphism $A_{\notin}\to A_{\text{li}}$. Hence, a morphism $A_{\notin}\to A_{\text{pt}1}\times A_{\text{li}}$ is induced.

Let $G_{\parallel\#}(R)$ be the group of invertible matrices over the local ring $R$ of the form
$$\begin{pmatrix}
1 & 0 & 0 \\
\alpha_1 & \beta_1 & \gamma_1 \\
0 & 0 & 1
\end{pmatrix}.$$
Let $A_{\parallel\#}$ be the object of $G_{\parallel\#}(R)$-orbits of $a_A$. The group monomorphisms $G_{\parallel\#}(R)\to G_{\text{li}'}(R)$ and $G_{\parallel\#}(R)\to G_{\text{li}}(R)$ induce morphisms $A_{\parallel\#}\to A_{\text{li}'}$ and $A_{\parallel\#}\to A_{\text{li}'}$ respectively. Hence, a morphism $A_{\parallel\#} \to A_{\text{li}'}\times A_{\text{li}}$ is induced.

To construct the above structure from the local ring $R$ and the $G(R)$-torsor $a_A$, we have only used finite limits and colimits. Therefore, the construction is preserved by inverse images of geometric morphisms.

\begin{lem}
Let $\ca A$ be an affine plane, and let $R_{\ca A}$ be its ring of trace preserving homomorphisms, and let $\omega(\ca A)$ be the $G(R_{\ca A})$-torsor as constructed in Theorem \ref{thrmgtorsor}. Then, the above construction applied to the local ring $R_{\ca A}$ and the $G(R_{\ca A})$-torsor $\omega(\ca A)$ gives a structure isomorphic to the affine plane $\ca A$.
\end{lem}

\begin{proof}
For the affine plane $\ca A$, we write $\ca A_{\text{pt}}$ for its object of points and $\ca A_{\text{li}}$ for its object of lines. We write $\ca A_{\#\text{pt}}$ for the subobject of $\ca A_{\text{pt}}\times \ca A_{\text{pt}}$ of points that are apart from each other. We write $\ca A_{\in}$ for the subobject of $\ca A_\times \ca A_{\text{li}}$ of pairs $(A,l)$ such that $A\in l$. We write $\ca A_{\notin}$ for the subobject of $\ca A_{\text{pt}}\times \ca A_{\text{li}}$ of pairs $(A,l)$ such that $A\notin l$. We write $\ca A_{\parallel \#}$ for the subobject of $\ca A_{\text{li}}\times \ca A_{\text{li}}$ of pairs of lines that are parallel and apart from each other.

Let $R_{\ca A}$ be the ring of trace preserving homomorphisms of $\ca A$ and let $\omega(\ca A)$ be the object of triples of non-collinear points of $\ca A$. The $G(R_{\ca A})$-torsor constructed in Chapter \ref{chalocal} is the right $G(R_{\ca A})$-action on $\omega(\ca A)$ where an element  $\begin{pmatrix}
\alpha_0 & \beta_0 & \gamma_0 \\
\alpha_1 & \beta_1 & \gamma_1 \\
0 & 0 & 1
\end{pmatrix}$ of $G(R_{\ca A})$ acts on $\omega(\ca A)$ by mapping a triple $(A,B,C)$ of non-collinear points to $(\tau_{CA}^{\alpha_0+\gamma_0}\tau_{CB}^{\alpha_1+\gamma_1} (C), \tau_{CA}^{\beta_0+\gamma_0}\tau_{CB}^{\beta_1+\gamma_1}(C), \tau_{CA}^{\gamma_0}\tau_{CB}^{\gamma_1}(C))$.

The epimorphisms we describe below, exhibit the isomorphism from the affine plane $\ca A$ to the affine plane constructed from the local ring $R_{\ca A}$ and the $G(R_{\ca A})$-torsor $\omega(\ca A)$ as above.

The object of $G_3(R_{\ca A})$-orbits is $\ca A_{\text{pt}}$ via the epimorphism $\omega(\ca A) \twoheadrightarrow \ca A_{\text{pt}}$ which maps $(A,B,C)$ of $\omega(\ca A)$ to $C$.

The object of $G_1(R_{\ca A})$-orbits is $\ca A_{\text{pt}}$ via the epimorphism $\omega(\ca A)\twoheadrightarrow \ca A_{\text{pt}}$ which maps $(A,B,C)$ of $\omega(\ca A)$ to $A$.

The object of $G_2(R_{\ca A})$-orbits is $\ca A_{\text{pt}}$ via the epimorphism $\omega(\ca A)\twoheadrightarrow \ca A_{\text{pt}}$ which maps $(A,B,C)$ of $\omega(\ca A)$ to the point $B$.

The object $G_{23}(R_{\ca A})$-orbits is $\ca A_{\#\text{pt}}$ via the epimorphism $\omega(\ca A)\twoheadrightarrow \ca A_{\#\text{pt}}$ which maps $(A,B,C)$ of $\omega(\ca A)$ to the pair of points $(B,C)$. Moreover, the two projections from $\mathbb A_{\#\text{pt}}$ to $\mathbb A_{\text{pt}}$ are the ones induced by the group monomorphisms $G_{23}(R_{\ca A})\to G_2(R_{\ca A})$ and $G_{23}(R_{\ca A})\to G_3(R_{\ca A})$.

The object of $G_\text{li}(R_{\ca A})$-orbits is $\ca A_{\text{li}}$ via  the epimorphism $\omega(\ca A)\twoheadrightarrow \ca A_{\text{li}}$ which maps $(A,B,C)$ of $\omega(\ca A)$ to the line $\ovl{BC}$.

The object of $G_{\in}(R_{\ca A})$-orbits is $\ca A_{\in}$ via the epimorphism $\omega(\ca A)\twoheadrightarrow \ca A_{\in}$ which maps $(A,B,C)$ of $\omega(\ca A)$ to the pair of the point $C$ and the line $\ovl{BC}$. Moreover, the projection from $\mathbb A_{\in}$ to $\mathbb A_{\text{pt}}$ is the morphism induced by the monomorphism $G_{\in}(R_{\ca A})\to G_3(R_{\ca A})$ and the projection $\mathbb A_{\in} \to \mathbb A_{\text{li}}$ is the morphism induced by the monomorphism $G_{\in}(R_{\ca A})\to G_{\text{li}}(R_{\ca A})$.

The object of $G_{\notin}(R_{\ca A})$-orbits is $\ca A_{\notin}$ via  the epimorphism $\omega(\ca A)\twoheadrightarrow \ca A_{\notin}$ which maps $(A,B,C)$ of $\omega(\ca A)$ to the pair of the point $A$ and the line $\ovl{BC}$. The projection from $\mathbb A_{\notin}$ to $\mathbb A_{\text{pt}}$ is the morphism induced by the monomorphism $G_{\notin}(R_{\ca A})\to G_1(R_{\ca A})$ and the projection $\mathbb A_{\notin}$ to $\mathbb A_{\text{li}}$ is the morphism induced by the monomorphism $G_{\notin}(R_{\ca A})\to G_{\text{li}}(R_{\ca A})$.

The object of $G_{\text{li}'}(R_{\ca A})$-orbits is $\ca A_{\text{li}}$ via  the epimorphism $\omega(\ca A)\twoheadrightarrow \ca A_{\text{li}}$ which maps $(A,B,C)$ of $\omega(\ca A)$ to the line through $A$ and parallel to $\ovl{BC}$.

The object of $G_{\parallel\#}(R_{\ca A})$-orbits is $\ca A_{\parallel\#}$ via the epimorphism $\omega(\ca A)\twoheadrightarrow \ca A_{\parallel\#}$ which maps $(A,B,C)$ of $\omega(\ca A)$ to the pair of the line through $A$ and parallel to $\ovl{BC}$, and the line $\ovl{BC}$. The two projections from $\mathbb A_{\parallel\#}$ to $\mathbb A_{\text{li}}$ are the ones induced by the monomorphisms $G_{\parallel\#}(R_{\ca A})\to G_{\text{li}'}(R_{\ca A})$ and $G_{\parallel\#}(R_{\ca A})\to G_{\text{li}}(R_{\ca A})$.
\end{proof}

\begin{lem} \label{lemgag}
In $\ca Z[G]$, the above construction applied to the local ring $(M,\pi_2)$ and the generic $G$-torsor $m_G:(G,m_G)\times (G,\pi_2)\to (G,m_G)$ gives a structure which is isomorphic to the affine plane $(\mathbb A(M),a_{\mathbb A})$.
\end{lem}

\begin{proof}
For the affine plane $(\mathbb A(M),a_{\mathbb A})$, we write $(\mathbb A_{\#\text{pt}},a_{\#\text{pt}})$ for the subobject of $(\mathbb A_{\text{pt}},a_{\text{pt}})\times (\mathbb A_{\text{pt}},a_{\text{pt}})$ of points that are apart from each other. We write $(\mathbb A_{\in},a_{\in})$ for the subobject of $(\mathbb A_{\text{pt}},a_{\text{pt}})\times (\mathbb A_{\text{li}},a_{\text{li}})$ of pairs $(A,l)$ such that $A\in l$. We write $(\mathbb A_{\notin},a_{\notin})$ for the subobject of $(\mathbb A_{\text{pt}},a_{\text{pt}})\times (\mathbb A_{\text{li}},a_{\text{li}})$ of pairs $(A,l)$ such that $A\notin l$. We write $(\mathbb A_{\parallel \#},a_{\parallel \#})$ for the subobject of $(\mathbb A_{\text{li}},a_{\text{li}})\times (\mathbb A_{\text{li}},a_{\text{li}})$ of pairs of lines that are parallel and apart from each other.

Observe that $(\mathbb A_{\text{pt}}(M),a_{\text{pt}})$ is the object of $G_3(M,\pi_2)$-orbits via the epimorphism $(G,m_G)\twoheadrightarrow (\mathbb A_{\text{pt}}(M),a_{\text{pt}})$ which sends $\begin{pmatrix}
a_0 & b_0 & c_0\\
a_1 & b_1 & c_1 \\
0 & 0 & 1
\end{pmatrix}$ to the point $(c_0,c_1)$.

Observe that $(\mathbb A_{\text{pt}}(M),a_{\text{pt}})$ is the object of $G_1(M,\pi_2)$-orbits via the epimorphism $(G,m_G)\twoheadrightarrow (\mathbb A_{\text{pt}}(M),a_{\text{pt}})$ which sends $\begin{pmatrix}
a_0 & b_0 & c_0\\
a_1 & b_1 & c_1 \\
0 & 0 & 1
\end{pmatrix}$ to the point $(a_0+c_0,a_1+c_1)$.

Observe that $(\mathbb A_{\text{pt}}(M),a_{\text{pt}})$ is the object of $G_2(M,\pi_2)$-orbits via the epimorphism $(G,m_G)\twoheadrightarrow (\mathbb A_{\text{pt}}(M),a_{\text{pt}})$ which sends $\begin{pmatrix}
a_0 & b_0 & c_0\\
a_1 & b_1 & c_1 \\
0 & 0 & 1
\end{pmatrix}$ to the point $(b_0+c_0,b_1+c_1)$.

Observe that $(\mathbb A_{\#\text{pt}}(M),a_{\#\text{pt}})$ is the object of $G_{23}(M,\pi_2)$-orbits via the epimorphism $(G,m_G)\twoheadrightarrow (\mathbb A_{\#\text{pt}}(M),a_{\#\text{pt}})$ which sends $\begin{pmatrix}
a_0 & b_0 & c_0\\
a_1 & b_1 & c_1 \\
0 & 0 & 1
\end{pmatrix}$ to the pair of points $(b_0+c_0,b_1+c_1)$ and $(c_0,c_1)$. Moreover, the two projections from $(\mathbb A_{\#\text{pt}}(M),a_{\#\text{pt}})$ to $(\mathbb A_{\text{pt}}(M),a_{\text{pt}})$ are the ones induced by the group monomorphisms $G_{23}(M,\pi_2)\to G_2(M,\pi_2)$ and $G_{23(M,\pi_2)}\to G_3(M,\pi_2)$.

Observe that $(\mathbb A_{\text{li}}(M),a_{\text{li}})$ is the object of $G_{\text{li}}(M,\pi_2)$-orbits via the epimorphism $(G,m_G)\twoheadrightarrow (\mathbb A_{\text{li}}(M),a_{\text{li}})$ which sends $\begin{pmatrix}
a_0 & b_0 & c_0\\
a_1 & b_1 & c_1 \\
0 & 0 & 1
\end{pmatrix}$ to the (equivalence class of the) line $(b_1,-b_0,b_0c_1-b_1c_0)$.

Observe that $(\mathbb A_{\text{li}}(M),a_{\text{li}})$ is the object of $G_{\text{li}'}(M,\pi_2)$-orbits via the epimorphism $(G,m_G)\twoheadrightarrow (\mathbb A_{\text{li}}(M),a_{\text{li}})$ which sends $\begin{pmatrix}
a_0 & b_0 & c_0\\
a_1 & b_1 & c_1 \\
0 & 0 & 1
\end{pmatrix}$ to the line $(b_1,-b_0,a_1b_0-a_0b_1+b_0c_1-b_1c_0)$.

Observe that $(\mathbb A_{\#\text{li}}(M),a_{\#\text{li}})$ is the object of $G_{\#\text{li}}(M,\pi_2)$-orbits via the epimorphism $(G,m_G)\twoheadrightarrow (\mathbb A_{\text{li}}(M),a_{\text{li}})$ which sends $\begin{pmatrix}
a_0 & b_0 & c_0\\
a_1 & b_1 & c_1 \\
0 & 0 & 1
\end{pmatrix}$ to the pair of lines $(b_1,-b_0,b_0c_1-b_1c_0)$ and $(b_1,-b_0,a_1b_0-a_0b_1+b_0c_1-b_1c_0)$ which are apart from each other. Moreover, the two projections from $(\mathbb A_{\#\text{li}}(M),a_{\#\text{li}})$ to $(\mathbb A_{\text{li}}(M),a_{\text{li}})$ are the ones induced by the group monomorphisms $G_{\#\text{li}}(M,\pi_2)\to G_{\text{li}}(M,\pi_2)$ and $G_{\#\text{li}}(M,\pi_2)\to G_{\text{li}'}(M,\pi_2)$.

Observe that $(\mathbb A_{\in}(M),a_{\in})$ is the object of $G_{\in}(M,\pi_2)$-orbits via the epimorphism $(G,m_G)\twoheadrightarrow (\mathbb A_{\in}(M),a_{\in})$ which sends $\begin{pmatrix}
a_0 & b_0 & c_0\\
a_1 & b_1 & c_1 \\
0 & 0 & 1
\end{pmatrix}$ to the pair of the point $(c_0,c_1)$ and the line $(b_1,-b_0,b_0c_1-b_1c_0)$. Moreover, the projection from $(\mathbb A_{\in}(M),a_{\in})$ to $(\mathbb A_{\text{pt}}(M),a_{\text{pt}})$ is the morphism induced by the group monomorphism $G_{\in}(M,\pi_2)\to G_3(M,\pi_2)$ and the projection $(\mathbb A_{\in}(M),a_{\in})$ to $(\mathbb A_{\text{li}}(M),a_{\text{li}})$ is the morphism induced by the group monomorphism $G_{\in}(M,\pi_2)\to G_{\text{li}}(M,\pi_2)$.

Observe that $(\mathbb A_{\notin}(M),a_{\notin})$ is the object of $G_{\notin}(M,\pi_2)$-orbits via the epimorphism $(G,m_G)\twoheadrightarrow (\mathbb A_{\notin}(M),a_{\notin})$ which sends $\begin{pmatrix}
a_0 & b_0 & c_0\\
a_1 & b_1 & c_1 \\
0 & 0 & 1
\end{pmatrix}$ to the pair of the point $(a_0+c_0,a_0+c_1)$ and the line $(b_1,-b_0,b_0c_1-b_1c_0)$. The projection from $(\mathbb A_{\notin}(M),a_{\notin})$ to $(\mathbb A_{\text{pt}}(M),a_{\text{pt}})$ is the morphism induced by the group monomorphism $G_{\notin}\to G_1$ and the projection $(\mathbb A_{\notin}(M),a_{\notin})$ to $(\mathbb A_{\text{li}}(M),a_{\text{li}})$ is the morphism induced by the group monomorphism $G_{\notin}\to G_{\text{li}}$.

Observe that $(\mathbb A_{\parallel\#}(M),a_{\text{pt}})$ is the object of $G_{\parallel\#}(M,\pi_2)$-orbits via the epimorphism $(G,m_G)\twoheadrightarrow (\mathbb A_{\parallel\#}(M),a_{\parallel\#})$ which sends $\begin{pmatrix}
a_0 & b_0 & c_0\\
a_1 & b_1 & c_1 \\
0 & 0 & 1
\end{pmatrix}$ to the pair of lines $(b_1,-b_0,a_1b_0-a_0b_1+b_0c_1-b_1c_0)$ and $(b_1,-b_0,b_0c_1-b_1c_0)$. The two projections from $(\mathbb A_{\parallel\#}(M),a_{\parallel\#})$ to $(\mathbb A_{\text{li}}(M),a_{\text{li}})$ are the ones induced by the group monomorphisms $G_{\parallel\#}(M,\pi_2)\to G_{\text{li}'}(M,\pi_2)$ and $G_{\parallel\#}(M,\pi_2)\to G_{\text{li}}(M,\pi_2)$.
\end{proof}


We now return to the construction from the beginning of this section. The local ring $R$ in topos $\ca E$ corresponds to a geometric morphism $f:\ca E\to \ca Z$ such that $f^*(M)$ is isomorphic to $R$. By Diaconescu's theorem, a right $G(R)$-torsor (or equivalently a right $f^*(G)$-torsor) $a_A:A\times G(R)\to A$ corresponds to a factorization $f':\ca E \to \ca Z[G]$ of $f$ through the geometric morphism $\ca Z[G]\to \ca Z$ which is such that $f'^*$ maps the generic $G$-torsor to the $G(R)$-torsor $a_A$. The construction of the affine plane structure is preserved by inverse images of geometric morphisms, therefore the construction applied to the local ring $R$ and the $G(R)$-torsor $a_A$ is isomorphic to $f'^*(\mathbb A(M),a_{\mathbb A})$ or equivalently $f'^*f_{\mathbb A}^*(\ca Ag)$. Hence, since the inverse image of an affine plane is an affine plane we have the following:

\begin{thrm}
The structure constructed in the beginning of this section is an affine plane.
\end{thrm}

\section{Geometric morphisms over $\ca Z$} \label{secaffoverZ}

In Chapter \ref{chalocal}, given an affine plane $\ca A$, we construct the local ring of trace preserving homomorphisms. Let $R_{\ca Ag}$ be the local ring of trace preserving homomorphisms of the generic affine plane $\ca Ag$ of $\mathbf{Aff}$. The local ring $R_{\ca Ag}$ in $\mathbf{Aff}$ corresponds to a geometric morphism 
$$R_{\mathbb A}:\mathbf{Aff}\to \ca Z$$
such that $R_{\mathbb A}^*(M)$ is isomorphic to $R_{\ca Ag}$. Thus, $\mathbf{Aff}$ is a $\ca Z$-topos via $R_{\mathbb A}$. $\ca Z[G]$ is also a $\ca Z$-topos via the geometric morphism induced by the group homomorphism $G\to 1$. In Theorem \ref{thrmgtorsor}, we also construct a $G(R_{\ca Ag})$-torsor in $\mathbf{Aff}$. By Diaconescu's theorem, this $G(R_{\ca Ag})$-torsor corresponds to a geometric morphism 
$$f_G: \mathbf{Aff}\to \ca Z[G]$$
over $\ca Z$.

We wish to show that the geometric morphism $f_{\mathbb A}: \ca Z[G]\to \mathbf{Aff}$ is also a geometric morphism over $\ca Z$. The construction of the local ring of trace preserving homomorphisms is preserved under inverse images of geometric morphisms. Therefore, it suffices to show that the local ring of trace preserving homomorphisms of the affine plane $(\mathbb A(M), a_{\mathbb A})$ is isomorphic to $(M,\pi_2)$.

In Chapter \ref{chalocal}, we gave explicit descriptions of the objects of translations and trace preserving homomorphisms of an affine plane over a local ring. We can apply these constructions to the affine plane over the generic local ring $\mathbb A(M)$ in the Zariski topos. In Section \ref{seclocalpres}, we show how these constructions are equivalent to constructions based on limits and colimits. The forgetful $\ca Z[G]\to \ca Z$ creates limits and colimits. In the following discussion, we describe cocones of specific diagrams in $\ca Z[G]$. The forgetful functor maps them to colimits in $\ca Z$, therefore the cocones where already colimits in $\ca Z[G]$.

Recall the construction of the local ring $\text{Tp}$ in Section \ref{seclocalpres}. The object $\text{Tn}_\#$ is defined as a quotient of pairs of points that are apart from each other. In the case of the affine plane $(\mathbb A(M),a_{\mathbb A})$, $\text{Tn}_{\#}(\mathbb A(M),a_{\mathbb A})$ is the object $\{(x,y,0):M^3|\text{inv}(x)\vee\text{inv}(y)\}$ whose left $G$-action is matrix multiplication.  It is a quotient of the object of pairs of points that are apart from each other via the epimorphism which sends such a pair $(a_0,a_1,1)$ and $(b_0,b_1,1)$ to $(b_0-a_0,b_1-a_1,0)$, or equivalently to the product $\begin{pmatrix}
a_0 & b_0 \\
a_1 & b_1 \\
1 & 1
\end{pmatrix}
\begin{pmatrix}
-1 \\
1
\end{pmatrix}$.

The object of translations is defined as a quotient of the object of pairs of points. In the case of the affine plane $(\mathbb A(M),a_{\mathbb A})$, $\text{Tn}(\mathbb A(M),a_{\mathbb A})$ is the object $\{(x,y,0):M^3\}$ whose left $G$-action is matrix multiplication.  It is a quotient of the object of pairs of points via the epimorphism which sends the pair of points $(a_0,a_1,1)$ and $(b_0,b_1,1)$ to $(b_0-a_0,b_1-a_1,0)$, or equivalently to the product $\begin{pmatrix}
a_0 & b_0 \\
a_1 & b_1 \\
1 & 1
\end{pmatrix} \begin{pmatrix}
-1 \\
1
\end{pmatrix}$.

$\text{Tn}$ has a group structure. In the case of the affine plane $(\mathbb A(M),a_{\mathbb A})$, this group structure is given in the following way:
The unit of the group is the morphism $(1,!)\to \text{Tn}(\mathbb A(M),a_{\mathbb A})$ which sends the unique element of $1$ to $(0,0,0)$ (and notice that this commutes with the two $G$-actions). The group operation is addition of the vectors.

The $\text{Tn}(\mathbb A(M),a_{\mathbb A})$-action on $(\mathbb A_{\text{pt}}(M),a_{\text{pt}})$ sends a pair of a translation $(x,y,0)$ and a point $(a_0,a_1,1)$ to $(a_0+x,a_0+y,1)$.

The object of trace preserving homomorphisms is defined as a quotient of the object $\{(A,B,C)| (A\#B) \wedge (C\in \ovl{AB})\}$. We construct this in two steps here. Let us consider the interpretation of the relation $C\in \ovl{AB}$ in $(\mathbb A(M),a_{\mathbb A})$. It is the object of matrices over the generic local ring of the form 
$ \begin{pmatrix}
a_0 & b_0 & c_0 \\
a_1 & b_1 & c_1 \\
1 & 1 & 1
\end{pmatrix}$ whose determinant is $0$ and such that at least one of $b_0-a_0$ and $b_1-a_1$ is invertible. The $G$-action is left matrix multiplication. Right multiplication of the above object by $\begin{pmatrix}
-1 & -1 \\
1 & 0 \\
0 & 1
\end{pmatrix}$ gives an epimorphism in $\ca Z[G]$ to the object of matrices over the generic local ring of the form $\begin{pmatrix}
x_0 & y_0 \\
x_1 & y_1 \\
0 & 0
\end{pmatrix}$, where at least one of $x_0$ and $x_1$ is invertible and $x_0 y_1=x_1 y_0$. The $G$-action on this object is again left matrix multiplication. By the above discussion, this is isomorphic the object $\{(\tau,\tau'): \text{Tn}^2| (\exists A) (A\#\tau(A))\wedge (\tau'(A)\in \ovl{A\tau(A)})\}$.

We have an epimorphism from the above to the generic local ring in the following way: Given $\begin{pmatrix}
x_0 & y_0 \\
x_1 & y_1 \\
0 & 0
\end{pmatrix}$, if $x_0$ is invertible it is sent to $x_0^{-1}y_0$ and if $x_1$ is invertible it is sent to $x_1^{-1}y_1$. Notice that the two definitions agree when both $x_0$ and $x_1$ are invertible. Also notice that the morphism commutes with the two actions.

The composite of the two above epimorphisms
$$\{(A,B,C)| (A\#B) \wedge (C\in \ovl{AB})\}\to \{(\tau,\tau'): \text{Tn}^2| (\exists A) (A\#\tau(A))\wedge (\tau'(A)\in \ovl{A\tau(A)})\}\to M$$
is the quotient of $\{(A,B,C)| (A\#B) \wedge (C\in \ovl{AB})\}$
via the equivalence relation $\sim_{\text{Tp}}$ described in Lemma \ref{lemsimtp}. Hence, the object $\text{Tp}(\mathbb A(M),a_{\mathbb A})$ of trace preserving homomorphisms of $(\mathbb A(M),a_{\mathbb A})$ is isomorphic to the object $(M,\pi_2)$. It is easy to show that the ring operations on $(M,\pi_2)$ coincide with the ones on $\text{Tp}(\mathbb A(M),a_{\mathbb A})$ described in Section \ref{seclocalpres}. Thus, we have the following:

\begin{lem}
The geometric morphism $f_{\mathbb A}: \ca Z[G]\to \mathbf{Aff}$ is a geometric morphism over $\ca Z$.
\end{lem}

\section{$\ca Z[G]\xrightarrow{f_{\mathbb A}} \mathbf{Aff}\xrightarrow{f_G} \ca Z[G]$ is isomorphic to the identity} \label{secgid}

By the results of the previous section, $\ca Z[G]\xrightarrow{f_{\mathbb A}} \mathbf{Aff}\xrightarrow{f_G} \ca Z[G]$ is a geometric morphism over $\ca Z$. By Diaconescu's theorem, to prove that this composite is isomorphic to the identity it suffices to show that $f_{\mathbb A}^*f_G^*$ maps the generic $G$-torsor to a $G$-torsor isomorphic to the generic one. We have shown that the local ring of trace preserving homomorphisms constructed from the affine plane $(\mathbb A(M),a_{\mathbb A})$ is isomorphic to the local ring $(M,\pi_2)$. Therefore it suffices to show that the $G(\text{Tp}(\mathbb A(M),a_{\mathbb A}))$-torsor constructed from $(\mathbb A(M),a_{\mathbb A})$ is isomorphic to the generic $G$-torsor in $\ca Z[G]$.

Notice that the object of three non-collinear points of $(\mathbb A(M),a_{\mathbb A})$ is the object $(\omega(M),a_{\omega})$, where $\omega(M)$ is the object of invertible matrices of the form 
$$\begin{pmatrix}
a_0 & b_0 & c_0\\
a_1 & b_1 & c_1 \\
1 & 1 & 1
\end{pmatrix}$$
and the action $a_{\omega}$ is left matrix multiplication. Theorem \ref{thrmgtorsor}, constructs a right $G(\text{Tp}(\mathbb A(M),a_{\mathbb A}))$-torsor whose underlying object is  $(\omega(M),a_{\omega})$. Via the isomorphism of $\text{Tp}(\mathbb A(M),a_{\mathbb A})$ and $(M,\pi_2)$ from the previous section $G(\text{Tp}(\mathbb A(M),a_{\mathbb A}))\cong (G,\pi_2)$ and therefore $(\omega(M),a_{\omega})$ is a $(G,\pi_2)$-torsor. Using the results from the previous section, we can explicitly describe the $(G,\pi_2)$-action on  $(\omega(M),a_{\omega})$ as the morphism $(\omega(M),a_{\omega})\times (G,\pi_2)$ which maps the pair $\begin{pmatrix}
a_0 & b_0 & c_0\\
a_1 & b_1 & c_1 \\
1 & 1 & 1
\end{pmatrix}$ of $\omega(M)$ and  $g$ of $G$ to 
$$\begin{pmatrix}
a_0 & b_0 & c_0\\
a_1 & b_1 & c_1 \\
1 & 1 & 1
\end{pmatrix}\begin{pmatrix}
1 & 0 & 0 \\
0 & 1 & 0 \\
-1 & -1 & 1
\end{pmatrix}
g
\begin{pmatrix}
1 & 0 & 0 \\
0 & 1 & 0 \\
1 & 1 & 1
\end{pmatrix}.$$

Recall the isomorphism $(\omega(M),a_{\omega})\to (G,m_G)$ which maps $\begin{pmatrix}
a_0 & b_0 & c_0\\
a_1 & b_1 & c_1 \\
1 & 1 & 1
\end{pmatrix}$ to the product $\begin{pmatrix}
a_0 & b_0 & c_0\\
a_1 & b_1 & c_1 \\
1 & 1 & 1
\end{pmatrix}
\begin{pmatrix}
1 & 0 & 0 \\
0 & 1 & 0 \\
-1 & -1 & 1
\end{pmatrix}$
or equivalently to 
$\begin{pmatrix}
a_0-c_0 & b_0-c_0 & c_0\\
a_1-c_1 & b_1-c_1 & c_1 \\
0 & 0 & 1
\end{pmatrix}$.
This isomorphism commutes with the two left $G$-actions $a_{\omega}$ and $m_G$.
Furthermore, it commutes with the right $(G,\pi_2)$-actions that make $(\omega(M),a_{\omega})$ and $(G,m_G)$ right $G$-torsors. Hence, the $G$-torsor corresponding to the composite $\ca Z[G]\to \mathbf{Aff}\to \ca Z[G]$ is isomorphic to the generic $G$-torsor, and therefore the composite geometric morphism is isomorphic to the identity.

\section{$\mathbf{Aff}\xrightarrow{f_G} \ca Z[G] \xrightarrow{f_{\mathbb A}} \mathbf{Aff}$ is isomorphic to the identity}

The geometric morphism $f_{\mathbb A}f_G$ corresponds to an affine plane in $\mathbf{Aff}$. Let $\ca Ag$ be the generic affine plane in $\mathbf{Aff}$ and let $\omega(\ca Ag)$ be the object of triples of non-collinear points. The geometric morphism $f_G$ corresponds to the pair of the local ring $R_{\ca Ag}$ of trace preserving homomorphisms and the right $G(R_{\ca Ag})$-torsor $\omega(\ca Ag)$ as constructed in Chapter \ref{chalocal}.

Hence, the affine plane corresponding to the composite $\mathbf{Aff}\xrightarrow{f_G} \ca Z[G] \xrightarrow{f_{\mathbb A}} \mathbf{Aff}$ is isomorphic to the affine plane constructed from the pair of the local ring $R_{\ca Ag}$ and the right $G(R_{\ca Ag})$-torsor $\omega(\ca Ag)$ as in Section \ref{secGtoraff}. By Lemma \ref{lemgag}, this new affine plane is isomorphic to $\ca A$. Therefore, the geometric  morphism is isomorphic to the identity.

We have already proved in Section \ref{secgid} that $f_G f_{\mathbb A}$ is isomorphic to the identity, hence we conclude the following:

\begin{thrm}
The classifying topos for the theory of affine planes $\mathbf{Aff}$ is equivalent to the topos $\ca Z[G]$.
\end{thrm}

\section{$\mathbf{Aff}_{\text{Pt}}\simeq \ca Z[G_3]$} \label{secaffpt}

Let $\mathbf{Aff}_{\text Pt}$ be the classifying topos for the theory of affine planes with a chosen point.

Let $G_3$ be as before the group in $\ca Z$ of matrices over the generic local ring of the form:
\begin{displaymath}
\begin{pmatrix}
\alpha_0 & \beta_0 & 0 \\
\alpha_1 & \beta_1 & 0 \\
0 & 0 & 1
\end{pmatrix} .
\end{displaymath}
Notice that $G_3$ is a subgroup of $G$, therefore the inclusion $\iota : G_3\to G$ induces a geometric morphism $\iota: \ca Z[G_3]\to \ca Z[G]$. This geometric morphism sends the generic affine plane of $\ca Z[G]$ to an affine plane in $\ca Z[G_3]$. The object of points of this affine plane is the one we get when we restrict the various $G$-actions of the generic affine plane in $\ca Z[G]$ to $G_3$-actions. The point $(0,0)$ is stable under the $G_3$-action therefore it is a global section of the object of points of this affine plane.

This affine plane with a chosen point in $\ca Z[G_3]$ induces a geometric morphism $\ca Z[G_3]\to \mathbf{Aff}_{\text Pt}$. We claim that this is an equivalence and we can prove this in a similar way to the way we proved that $\mathbf{Aff}\simeq \ca Z[G]$.

We give an alternative proof of this result, using Theorem \ref{thrmlocal}. $G_3\to G$ is a group monomorphism, therefore $\ca Z[G_3]\to \ca Z[G]$ is a local homeomorphism and in particular $\ca Z[G_3]$ is equivalent to $\ca Z[G]/J$, where $J$ is the coequalizer of the morphisms $m_{G_3}(1_{G_3}\times \iota),\pi_1:(G\times G_3,m_G\times 1_{G_3}) \rightrightarrows (G,m_G)$, or equivalently the object of $G_3$-orbits of $(G,m_G)$. As we have also mentioned in the proof of Theorem \ref{lemgag} the object of $G_3$-orbits is the object $(\mathbb A_{\text{pt}}(M),a_{\text{pt}})$ of points of the generic affine plane in $\ca Z[G]$. Thus, we have the following:

\begin{thrm}
$\mathbf{Aff}_{\text{Pt}}\simeq \ca Z[G_3]$.
\end{thrm}

\begin{proof}
By Theorem \ref{thrmlocal} and the above discussion we have a geometric isomorphism $\ca Z[G_3]\to \ca Z[G]/(\mathbb A_{\text{pt}}(M),a_{\text{pt}})$ where $(\mathbb A_{\text{pt}}(M),a_{\text{pt}})$ is the object of points of the generic affine plane in $\ca Z[G]$. $\ca Z[G]$ is the classifying topos for the theory of affine planes, therefore by \cite[B3.2.8(b)]{Elephant1}, $\ca Z[G]/(\mathbb A_{\text{pt}}(M),a_{\text{pt}})$ is the classifying topos for the theory of projective planes with a chosen point. $\ca Z[G_3]\simeq \ca Z[G]/(\mathbb A_{\text{pt}}(M),a_{\text{pt}})$, therefore $\ca Z[G_3]\simeq \mathbf{Aff}_{\text{Pt}}$.
\end{proof}

\chapter{The classifying topos for projective planes} \label{chaZH}

In this chapter, we prove that the classifying topos of the theory of projective planes is equivalent to the topos $\ca Z[H]$, where $\ca Z$ is the Zariski topos and $H$ is the projective general linear group of the generic local ring. The proof contains a construction of a projective plane from a pair of a local ring $R$ and an $H(R)$-torsor. It also uses results from Chapter \ref{chaprojcoo}.

\section{Projective planes in $\ca Z$ and $\ca Z[H]$} \label{secprojZH}

In Chapter \ref{chaprojplanes}, we construct a projective plane over a given local ring in a topos. Hence, we construct a projective plane over the generic local ring $M$ of $\ca Z$. We denote this projective plane as $\mathbb P(M)$.

Let $H$ be the projective general linear group over the generic local ring, i.e. $H$ is the quotient by scalar multiplication of the object of invertible matrices over the generic local ring of the form 
\begin{displaymath}
\begin{pmatrix}
\alpha_0 & \beta_0 & \gamma_0 \\
\alpha_1 & \beta_1 & \gamma_1 \\
\alpha_2 & \beta_2  & \gamma_2
\end{pmatrix}.
\end{displaymath}

By Theorem \ref{thrmprojmorph}, an automorphism of the projective plane over the generic local ring is uniquely of the form $h\circ \mathbb P(\alpha)$ where $h$ is in $H$ and $\alpha$ is an automorphism of the generic local ring. In fact, by Theorem \ref{thrmautgen} the unique automorphism of the generic local ring is the identity, and therefore the group of automorphisms of $\mathbb P(M)$ is $H$. However, all we need to know is that the projective plane $\mathbb P(M)$ has a left $H$-action.

$H$ acts on points via left matrix multiplication and we denote this action by $a_{\text{pt}}:H\times\mathbb P_{\text{pt}}(M)\to \mathbb P_{\text{pt}}(M)$.

An element of $H$ represented by a matrix $h$ acts on a line represented by $\pt{\lambda}$ via left matrix multiplication by $(h^{-1})^T $ (and this does not depend on the choice of representatives). We denote this action by $a_{\text{li}}:H\times\mathbb P_{\text{li}}(M)\to \mathbb P_{\text{li}}(M)$.

We shall denote the described $H$-action on $\mathbb P(M)$ by $a_{\mathbb P}$. $a_{\mathbb P}$ is the $H$-action induced by the fact that $H$ is the group of automorphisms of $\mathbb P(M)$. Hence, this $H$-action preserves the subobjects $\#_{\text{pt}}$, $\#_{\text{li}}$, $\in$, $\notin$ of $\mathbb P_{\text{pt}}(M)\times \mathbb P_{\text{pt}}(M)$, $\mathbb P_{\text{li}}(M)\times \mathbb P_{\text{li}}(M)$, $\mathbb P_{\text{pt}}(M) \times \mathbb P_{\text{li}}(M)$, $\mathbb P_{\text{pt}}(M)\times \mathbb P_{\text{li}}(M)$ respectively.

This structure in $\ca Z[H]$ satisfies the axioms of projective planes because the forgetful functor $\ca Z[H]\to \ca Z$ reflects monomorphisms and it maps this structure to the projective plane over the generic local ring in $\ca Z$ (which satisfies the axioms for projective planes). We denote this projective plane of $\ca Z[H]$ as $(\mathbb P(M),a_{\mathbb P})$. Let $\mathbf{Proj}$ be the classifying topos of the theory of projective planes and let $\ca Pg$ be its generic projective plane. Then, the projective plane $(\mathbb P(M), a_{\mathbb P})$ of  $\ca Z[H]$, corresponds to a geometric morphism 
$$f_{\mathbb P}: \ca Z[H]\to \mathbf{Proj}$$
such that  $f_{\mathbb P}^*(\ca Pg)$ is isomorphic to $(\mathbb P(M),a_{\mathbb P})$.

\section{Explicit construction of a projective plane from an $H$-torsor} \label{secGtoproj}
Let $R$ be a local ring in a topos $\ca E$, and let $a_A:A\times H(R)\to H(R)$ be a right $H(R)$-torsor in $\ca E$. We shall describe a structure in $\ca E$ in the language of projective planes with the intention of proving later in this chapter that it is a projective plane. The intentional interpretation of the objects and morphisms that will follow is the following: $A_{\text{pt}1}$, $A_{\text{pt}2}$, $A_{\text{pt}3}$ are all isomorphic and will be the object of points of the projective plane. $A_{\text{li}}$ and $A_{\text{li}'}$ are isomorphic and they will be the object of lines of the projective plane. $A_{\#\text{pt}}\to A_{\text{pt}1}\times A_{\text{pt}2}$ is a monomorphism and it will be the subobject of pairs of points that are apart from each other. $A_{\#\text{li}}\to A_{\text{li}}\times A_{\text{li}'}$ is a monomorphism and it will be the subobject of pairs of lines that are apart from each other. $A_{\in}\to A_{\text{pt}1}\times A_{\text{li}}$ is a monomorphism and it will be the subobject of pairs $(P,l)$ of a point $P$ lying on a line $l$. $A_{\notin}\to A_{\text{pt}3}\times A_{\text{li}}$ is a monomorphism and it will be the subobject of pairs $(P,l)$ of a point $P$ lying outside a line $l$.

Let $H_1(R)$ be the group of invertible matrices over the local ring $R$ of the form $$\begin{pmatrix}
\alpha_0 & \beta_0 & \gamma_0 \\
0 & \beta_1 & \gamma_1 \\
0 & \beta_2  & \gamma_2
\end{pmatrix},$$
and let $A_{\text{pt}1}$ be the object of $H_1(R)$-orbits of $a_A$.

Let $H_2(R)$ be the group of invertible matrices over the local ring $R$ of the form $$\begin{pmatrix}
\alpha_0 & 0 & \gamma_0 \\
\alpha_1 & \beta_1 & \gamma_1 \\
\alpha_2 & 0  & \gamma_2
\end{pmatrix},$$
and let $A_{\text{pt}2}$ be the object of $H_1(R)$-orbits of $a_A$.

Let $H_3(R)$ be the group of invertible matrices over the local ring $R$ of the form $$\begin{pmatrix}
\alpha_0 & \beta_0 & 0 \\
\alpha_1 & \beta_1 & 0 \\
\alpha_2 & \beta_2  & \gamma_2
\end{pmatrix},$$
and let $A_{\text{pt}3}$ be the object of $H_3(R)$-orbits of $a_A$.

Let $H_{12}(R)$ be the group of invertible matrices over the local ring $R$ of the form
$$\begin{pmatrix}
\alpha_0 & 0 & \gamma_0 \\
0 & \beta_1 & \gamma_1 \\
0 & 0  & \gamma_2
\end{pmatrix},$$
and let $A_{\#\text{pt}}$ be the object of $H_{12}(R)$-orbits of $a_A$. Notice that the group monomorphisms $H_{12}(R)\to H_1(R)$ and $H_{12}(R)\to H_2(R)$ induce morphisms $A_{\#\text{pt}}\to A_{\text{pt}1}$ and $A_{\#\text{pt}}\to A_{\text{pt}2}$ respectively. Hence, a morphism $A_{\#\text{pt}}\to A_{\text{pt}1}\times A_{\text{pt}2}$ is also induced.

Let $H_\text{li}(R)$ be the group of invertible matrices over the local ring $R$ of the form
$$\begin{pmatrix}
\alpha_0 & \beta_0 & \gamma_0 \\
\alpha_1 & \beta_1 & \gamma_1 \\
0 & 0 & \gamma_2
\end{pmatrix},$$
and let $A_{\text{li}}$ be the object of $H_{\text{li}}(R)$-orbits of $a_A$.

Let $H_{\text{li}'}(R)$ be the group of invertible matrices over the local ring $R$ of the form
$$\begin{pmatrix}
\alpha_0 & \beta_0 & \gamma_0 \\
0 & \beta_1 & 0 \\
\alpha_2 & \beta_2 & \gamma_2
\end{pmatrix},$$
and let $A_{\text{li}'}$ be the object of $H_{\text{li}}(R)$-orbits of $a_A$.

Let $H_{\#\text{li}}(R)$ be the group of invertible matrices over the local ring $R$ of the form
$$\begin{pmatrix}
\alpha_0 & \beta_0 & \gamma_0 \\
0 & \beta_1 & 0 \\
0 & 0 & \gamma_2
\end{pmatrix},$$
and let $A_{\#\text{li}}$ be the object of $H_{\#\text{li}}(R)$-orbits of $a_A$. The group monomorphisms $H_{\#\text{li}}(R)\to H_{\text{li}}(R)$ and $H_{\#\text{li}}(R)\to H_{\text{li}'}(R)$ induce morphisms $A_{\#\text{li}}\to A_{\text{li}}$ and $A_{\#\text{li}}\to A_{\text{li}'}$ respectively. Hence, a morphism $A_{\#\text{li}}\to A_{\text{li}}\times A_{\text{li}'}$ is also induced.

Let $H_{\in}(R)$ be the group of invertible matrices over the local ring $R$ of the form
$$\begin{pmatrix}
\alpha_0 & \beta_0 & \gamma_0 \\
0 & \beta_1 & \gamma_1 \\
0 & 0 & \gamma_2
\end{pmatrix},$$
and let $A_{\in}$ be the object of $H_{\in}(R)$-orbits of $a_A$. The group monomorphism $H_{\in}(R)\to H_1(R)$ induces a morphism $A_{\in}\to A_{\text{pt}1}$. The group monomorphism $H_{\in}(R)\to H_{\text{li}}(R)$ induces a morphism $A_{\in}\to A_{\text{li}}$. Hence a morphism $A_{\in}\to A_{\text{pt}1}\times A_{\text{li}}$ is induced.

Let $H_{\notin}(R)$ be the group of invertible matrices over the local ring $R$ of the form
$$\begin{pmatrix}
\alpha_0 & \beta_0 & 0 \\
\alpha_1 & \beta_1 & 0 \\
0 & 0 & \gamma_2
\end{pmatrix}.$$
and let $A_{\notin}$ be the object of $H_{\notin}(R)$-orbits of $a_A$. The group monomorphism $H_{\notin}(R)\to H_3(R)$ induces a morphism $A_{\notin}\to A_{\text{pt}3}$. The group monomorphism $H_{\notin}(R)\to H_{\text{li}}(R)$ induces a morphism $A_{\notin}\to A_{\text{li}}$. Hence a morphism $A_{\notin}\to A_{\text{pt}3}\times A_{\text{li}}$ is induced.

To construct the above structure from the local ring $R$ and the $H(R)$-torsor $a_A$, we have only used finite limits and colimits. Therefore, the construction is preserved by inverse images of geometric morphisms. 

\begin{lem} \label{lemhph}
Let $\ca P$  be a projective plane, and let $R_{\ca P}$ be the local ring and $\omega_4(\ca P)$ the $H(R_{\ca P})$-torsor as constructed in Chapter \ref{chaprojcoo}. Then, the above construction applied to the local ring $R_{\ca P}$ and the $H(R_{\ca P})$-torsor $\omega_4(\ca P)$ gives a structure isomorphic to the projective plane $\ca P$.
\end{lem}

\begin{proof}
For the projective plane $\ca P$, let us write $\ca P_{\text{pt}}$ for its object of points. We write $\ca P_{\text{li}}$ for its object of lines. We write $\ca P_{\#\text{pt}}$ for the subobject of $\ca P_{\text{pt}}\times \ca P_{\text{pt}}$ of points that are apart from each other. We write $\ca P_{\in}$ for the subobject of $\ca P_{\text{pt}}\times \ca P_{\text{li}}$ of pairs $(A,l)$ such that $A\in l$. We write $\ca P_{\notin}$ for the subobject of $\ca P_{\text{pt}}\times \ca P_{\text{li}}$ of pairs $(A,l)$ such that $A\notin l$.

Let $R_{\ca P}$ be the local ring as constructed in Chapter \ref{chaprojcoo}, and given $(A,B,O,I)$ let $\psi_{ABOI}:\mathbb P(R_{\ca P})\to \ca P$ be the projective plane isomorphism as constructed in Chapter \ref{chaprojcoo}. The local ring $R_{\ca P}$ satisfies the conditions of Definition \ref{defnprojring} via the projective plane isomorphisms $\psi_-$.

Recall that the right $H(R_{\ca P})$-torsor is the object $\omega_4(\ca P)$ of quadruples of points of $\ca P$ in general position. The right $H(R_{\ca P})$-action on $\omega_4(\ca P)$ is the following: $h$ acts on $(A,B,O,I)$ by mapping it to $\psi_{ABOI}(h)$ which is the quadruple of the points
$A'=\psi_{ABOI} \begin{pmatrix} h\begin{pmatrix} 1 \\ 0 \\ 0 \end{pmatrix}\end{pmatrix}$,
$B'=\psi_{ABOI} \begin{pmatrix}h\begin{pmatrix} 0 \\ 1 \\ 0 \end{pmatrix}\end{pmatrix}$,
$O'=\psi_{ABOI} \begin{pmatrix} h\begin{pmatrix} 0 \\ 0 \\1 \end{pmatrix}\end{pmatrix}$ and
$I'=\psi_{ABOI} \begin{pmatrix} h\begin{pmatrix} 1 \\ 1 \\ 1 \end{pmatrix}\end{pmatrix}$. 

The epimorphisms we describe below, exhibit the isomorphism from the projective plane $\ca P$ to the one constructed from the local ring $R_{\ca P}$ and the $H(R_{\ca P})$-torsor $\omega_4$ above.

The object of $H_1(R_{\ca P})$-orbits is $\ca P_{\text{pt}}$ via the epimorphism $\omega_4(\ca P)\twoheadrightarrow \ca P_{\text{pt}}$ which sends $(A,B,O,I)$ to the point $A$.

The object of $H_2(R_{\ca P})$-orbits is $\ca P_{\text{pt}}$ via the epimorphism $\omega_4(\ca P)\twoheadrightarrow \ca P_{\text{pt}}$ which sends $(A,B,O,I)$ to the point $B$.

The object of $H_3(R_{\ca P})$-orbits is $\ca P_{\text{pt}}$ via the epimorphism $\omega_4(\ca P)\twoheadrightarrow \ca P_{\text{pt}}$ which sends $(A,B,O,I)$ to the point $O$.

The object of $H_{12}(R_{\ca P})$-orbits is $\ca P_{\#\text{pt}}$ via the epimorphism $H(R_{\ca P})\twoheadrightarrow \ca P_{\#\text{pt}}$ which sends $(A,B,O,I)$ to the pair of points $A$ and $B$. The two projections from $\ca P_{\#\text{pt}}$ to $\ca P_{\text{pt}}$ are the morphisms induced by the group monomorphisms $H_{12}\to H_1$ and $H_{12}\to H_2$.

The object of $H_\text{li}(R_{\ca P})$-orbits is $\ca P_{\text{li}}$ via the epimorphism $\omega_4(\ca P)\twoheadrightarrow \ca P_{\text{li}}$ which sends $(A,B,O,I)$ to the line $\ovl{AB}$.

The object of $H_{\in}(R_{\ca P})$-orbits is $\ca P_{\in}$ via the epimorphism $\omega_4(\ca P)\twoheadrightarrow \ca P_{\in}$ which sends $(A,B,O,I)$ to the pair of the point $A$ and the line $\ovl{AB}$. The projection $\ca P_{\in}\to \ca P_{\text{pt}}$ is the morphism induced by the group monomorphism $H_{\in}\to H_1$ and the projection $\ca P_{\in}\to \ca P_{\text{li}}$ is the morphism induced by the group monomorphism $H_{\in}\to H_{\text{li}}$.

The object of $H_{\notin}(R_{\ca P})$-orbits is $\ca P_{\notin}$ via the epimorphism $\omega_4(\ca P)\twoheadrightarrow \ca P_{\notin}$ which sends $(A,B,O,I)$ to the pair of the point $O$ and the line $\ovl{AB}$. The projection $\ca P_{\notin}\to \ca P_{\text{pt}}$ is the morphism induced by the group monomorphism $H_{\notin}\to H_3$ and the projection $\ca P_{\notin}\to \ca P_{\text{li}}$ is the morphism induced by the group monomorphism $H_{\notin}\to H_{\text{li}}$.
\end{proof}

\begin{lem} \label{lemhah}
The above construction in $\ca Z[H]$, applied to the local ring $(M,\pi_2)$ and the generic $H$-torsor $m_H:(H,m_H)\times (H,\pi_2)\to (H,m_H)$ gives a structure which is isomorphic to the projective plane $(\mathbb P(M),a_{\mathbb P})$.
\end{lem}


\begin{proof}
For the projective plane $(\mathbb P(M),a_{\mathbb P})$, let us write $(\mathbb P_{\#\text{pt}},a_{\#\text{pt}})$ for the subobject of $(\mathbb P_{\text{pt}},a_{\text{pt}})\times (\mathbb P_{\text{pt}},a_{\text{pt}})$ of points that are apart from each other. We write $(\mathbb P_{\#\text{li}},a_{\#\text{li}})$ for the subobject of $(\mathbb P_{\text{li}},a_{\text{li}})\times (\mathbb P_{\text{li}},a_{\text{li}})$ of lines that are apart from each other. We write $(\mathbb P_{\in},a_{\in})$ for the subobject of $(\mathbb P_{\text{pt}},a_{\text{pt}})\times (\mathbb P_{\text{li}},a_{\text{li}})$ of pairs $(A,l)$ such that $A\in l$. We write $(\mathbb P_{\notin},a_{\notin})$ for the subobject of $(\mathbb P_{\text{pt}},a_{\text{pt}})\times (\mathbb P_{\text{li}},a_{\text{li}})$ of pairs $(A,l)$ such that $A\notin l$.

Observe that $(\mathbb P_{\text{pt}}(M),a_{\text{pt}})$ is the object of $H_1(M,\pi_2)$-orbits via the epimorphism $(H,m_H)\twoheadrightarrow (\mathbb P_{\text{pt}}(M),a_{\text{pt}})$ which maps $\begin{pmatrix}
a_0 & b_0 & c_0\\
a_1 & b_1 & c_1 \\
a_2 & b_2 & c_2
\end{pmatrix}$ to the point $(a_0,a_1,a_2)$.

Observe that $(\mathbb P_{\text{pt}}(M),a_{\text{pt}})$ is the object of $H_2(M,\pi_2)$-orbits via the epimorphism $(H,m_H)\twoheadrightarrow (\mathbb P_{\text{pt}}(M),a_{\text{pt}})$ which maps $\begin{pmatrix}
a_0 & b_0 & c_0\\
a_1 & b_1 & c_1 \\
a_2 & b_2 & c_2
\end{pmatrix}$ to the point $(b_0,b_1,b_2)$.

Observe that $(\mathbb P_{\text{pt}}(M),a_{\text{pt}})$ is the object of $H_3(M,\pi_2)$-orbits via the epimorphism $(H,m_H)\twoheadrightarrow (\mathbb P_{\text{pt}}(M),a_{\text{pt}})$ which maps $\begin{pmatrix}
a_0 & b_0 & c_0\\
a_1 & b_1 & c_1 \\
a_2 & b_2 & c_2
\end{pmatrix}$ to the point $(c_0,c_1,c_2)$.

Observe that $ (\mathbb P_{\#\text{pt}}(M),a_{\#\text{pt}})$ is the object of $H_{12}(M,\pi_2)$-orbits via the epimorphism $(H,m_H)\twoheadrightarrow (\mathbb P_{\#\text{pt}}(M),a_{\#\text{pt}})$ which maps $\begin{pmatrix}
a_0 & b_0 & c_0\\
a_1 & b_1 & c_1 \\
a_2 & b_2 & c_2
\end{pmatrix}$ to the pair of points $(b_0+c_0,b_1+c_1)$ and $(c_0,c_1)$. The two projections from $(\mathbb P_{\#\text{pt}}(M),a_{\#\text{pt}})$ to $ (\mathbb P_{\text{pt}}(M),a_{\text{pt}})$ are the ones induced by the group monomorphisms $H_{12}\to H_1$ and $H_{12}\to H_2$.

Observe that $(\mathbb P_{\text{li}}(M),a_{\text{li}})$ is the object of $H_{\text{li}}(M,\pi_2)$-orbits via the epimorphism $(H,m_H)\twoheadrightarrow (\mathbb P_{\text{li}}(M),a_{\text{li}})$ which maps $\begin{pmatrix}
a_0 & b_0 & c_0\\
a_1 & b_1 & c_1 \\
a_2 & b_2 & c_2
\end{pmatrix}$ to the line $(a_1b_2-a_2b_1,a_2b_0-a_0b_2,a_0b_1-a_1b_0)$.

Observe that $(\mathbb P_{\text{li}}(M),a_{\text{li}})$ is the object of $H_{\text{li}'}(M,\pi_2)$-orbits via the epimorphism $(H,m_H)\twoheadrightarrow (\mathbb P_{\text{li}}(M),a_{\text{li}})$ which maps $\begin{pmatrix}
a_0 & b_0 & c_0\\
a_1 & b_1 & c_1 \\
a_2 & b_2 & c_2
\end{pmatrix}$ to the line $(a_1c_2-a_2c_1,a_2c_0-a_0c_2,a_0c_1-a_1c_0)$.

Observe that $(\mathbb P_{\#\text{li}}(M),a_{\#\text{li}})$ is the object of $H_{\#\text{li}}(M,\pi_2)$-orbits via the epimorphism $(H,m_H)\twoheadrightarrow (\mathbb P_{\text{li}}(M),a_{\text{li}})$ which maps $\begin{pmatrix}
a_0 & b_0 & c_0\\
a_1 & b_1 & c_1 \\
a_2 & b_2 & c_2
\end{pmatrix}$ to the pair of lines $(a_1b_2-a_2b_1,a_2b_0-a_0b_2,a_0b_1-a_1b_0)$ and $(a_1c_2-a_2c_1,a_2c_0-a_0c_2,a_0c_1-a_1c_0)$ which are apart from each other. The two projections from $(\mathbb P_{\#\text{li}}(M),a_{\#\text{li}})$ to $ (\mathbb P_{\text{li}}(M),a_{\text{li}})$ are the ones induced by the group monomorphisms $H_{\#\text{li}}\to H_{\text{li}}$ and $H_{\#\text{li}}\to H_{\text{li}'}$.

Observe that $(\mathbb P_{\in}(M),a_{\in})$ is the object of $H_{\in}(M,\pi_2)$-orbits via the epimorphism $(H,m_H)\twoheadrightarrow (\mathbb P_{\in}(M),a_{\in})$ which maps $\begin{pmatrix}
a_0 & b_0 & c_0\\
a_1 & b_1 & c_1 \\
a_2 & b_2 & c_2
\end{pmatrix}$ to the pair of the point $(a_0,a_1,a_2)$ and the line $(a_1b_2-a_2b_1,a_2b_0-a_0b_2,a_0b_1-a_1b_0)$. The projection $(\mathbb P_{\in}(M),a_{\in})\to (\mathbb P_{\text{pt}}, a_{\text{pt}})$ is the morphism induced by the group monomorphism $H_{\in}(M,\pi_2)\to H_1(M,\pi_2)$ and the projection $(\mathbb P_{\in}(M),a_{\in})\to (\mathbb P_{\text{li}}, a_{\text{li}})$ is the morphism induced by the group monomorphism $H_{\in}(M,\pi_2)\to H_{\text{li}}(M,\pi_2)$.

Observe that $(\mathbb P_{\notin}(M),a_{\notin})$ is the group of $H_{\notin}(M,\pi_2)$-orbits via the epimorphism $(H,m_H)\twoheadrightarrow (\mathbb P_{\notin}(M),a_{\notin})$ which maps $\begin{pmatrix}
a_0 & b_0 & c_0\\
a_1 & b_1 & c_1 \\
a_2 & b_2 & c_2
\end{pmatrix}$ to the pair of the point $(c_0,c_1,c_2)$ and the line $(a_1b_2-a_2b_1,a_2b_0-a_0b_2,a_0b_1-a_1b_0)$. Notice that the projection $(\mathbb P_{\notin}(M),a_{\notin})\to (\mathbb P_{\text{pt}}, a_{\text{pt}})$ is the morphism induced by the group monomorphism $H_{\notin}(M,\pi_2)\to H_3(M,\pi_2)$ and the projection $(\mathbb P_{\notin}(M),a_{\notin})\to (\mathbb P_{\text{li}}, a_{\text{li}})$ is the morphism induced by the group monomorphism $H_{\notin}(M,\pi_2)\to H_{\text{li}}(M,\pi_2)$.
\end{proof}

Let us now return to the construction from the beginning of the section. The local ring $R$ in $\ca E$ corresponds to a geometric morphism $f:\ca E\to \ca Z$ such that $f^*(M)$ is isomorphic to $R$. By Diaconescu's theorem, the right $H(R)$-torsor (or equivalently the right $f^*(H)$-torsor) $a_A:A\times H(R)\to A$ corresponds to a factorization $f':\ca E \to \ca Z[H]$ of $f$ through the geometric morphism $\ca Z[H]\to \ca Z$ which is such that $f'^*$ maps the generic $H$-torsor to the $H(R)$-torsor $a_A$. Since, the construction of the projective plane structure is preserved by inverse images of geometric morphisms it means that the construction for the local ring $R$ and the $H(R)$-torsor $a_A$ is isomorphic to $f'^*(\mathbb P(M),a_{\mathbb P})$ or equivalently to $f'^*f_{\mathbb P}^*(\ca Pg)$. Hence, since the inverse image of a projective plane is a projective plane, we have the following:

\begin{thrm}
The structure constructed in the beginning of the section is a projective plane.
\end{thrm}

\section{Geometric morphisms over $\ca Z$}

In Chapter \ref{chaprojcoo}, we construct a local ring from a projective plane $\ca P$ in a topos. This construction applied to the generic projective plane $\ca Pg$ in $\mathbf{Proj}$ gives the local ring $R_{\ca Pg}$. The local ring $R_{\ca Pg}$ corresponds to a geometric morphism 
$$R_{\mathbb P}: \mathbf{Proj}\to \ca Z$$
such that $R_{\mathbb P}^*(M)$ is isomorphic to $R_{\ca Pg}$. Thus, $\mathbf{Proj}$ is a $\ca Z$-topos via $R_{\mathbb P}$. $\ca Z[H]$ is also a $\ca Z$-topos via the geometric morphism induced by the group homomorphism $H\to 1$. We use the construction from Lemma \ref{lemprojtors} to construct an $H(R_{\ca Ag})$-torsor in $\mathbf{Proj}$. By Diaconescu's theorem, this $H(R_{\ca Pg})$-torsor corresponds to a geometric morphism 
$$f_H: \mathbf{Proj}\to \ca Z[H]$$
over $\ca Z$.

We wish to show that the geometric morphism $f_{\mathbb P}: \ca Z[H]\to \mathbf{Proj}$ is also a geometric morphism over $\ca Z$. Equivalently, we wish to show that the local ring $f_{\mathbb P}^*(R_{\ca Pg})$ is isomorphic to $(M,\pi_2)$. $f_{\mathbb P}^*(R_{\ca Pg})$ is a coordinate ring of $\ca Pg$. Therefore, by Lemma \ref{lemprojunique} to construct a ring isomorphism $f_{\mathbb P}^*(R_{\ca Pg})\to (M,\pi_2)$ it suffices to prove that the local ring $(M,\pi_2)$ also satisfies these properties.

In $\ca Z[H]$, let $(\omega_4(M),a_{\omega4})$ be the object of quadruples of points in general position of $(\mathbb P(M),a_{\mathbb P})$. There is an isomorphism $(H,m_H)\to (\omega_4(M),a_{\omega4})$  which maps an element of $H$ represented by a matrix $h$ to the quadruple of points $h(1,0,0)$, $h(0,1,0)$, $h(0,0,1)$ and $h(1,1,1)$. This morphism commutes with the two left $H$-actions. Its inverse is the morphism which sends a quadruple of points $(A,B,O,I)$ in $\omega_4(M)$ to the unique element of $H$ which maps the points $(1,0,0)$, $(0,1,0)$, $(0,0,1)$ and $(1,1,1)$ to $A$, $B$, $O$ and $I$. Henceforth, we shall identify $(H,m_H)$ with $(\omega_4(M),a_{\omega4})$ via this isomorphism.

For each element of $(H,m_H)$ we have an isomorphism from the projective plane over the local ring $(M,\pi_2)$ to $(\mathbb P(H), a_{\mathbb P})$ via the following two morphisms:
$$a_{\text{pt}}: (H,m_H)\times (\mathbb P_{\text{pt}}(M),\pi_2)\to (\mathbb P_{\text{pt}}(M),a_{\text{pt}}).$$

$$a_{\text{li}}: (H,m_H)\times (\mathbb P_{\text{li}}(M),\pi_2)\to (\mathbb P_{\text{li}}(M),a_{\text{li}}).$$

For each quadruple $(A,B,O,I)$ of points in $(\omega_4(M),a_{\omega4})$ (or equivalently of $(H,m_H)$), these two arrows give an isomorphism of projective planes. This isomorphism maps $(1,0,0)$ to $A$ since the composite 
$$(\omega_4(M),a_{\omega4}) \cong (H,m_H)\times (1,!)\xrightarrow{1_H\times (1,0,0)} (H,m_H)\times (\mathbb P_{\text{pt}}(M),\pi_2)\xrightarrow{a_{\text{pt}}} (\mathbb P_{\text{pt}}(M),a_{\text{pt}})$$
maps $(A,B,O,I)$ to $A$. Similarly, this projective plane isomorphism maps $(0,1,0)$, $(0,0,1)$ and $(1,1,1)$ to $B$, $O$ and $I$ respectively.

The inverses of this isomorphisms are given by the following two morphisms: 
$$b_{\text{pt}}: (H,m_H)\times (\mathbb P_{\text{pt}}(M),a_{\text{pt}}) \to (\mathbb P_{\text{pt}}(M),\pi_2),$$
which maps the pair $h$, $\pt{x}$ to $a_{\text{pt}}(h^{-1},\pt{x})$, and 
$$b_{\text{li}}: (H,m_H)\times (\mathbb P_{\text{li}}(M),a_{\text{li}})\to (\mathbb P_{\text{li}}(M),\pi_2),$$
which maps the pair $h$, $\pt{\lambda}$ to $a_{\text{li}}(h^{-1},\pt{\lambda})$.

Adopting the notation from Chapter \ref{chaprojcoo}, given $(A,B,O,I)$ in $(\omega_4(M),a_{\omega4})$ corresponding to $h$ in $(H,m_H)$, we write $\chi_{ABOI}(X)$ for $a_{\text{pt}}(h, X)$ and $\chi_{ABOI}^{-1}(X)$ for $b_{\text{pt}}(h, X)$.

Let us consider the $(\omega_4(M),a_{\omega4})\times (\omega_4(M),a_{\omega4})$-indexed projective plane isomorphism given by 
$$(\omega_4(M),a_{\omega4})\times (\omega_4(M),a_{\omega4})\times (\mathbb P(M),a_{\mathbb P})\to (\mathbb P(M),a_{\mathbb P}),$$
which maps the triple $(A',B',O',I')$, $(A,B,O,I)$, $X$ to $\chi_{A'B'O'I'}^{-1}\chi_{ABOI}(X)$. Via the isomorphism $(\omega_4(M),a_{\omega4})\cong (H,m_H)$ this becomes the morphism
$$\tilde{\chi}:(H,m_H)\times (H,m_H)\times (\mathbb P(M),a_{\mathbb P})\to (\mathbb P(M),a_{\mathbb P})$$
where $(h,k)$ in $(H,m_H)\times (H,m_H)$ induces the projective plane isomorphism $h^{-1}k$ of $(H,\pi_2)$.
Consider the morphism $t:(H,m_H)\times (H,m_H)\to (H,\pi_2)$ mapping $(h,k)$ to $h^{-1}k$. Then, the triangle 
\begin{displaymath}
\xymatrix{
(H,m_H)\times (H,m_H) \times (\mathbb P(M),a_{\mathbb P}) \ar[rr]^-{\tilde{\chi}} \ar[d]^{t\times 1_{\mathbb P(M)}} & & \mathbb P(M) \\
(H,\pi_2)\times (\mathbb P(M),a_{\mathbb P}) \ar[urr]_{a_{\mathbb P}}
}
\end{displaymath}
commutes. Hence, all the conditions of Definition \ref{defnprojring} are satisfied for the local ring $(M,\pi_2)$, and therefore it is a coordinate ring for $(\mathbb P(M),a_{\mathbb P})$.

Hence, by Lemma \ref{lemprojunique} $(M,\pi_2)$ is isomorphic to $f_{\mathbb P}^*(R_{\ca Pg})$, and therefore we have proved the following:

\begin{lem}
The geometric morphism $\ca Z[H] \to \mathbf{Proj}$ is a geometric morphism over $\ca Z$.
\end{lem}

\section{$\ca Z[H]\xrightarrow{f_{\mathbb P}} \mathbf{Proj}\xrightarrow{f_H} \ca Z[H]$ is isomorphic to the identity} \label{sechid}

By the results of the previous section, $\ca Z[H]\xrightarrow{f_{\mathbb P}} \mathbf{Proj}\xrightarrow{f_H} \ca Z[H]$ is a geometric morphism over $\ca Z$. By Diaconescu's theorem to prove that this composite is isomorphic to the identity it suffices to show that $f_{\mathbb P}^*f_H^*$ sends the generic $H$-torsor to an $H$-torsor isomorphic to the generic one.

The inverse image of $f_Hf_{\mathbb P}$ sends the generic $H$-torsor to the torsor $(\omega_4(M), a_{\omega4})$ with the right $(H,\pi_2)$-action described in Lemma \ref{lemprojtors}: an element $h$ of $(H,\pi_2)$ acts on $(A,B,O,I)$ by mapping it to the quadruple of points $\chi_{ABOI}(h(1,0,0))$, $\chi_{ABOI}(h(0,1,0))$, $\chi_{ABOI}(h(0,0,1))$, $\chi_{ABOI}(h(1,1,1))$. Via the isomorphism from $(H,m_H)$ to $(\omega_4(M),a_{\omega4})$, this is the right $(H,\pi_2)$-action on $(H,m_H)$ via right group multiplication. Hence the $(H,\pi_2)$-torsor corresponding to the geometric morphism $f_Hf_{\mathbb P}$ is isomorphic to the generic $H$-torsor of $\ca Z[H]$. Therefore, the geometric morphism $f_Hf_{\mathbb P}$ is isomorphic to the identity.

\section{$\mathbf{Proj}\xrightarrow{f_H} \ca Z[H] \xrightarrow{f_{\mathbb P}} \mathbf{Proj}$ is isomorphic to the identity}

The geometric morphism $f_{\mathbb P}f_H$ corresponds to a projective plane in $\mathbf{Proj}$. Let $\ca Pg$ be the generic affine plane in $\mathbf{Proj}$ and let $\omega_4$ be the object of quadruples of points in general position. The geometric morphism $f_H$ corresponds to the pair of the local ring $R_{\ca Pg}$ and the right $H(R_{\ca Pg})$-torsor $\omega_4(\ca Pg)$ in $\mathbf{Proj}$ as constructed in Chapter \ref{chaprojcoo}.

Hence, the projective plane corresponding to the composite $\mathbf{Proj}\xrightarrow{f_H} \ca Z[H] \xrightarrow{f_{\mathbb P}} \mathbf{Proj}$ is isomorphic to the projective plane constructed from the pair of the local ring $R_{\ca Pg}$ and the right $H(R_{\ca Pg})$-torsor $\omega_4(\ca Pg)$ as in Section \ref{secGtoproj}. By Lemma \ref{lemhph}, this new projective plane is isomorphic to $\ca P$. Therefore, the geometric  morphism is isomorphic to the identity.

We have already seen in Section \ref{sechid}  that $f_Hf_{\mathbb P}$  is isomorphic to the identity, therefore we conclude the following theorem.

\begin{thrm}
The classifying topos for the theory of projective planes $\mathbf{Proj}$ is equivalent to the topos $\ca Z[H]$.
\end{thrm}

\chapter{Geometric morphisms between $\ca Z$, $\mathbf{Aff}$ and $\mathbf{Proj}$} \label{chageom}

In this chapter, we conclude the thesis by describing geometric morphisms between the various toposes we have considered. We have identified our leading toposes as extensions of the Zariski topos by the groups $G$ and $H$ and we describe the geometric morphisms in these terms.

\section{Overview}

Let $\ca Z$ be the Zariski topos, and let $G$ and $H$ be the group objects in $\ca Z$ as defined earlier. Note that $G$ is a subgroup of $H$. $\mathbf{Aff}\simeq \ca Z[G]$ and $\mathbf{Proj}\simeq \ca Z[H]$. Let $1$ be the trivial group in the Zariski topos. Then, $\ca Z$ is equivalent to $\ca Z[1]$. $1$ is a subgroup of both $G$ and $H$ and we have unique group homomorphisms from both $G$ and $H$ to $1$. Hence, we have the following group homomorphisms:

\begin{displaymath}
\xymatrix{
1\ar[r] &  G \ar[r] & H \ar[r] & 1.
}
\end{displaymath}

These induce the following geometric morphisms (where the vertical arrows are the equivalences mentioned above):
\begin{displaymath}
\xymatrix{
\ca Z[1]\ar[r] \ar[d]^{\simeq} &  \ca Z[G] \ar[r] \ar[d]^{\simeq} & \ca Z[H] \ar[r] \ar[d]^{\simeq} & \ca Z[1] \ar[d]^{\simeq} \\
\ca Z \ar[r] & \mathbf{Aff} \ar[r] & \mathbf{Proj}  \ar[r] & \ca Z .
}
\end{displaymath}

Moreover, there is a group isomorphism $d:H\to H$ which sends a matrix to the transpose of its inverse  and that induces a geometric morphism $\ca Z[H]\to \ca Z[H]$ and therefore also a geometric morphism $\mathbf{Proj}\to \mathbf{Proj}$.

\section{$\ca Z$ as a slice of the topos $\mathbf{Aff}$}

In $\ca Z$, there is a unique group monomorphism from the trivial group to the group $G$. This induces a geometric morphism over $\ca Z$ from $\ca Z[1]\simeq \ca Z$ to $\ca Z[G]$. By Theorem \ref{thrmlocal}, this geometric morphism $\ca Z\to \ca Z[G]$ is a local homeomorphism. In particular, we have a geometric isomorphism $\ca Z\to \ca Z[G]/(G,m_G)$ such that the triangle 
\begin{displaymath}
 \xymatrix{
\ca Z\ar[r] \ar[d] & \ca Z[G]/(G,m_G) \ar[dl]\\
\ca Z[G]
}
\end{displaymath}
commutes.

Via the equivalence $\ca Z[G]\simeq \mathbf{Aff}$ from Chapter \ref{chaZG}, $(G,m_G)$ corresponds to the object $\omega$ of triples of non-collinear points of the generic affine plane in $\mathbf{Aff}$. Hence, we have an equivalence $\mathbf{Aff}/\omega\simeq \ca Z$. By \cite[B3.2.8(b)]{Elephant1}, $\mathbf{Aff}/\omega$ is the classifying topos for the theory of affine planes with three added constants of sorts of points, and an axiom stating that these three points are non-collinear. Hence, the theory of local rings is Morita equivalent to the theory of affine planes with three chosen points which are non-collinear.

Via the equivalence $\mathbf{Aff}/\omega \simeq \ca Z$, the generic local ring in $\ca Z$ corresponds to the local ring of trace preserving homomorphisms of the generic affine plane in $\mathbf{Aff}$. Via the same equivalence, the generic affine plane with a choice of three non-collinear points corresponds to the affine plane over the generic local ring $M$ in $\ca Z$ with the choice of the three non-collinear points $(0,0)$, $(1,0)$ and $(0,1)$. Notice that by Lemma \ref{lemmatrixautoaf} any choice of three non-collinear points would have given an isomorphic model of the theory of affine planes with a choice of three non-collinear points.

We have already seen the geometric morphism $\mathbf{Aff}\to \ca Z$. By the equivalence $\ca Z \simeq \mathbf{Aff}/\omega$, we have a geometric morphism $\mathbf{Aff}\to \mathbf{Aff}/\omega$. Given an affine plane $\ca A$ in a topos $\ca E$, we have a (unique up to isomorphism) geometric morphism $f:\ca E\to \mathbf{Aff}$ such that $\ca A$ is isomorphic to $f^*(\ca Ag)$. By postcomposing $f$ with the above geometric morphism $\mathbf{Aff}\to \mathbf{Aff}/\omega$, and taking the inverse image of generic affine plane with a choice of three non-collinear points, we construct an affine plane with a choice of three non-collinear points in $\ca E$. This new affine plane is the affine plane over the local ring of trace preserving homomorphisms of $\ca A$ and the choice of three non-collinear points is $(0,0)$, $(1,0)$ and $(0,1)$ (and as before any choice of three non-collinear points would have given an isomorphic affine plane with a choice of a triple of non-collinear points).

\section{$\ca Z$ as a slice of the topos $\mathbf{Proj}$}

In $\ca Z$, there is a unique group monomorphism from the trivial group to the group $H$. This induces a geometric morphism over $\ca Z$ from $\ca Z[1]\simeq \ca Z$ to $\ca Z[H]$. By Theorem \ref{thrmlocal}, this geometric morphism $\ca Z\to \ca Z[H]$ is a local homeomorphism. In particular, we have a geometric isomorphism $\ca Z\to \ca Z[H]/(H,m_H)$ such that the triangle 
\begin{displaymath}
 \xymatrix{
\ca Z\ar[r] \ar[d] & \ca Z[H]/(H,m_H) \ar[dl]\\
\ca Z[H]
}
\end{displaymath}
commutes.

Via the equivalence $\ca Z[H]\simeq \mathbf{Proj}$ from Chapter \ref{chaZH}, $(H,m_H)$ corresponds to the object $\omega_4$ of quadruples of points in general position of the generic projective plane $\ca Pg$ in $\mathbf{Proj}$. Hence, we have an equivalence $\mathbf{Proj}/\omega_4\simeq \ca Z$. By \cite[B3.2.8(b)]{Elephant1}, $\mathbf{Proj}/\omega_4$ is the classifying topos for the theory of projective planes with four added constants for points, and an axiom stating that they are in general position. Hence, the theory of local rings is Morita equivalent to the theory of projective planes with a choice of a quadruple of points in general position.

Via the equivalence $\mathbf{Proj}/\omega_4 \simeq \ca Z$, the generic local ring in $\ca Z$ corresponds to the local ring constructed from the generic projective plane of $\mathbf{Proj}$ as in Chapter \ref{chaprojcoo}. Via the same equivalence, the generic projective plane with a choice of an element of $\omega_4$ corresponds to the projective plane over the generic local ring $M$ in $\ca Z$ with the choice of the quadruple of points $(1,0,0)$, $(0,1,0)$, $(0,0,1)$ and $(1,1,1)$. Notice that by Lemma \ref{lemmatrixauto} any other choice of a quadruple of points in general position would have given an isomorphic model of the theory of projective planes with a choice of a quadruple of points in general position.

We have already seen the geometric morphism $\mathbf{Proj}\to \ca Z$. By the equivalence $\ca Z \simeq \mathbf{Proj}/\omega_4$, we have a geometric morphism $\mathbf{Proj}\to \mathbf{Proj}/\omega_4$. Given a projective plane $\ca P$ in a topos $\ca E$, we have a (unique up to isomorphism) geometric morphism $f:\ca E\to \mathbf{Proj}$ such that $\ca P$ is isomorphic to $f^*(\ca Pg)$. By postcomposing $f$ with the above geometric morphism $\mathbf{Proj}\to \mathbf{Proj}/\omega_4$, and taking the inverse image of the generic projective plane with a choice of quadruple of points in general position, we construct a projective plane with such a choice of four points in $\ca E$. This new projective plane is the projective plane over the local ring of coordinates of $\ca P$ and the choice of the four points is $(1,0,0)$, $(0,1,0)$, $(0,0,1)$ and $(1,1,1)$ (and as before any choice a quadruple of points in general position would have given an isomorphic projective plane with such a choice of points).

\section{$\mathbf{Aff}\to \mathbf{Proj}$}

Consider the group monomorphism $\theta: G\to H$ in $\ca Z$. It induces a geometric morphism $\ca Z[G]\to \ca Z[H]$. This group homomorphism is injective, and therefore by Theorem \ref{thrmlocal} the induced geometric morphism is a local homeomorphism. By going through the proof of \ref{thrmlocal}, we see that $\ca Z[G]$ is equivalent to $\ca Z[H]/K$ where $K$ is the coequalizer in $\ca Z[H]$ of the diagram: $(H\times G,m_H\times 1_G) \rightrightarrows (H,m_H)$ where the two arrows are $m_H(1_H\times \theta)$ and $\pi_1$. Equivalently, $K$ is the object of $(G,\pi_2)$-orbits of $(H,m_H)$. Notice that $(G,\pi_2)$ is the group $H_{\text{li}}(M,\pi_2)$ and as mentioned in the proof of Lemma \ref{lemhah}, the object of $H_{\text{li}}(M,\pi_2)$-orbits is the object $(\mathbb P_{\text{li}},a_{\text{li}})$ of lines of the generic projective plane in $\ca Z[H]$.

By Theorem \ref{thrmlocal} and the above discussion, we have a geometric isomorphism $\ca Z[G]\to \ca Z[H]/(\mathbb P_{\text{li}}(M),a_{\text{li}})$ such that the triangle:
\begin{displaymath}
 \xymatrix{
\ca Z[G]\ar[r] \ar[d]^{\theta} & \ca Z[H]/(\mathbb P_{\text{li}}(M),a_{\text{li}}) \ar[dl]\\
\ca Z[H]
}
\end{displaymath}
commutes (up to isomorphism). Via the equivalences $\ca Z[G]\simeq \mathbf{Aff}$ and $\ca Z[H]\simeq  \mathbf{Proj}$, the geometric isomorphism $\ca Z[G]\to \ca Z[H]/(\mathbb P_{\text{li}}(M),a_{\text{li}})$ becomes a geometric isomorphism $\mathbf{Aff}\to \mathbf{Proj}/\ca P_{\text{li}}$ where $\ca P_{\text{li}}$ is the line object of the generic projective plane in $\mathbf{Proj}$. Moreover, this geometric morphism is such that the triangle
\begin{displaymath}
 \xymatrix{
\mathbf{Aff}\ar[d] \ar[r] & \mathbf{Proj}/\ca P_{\text{li}} \ar[dl]\\
\mathbf{Proj} 
}
\end{displaymath}
commutes.

\begin{thrm}
The theory of affine planes is Morita equivalent to the theory of projective planes with a choice of line.
\end{thrm}

\begin{proof}
By \cite[B3.2.8(b)]{Elephant1}, $\mathbf{Proj}/\ca P_{\text{li}}$ is the classifying topos for the theory of projective planes with an added constant for the sort of lines. By the above discussion $\mathbf{Aff}\simeq \mathbf{Proj}/\ca P_{\text{li}}$, hence the result.
\end{proof}

The construction of an affine plane from a projective plane with a line corresponding to the equivalence $\mathbf{Aff}\simeq \mathbf{Proj}/\ca P_{\text{li}}$ is the one described in Section \ref{secprojtoaff}.

The construction of a projective plane with a line from an affine plane is more complicated. Using the classifying toposes and the geometric morphism $\ca Z[G]\to \ca Z[H]$, we can see that we can construct the points of the projective plane as a quotient of pairs of lines of the affine plane that are apart from each other. This is also the case in classical geometric algebra and in particular in the construction of a projective plane from an affine plane in \cite{Harts}.

\section{$\mathbf{Aff}\to \mathbf{Aff}_{\text{Pt}}$}

We have already seen in Section \ref{secaffpt} that the classifying topos $\mathbf{Aff}_{\text{Pt}}$ for the theory of affine planes with a chosen point is the topos $\ca Z[G_3]$. In Lemma \ref{lemplanetrans}, we have proved that the object of translations can be viewed as the object of points of an affine plane with a chosen point. This construction of an affine plane with a chosen point from an affine plane corresponds to a geometric morphism $\mathbf{Aff}\to \mathbf{Aff}_{\text{Pt}}$. Via the equivalences $\mathbf{Aff}\simeq \ca Z[G]$ and $\mathbf{Aff}_{\text{Pt}}\simeq \ca Z[G_3]$, this geometric morphism corresponds to a geometric morphism $\ca Z[G]\to \ca Z[G_3]$.

\begin{thrm}
The geometric morphism $\ca Z[G]\to \ca Z[G_3]$ described above is isomorphic to the geometric morphism $k:\ca Z[G]\to \ca Z[G_3]$ induced by the group homomorphism $k:G\to G_3$ which sends $\begin{pmatrix}
\alpha_0 & \beta_0 & \gamma_0 \\
\alpha_1 & \beta_1 & \gamma_1 \\
0 & 0 & 1
\end{pmatrix}$ to $\begin{pmatrix}
\alpha_0 & \beta_0 & 0 \\
\alpha_1 & \beta_1 & 0 \\
0 & 0 & 1
\end{pmatrix}$.
\end{thrm}

\begin{proof}
In $\ca Z[G]$, let $(\mathbb A_T(M),a_{\mathbb AT})$ be the affine plane of translations constructed as in Lemma \ref{lemplanetrans} from the generic affine plane $(\mathbb A(M),a_{\mathbb A})$ of $\ca Z[G]$ (and notice that it comes with a choice of a point which is the unit of the group of translations).

Also in $\ca Z[G]$, let $(\mathbb A_k(M),a_{\mathbb Ak})$ be the image under $k^*$ of the generic affine plane with a chosen point of $\ca Z[G_3]$.

To prove that the two geometric morphisms are isomorphic it suffices to show that these two affine planes with their choices of points are isomorphic.

The object of points of $(\mathbb A_T(M),a_{\mathbb AT})$ is the object of translations of $(\mathbb A(M),a_{\mathbb A})$. In Section \ref{secaffoverZ}, we show that this object of translations is the object $\{(a_0,a_1,0):M^3\}$ whose left $G$-action is left matrix multiplication.

The object of points of $(\mathbb A_k(M),a_{\mathbb Ak})$ is the object $\{(a_0,a_1,1):M^3\}$ with an element  $\begin{pmatrix}
\alpha_0 & \beta_0 & \gamma_0 \\
\alpha_1 & \beta_1 & \gamma_1 \\
0 & 0 & 1
\end{pmatrix}$ of $G$ acting on it via left matrix multiplication by $\begin{pmatrix}
\alpha_0 & \beta_0 & 0 \\
\alpha_1 & \beta_1 & 0 \\
0 & 0 & 1
\end{pmatrix}$.

These objects of points are isomorphic via the morphism which maps $(a_0,a_1,0)$ to $(a_0,a_1,1)$ (which commutes with the two actions). By going through the construction of the affine plane of translations in Lemma \ref{lemplanetrans}, we show that this isomorphism can be extended to an isomorphism of affine planes. Hence, these two affine planes are isomorphic, and therefore the two geometric morphisms are isomorphic.
\end{proof}

\section{Duality of the projective plane}

Notice that the theory of projective planes is self-dual.  Moreover, projective planes over local rings are isomorphic to their duals. However, that is not the case for all projective planes. We explain this subtle point in this final section.

Consider the group isomorphism $d:H\to H$ which sends a matrix to the transpose of its inverse. Notice that $d\circ d=1_H$. The isomorphism $d$ induces a functor $\ca Z[H]\to \ca Z[H]$ which sends an $H$-object $a:H\times A\to A$ to $H\times A\xrightarrow{d\times 1_A} H\times A \xrightarrow{a} A$. This functor is self-adjoint and therefore it is the inverse and direct image of a geometric morphism 
$$d:\ca Z[H]\to \ca Z[H].$$
Notice that the geometric morphism $d\circ d$ is the identity on $\ca Z[H]$.

The inverse image of $d$ sends the generic projective plane to its dual plane. The dual plane is not isomorphic to the generic projective plane. In particular, we have the following:

\begin{thrm}
The object of points of the generic projective plane is not isomorphic to its object of lines, and moreover there is no morphism from the object of points to the object of lines. Dually, there is no morphism from the object of lines to the object of points.
\end{thrm}

\begin{proof}
Consider the object of vectors
$\begin{pmatrix}
a_0 \\ a_1 \\ a_2
\end{pmatrix}$
over the generic local ring in $\ca Z$
where at least one of $a_0$, $a_1$, $a_2$ is invertible. We now consider the quotient of the above by the equivalence relation which relates a vector with a scalar multiple of it by (an invertible) element of the generic local ring and we denote this object of $\ca Z$ by $A$. We equip $A$ with two left $H$-actions. The first one is left matrix multiplication and this gives the object of points of the generic projective plane in $\ca Z[H]$. The second one is left matrix multiplication by the transpose of the inverse of matrices of $H$ and it gives the object of lines of the projective plane in $\ca Z[H]$.

A morphism from the object of points to the object of lines is a morphism $A\to A$ which commutes with the two $H$-actions. Let $f$ be such a morphism. $f$ sends the element of $A$ represented by $(1,0,0)$ to an element represented by $\pt{b}$. For a matrix $h$ representing an element of $H$, the commutativity of the actions implies that 
$$f \begin{pmatrix}h \begin{pmatrix}
1 \\ 0 \\ 0
\end{pmatrix}
\end{pmatrix}= r(h^{-1})^T
\begin{pmatrix}
b_0 \\ b_1 \\ b_2
\end{pmatrix}$$
for an invertible scalar $r$. At least one of $b_0$, $b_1$ and $b_2$ is invertible, therefore by symmetry let us suppose that $b_0$ is invertible.

In what follows we shall only consider matrices $h$ such that $h\begin{pmatrix}
1 \\ 0 \\ 0
\end{pmatrix}=\begin{pmatrix}
1 \\ 0 \\ 0
\end{pmatrix}$, hence the above equation becomes  $$\begin{pmatrix}
b_0 \\ b_1 \\ b_2
\end{pmatrix}=
r(h^{-1})^T
\begin{pmatrix}
b_0 \\ b_1 \\ b_2
\end{pmatrix}.$$

Let $h$ be $\begin{pmatrix}
1 & 0 &0 \\
0 & 0 &1 \\
0 & 1 & 0 
\end{pmatrix}$. Then, $(h^{-1})^T = h$ and by the commutativity of the actions we can see that
$\begin{pmatrix}
b_0 \\ b_1 \\ b_2
\end{pmatrix}$ is a multiple of
$(h^{-1})^T \mathbf{b}=\begin{pmatrix}
b_0 \\ b_2 \\ b_1
\end{pmatrix}$. $b_0$ is invertible, hence $b_1=b_2$.

By letting $h$ be $\begin{pmatrix}
1 & 0 &0 \\
0 & 1 & 0 \\
0 & -1 & 1 
\end{pmatrix}$ and since $(h^{-1})^T = \begin{pmatrix}
1 & 0 &0 \\
0 & 1 & 1 \\
0 & 0 & 1 
\end{pmatrix}$, we conclude that
$\begin{pmatrix}
b_0 \\ b_1 \\ b_2
\end{pmatrix}$ is a multiple of
$\begin{pmatrix}
b_0 \\ b_1+b_2 \\ b_2
\end{pmatrix}$. $b_0$ is invertible, hence $b_1+b_2=b_1$. Therefore, $b_1=0$ and by the above $b_2$ is also $0$.

For $h=\begin{pmatrix}
1 & 1 &0 \\
0 & 1 & 0 \\
0 & 0 & 1 
\end{pmatrix}$, $(h^{-1})^T=\begin{pmatrix}
1 & 0 &0 \\
-1 & 1 & 0 \\
0 & 0 & 1 
\end{pmatrix}$. Therefore,
$\begin{pmatrix}
b_0 \\ b_1 \\ b_2
\end{pmatrix}$ is a multiple of
$\begin{pmatrix}
b_0 \\ -b_0 + b_1 \\ b_2
\end{pmatrix}$
and since $b_0$ is invertible we conclude that $b_1=-b_0+b_1$. Hence $b_0=0$ giving a contradiction.

Thus, there exists no morphism from the object of points to the object of lines.
\end{proof}

An alternative way to see that there is no isomorphism from the object of points to the object of lines is by considering the model of projective planes in $\ca Z[G]$ corresponding to the geometric morphism $\ca Z[G]\to\ca Z[H]\simeq \mathbf{Proj}$ (induced by the group monomorphism $G\to H$). The object of lines of this projective plane in $\ca Z[G]$ has a global section but its object of points does not. Therefore, they are not isomorphic. Hence, there is a projective plane whose object of points is not isomorphic to its object of lines. Thus, the object of points of the generic projective plane is not isomorphic its object of lines.

\backmatter


\bibliographystyle{alpha}
\bibliography{refthesis}

\end{document}

%% file: coverpage.tex
\begin{center}
\mbox{}
{\huge{\bfseries A constructive approach to affine and projective planes}}
\vskip 1.2in
\includegraphics[height=2in]{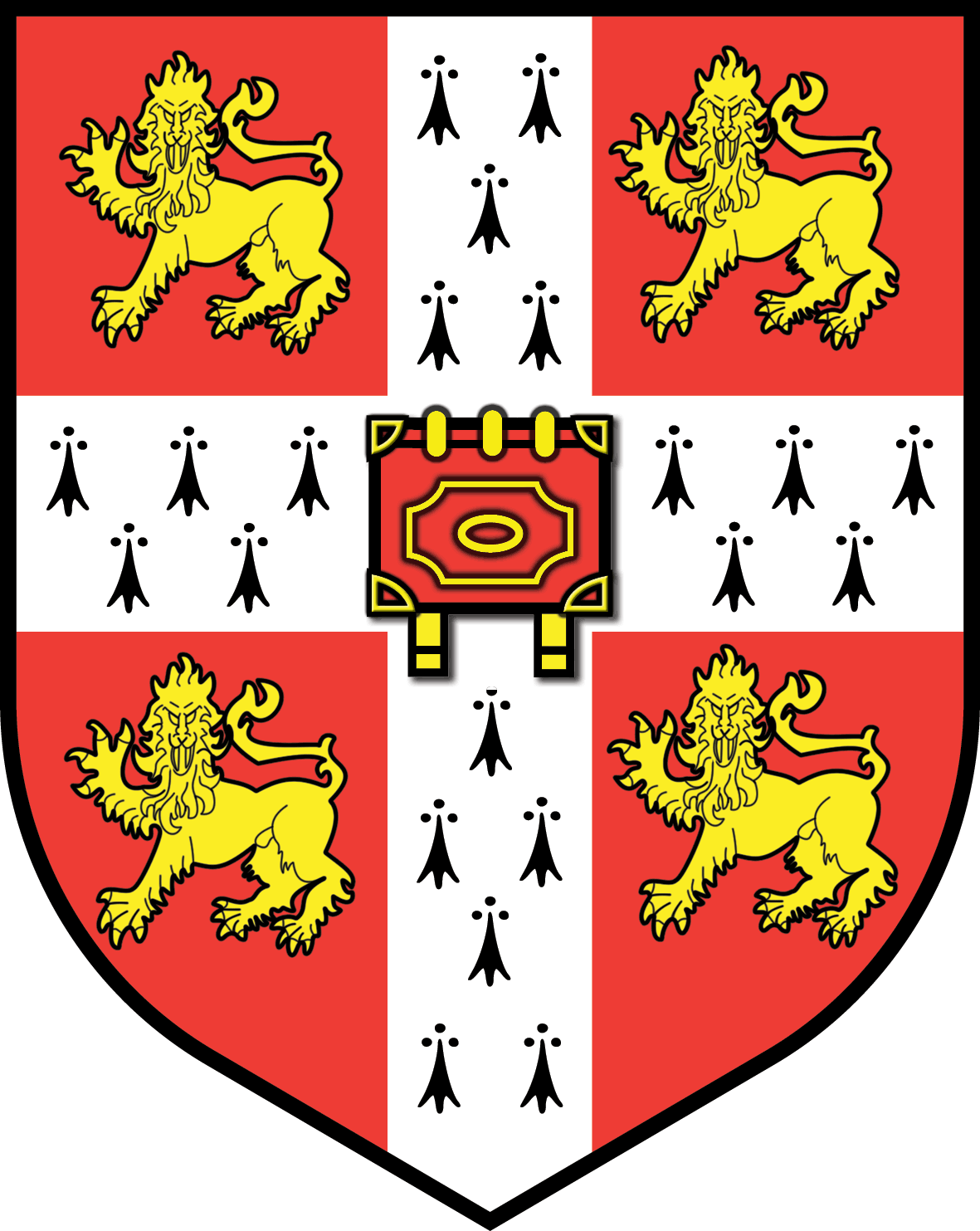}
\vskip 1.2in
{\large Achilleas Kryftis}
\vskip 0.4in
Trinity College \\
and \\
Department of Pure Mathematics and Mathematical Statistics \\
University of Cambridge
\vskip 0.4in
This dissertation is submitted for the degree of \\
\emph{Doctor of Philosophy}
\vskip 0.4in
May 19, 2015
\end{center}

%% file: declaration.tex
\newpage
\mbox{}
\vskip 2in
\begin{quote}
This dissertation is the result of my own work and includes nothing that is the outcome of work done in collaboration except where specifically indicated in the text. No part of this dissertation has been submitted for any other qualification.
\vskip 1in

\hfill Achilleas Kryftis

\hfill May 19, 2015
\end{quote}

%% file: abstract.tex
\newpage
\begin{center}
{\large{\bfseries A constructive approach to affine and projective planes\\}}
\vskip 0.1in
Achilleas Kryftis
\vskip 0.3in
{\sc Abstract}
\vskip 0.1in
\end{center}
\begin{quote}
\begin{singlespace}
In classical geometric algebra, there have been several treatments of affine and projective planes based on fields. In this thesis we approach affine and projective planes from a constructive point of view and we base our geometry on local rings instead of fields.

We start by constructing projective and affine planes over local rings and establishing forms of Desargues' Theorem and Pappus' Theorem which hold for these. From this analysis we derive coherent theories of projective and affine planes. 

The great Greek mathematicians of the classical period used geometry as the basis for their theory of quantities. The modern version of this idea is the reconstruction of algebra from geometry. We show how we can construct a local ring whenever we are given an affine or a projective plane. This enables us to describe the classifying toposes of our theories of affine and projective planes as extensions of the Zariski topos by certain group actions.

Through these descriptions of the classifying toposes, the links between the theories of local rings, affine and projective planes become clear. In particular, the geometric morphisms between these classifying toposes are all induced by group homomorphisms even though they demonstrate complicated constructions in geometry.

In this thesis, we also prove results in topos theory which are applied to these geometric morphisms to give Morita equivalences between some further theories.
\end{singlespace}
\end{quote}

%% file: acknowledgements.tex
\newpage
\mbox{}
\vskip 1in
\begin{center}
{\sc Acknowledgements}
\vskip 0.1in
\end{center}
\begin{singlespace}
I would like thank my supervisor, Professor Martin Hyland, for his constant support and encouragement throughout my PhD and for sharing his wisdom with me.

The category theory group of Cambridge has been a very important part of my life in Cambridge. My mathematical interactions with the group have been very fruitful and we have also shared fun moments in Cambridge and at conferences. I would like to thank Christina, Tamara, Guilherme, Zhen Lin, Sori, Enrico, Paige and all the members of the category theory group.

I would like to thank Trinity College and DPMMS for financially supporting me during my PhD.

I have been very fortune during my PhD to be surrounded by great friends. I would like to thank John, Gabriele, Ugo, Evangelia, Moses, Juhan, Giulio, Eleni, Dimitris, Anastasia, Richard and Anna for being such a great company.

I would also like to thank my parents Georgios and Theano, and my siblings Maria and Yiannos for their constant love and support.
\end{singlespace}

%% file: intro.tex
Algebra and Geometry are two central subjects in mathematics, and the links between them are of fundamental importance. Already, in \emph{The Elements}, Euclid used Geometry to express algebraic properties of the natural numbers. Descartes founded coordinate geometry, the ancestor of modern algebraic geometry. However, it was not until Hilbert's \emph{The Foundations of Geometry} \cite{Hilbert} that systematic connections were made between axiomatic geometry and the algebraic structure.

In this thesis, the particular part of algebra we are interested in is the notion of fields and its constructive cousin: local rings. While developing the theory, we shall also meet symmetry groups and endomorphism monoids.


In the classical approach (as described in \cite{Harts}, \cite{Artin} and \cite{Seidenberg}), the affine plane and the projective plane can be approached from a synthetic and an analytic point of view. The analytic approach involves constucting the affine and the projective planes over a given field. Alternatively, we can think of affine and projective planes synthetically as models of a first order theory with Points and Lines as sorts. The basic foundational result is that the two approaches are essentially the same.


We are interested in a constructive version of these results. In constructive mathematics, models of the geometric theory of fields have decidable equality, while the equality of  quantities with natural geometric significance is not decidable. In particular, the ring of real numbers from a constructive point of view is not necessarily a field but it is a local ring. Specifically, in a topos with a natural number object the ring of Dedekind reals is a local ring but not necessarily a field \cite[D4.7]{Elephant2}. Moreover, \cite[Theorem 4.1]{Kockproj} shows how in constructive logic, the geometric sequents which are true for local rings are exactly the ones intuitionistically derivable from the field axioms given in the same paper. Hence, one can claim that in constructive mathematics, local rings should play the role that fields play in classical mathematics. Therefore, we choose to base our geometry on local rings instead of (geometric) fields.

We present coherent theories of affine and projective planes and construct their classifying toposes. Our theories of affine and projective planes are written in extensions of the languages of the corresponding theories in their classical presentation. In particular, the language is extended by an apartness relation on lines and points and an outside relation between points and lines.








Some papers on affine and projective planes treated constructively are \cite{Heytinggerm}, \cite{Heytingproj}, \cite{Heytingaff}, \cite{Kockproj}, \cite{Dalenoutside}, \cite{Mandcoo}, \cite{Mandcom}, \cite{Mandext}. None of these treatments of affine and projective planes is in terms of coherent (or geometric) theories.  In particular, the classifying toposes of the theories of affine and projective planes have not been constructed elsewhere. Also, apart from \cite{Kockproj} the geometries considered were based on fields and not on local rings as they are here.


There has been some work done on affine and projective planes over local rings in \cite{Kling1}, \cite{Kling2}, \cite{Hj1}, \cite{Hj2}. The theories of Klingenberg affine and projective planes are written in a language with a neighbouring relation on points and lines instead of an apartness relation which is used here. In classical logic, this neighbouring relation would be the complement of the apartness relation of our axiomatization. The theories of Hjelmslev planes have some further axioms. The underlying logic of these papers is not constructive and the axiomatization is again not in coherent logic.

The thesis is divided into two parts. In the first part, we describe the structures of projective and affine planes over given local rings and we present the theories of projective and affine planes. We describe constructions which relate projective planes, affine planes  and local rings.

In Chapter \ref{chaprojplanes}, we construct the projective plane over a local ring. We then describe the coherent theory of preprojective planes which consists of axioms satisfied by projective planes over local rings. This is followed by a discussion on morphisms of preprojective planes. We give a complete description of morphisms between projective planes over local rings. We present Desargues' theorem on the projective plane in a different and more symmetric form than the classical one. The theorem is self dual. The classical presentations of Desargues' theorem rely on the fact that in classical projective planes (over fields) there is always a line passing through two given points which is not always the case for projective planes over local rings.  That made classical versions of Desargues' unsuitable for our theory of projective planes. Furthermore, our version of Desargues' theorem can be stated as a geometric sequent and it is used as an axiom of the (coherent) theory of projective planes. Our version of Pappus' theorem is quite similar to classical presentations of it.

In Chapter \ref{chaaffplanes}, we first construct the affine plane over a local ring and then give a more general construction of a preaffine plane from a preprojective plane with a chosen line. We use the propositions from Chapter \ref{chaprojplanes} to prove several propositions for affine planes over local rings and then we present a coherent theory of preaffine planes. As before, we continue with a discussion on morphisms between preaffine planes followed by a few theorems on how morphisms of preprojective planes with a line interact with the morphims of the induced preaffine planes. We give a complete description of morphisms between affine planes over local rings. Finally, we give Desargues' small and big axioms, and Pappus' axiom on the affine plane and prove them for analytic planes and show some of their consequences.

Chapter \ref{chalocal} gives a construction of a local ring from a synthetic affine plane. We use methods similar to the ones used in  \cite{Artin} and \cite{Harts}. The definition of dilatations is very similar to the one in \cite{Harts}. The definition of translations had to be modified and the proofs of theorems about dilatations and translations were very different. For example, proving that translations are closed under composition requires Desargues' theorem while in \cite{Harts} it is a simple consequence of the definition of translations. These proofs have been very instructive in understanding which versions of Desargues' axioms we would need in our theory of affine planes. The local ring we construct is the ring of trace preserving homorphisms of the group of translations. We demonstrate how this is the coordinate ring of our affine plane. We then revisit the construction of the local ring to show that it is preserved by inverse images of geometric morphisms. We also show that an alternative construction of a local ring gives an isomorphic ring.


In Chapter \ref{chaprojcoo}, we use the results of Chapter \ref{chalocal} to construct a local ring from a given projective plane. We again show that the constructed ring is in a sense the coordinate ring of the projective plane. We also show that any such a ring is unique up to isomorphism.

In the second part of the thesis, we throw light on the constructive theory developed in the first part from the point of view of classifying toposes. In particular, we give more concrete descriptions of the classifying toposes for projective and affine planes.

Sometimes classifying toposes are identified with Grothendieck toposes over the base category of $\mathbf{Sets}$. That enables comparisons with traditional model theory, as for example is done by Olivia Caramello in a series of papers (see for particularly telling instances \cite{Oliviafraisse} and \cite{Oliviafields}). In particular, one can make use of the conceptual completeness for coherent theories (which is essentially the completeness theorem for first order logic). However, though we may often write as if we were working over $\set$, we believe that the arguments of the thesis go through for Grothendieck toposes (that is for any bounded topos) over an arbitrary base topos with a natural number object. In that reading we are frequently arguing in the first part in the internal logic of a topos. We do not draw explicit attention to this.

In Chapter \ref{chaDiac}, we prove a new version of Diaconescu's theorem. While strictly speaking this is not needed for the construction of the classifying toposes of the theories of affine and projective planes, it throws light on what we do later and it can be used in an approach using conceptual completeness on the $\set$-based case. Given an internal category $\mathbb C$ in a topos $\ca S$ (which is a classifying topos over $\set$), Diaconescu's theorem describes what $[\mathbb C, \ca S]$ classifies as an $\ca S$-topos while our version explains what $[\mathbb C,\ca S]$ classifies as a $\set$-topos.

The main goal of Chapter \ref{chaEG} is to prove that when $G$ is a subgroup of $H$ in a topos $\ca E$, then the group homomorphism $G\to H$ induces a local homeomorphism $\ca E[G]\to \ca E[H]$. This result is used in the final chapter to explain how the theories of local rings, affine planes and projective planes interact with each other.

In Chapters \ref{chaZG} and \ref{chaZH}, we construct the classifying toposes of the theories of affine and projective planes, identifying them with extensions of the Zariski topos by certain groups. 


Chapter \ref{chageom} describes the geometric morphisms between the classifying toposes of the theories of local rings, affine planes and projective planes using their descriptions from Chapters \ref{chaZG} and \ref{chaZH}. We use these presentations and general results about toposes to get a better understanding of how these theories are related. For example, Theorem \ref{thrmlocal} applied to the geometric morphisms between the classifying toposes demonstrates Morita equivalences between further theories.


Just to fix terminology, in this thesis by ring we mean a commutative, unital ring.

%% file: thesis.bbl
\begin{thebibliography}{Joh02b}

\bibitem[Art88]{Artin}
E.~Artin.
\newblock {\em {Geometric algebra}}.
\newblock {Wiley Classics Library}. John Wiley \& Sons, Inc., New York, 1988.
\newblock Reprint of the 1957 original, A Wiley-Interscience Publication.

\bibitem[Bac78]{Kling1}
P.~Y. Bacon.
\newblock Desarguesian {K}lingenberg planes.
\newblock {\em Trans. Amer. Math. Soc.}, 241:343--355, 1978.

\bibitem[BL10]{Kling2}
T.~Bisztriczky and J.~W. Lorimer.
\newblock Translations in affine {K}lingenberg spaces.
\newblock {\em J. Geom.}, 99(1-2):15--42, 2010.

\bibitem[Car14]{Oliviafraisse}
Olivia Caramello.
\newblock Fra\"\i ss\'e's construction from a topos-theoretic perspective.
\newblock {\em Log. Univers.}, 8(2):261--281, 2014.

\bibitem[CJ09]{Oliviafields}
Olivia Caramello and Peter Johnstone.
\newblock De {M}organ's law and the theory of fields.
\newblock {\em Adv. Math.}, 222(6):2145--2152, 2009.

\bibitem[Dra68]{Hj2}
David~A. Drake.
\newblock Projective extensions of uniform affine {H}jelmslev planes.
\newblock {\em Math. Z.}, 105:196--207, 1968.

\bibitem[EM65]{Adjointtriples}
Samuel Eilenberg and John~C. Moore.
\newblock Adjoint functors and triples.
\newblock {\em Illinois Journal of Mathematics}, 9(3):381--398, 09 1965.

\bibitem[Har67]{Harts}
Robin Hartshorne.
\newblock {\em {Foundations of projective geometry}}, volume 1966/67 of {\em
  {Lecture Notes, Harvard University}}.
\newblock W. A. Benjamin, Inc., New York, 1967.

\bibitem[Hey28]{Heytinggerm}
A.~Heyting.
\newblock Zur intuitionistischen {A}xiomatik der projektiven {G}eometrie.
\newblock {\em Math. Ann.}, 98(1):491--538, 1928.

\bibitem[Hey59]{Heytingaff}
A.~Heyting.
\newblock {Axioms for intuitionistic plane affine geometry}.
\newblock In {\em {The axiomatic method. {W}ith special reference to geometry
  and physics. {P}roceedings of an {I}nternational {S}ymposium held at the
  {U}niv. of {C}alif., {B}erkeley, {D}ec. 26, 1957-{J}an\ 4, 1958 (edited by
  {L}. {H}enkin, {P}. {S}uppes and {A}. {T}arski)}}, {Studies in Logic and the
  Foundations of Mathematics}, pages 160--173. North-Holland Publishing Co.,
  Amsterdam, 1959.

\bibitem[Hey80]{Heytingproj}
A.~Heyting.
\newblock {\em {Axiomatic projective geometry}}.
\newblock {Bibliotheca Mathematica [Mathematics Library], V}. Wolters-Noordhoff
  Scientific Publications, Ltd., Groningen; North-Holland Publishing Co.,
  Amsterdam-New York, second edition, 1980.

\bibitem[Hil59]{Hilbert}
David Hilbert.
\newblock {\em {The foundations of geometry}}.
\newblock {Authorized translation by E. J. Townsend. Reprint edition}. The Open
  Court Publishing Co., La Salle, Ill., 1959.

\bibitem[Joh77]{PTJtopos}
P.~T. Johnstone.
\newblock {\em {Topos theory}}.
\newblock Academic Press [Harcourt Brace Jovanovich, Publishers], London-New
  York, 1977.
\newblock London Mathematical Society Monographs, Vol. 10.

\bibitem[Joh02a]{Elephant1}
Peter~T. Johnstone.
\newblock {\em {Sketches of an elephant: a topos theory compendium. {V}ol. 1}},
  volume~43 of {\em {Oxford Logic Guides}}.
\newblock The Clarendon Press Oxford University Press, New York, 2002.

\bibitem[Joh02b]{Elephant2}
Peter~T. Johnstone.
\newblock {\em {Sketches of an elephant: a topos theory compendium. {V}ol. 2}},
  volume~44 of {\em {Oxford Logic Guides}}.
\newblock The Clarendon Press Oxford University Press, Oxford, 2002.

\bibitem[Koc77]{Kockproj}
Anders Kock.
\newblock Universal projective geometry via topos theory.
\newblock {\em J. Pure Appl. Algebra}, 9(1):1--24, 1976/77.

\bibitem[Kre91]{Hj1}
Alexander Kreuzer.
\newblock A system of axioms for projective {H}jelmslev spaces.
\newblock {\em J. Geom.}, 40(1-2):125--147, 1991.

\bibitem[Law66]{Lawverecategories}
F.~William Lawvere.
\newblock The category of categories as a foundation for mathematics.
\newblock In {\em Proc. {C}onf. {C}ategorical {A}lgebra ({L}a {J}olla,
  {C}alif., 1965)}, pages 1--20. Springer, New York, 1966.

\bibitem[Man07]{Mandcoo}
Mark Mandelkern.
\newblock Constructive coordinatization of {D}esarguesian planes.
\newblock {\em Beitr\"age Algebra Geom.}, 48(2):547--589, 2007.

\bibitem[Man13]{Mandcom}
Mark Mandelkern.
\newblock The common point problem in constructive projective geometry.
\newblock {\em Indag. Math. (N.S.)}, 24(1):111--114, 2013.

\bibitem[Man14]{Mandext}
Mark Mandelkern.
\newblock Constructive projective extension of an incidence plane.
\newblock {\em Trans. Amer. Math. Soc.}, 366(2):691--706, 2014.

\bibitem[Mar02]{Marker}
David Marker.
\newblock {\em Model theory}, volume 217 of {\em Graduate Texts in
  Mathematics}.
\newblock Springer-Verlag, New York, 2002.
\newblock An introduction.

\bibitem[MM94]{Sheaves}
Saunders {Mac Lane} and Ieke Moerdijk.
\newblock {\em {Sheaves in geometry and logic}}.
\newblock {Universitext}. Springer-Verlag, New York, 1994.
\newblock A first introduction to topos theory, Corrected reprint of the 1992
  edition.

\bibitem[Sei12]{Seidenberg}
A.~Seidenberg.
\newblock {\em {Lectures in Projective Geometry}}.
\newblock {Dover books on mathematics}. Dover Publications, 2012.

\bibitem[vD96]{Dalenoutside}
Dirk van Dalen.
\newblock ``{O}utside'' as a primitive notion in constructive projective
  geometry.
\newblock {\em Geom. Dedicata}, 60(1):107--111, 1996.

\end{thebibliography}
